\documentclass{article}
\usepackage{longtable}
\usepackage{float}
\usepackage{siunitx}
\usepackage{enumitem}
\usepackage{pifont}

\usepackage[most]{tcolorbox}
\usepackage{amsmath} 
\DeclareMathOperator{\supp}{supp}
\usepackage{bm}       
\numberwithin{equation}{section} 
\usepackage{amsfonts} 
\usepackage{amsthm}
\usepackage{amsmath}
\usepackage{siunitx}

\usepackage{empheq} 
\usepackage{authblk}
\usepackage{booktabs, tabularx, array, makecell}
\newcolumntype{Y}{>{\raggedright\arraybackslash}X}
\usepackage{graphicx} 
\usepackage[utf8]{inputenc} 
\usepackage{caption} 
\usepackage{subcaption} 
\usepackage{tikz} 
\usepackage{bm} 
\usepackage{graphics}
\usepackage{enumitem} 
\usepackage{geometry} 
\usepackage{amssymb} 
\usepackage{physics} 
\usepackage{booktabs} 
\usepackage{fancyhdr}
\fancypagestyle{titlepagefooter}{ 
	\fancyhf{} 
	\fancyfoot[C]{Email: \texttt{lillian.kleess@baruchmail.cuny.edu}} 
}
\newtheoremstyle{boldplain}
{6pt}{6pt}
{\itshape}
{}
{\bfseries}
{.}
{.5em}
{}

\newtheoremstyle{bolddef}
{6pt}{6pt}
{\normalfont}
{}
{\bfseries}
{.}
{.5em}
{}

\newtheoremstyle{boldremark}
{6pt}{6pt}
{\normalfont}
{}
{\bfseries}
{.}
{.5em}
{\thmname{#1}\thmnumber{ #2}\thmnote{ \textnormal{(#3)}}}

\theoremstyle{boldplain}
\newtheorem{theorem}{Theorem}[section]
\newtheorem{proposition}{Proposition}[section]
\newtheorem{corollary}{Corollary}[section]
\newtheorem{lemma}{Lemma}[section]

\theoremstyle{bolddef}
\newtheorem{definition}{Definition}[section]

\theoremstyle{boldremark}
\newtheorem{remark}{Remark}[section]
\graphicspath{{figures/}}
\title{Phase-Textured Complex Viscosity in Linear Viscous Flows:\\[0.6ex]
	{\large\itshape Non-Normality Without Advection, Corner Defects, and 3D Mode Coupling}}

\author{L.S. Kleess}
\affil{Baruch College, Department of Natural Sciences, Manhattan, USA}
\date{January 2026}
\begin{document}
	
	\maketitle

	\begin{abstract}
		We develop a mathematically consistent framework for oscillatory incompressible flow with a \emph{complex, spatially heterogeneous viscosity field}
		$\mu^*(\mathbf{x},\omega)$ at a fixed forcing frequency $\omega>0$, with emphasis on \emph{constitutive phase textures}
		$\varphi(\mathbf{x})=\arg\mu^*(\mathbf{x},\omega)$ and their measurable consequences. A principal novelty appears already at the linear,
		frequency-domain level in genuinely three-dimensional settings: when the spanwise direction is periodic and $\mu^*$ varies in $z$, multiplication by
		$\mu^*$ induces convolution in spanwise Fourier index, yielding an operator-valued Toeplitz/Laurent coupling of modes. As a result, even
		spanwise-uniform forcing generically produces $\kappa\neq 0$ response (sidebands and patterning) as a \emph{linear constitutive} effect, without any
		nonlinear transfer mechanism.
		
		Rather than inserting a complex coefficient into the Laplacian ad hoc, we place $\mu^*$ at the closure level where it is both physically and
		mathematically canonical:
		$\hat{\bm{\tau}}(\mathbf{x};\omega)=2\,\mu^*(\mathbf{x},\omega)\mathbf{D}(\hat{\mathbf{v}})$,
		the harmonic linear-response representation of a causal stress memory kernel. Under a passivity condition
		$\Re\mu^*(\mathbf{x},\omega)\ge \mu_{\min}>0$, we establish well-posedness for oscillatory Stokes and Oseen-type problems on bounded Lipschitz domains
		(including polygonal/cornered geometries) via coercive sectorial forms. We then adopt a resolvent-centered operator viewpoint appropriate for the
		resulting non-selfadjoint viscous cores (compact resolvent on bounded truncations, discrete spectrum, numerical range bounds, and pseudospectral
		stability). This lens isolates a mechanism absent from constant-viscosity and magnitude-only variable-viscosity models: \emph{spatial variation of
			$\arg\mu^*$ renders the viscous operator intrinsically non-normal even prior to advection}, so eigenvalues alone are not predictive and large,
		frequency-selective resolvent gains can occur in linear regimes.
		
		Complimenting the operator picture, for Tier~II regularity $\mu^*\in W^{1,\infty}$ we derive an exact vorticity decomposition in which the viscous curl
		splits into a diffusion-like contribution plus an explicit \emph{texture-gradient commutator} that vanishes when $\nabla\mu^*\equiv 0$ and is
		controlled by $\|\nabla\mu^*\|_{L^\infty}\|\mathbf{D}(\hat{\mathbf{v}})\|_{L^2}$. In the pure-phase class
		$\mu^*(\mathbf{x},\omega)=\mu_0(\omega)e^{i\varphi(\mathbf{x})}$, this reduces to a single physically interpretable control parameter
		$\mu_0(\omega)\|\nabla\varphi\|_{L^\infty}$, providing a rigorous pathway by which constitutive phase gradients generate and localize
		phase-sensitive vortical response in oscillatory flow. A phase-compensated change of variables removes the leading complex factor and exposes an
		unavoidable drift-like first-order coupling proportional to $\nabla\varphi$, clarifying which phase effects persist after compensation.
		
		We anchor the analysis to classical benchmark flows (oscillatory pressure-driven channel/pipe flow, Stokes' second problem, and backward-facing-step
		(BFS) geometries) and identify quantitative signatures that isolate the constitutive novelty, including resolvent gain maps $G(\omega,\kappa)$, sideband
		energy ratios, corner-local strain/enstrophy diagnostics, and impedance or traction phase (e.g.\ $\arg Z^*(\omega)$). Finally, we propose a
		reproducible viscosity-texture library and a computational protocol consistent with the theory (stable saddle-point treatment and matrix-free
		singular-value/resolvent computations).
	\end{abstract}
	\newpage
	    \thispagestyle{titlepagefooter} 
	\tableofcontents
\newpage

\section{Motivation and Modeling Principle: Complex Viscosity as a Spatially Resolved Constitutive Field.}
\label{subsec:motivation}
\subsection{Introduction and Summary of Proposed Contributions.}

Viscous stresses are the mechanism by which continuum models encode microscopic dissipation and, in many materials, microscopic \emph{memory}.
In simple Newtonian liquids, stress responds essentially instantaneously to strain rate, so a single real viscosity coefficient provides an accurate
closure in broad regimes \cite{Temam2001,Galdi2011}.
In contrast, many fluids of contemporary interest (polymer solutions and melts, suspensions and emulsions, biofluids and
mucus-like gels, and engineered structured media) exhibit a measurable phase lag between imposed strain-rate and stress response under oscillatory
forcing \cite{BirdArmstrongHassager1987,Ferry1980,Larson1999,BarnesHuttonWalters1989,Morrison2001}.
In the linear response regime, that lag is not an \emph{add-on} to the momentum balance: it is precisely what a frequency-domain
constitutive coefficient is designed to represent \cite{BirdArmstrongHassager1987,Ferry1980,Larson1999,HackleyFerraris2001}.

A useful intuition is to view viscosity not merely as "how much damping" is present, but as a \emph{local constitutive map} from deformation rate
to stress \cite{BirdArmstrongHassager1987,Larson1999}.
In a time-harmonic experiment at fixed angular frequency $\omega$, that map is naturally complex-valued: the real part quantifies the
in-phase stress component responsible for cycle-averaged dissipation, while the imaginary part quantifies the out-of-phase (reactive) component
associated with internal relaxation and storage \cite{Ferry1980,Larson1999,Morrison2001,HackleyFerraris2001}.
When the medium is spatially structured, the relaxation is likewise spatial: different locations
can respond with different stress-lag to the same imposed oscillation \cite{BirdArmstrongHassager1987,Larson1999}.
In polar form,
\[
\mu^*(\mathbf{x},\omega)=|\mu^*(\mathbf{x},\omega)|\,e^{i\varphi(\mathbf{x},\omega)},
\]
one may regard $\varphi(\mathbf{x},\omega)$ as a spatial "map of local stress-lag": a field of microscopic phase offsets that need not be
synchronized across the domain \cite{Ferry1980,Larson1999}.
The central theme of this paper is that \emph{spatial variation of this constitutive phase} is not innocuous. Even
when $|\mu^*|$ is held fixed, gradients of $\varphi$ alter the momentum operator and can induce frequency-selective amplification, phase sensitivity,
and localized vortical response through strictly linear mechanisms \cite{TrefethenEmbree2005,SchmidHenningson2001,Schmid2007}.

Spatially heterogeneous viscoelastic response is natural in structured fluids (microstructure and concentration gradients), in biological settings
(heterogeneous composition and hydration), and in effective-medium descriptions (engineered metamaterials and wall-adjacent layers) \cite{Larson1999,BirdArmstrongHassager1987}.
Once $\mu^*$
varies in space, the viscous term is no longer a scalar multiple of a Laplacian: the variable-coefficient divergence-form operator must be retained
\cite{Evans2010,GiraultRaviart1986}.
The incompressible momentum balance (per unit volume) is
\begin{equation}
	\rho\big(\partial_t \mathbf{v}+\mathbf{v}\cdot\nabla \mathbf{v}\big)
	=
	-\nabla p + \nabla\cdot\bm{\tau} + \mathbf{f},
	\qquad
	\nabla\cdot\mathbf{v}=0,
	\label{eq:intro_NS_balance}
\end{equation}
and viscosity enters \emph{only} through the constitutive specification of $\bm{\tau}$ \cite{Temam2001,Galdi2011}.
Under harmonic forcing at frequency $\omega$, substituting
\eqref{eq:intro_constitutive} into \eqref{eq:intro_NS_balance} yields oscillatory Stokes or Oseen balances featuring the variable-coefficient
divergence-form operator
\[
\nabla\cdot\big(2\mu^*(\mathbf{x},\omega)\mathbf{D}(\cdot)\big),
\qquad
\mathbf{D}(\mathbf{v})=\tfrac12(\nabla\mathbf{v}+(\nabla\mathbf{v})^{\mathsf T}),
\]
as standard in incompressible viscous formulations \cite{GiraultRaviart1986,BoffiBrezziFortin2013}.
This operator contains $\nabla\mu^*$ couplings and, in general, cannot be replaced by $\mu^*(\mathbf{x},\omega)\Delta(\cdot)$ without additional
hypotheses that effectively suppress those couplings (for example, constant coefficients, or restrictive symmetry/regularity assumptions)
\cite{Evans2010}.
Mechanically, these couplings are the continuum manifestation of the fact that different parts of a spatially structured medium "relax out of
phase" and therefore exchange stress with their neighbors in a way that cannot be captured by any spatially uniform complex coefficient
\cite{Larson1999,Ferry1980}.
Mathematically, they are the structural origin of the two main mechanisms developed in this paper:
\begin{enumerate}
	\item \emph{Intrinsic non-selfadjointness and non-normality of the viscous core} driven by phase textures \cite{TrefethenEmbree2005,GustafsonRao1997}.
	\item An explicit \emph{texture-gradient commutator} that acts as a distributed vorticity source in oscillatory flow \cite{MajdaBertozzi2002}.
\end{enumerate}

A standard linear viscoelastic closure expresses the deviatoric stress as a causal memory law: present stress depends on past strain-rate through a
relaxation kernel \cite{BirdArmstrongHassager1987,Ferry1980,Larson1999}.
In an Eulerian description, one may write
\begin{equation}
	\bm{\tau}(\mathbf{x},t)
	=
	2\int_{0}^{\infty}\mu(\mathbf{x},s)\,\mathbf{D}(\mathbf{v})(\mathbf{x},t-s)\,ds,
	\label{eq:intro_memory_law}
\end{equation}
where $s\ge 0$ is lag time and $\mu(\mathbf{x},s)$ is a relaxation kernel \cite{Ferry1980,Larson1999}.
For passive materials one expects $\mu(\mathbf{x},s)\ge 0$ (in the
appropriate operator/kernel sense) and sufficient integrability so that the Laplace--Fourier transform is well-defined \cite{Ferry1980,Larson1999}.
For time-harmonic velocities
$\mathbf{v}(\mathbf{x},t)=\Re\{\hat{\mathbf{v}}(\mathbf{x};\omega)e^{i\omega t}\}$, the convolution \eqref{eq:intro_memory_law} reduces to the
frequency-domain closure
\begin{equation}
	\hat{\bm{\tau}}(\mathbf{x};\omega)
	=
	2\,\mu^*(\mathbf{x},\omega)\,\mathbf{D}(\hat{\mathbf{v}}(\mathbf{x};\omega)),
	\qquad
	\mu^*(\mathbf{x},\omega)
	:=
	\int_0^\infty \mu(\mathbf{x},s)\,e^{-i\omega s}\,ds,
	\label{eq:intro_constitutive}
\end{equation}
as in standard small-amplitude oscillatory linear response \cite{BirdArmstrongHassager1987,Ferry1980,Larson1999,HackleyFerraris2001}.
Thus, $\mu^*(\cdot,\omega)$ is the harmonic linear-response coefficient mapping strain-rate amplitude to stress amplitude at frequency $\omega$
\cite{Ferry1980,Morrison2001}.
It is convenient to separate dissipative and reactive components by defining
\[
\mu'(\mathbf{x},\omega):=\Re\mu^*(\mathbf{x},\omega),
\qquad
\mu"(\mathbf{x},\omega):=-\Im\mu^*(\mathbf{x},\omega),
\]
so that $\mu^*=\mu'-i\mu"$ under the convention \eqref{eq:intro_constitutive} \cite{HackleyFerraris2001,Larson1999}.
In this notation, $\mu'$ controls the in-phase stress component
(responsible for cycle-averaged dissipation), while $\mu"$ controls the quadrature component (reactive/storage-like exchange at frequency $\omega$)
\cite{Ferry1980,Larson1999,Morrison2001}.
Equivalently,
\[
\mu^*(\mathbf{x},\omega)=|\mu^*(\mathbf{x},\omega)|\,e^{i\varphi(\mathbf{x},\omega)},
\qquad
\varphi(\mathbf{x},\omega)=\arg\mu^*(\mathbf{x},\omega),
\]
and \emph{phase texture} refers to spatial variation of $\varphi(\cdot,\omega)$.

We adopt passivity (uniform positive dissipation) as the minimal physically meaningful assumption at the forcing frequency
\cite{Ferry1980,Larson1999}.
The central structural requirement is
\begin{equation}
	\mu^*(\cdot,\omega)\in L^\infty(\Omega;\mathbb{C}),
	\qquad
	\Re\mu^*(\mathbf{x},\omega)\ge \mu_{\min}>0\quad \text{for a.e.\ $\mathbf{x}\in\Omega$,}
	\label{eq:intro_passivity}
\end{equation}
together with density bounds $0<\rho_{\min}\le\rho\le\rho_{\max}$ a.e.
Condition \eqref{eq:intro_passivity} is physically transparent: it encodes
nonnegative cycle-averaged viscous dissipation and rules out frequency-local active behavior \cite{Ferry1980,Larson1999}.
Mathematically, it implies coercivity of the real part
of the viscous form (via Korn and Poincar\'e), which yields existence, uniqueness, and stability for oscillatory Stokes problems and for Oseen-type
linearizations on bounded Lipschitz domains, even in the presence of spatially heterogeneous complex coefficients
\cite{Horgan1995,Ciarlet1988,GiraultRaviart1986,BoffiBrezziFortin2013,Kato1995}.
Importantly, passivity does
\emph{not} force the operator to be symmetric: the imaginary part of $\mu^*$ and, more sharply, spatial variation of $\arg\mu^*$ can render the
viscous core non-selfadjoint and non-normal while remaining strictly dissipative in the cycle-averaged sense \cite{TrefethenEmbree2005,GustafsonRao1997}.
This decoupling (dissipative
admissibility versus operator normality) is central to the amplification and phase-sensitivity phenomena we isolate
\cite{TrefethenEmbree2005,Schmid2007}.

\medskip
\noindent The key modeling principle, therefore, is the following: \emph{at fixed $\omega$, treat $\mu^*(\mathbf{x},\omega)$ as a given resolved
	field subject to passivity, and analyze the resulting oscillatory operator in its correct divergence-of-symmetric-gradient form}
\cite{Evans2010,GiraultRaviart1986,BoffiBrezziFortin2013}.
This choice preserves the constitutive meaning of complex viscosity as a linear-response coefficient \cite{Ferry1980,Larson1999},
retains the commutator structures generated by spatial heterogeneity \cite{Evans2010},
and makes it possible to connect mechanics to operator theory within a unified, frequency-domain framework
\cite{Kato1995,TrefethenEmbree2005}.

\subsection{Constitutive Complex Viscosity Textures as an Operator-Level Source of Non-Normality.}

As previously stated, a central premise of this work is that, at a fixed forcing frequency $\omega>0$, the effective viscosity is not merely a scalar
parameter but a \emph{spatially resolved complex field} $\mu^*(\mathbf{x},\omega)$ arising from linear viscoelastic response
\cite{BirdArmstrongHassager1987,Ferry1980,Larson1999,BarnesHuttonWalters1989,Morrison2001}.
The theoretical novelty begins once the \emph{constitutive phase}
\begin{equation}
	\varphi(\mathbf{x},\omega):=\arg\mu^*(\mathbf{x},\omega)
\end{equation}
is permitted to vary in space.

In classical constant-viscosity or constant-phase settings, the viscous term can be interpreted (up to an immaterial global complex rotation) as the
realization of a symmetric, coercive form, with the unsteady mass contribution entering as the purely imaginary shift $i\omega\rho$
\cite{Kato1995,Temam2001,Galdi2011,Horgan1995,Ciarlet1988}.
In that regime, spectral location and standard energy estimates already capture most of what can occur at the harmonic level
\cite{Temam2001,Galdi2011}.
By contrast, if $\varphi(\cdot,\omega)$ varies spatially, the viscous core
\begin{equation}
	\mathbf{u}\ \mapsto\ -\nabla\cdot\big(2\mu^*(\cdot,\omega)\,\mathbf{D}(\mathbf{u})\big)
\end{equation}
becomes generically \emph{non-selfadjoint and non-normal} even before advection is introduced; see, e.g., general discussions of non-normality,
resolvent growth, and pseudospectra in operator and hydrodynamic settings \cite{TrefethenEmbree2005,GustafsonRao1997,SchmidHenningson2001,Schmid2007}.
The underlying reason is structural: multiplication by $e^{i\varphi(\mathbf{x},\omega)}$ does not commute with differentiation, so the
variable-coefficient viscous operator contains irreducible commutator terms that are absent from constant-phase models
(cf.\ variable-coefficient divergence-form operators in standard PDE treatments \cite{Evans2010}).

A concrete implication follows immediately: harmonic receptivity is no longer governed by eigenvalues alone.
Large gains may occur even when the spectrum is well separated from the origin, because amplification is controlled by the resolvent norm
(equivalently, by pseudospectral proximity and numerical-range geometry) rather than by spectral abscissa
\cite{TrefethenEmbree2005,GustafsonRao1997,SchmidHenningson2001,Schmid2007}.
The same constitutive mechanism has a sharp fluid-mechanical manifestation once one tracks vorticity rather than velocity
\cite{MajdaBertozzi2002}.

Assuming Tier~II coefficient regularity $\mu^*(\cdot,\omega)\in W^{1,\infty}(\Omega;\mathbb{C})$, the variable-coefficient viscous operator admits a
gradient-resolved decomposition: the familiar diffusion-like contribution is accompanied by an \emph{explicit commutator} involving $\nabla\mu^*$.
Taking the curl of the harmonic momentum balance eliminates pressure and yields a vorticity identity of the form
\begin{equation}
	i\omega\rho\,\hat{\boldsymbol{\omega}}
	=
	\mu^*\,\Delta\hat{\boldsymbol{\omega}}
	+
	\mathcal{G}_{\mu^*}[\hat{\mathbf{v}}]
	+
	\nabla\times\hat{\mathbf{f}},
	\qquad
	\hat{\boldsymbol{\omega}}:=\nabla\times\hat{\mathbf{v}},
\end{equation}
where $\mathcal{G}_{\mu^*}$ is linear in $\nabla\mu^*$ and first derivatives of $\hat{\mathbf{v}}$
(cf.\ distributional manipulations and curl identities in PDE/fluids references \cite{Evans2010,MajdaBertozzi2002}).
This term vanishes identically when $\nabla\mu^*\equiv 0$, but for textured viscosity it acts as a \emph{distributed, linear vorticity source} in
oscillatory flow.
The associated estimates are quantitative and transparent:
\begin{equation}
	\|\mathcal{G}_{\mu^*}[\hat{\mathbf{v}}]\|_{H^{-1}(\Omega)}
	\ \lesssim\
	\|\nabla\mu^*(\cdot,\omega)\|_{L^\infty(\Omega)}\,\|\mathbf{D}(\hat{\mathbf{v}})\|_{L^2(\Omega)}.
\end{equation}
Thus, texture gradients inject vorticity in direct proportion to a single coefficient scale $\|\nabla\mu^*\|_{L^\infty}$ and the natural strain-energy
scale $\|\mathbf{D}(\hat{\mathbf{v}})\|_{L^2}$.
In cornered and separated-flow geometries (backward-facing steps, L-bends, cavities), strain and vorticity already concentrate in singular layers and
shear layers; the commutator forcing therefore becomes a mechanically legible route to localized pre-turbulent vortical structure, strongest precisely
where the flow is already predisposed to large shear \cite{ArmalyDurstPereiraSchonung1983,KaiktsisKarniadakisOrszag1996,BarkleyGomesHenderson2002}.

\medskip
\noindent A particularly revealing specialization is the \emph{phase-only} texture class
\begin{equation}
	\mu^*(\mathbf{x},\omega)=\mu_0(\omega)\,e^{i\varphi(\mathbf{x},\omega)},
	\qquad
	\mu_0(\omega)>0\ \text{constant in }\mathbf{x},
\end{equation}
which is the most direct way to isolate phase-texture effects from magnitude-only heterogeneity in dissipation
\cite{BirdArmstrongHassager1987,Ferry1980,Larson1999,BarnesHuttonWalters1989}.
In this class,
\begin{equation}
	\nabla\mu^*(\mathbf{x},\omega)= i\,\mu^*(\mathbf{x},\omega)\,\nabla\varphi(\mathbf{x},\omega),
\end{equation}
so the texture strength reduces to a single, physically interpretable control knob
$\mu_0(\omega)\|\nabla\varphi(\cdot,\omega)\|_{L^\infty(\Omega)}$ (or, in dimensionless form,
$\Pi_\varphi(\omega)=L\|\nabla\varphi\|_{L^\infty}$ for a chosen geometric length scale $L$).
This provides a clean parametric axis for numerical and analytical studies: one can vary $\|\nabla\varphi\|$ while keeping $|\mu^*|$ fixed, thereby
ruling out explanations in terms of spatially varying dissipation magnitude.

Moreover, the phase-only class admits a principled \emph{phase-compensation} rewrite
\begin{equation}
	\hat{\mathbf{v}}=e^{-i\varphi}\hat{\mathbf{w}}.
\end{equation}
In the compensated variable $\hat{\mathbf{w}}$, the leading dissipative part of the Stokes form becomes real and symmetric (up to lower-order terms),
while the entire effect of phase texture survives as an explicit first-order coupling proportional to $\nabla\varphi$; see, e.g., standard discussions
of perturbations of coercive forms and lower-order couplings \cite{Kato1995}.

Here we identify what is truly irreducible (phase-gradient couplings) and what can be
interpreted as a removable global rotation (constant phase). It is important to note that texture-driven commutator forcing does not act in isolation: it is amplified by the same non-normal operator geometry that
spatial phase variation creates.
When the vorticity identity is combined with resolvent bounds for the oscillatory Stokes/Oseen operator, one obtains a \emph{square-resolvent}
amplification pathway: the texture gradient produces distributed vorticity injection proportional to $\|\nabla\mu^*\|$ (or $\|\nabla\varphi\|$ in the
phase-only class), and the non-normal resolvent amplifies and localizes the resulting response at selected frequencies
\cite{TrefethenEmbree2005,Schmid2007,McKeonSharma2010}.
Mechanically, this formalizes a two-stage receptivity picture:
\begin{enumerate}
	\item \emph{Injection}: viscosity gradients create a linear, spatially distributed vorticity source tied to local strain,
	\item \emph{Selection/Amplification}: the non-normal resolvent concentrates and amplifies that injected response in frequency and space.
\end{enumerate}
This explains, in a quantitative and testable way, why apparently modest phase defects (especially when supported near corners or within shear layers)
can produce disproportionate changes in amplitude and phase, changes that are invisible to eigenvalue-only diagnostics and that persist even before
inertial nonlinearity is restored.

\begin{figure}[h]
\centering
\includegraphics[width=\linewidth]{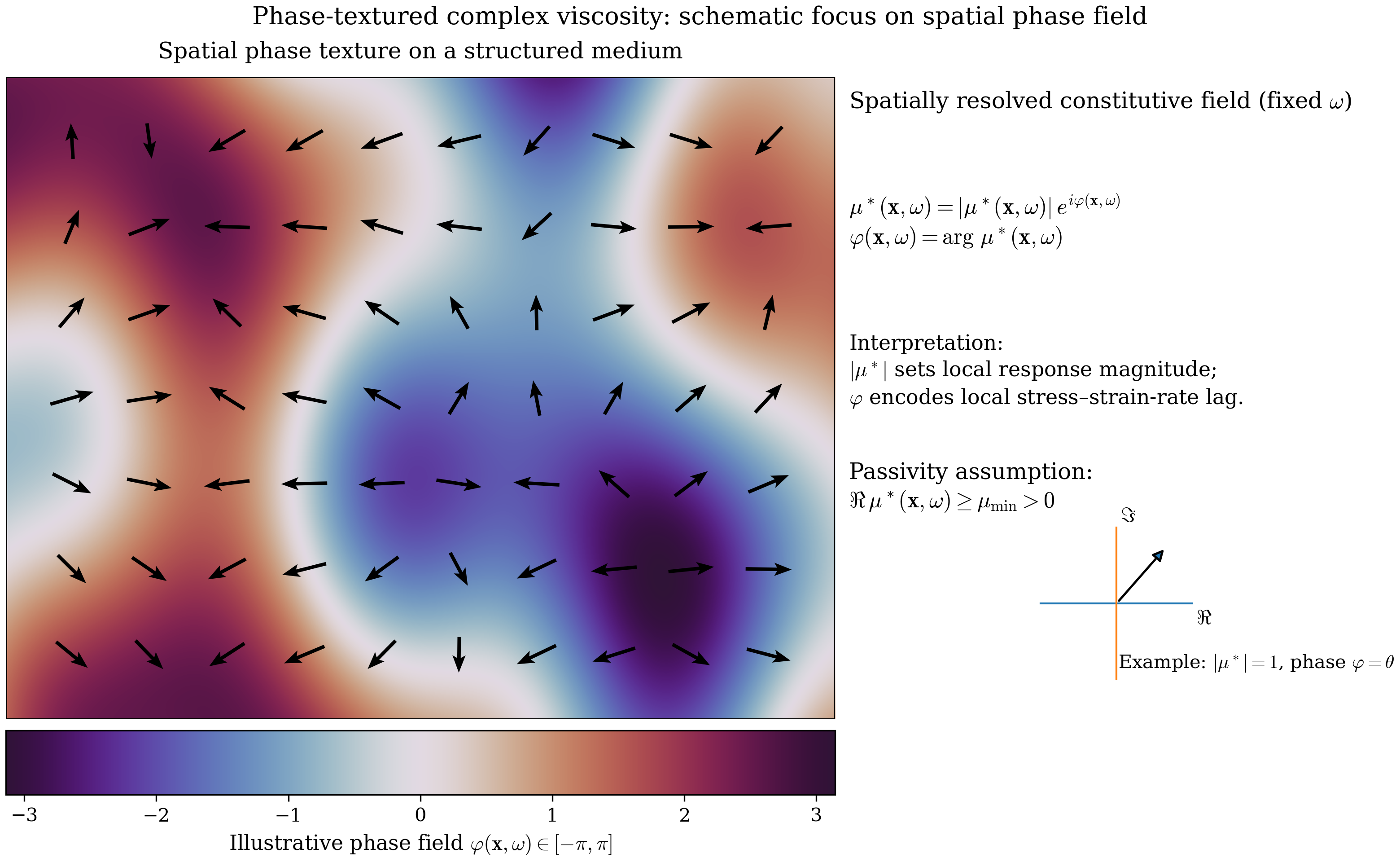}
\caption{\textbf{Schematic of a constitutive phase texture.}
At a fixed forcing frequency $\omega$, the complex viscosity field is written in polar form
$\mu^*(\mathbf{x},\omega)=|\mu^*(\mathbf{x},\omega)|e^{i\varphi(\mathbf{x},\omega)}$.
The colormap encodes the spatially varying constitutive phase $\varphi(\mathbf{x},\omega)=\arg\mu^*(\mathbf{x},\omega)$ (a local stress--strain-rate lag),
while the arrow glyphs indicate the associated phase-gradient direction (and relative magnitude), emphasizing the \emph{texture-gradient} structure that
enters the divergence-form viscous operator through $\nabla\mu^*= i\,\mu^*\,\nabla\varphi$ in the phase-only class.
The inset provides the local phasor convention: the in-phase component $\mu'=\Re\mu^*$ (dissipative) and the quadrature component
$\mu''=-\Im\mu^*$ (reactive) under the $e^{i\omega t}$ time dependence.
The figure is qualitative and intended to visualize how spatial variation of $\varphi$ generates first-order, gradient-controlled couplings even when
$|\mu^*|$ is uniform.}
\label{fig:phase_texture_schematic}
\end{figure}

\newpage

\subsection{A Genuinely Three-Dimensional Linear Mechanism: Toeplitz/Laurent Mode Coupling from $z$-Dependent Textures.}
\label{subsec:toeplitz_mechanism_intro}

\paragraph{Fourier algebra: multipliers become convolutions.}
A further consequence appears uniquely in three-dimensional periodic settings, and it is best understood as an operator-theoretic statement
about how \emph{variable coefficients} interact with Fourier representations. Suppose $\Omega=\Omega_{2D}\times(0,L_z)$ with $z$-periodic boundary
conditions, and consider the harmonic oscillatory Stokes/Oseen operator at a fixed forcing frequency $\omega>0$ with viscosity field
$\mu^*(x,y,z;\omega)$. When coefficients are independent of $z$, the spanwise Fourier transform in $z$ diagonalizes the operator:
each spanwise Fourier mode evolves independently, yielding a family of decoupled two-dimensional resolvent problems indexed by the spanwise
wavenumber $\kappa$. This classical decoupling is the structural reason that stability/resolvent analyses in spanwise-periodic channels and
backward-facing-step (BFS) truncations are typically performed ``mode by mode'' \cite{SchmidHenningson2001,Schmid2007}.

The present mechanism is the minimal way to \emph{break} that decoupling without invoking any nonlinear energy transfer.
If $\mu^*$ depends on $z$ and is $L_z$-periodic, then multiplication by $\mu^*$ in physical space becomes \emph{convolution} in Fourier space.
Because the viscous term appears in the correct constitutive form $\nabla\cdot\big(2\mu^*\mathbf{D}(\cdot)\big)$, this convolution enters
\emph{inside} the linear operator. Consequently, the $z$-Fourier modes of the velocity and pressure no longer satisfy independent equations:
they satisfy a coupled infinite system whose \emph{off-diagonal structure} is Toeplitz/Laurent in the mode index (fixed index shifts determined by
the coefficient harmonics) \cite{BottcherSilbermann2006}. The conceptual message, therefore, is that
\emph{spanwise patterning can arise as a purely linear constitutive effect.} Spanwise-uniform forcing can generate $\kappa\neq 0$ responses solely
because the coefficients carry spanwise harmonics.
\medskip

\noindent To demonstrate this, let $\kappa_m:=2\pi m/L_z$ for $m\in\mathbb{Z}$. For any $z$-periodic field $g(x,y,z)$, write
\[
g(x,y,z)=\sum_{m\in\mathbb{Z}} g_m(x,y)\,e^{i\kappa_m z},
\qquad
g_m(x,y):=\frac{1}{L_z}\int_0^{L_z} g(x,y,z)\,e^{-i\kappa_m z}\,dz.
\]
If $\mu^*$ and $\hat{\mathbf{u}}$ have Fourier series
\[
\mu^*(x,y,z;\omega)=\sum_{n\in\mathbb{Z}}\mu_n(x,y;\omega)\,e^{i\kappa_n z},
\qquad
\hat{\mathbf{u}}(x,y,z;\omega)=\sum_{m\in\mathbb{Z}}\hat{\mathbf{u}}_m(x,y;\omega)\,e^{i\kappa_m z},
\]
then their product obeys the standard convolution identity
\[
\mu^*(x,y,z;\omega)\,\hat{\mathbf{u}}(x,y,z;\omega)
=
\sum_{m\in\mathbb{Z}}
\Big(\sum_{n\in\mathbb{Z}}\mu_n(x,y;\omega)\,\hat{\mathbf{u}}_{m-n}(x,y;\omega)\Big)e^{i\kappa_m z}.
\]
This ``multiplier $\Rightarrow$ convolution'' principle is the algebraic core of Toeplitz/Laurent coupling. Importantly, it survives the presence
of derivatives in the viscous term. Indeed, in the operator $\nabla\cdot(2\mu^*\mathbf{D}(\hat{\mathbf{u}}))$ one encounters both
$\mu^*\,\nabla\hat{\mathbf{u}}$ and $(\nabla\mu^*)\hat{\mathbf{u}}$ contributions; in Fourier space, $\partial_z$ acting on a Fourier mode
multiplies by $i\kappa_m$, while $\partial_z$ acting on $\mu^*$ multiplies by $i\kappa_n$ at the coefficient level. In either case, the
\emph{mode-shift structure} is unchanged: the $m$th output mode depends on $\hat{\mathbf{u}}_{m-n}$ with weights determined by $\mu_n$
(and, in gradient terms, by factors such as $\kappa_n\mu_n$).

To make this concrete, consider the oscillatory Oseen problem in a $z$-periodic domain (the Stokes case is recovered by setting $\mathbf{V}_0\equiv 0$):
\begin{equation}
	i\omega\rho\,\hat{\mathbf{u}} + \rho(\mathbf{V}_0\cdot\nabla)\hat{\mathbf{u}} + \rho(\hat{\mathbf{u}}\cdot\nabla)\mathbf{V}_0
	-\nabla\cdot\big(2\mu^*(x,y,z;\omega)\mathbf{D}(\hat{\mathbf{u}})\big) + \nabla \hat{p}=\hat{\mathbf{f}},
	\qquad
	\nabla\cdot\hat{\mathbf{u}}=0,
\end{equation}
with $z$-periodic boundary conditions and (for simplicity) a base flow $\mathbf{V}_0=\mathbf{V}_0(x,y)$ independent of $z$.
Projecting onto $e^{i\kappa_m z}$ yields, for each $m\in\mathbb{Z}$, a \emph{coupled} family of two-dimensional problems of the form
\begin{equation}
	\label{eq:toeplitz_generic_coupling}
	\mathcal{L}_\omega(\kappa_m)\,\hat{\mathbf{u}}_m
	\;+\;
	\sum_{n\neq 0}\mathcal{C}_{\omega,n}(\kappa_{m-n}\to\kappa_m)\,\hat{\mathbf{u}}_{m-n}
	=
	\hat{\mathbf{f}}_m,
\end{equation}
where $\mathcal{L}_\omega(\kappa_m)$ is the classical $\kappa$-reduced operator associated with the $z$-independent part of the coefficients (in
particular the $n=0$ viscosity mode $\mu_0$), and where $\mathcal{C}_{\omega,n}$ is the linear coupling map induced by the $n$th viscosity mode
$\mu_n$ (including the explicit $\kappa_n$ factors coming from $\partial_z\mu^*$ terms). The essential structural fact is that the coupling
\emph{occupies fixed index offsets} determined by $n$ (through $m-n$): for each nonzero coefficient mode $n$, the corresponding coupling terms
populate the $n$th block off-diagonal across the infinite system \cite{BottcherSilbermann2006}. (The block \emph{operators} themselves may still
depend on $\kappa_m$ via the $\partial_z\mapsto i\kappa_m$ substitution.) From this, two immediate consequences follow.
\begin{enumerate}
	\item \noindent\emph{Linear Generation of Spanwise Sidebands.}
	If the forcing is spanwise-uniform, then $\hat{\mathbf{f}}_m\equiv 0$ for $m\neq 0$. When coefficients are $z$-independent, this implies
	$\hat{\mathbf{u}}_m\equiv 0$ for $m\neq 0$ as well (perfect decoupling). Under a $z$-dependent texture, \eqref{eq:toeplitz_generic_coupling}
	generically forces $\hat{\mathbf{u}}_m\neq 0$ for $m\neq 0$ because the equation for $\hat{\mathbf{u}}_m$ contains contributions from
	$\hat{\mathbf{u}}_{m-n}$ whenever $\mu_n\not\equiv 0$.
	
	\smallskip
	\item \noindent\emph{Bandwidth = Texture Bandwidth.}
	If $\mu^*$ has finitely many spanwise Fourier modes (a finite-band texture), then the sum in \eqref{eq:toeplitz_generic_coupling} is finite and the
	coupled system is banded: only finitely many block off-diagonals are populated. Conversely, phase-only textures with an exponential form typically
	have infinitely many Fourier coefficients but with rapid decay; the coupled system is then infinite-range but effectively band-limited in
	perturbative regimes.
\end{enumerate}
The clearest worked mechanism is a single-harmonic texture. For example, take
\begin{equation}
	\mu^*(x,y,z;\omega)=\mu_0^*(x,y;\omega)\bigl(1+\varepsilon e^{ik_0 z}\bigr),
	\qquad
	k_0=\frac{2\pi m_0}{L_z}=\kappa_{m_0},\quad m_0\in\mathbb{N},\quad 0<\varepsilon\ll 1.
\end{equation}
Then only the Fourier coefficients at indices $n=0$ and $n=m_0$ are nonzero (in this one-sided complex-amplitude form). Consequently,
\eqref{eq:toeplitz_generic_coupling} reduces to a one-step \emph{index-shift} recursion:
\begin{equation}
	\label{eq:toeplitz_onesided_shift}
	\mathcal{L}_\omega(\kappa_m)\hat{\mathbf{u}}_m
	\;+\;
	\varepsilon\,\mathcal{C}_{\omega,m_0}(\kappa_{m-m_0}\to\kappa_m)\hat{\mathbf{u}}_{m-m_0}
	=
	\hat{\mathbf{f}}_m.
\end{equation}

For a phase-only texture of unit modulus,
$\mu^*(x,y,z;\omega)=\mu_0^*(x,y;\omega)e^{i\varepsilon\cos(k_0 z)}$,
its first-order expansion
$e^{i\varepsilon\cos(k_0 z)}=1+\tfrac{i\varepsilon}{2}e^{ik_0 z}+\tfrac{i\varepsilon}{2}e^{-ik_0 z}+\mathcal{O}(\varepsilon^2)$
yields symmetric $\pm m_0$ coupling at $\mathcal{O}(\varepsilon)$:
\begin{equation}
	\label{eq:toeplitz_symmetric_shift}
	\mathcal{L}_\omega(\kappa_m)\hat{\mathbf{u}}_m
	+\frac{i\varepsilon}{2}\,\mathcal{C}_{\omega,m_0}(\kappa_{m-m_0}\to\kappa_m)\hat{\mathbf{u}}_{m-m_0}
	+\frac{i\varepsilon}{2}\,\mathcal{C}_{\omega,-m_0}(\kappa_{m+m_0}\to\kappa_m)\hat{\mathbf{u}}_{m+m_0}
	=
	\hat{\mathbf{f}}_m
	+\mathcal{O}(\varepsilon^2).
\end{equation}
In this situation, ``pattern selection'' admits a precise meaning: the primary induced sidebands are at $\kappa=\pm k_0$, with higher sidebands
appearing only through iterated coupling and scaling like powers of $\varepsilon$ (modulo resolvent amplification). For analysis and computation we truncate to $m\in\{-M,\ldots,M\}$ and form the block vector
\begin{equation}
	\mathbf{U}^{[M]}:=(\hat{\mathbf{u}}_{-M},\ldots,\hat{\mathbf{u}}_0,\ldots,\hat{\mathbf{u}}_M)^{\mathsf{T}},
	\qquad
	\mathbf{F}^{[M]}:=(\hat{\mathbf{f}}_{-M},\ldots,\hat{\mathbf{f}}_0,\ldots,\hat{\mathbf{f}}_M)^{\mathsf{T}}.
\end{equation}
The truncated coupled system becomes a finite block system
\begin{equation}
	\label{eq:toeplitz_truncated_system}
	\mathbb{A}_M(\omega,\varepsilon)\mathbf{U}^{[M]}=\mathbf{F}^{[M]},
\end{equation}
whose diagonal blocks are $\mathcal{L}_\omega(\kappa_m)$ and whose off-diagonal blocks occupy fixed index offsets determined by the active
texture harmonics (e.g.\ offsets $\pm m_0$ for a single cosine/linearized phase-only harmonic). Thus, while the diagonal physics remains
mode-dependent through $\kappa_m$, the \emph{coupling pattern in index space} is Toeplitz/Laurent (fixed-shift) \cite{BottcherSilbermann2006}.
This makes the mechanism ``plot-ready'': one can solve a structured linear system and directly quantify how energy migrates from $m=0$ forcing into
$m\neq 0$ response modes.

The Toeplitz/Laurent mode-coupling mechanism is purely linear and purely constitutive. It does not require nonlinear interactions, transient growth
driven by quadratic terms, or an instability that breaks spanwise symmetry through bifurcation. Rather, three-dimensionality is introduced through
the coefficients: the operator itself mixes spanwise Fourier modes because its coefficients carry spanwise harmonics. This persists in oscillatory
Stokes settings (no advection) and therefore remains present as inertia is increased. Advection may amplify or reshape the coupled response, but it is
not the origin of the coupling.

Furthermore, the Toeplitz coupling clarifies why the induced spanwise pattern can be disproportionate to the texture amplitude. In a perturbative
regime, one may view the coupled operator as a small structured perturbation of the block-diagonal decoupled operator. Schematically,
\begin{equation}
	\label{eq:toeplitz_perturbation_split}
	\mathbb{A}_M(\omega,\varepsilon)=\mathbb{D}_M(\omega)+\varepsilon\,\mathbb{B}_M(\omega),
\end{equation}
where $\mathbb{D}_M$ is block diagonal with blocks $\mathcal{L}_\omega(\kappa_m)$ and $\mathbb{B}_M$ is the Toeplitz/Laurent coupling operator
(fixed index offsets). When $\varepsilon\|\mathbb{D}_M(\omega)^{-1}\mathbb{B}_M(\omega)\|<1$, the inverse admits the Neumann expansion
\[
\mathbb{A}_M(\omega,\varepsilon)^{-1}
=
\bigl(\mathbb{I}+\varepsilon\,\mathbb{D}_M(\omega)^{-1}\mathbb{B}_M(\omega)\bigr)^{-1}\mathbb{D}_M(\omega)^{-1}
=
\sum_{j=0}^\infty (-\varepsilon)^j\bigl(\mathbb{D}_M(\omega)^{-1}\mathbb{B}_M(\omega)\bigr)^j\mathbb{D}_M(\omega)^{-1},
\]
which makes explicit that sideband response is filtered through products of modewise resolvents \cite{Kato1995}.
Accordingly, if $\mathcal{L}_\omega(\kappa_m)^{-1}$ has large norm at some $\kappa_m$ (for example, due to non-normal amplification in a BFS-type
truncation), then even small $\varepsilon$ can yield large sideband amplitudes because the coupling is mediated by resolvent factors. More broadly,
the size and location of the induced pattern are governed not only by the texture Fourier amplitude but also by the \emph{pseudospectral} and
numerical-range geometry of the decoupled modewise operators \cite{TrefethenEmbree2005,GustafsonRao1997}:
sidebands are strongest when either the forced mode ($m=0$) or the target mode ($m=\pm m_0$) lies in a region of high resolvent gain.

A convenient scalar measure of this constitutive patterning is the (truncated) sideband energy ratio
\begin{equation}
	\label{eq:toeplitz_sideband_energy_ratio}
	\mathcal{R}_M(\omega,\varepsilon)
	:=
	\frac{\sum_{\substack{|m|\le M\\ m\neq 0}}\|\hat{\mathbf{u}}_m(\omega;\varepsilon)\|_{\mathsf{E}}^2}
	{\|\hat{\mathbf{u}}_0(\omega;\varepsilon)\|_{\mathsf{E}}^2},
\end{equation}
in a fixed norm $\|\cdot\|_{\mathsf{E}}$ (e.g.\ kinetic energy on $\Omega_{2D}$, or an $X_{\kappa_m}$-consistent strain-energy norm). Under small
$\varepsilon$ and a single-harmonic texture, one expects $\mathcal{R}_M(\omega,\varepsilon)=\mathcal{O}(\varepsilon^2)$, with the leading
coefficient controlled by products of modewise resolvent norms and coupling-map norms \cite{TrefethenEmbree2005}. This makes the mechanism
experimentally and computationally legible: one can identify frequencies where the decoupled operator is high-gain and observe how the same
frequencies become hotspots for spanwise patterning once texture is introduced.

In summary, $z$-dependent complex viscosity textures induce an operator-valued Toeplitz/Laurent coupling of spanwise Fourier modes because
multiplication by the coefficient field becomes convolution in Fourier space. This yields a conservative and explicit pathway to spanwise patterning:
spanwise-uniform forcing can produce $\kappa\neq 0$ response already in linear oscillatory Stokes/Oseen theory. The Toeplitz structure supplies a
structured linear-algebra formulation, supports principled truncations, and connects pattern selection to resolvent and pseudospectral diagnostics
for the underlying decoupled operators \cite{BottcherSilbermann2006,TrefethenEmbree2005}.
\newpage

\section{Summary and Key Results}
\label{sec:highlights_key_results}

At a fixed forcing frequency $\omega>0$, we develop a functional-analytic and mechanism-facing framework for oscillatory Stokes/Oseen systems with
\emph{heterogeneous complex viscosity} $\mu^*(\mathbf{x},\omega)$, showing that:
\begin{enumerate}[leftmargin=2.2em, itemsep=0.35em]
	\item Under minimal passivity hypotheses, the associated sesquilinear forms generate $m$-sectorial realizations on bounded Lipschitz domains, and
	the corresponding realizations have compact resolvent on the solenoidal $L^2$ space (hence discrete spectrum with no finite accumulation
	point other than $\infty$). \cite{Kato1995}
	
	\item Spatial variation of the constitutive phase $\varphi(\mathbf{x},\omega)=\arg\mu^*(\mathbf{x},\omega)$ induces \emph{intrinsic non-normality}
	of the viscous core even in the absence of advection, so harmonic amplification is governed by resolvent norms (pseudospectral and
	numerical-range geometry), not by eigenvalues alone. \cite{TrefethenEmbree2005,GustafsonRao1997}
	
	\item Under Tier~II coefficient regularity (e.g.\ $\mu^*(\cdot,\omega)\in W^{1,\infty}$), viscosity gradients enter as commutator-type,
	distributed \emph{vorticity-source terms} whose strength is controlled by $\|\nabla\mu^*(\cdot,\omega)\|_{L^\infty}$ (and in the phase-only
	class by $\|\nabla\varphi(\cdot,\omega)\|_{L^\infty}$), producing a quantitative \emph{square-resolvent} amplification channel in which
	injection by texture gradients is filtered through non-normal resolvent gains. \cite{MajdaBertozzi2002,TrefethenEmbree2005}
	
	\item In three-dimensional periodic domains, $z$-dependent textures generate an operator-valued Toeplitz/Laurent coupling of spanwise Fourier
	modes, so spanwise-uniform forcing creates $\kappa\neq 0$ response already at the linear level.
\end{enumerate}

\paragraph{Setting (fixed forcing frequency; constitutive closure).}
Let $\Omega\subset\mathbb{R}^d$ with $d\in\{2,3\}$ be a bounded Lipschitz domain, allowing polygonal/polyhedral boundaries and the attendant
corner/edge singular behavior. Let $\Gamma_D\subset\partial\Omega$ be a Dirichlet boundary portion of positive surface measure (ensuring
Poincar\'e and Korn inequalities). \cite{Evans2010,Horgan1995,Ciarlet1988}
Fix $\omega>0$ and consider time-harmonic fields
$\mathbf{v}(\mathbf{x},t)=\Re\{\hat{\mathbf{v}}(\mathbf{x};\omega)e^{i\omega t}\}$, with deviatoric stress closed by the frequency-domain
linear-response law
\begin{equation}
	\hat{\bm{\tau}}(\mathbf{x};\omega)
	=
	2\,\mu^*(\mathbf{x},\omega)\,\mathbf{D}(\hat{\mathbf{v}}(\mathbf{x};\omega)),
	\qquad
	\mathbf{D}(\hat{\mathbf{v}}):=\tfrac12\big(\nabla\hat{\mathbf{v}}+(\nabla\hat{\mathbf{v}})^{\mathsf T}\big),
	\label{eq:HL_constitutive}
\end{equation}
where $\mu^*(\mathbf{x},\omega)\in\mathbb{C}$ is interpreted as a \emph{harmonic linear-response coefficient} induced by a causal stress memory
kernel, as is standard in linear viscoelasticity and structured-fluid settings. \cite{BirdArmstrongHassager1987,Ferry1980,Larson1999}
The corresponding variable-coefficient viscous operator is $\nabla\cdot\bigl(2\mu^*\mathbf{D}(\cdot)\bigr)$; replacing it by an \emph{ad hoc}
``complex Laplacian'' is generally incorrect once $\mu^*$ varies in space because derivatives acting on $\mu^*$ produce commutator terms that are
central to the theory and to the mechanisms identified in this paper.

Throughout, we assume uniform positive dissipation (passivity/accretivity) and bounded density:
\begin{equation}
	\mu^*(\cdot,\omega)\in L^\infty(\Omega;\mathbb{C}),
	\qquad
	\Re\mu^*(\mathbf{x},\omega)\ge \mu_{\min}>0\quad\text{for a.e.\ $\mathbf{x}\in\Omega$,}
	\qquad
	0<\rho_{\min}\le\rho(\mathbf{x})\le\rho_{\max}<\infty\quad\text{a.e.}
	\label{eq:HL_passivity}
\end{equation}
Condition~\eqref{eq:HL_passivity} is physically interpretable (nonnegative cycle-averaged dissipation density) and mathematically decisive
(coercivity of the real part of the viscous form, hence well-posedness and sectoriality). \cite{Kato1995,Horgan1995,Ciarlet1988}
We define the constitutive phase by
\[
\mu^*(\mathbf{x},\omega)=|\mu^*(\mathbf{x},\omega)|\,e^{i\varphi(\mathbf{x},\omega)},
\qquad
\varphi(\mathbf{x},\omega):=\arg\mu^*(\mathbf{x},\omega),
\]
so that (heuristically) $\varphi$ measures the local stress--strain-rate lag. The novelty of the paper hinges on the distinction between
\emph{global phase} (constant $\varphi$) and \emph{phase texture} (spatially varying $\varphi$).

\paragraph{Solenoidal spaces and equivalent formulations.}
Let
\[
H := L^2(\Omega;\mathbb{C}^d),
\qquad
V := \{\mathbf{u}\in H^1(\Omega;\mathbb{C}^d): \mathbf{u}=0 \text{ on }\Gamma_D\},
\]
and define the solenoidal subspaces
\[
H_\sigma := \overline{\{\mathbf{u}\in C_c^\infty(\Omega;\mathbb{C}^d): \nabla\cdot\mathbf{u}=0\}}^{\,L^2},
\qquad
V_\sigma := V\cap H_\sigma.
\]
Equivalently (and often more convenient for PDE readers), one may use the saddle-point formulation on $V\times L_0^2(\Omega)$ together with the
Babu\v{s}ka--Brezzi inf--sup condition; both viewpoints are compatible, and the results below can be stated in either language.
\cite{GiraultRaviart1986,BoffiBrezziFortin2013}

\subsection{Core Results (Operator Level.)}

\paragraph{R1. Well-posedness via coercive sectorial forms (oscillatory Stokes at fixed $\omega$).}
Define the viscous and mass forms on $V_\sigma$:
\[
\mathfrak{a}(\mathbf{u},\mathbf{v})
:=
\int_\Omega 2\,\mu^*(\mathbf{x},\omega)\,\mathbf{D}(\mathbf{u}):\overline{\mathbf{D}(\mathbf{v})}\,d\mathbf{x},
\qquad
\mathfrak{m}(\mathbf{u},\mathbf{v})
:=
\int_\Omega \rho(\mathbf{x})\,\mathbf{u}\cdot\overline{\mathbf{v}}\,d\mathbf{x}.
\]
At fixed $\omega>0$, define the oscillatory form
\begin{equation}
	\mathfrak{a}_\omega(\mathbf{u},\mathbf{v})
	:=
	\mathfrak{a}(\mathbf{u},\mathbf{v})
	+
	i\omega\,\mathfrak{m}(\mathbf{u},\mathbf{v}).
	\label{eq:HL_form}
\end{equation}
Then $\mathfrak{a}_\omega$ is bounded on $V_\sigma$. By Korn's inequality and \eqref{eq:HL_passivity},
\[
\Re \mathfrak{a}_\omega(\mathbf{u},\mathbf{u})
=
\int_\Omega 2\,\Re\mu^*(\mathbf{x},\omega)\,|\mathbf{D}(\mathbf{u})|^2\,d\mathbf{x}
\ \ge\
2\mu_{\min}\,\|\mathbf{D}(\mathbf{u})\|_{L^2(\Omega)}^2
\ \gtrsim\
\mu_{\min}\,\|\mathbf{u}\|_{H^1(\Omega)}^2,
\]
so $\mathfrak{a}_\omega$ is coercive in real part. Consequently, for any $\mathbf{F}\in V_\sigma^*$ there exists a unique
$\mathbf{u}\in V_\sigma$ satisfying
\[
\mathfrak{a}_\omega(\mathbf{u},\mathbf{v})=\langle \mathbf{F},\mathbf{v}\rangle
\quad\forall \mathbf{v}\in V_\sigma,
\]
with the a priori estimate
\begin{equation}
	\|\mathbf{u}\|_{H^1(\Omega)}
	\le
	C\,\|\mathbf{F}\|_{V_\sigma^*},
	\qquad
	C=C(\mu_{\min},\rho_{\max},\Omega,\omega).
	\label{eq:HL_apriori_stokes}
\end{equation}
The operator induced by $\mathfrak{a}_\omega$ on $H_\sigma$ via the form method is $m$-sectorial (in particular, sectorial with a sectorial
numerical range). \cite{Kato1995}
Accordingly, the standard holomorphic functional calculus for sectorial operators applies to $\mathcal{L}_\omega$ (with the usual domain
restrictions), and will be used as needed in later operator-theoretic arguments. \cite{Haase2006}
The term $i\omega\,\mathfrak{m}$ is purely imaginary and therefore does not compete with dissipation: it rotates phases but does not weaken
coercivity. All coercive stability is inherited from $\Re\mu^*\ge\mu_{\min}$.

\paragraph{R2. Mass renorming: the inertial term is canonically skew-adjoint in the density-weighted metric.}
Define the bounded, strictly positive mass operator $M:H_\sigma\to H_\sigma$ by
\[
(M\mathbf{u},\mathbf{v})_{L^2}=\mathfrak{m}(\mathbf{u},\mathbf{v}).
\]
Let $A$ denote the (pressure-eliminated) operator associated with the viscous form $\mathfrak{a}$ (at fixed $\omega$, i.e.\ with coefficient field
$\mu^*(\cdot,\omega)$). Then the oscillatory Stokes operator is
\[
\mathcal{L}_\omega = A + i\omega M.
\]
Introduce the density-weighted Hilbert space $H_\rho:=H_\sigma$ equipped with the inner product
\[
(\mathbf{u},\mathbf{v})_\rho := \mathfrak{m}(\mathbf{u},\mathbf{v}) = (M\mathbf{u},\mathbf{v})_{L^2}.
\]
In $H_\rho$, define $\widetilde{A}$ by the form identity
\[
(\widetilde{A}\mathbf{u},\mathbf{v})_\rho = \mathfrak{a}(\mathbf{u},\mathbf{v})
\qquad ( \mathbf{u}\in D(\widetilde{A}),\ \mathbf{v}\in V_\sigma ),
\]
so that the oscillatory operator is simply
\[
\widetilde{\mathcal{L}}_\omega := \widetilde{A}+i\omega I \quad \text{on } H_\rho.
\]
The realizations on $H_\sigma$ and $H_\rho$ are related by the canonical similarity transform induced by $M^{1/2}$:
\begin{equation}
	\mathcal{L}_\omega
	=
	M^{1/2}\,\widetilde{\mathcal{L}}_\omega\,M^{1/2},
	\qquad
	\mathcal{L}_\omega^{-1}
	=
	M^{-1/2}\,\widetilde{\mathcal{L}}_\omega^{-1}\,M^{-1/2},
	\label{eq:HL_similarity}
\end{equation}
whenever $0\in\rho(\mathcal{L}_\omega)$ (equivalently $0\in\rho(\widetilde{\mathcal{L}}_\omega)$).
This isolates the analytic structure: in the mass inner product, unsteady inertia is exactly the canonical skew-adjoint shift $i\omega I$, and all
accretivity resides in the viscous core. The operator $M$ changes only the metric (density-weighted kinetic energy); in that metric, inertia is a
pure $90^\circ$ phase rotation.

\paragraph{R3. Compact resolvent and discrete spectrum on bounded domains.}
On bounded Lipschitz $\Omega$, the embedding $V_\sigma\hookrightarrow H_\sigma$ is compact (Rellich--Kondrachov, restricted to solenoidal fields),
and the form method yields that $(\mathcal{L}_\omega-\lambda I)^{-1}$ maps $H_\sigma$ into $V_\sigma$ for $\lambda$ in the resolvent set.
Therefore, for any $\lambda\in\rho(\mathcal{L}_\omega)$, the resolvent
\[
(\mathcal{L}_\omega-\lambda I)^{-1}:H_\sigma\to H_\sigma
\]
is compact. Hence $\sigma(\mathcal{L}_\omega)$ is purely discrete: isolated eigenvalues of finite algebraic multiplicity with $|\lambda|\to\infty$
as the only accumulation, and isolated spectral components admit finite-rank Riesz projections via contour integrals. \cite{Evans2010,Kato1995}
Compactness here is geometric rather than delicate: it ultimately comes from ``$H^1$ controls $L^2$ on bounded domains.''

\paragraph{R4. Phase textures induce intrinsic non-normality of the viscous core (even without advection).}
If the constitutive phase $\varphi(\cdot,\omega)$ is spatially constant, then $\mu^*(\cdot,\omega)=e^{i\varphi_0}\,|\mu^*(\cdot,\omega)|$ has a
global complex factor that can be pulled out of the viscous form:
\[
\mathfrak{a}(\mathbf{u},\mathbf{v})
=
e^{i\varphi_0}\int_\Omega 2|\mu^*|\,\mathbf{D}(\mathbf{u}):\overline{\mathbf{D}(\mathbf{v})}\,d\mathbf{x},
\]
so the viscous core is a scalar complex multiple of a symmetric coercive form and is (after a global rotation) essentially normal.
In contrast, if $\varphi(\cdot,\omega)$ varies in space, multiplication by $e^{i\varphi(\mathbf{x},\omega)}$ does not commute with differentiation
in $\nabla\cdot\big(2\mu^*\mathbf{D}(\cdot)\big)$, and the viscous core becomes generically non-selfadjoint and non-normal even in unsteady Stokes
(no advective inertia). Consequently, eigenvalues alone do not control harmonic response; the appropriate amplification descriptors are resolvent norms
and pseudospectra. \cite{TrefethenEmbree2005}

A robust resolvent bound follows from the numerical range
\[
W(\mathcal{L}_\omega):=\{(\mathcal{L}_\omega u,u)_{H_\sigma}:\ u\in D(\mathcal{L}_\omega),\ \|u\|_{H_\sigma}=1\}.
\]
For $\lambda\notin \overline{W(\mathcal{L}_\omega)}$,
\begin{equation}
	\|(\mathcal{L}_\omega-\lambda I)^{-1}\|_{\mathcal{L}(H_\sigma)}
	\le
	\frac{1}{\mathrm{dist}(\lambda,W(\mathcal{L}_\omega))}.
	\label{eq:HL_numrange_bound}
\end{equation}
\cite{GustafsonRao1997,Kato1995}
Moreover, $\varepsilon$-pseudospectral proximity of $\lambda=0$ implies $\|\mathcal{L}_\omega^{-1}\|\gtrsim \varepsilon^{-1}$, yielding a rigorous
``hidden-gain'' mechanism in linear oscillatory response. \cite{TrefethenEmbree2005}
A constant phase is a global rotation of the stress--strain relation; a varying phase is a spatially varying rotation. Spatially varying rotations
necessarily generate commutators with derivatives, which is the operator-theoretic source of non-normality.

\paragraph{R5. Linearized Navier--Stokes (Oseen) as a bounded form perturbation; two independent non-normality channels.}
Let $(\mathbf{V}_0,P_0)$ be a steady base flow and define the Oseen sesquilinear form on $V_\sigma$ by
\[
\mathfrak{c}(\mathbf{u},\mathbf{v})
:=
\int_\Omega \rho(\mathbf{V}_0\cdot\nabla)\mathbf{u}\cdot\overline{\mathbf{v}}\,d\mathbf{x}
+
\int_\Omega \rho(\mathbf{u}\cdot\nabla)\mathbf{V}_0\cdot\overline{\mathbf{v}}\,d\mathbf{x}.
\]
If $\mathbf{V}_0\in W^{1,\infty}(\Omega)$ and $\rho\in L^\infty(\Omega)$, then $\mathfrak{c}$ is bounded on $V_\sigma\times V_\sigma$ and hence is a
bounded form perturbation of $\mathfrak{a}_\omega$. \cite{Kato1995}
Therefore the harmonic linearized operator $\mathcal{L}_{\omega,\mathrm{lin}}$ induced by $\mathfrak{a}_\omega+\mathfrak{c}$ is closed on $H_\sigma$
with $D(\mathcal{L}_{\omega,\mathrm{lin}})=D(\mathcal{L}_\omega)$; on bounded domains it inherits compact resolvent and discrete spectral structure.
Non-normality is now enforced by two independent channels:
\begin{enumerate}[leftmargin=2.2em, itemsep=0.25em]
	\item Advection (present even for real viscosity; the classical channel in open/shear flows),
	\item Spatially varying constitutive phase (present even without advection; the new channel isolated in this paper).
\end{enumerate}
The important point is that advection is not required for non-normality once the viscosity itself carries a spatially varying complex phase.

\paragraph{R6. Texture-gradient vorticity forcing and resolvent-controlled amplification (mechanics-to-operator bridge).}
Assume Tier~II regularity $\mu^*(\cdot,\omega)\in W^{1,\infty}(\Omega;\mathbb{C})$ in addition to \eqref{eq:HL_passivity}, and consider the
(pressure-inclusive) harmonic Stokes balance with forcing $\hat{\mathbf{f}}$ under the convention $\Re\{\hat{\mathbf{v}}e^{i\omega t}\}$:
\[
i\omega\rho\,\hat{\mathbf{v}} - \nabla\cdot\big(2\mu^*\mathbf{D}(\hat{\mathbf{v}})\big) + \nabla \hat{p} = \hat{\mathbf{f}},
\qquad
\nabla\cdot\hat{\mathbf{v}}=0.
\]
Taking curl eliminates pressure and yields, in distributions, a vorticity identity of the form
\begin{equation}
	i\omega\rho\,\hat{\boldsymbol{\omega}}
	=
	\mu^*\,\Delta\hat{\boldsymbol{\omega}}
	+
	\mathcal{G}_{\mu^*}[\hat{\mathbf{v}}]
	+
	\nabla\times \hat{\mathbf{f}},
	\qquad
	\hat{\boldsymbol{\omega}}:=\nabla\times\hat{\mathbf{v}}.
	\label{eq:HL_vorticity_identity}
\end{equation}
\cite{MajdaBertozzi2002}
Here $\mathcal{G}_{\mu^*}$ is a \emph{texture commutator}: it collects precisely the terms generated by derivatives hitting $\mu^*$ inside
$\nabla\cdot\big(2\mu^*\mathbf{D}(\cdot)\big)$. Structurally, $\mathcal{G}_{\mu^*}$ is linear in $\nabla\mu^*$ and linear in first derivatives of
$\hat{\mathbf{v}}$, so it represents distributed vorticity injection created by viscosity gradients even in purely linear oscillatory flow.
A dual estimate holds:
\begin{equation}
	\|\mathcal{G}_{\mu^*}[\hat{\mathbf{v}}]\|_{H^{-1}(\Omega)}
	\le
	C_\Omega\,\|\nabla\mu^*(\cdot,\omega)\|_{L^\infty(\Omega)}\,\|\mathbf{D}(\hat{\mathbf{v}})\|_{L^2(\Omega)}.
	\label{eq:HL_commutator_bound}
\end{equation}
\cite{Evans2010}
Thus, texture gradients act as sources of vorticity response with injection strongest where strain concentrates (e.g.\ geometric corners, shear layers,
and regions of separated flow). In regimes where $\|\mathcal{L}_\omega^{-1}\|$ is large, combining \eqref{eq:HL_commutator_bound} with resolvent control
of $\hat{\mathbf{v}}$ yields a two-stage (injection $\to$ amplification) pathway in which the overall response can scale like a square of the resolvent
norm (up to geometry- and norm-dependent constants).

\paragraph{R7. Phase-only textures: $\|\nabla\varphi\|$ as a single texture-strength axis and a phase-compensated rewrite.}
In the phase-only class
\[
\mu^*(\mathbf{x},\omega)=\mu_0(\omega)\,e^{i\varphi(\mathbf{x},\omega)},
\qquad
\cos\varphi(\cdot,\omega)\ge c_0>0,
\]
one has $\Re\mu^*=\mu_0(\omega)\cos\varphi \ge \mu_0(\omega)c_0$, so \eqref{eq:HL_passivity} holds with $\mu_{\min}=\mu_0c_0$, and
$\nabla\mu^* = i\mu^*\nabla\varphi$. Hence
\[
\|\nabla\mu^*\|_{L^\infty}
\le
\mu_0(\omega)\,\|\nabla\varphi(\cdot,\omega)\|_{L^\infty},
\]
identifying $\|\nabla\varphi\|_{L^\infty}$ as a single, physically interpretable control parameter for texture strength at fixed magnitude $|\mu^*|$.
Moreover, the phase-compensated transformation $\hat{\mathbf{v}}=e^{-i\varphi}\hat{\mathbf{w}}$ removes the leading complex factor from the symmetric
Stokes part but introduces an explicit first-order coupling proportional to $\nabla\varphi$ (a covariant-derivative structure). This supplies a
principled separation between:
\begin{enumerate}[leftmargin=2.2em, itemsep=0.25em]
	\item Phase-compensated dissipation (still governed by $\Re\mu^*$), and
	\item An irreducible phase-gradient mechanism (governed by $\nabla\varphi$) that cannot be removed by any global rotation.
\end{enumerate}
In phase-only textures, the magnitude of viscosity is unchanged; all new effects are commutator/geometry effects induced by a spatially varying local
phase lag.
\subsubsection{Core Results (3D Periodic Domains: Toeplitz/Laurent Coupling.)}
The results R8--R10 provide a linear, constitutive route to spanwise patterning in 3D periodic domains and supply structured,
plot-ready diagnostics for computation (Toeplitz truncations and sideband ratios). 
\paragraph{R8. $z$-dependent textures induce an operator-valued Laurent/Toeplitz coupling of spanwise Fourier modes.}
Let $\Omega=\Omega_{2D}\times(0,L_z)$ with $z$-periodic boundary conditions, and fix $\omega>0$.
Assume the coefficient field $\mu^*(x,y,z;\omega)\in L^\infty(\Omega)$ is $L_z$-periodic in $z$ with Fourier expansion
\[
\mu^*(x,y,z;\omega)=\sum_{n\in\mathbb{Z}}\mu_n(x,y;\omega)\,e^{i\kappa_n z},
\qquad
\kappa_n=\frac{2\pi n}{L_z},
\]
and expand the harmonic unknowns similarly:
\[
\hat{\mathbf{u}}(x,y,z;\omega)=\sum_{m\in\mathbb{Z}}\hat{\mathbf{u}}_m(x,y;\omega)\,e^{i\kappa_m z},
\qquad
\hat p(x,y,z;\omega)=\sum_{m\in\mathbb{Z}}\hat p_m(x,y;\omega)\,e^{i\kappa_m z}.
\]
Write $\hat{\mathbf{u}}_m=(\hat u_m,\hat v_m,\hat w_m)$ and let $\hat{\mathbf{u}}_{m,\parallel}:=(\hat u_m,\hat v_m)$.
Then the incompressibility constraint becomes the $\kappa$-reduced condition
\[
\nabla_{x,y}\cdot \hat{\mathbf{u}}_{m,\parallel} + i\kappa_m \hat w_m = 0 \qquad \text{in }\Omega_{2D}.
\]
Consider the (pressure-inclusive) oscillatory Stokes/Oseen balance with a $z$-independent base flow $\mathbf{V}_0=\mathbf{V}_0(x,y)$:
\[
-i\omega\rho\,\hat{\mathbf{u}}
+\rho(\mathbf{V}_0\cdot\nabla)\hat{\mathbf{u}}
+\rho(\hat{\mathbf{u}}\cdot\nabla)\mathbf{V}_0
-\nabla\cdot\big(2\mu^*\mathbf{D}(\hat{\mathbf{u}})\big)
+\nabla \hat p
=\hat{\mathbf{f}},
\qquad
\nabla\cdot\hat{\mathbf{u}}=0.
\]
Projecting onto the spanwise modes $e^{i\kappa_m z}$ yields an \emph{infinite coupled system} on $\Omega_{2D}$:
\begin{equation}
	\label{eq:HL_toeplitz_system}
	\mathcal{L}_{\omega,\mathrm{lin}}(\kappa_m)\,\hat{\mathbf{u}}_m
	\;+\;
	\sum_{n\neq 0}\mathcal{K}_{\omega,n}\!\big(\kappa_{m-n}\to\kappa_m\big)\,\hat{\mathbf{u}}_{m-n}
	=
	\hat{\mathbf{f}}_m,
	\qquad m\in\mathbb{Z},
\end{equation}
where $\mathcal{L}_{\omega,\mathrm{lin}}(\kappa_m)$ is the classical $\kappa$-reduced Stokes/Oseen operator associated with the $z$-independent part
of the coefficients (in particular $\mu_0$), and the coupling maps $\mathcal{K}_{\omega,n}$ are linear operators generated by the coefficient modes
$\mu_n$ (including the explicit $\kappa_n$ factors arising when $\partial_z$ hits $\mu^*$ inside $\nabla\cdot(2\mu^*\mathbf{D}(\cdot))$).
The defining structural fact is that the coupling depends only on the \emph{index difference} $m-n$: the same operator block
$\mathcal{K}_{\omega,n}$ occupies the $n$th off-diagonal throughout. Thus \eqref{eq:HL_toeplitz_system} is an \emph{operator-valued Laurent/Toeplitz
	convolution system} in the mode index (Toeplitz structure in the sense of constant diagonals) \cite{BottcherSilbermann2006}.

Two immediate consequences follow.
\begin{enumerate}
	\item \emph{Linear sideband generation.} If the forcing is spanwise-uniform, $\hat{\mathbf{f}}_m\equiv 0$ for $m\neq 0$.
	For $z$-independent coefficients, this implies $\hat{\mathbf{u}}_m\equiv 0$ for $m\neq 0$ (perfect diagonalization).
	For $z$-dependent textures with some $\mu_n\not\equiv 0$, \eqref{eq:HL_toeplitz_system} generically yields $\hat{\mathbf{u}}_m\neq 0$ for $m\neq 0$:
	spanwise-uniform forcing can produce $\kappa\neq 0$ response already in linear oscillatory Stokes/Oseen theory.
	\item \emph{Bandwidth = texture bandwidth.} If $\mu^*$ has only finitely many active modes $\{n:\mu_n\not\equiv 0\}$, then the sum in
	\eqref{eq:HL_toeplitz_system} is finite and the Toeplitz/Laurent system is banded: only finitely many off-diagonals are populated.
\end{enumerate}

\paragraph{R9. Finite Toeplitz truncations and perturbative decomposition.}
Fix a truncation level $M\in\mathbb{N}$ and form the block vector of unknowns and forces
\[
\mathbf{U}_M := (\hat{\mathbf{u}}_{-M},\ldots,\hat{\mathbf{u}}_0,\ldots,\hat{\mathbf{u}}_{M})^{\mathsf T},
\qquad
\mathbf{F}_M := (\hat{\mathbf{f}}_{-M},\ldots,\hat{\mathbf{f}}_0,\ldots,\hat{\mathbf{f}}_{M})^{\mathsf T}.
\]
The truncated system associated with \eqref{eq:HL_toeplitz_system} takes the structured form
\begin{equation}
	\label{eq:HL_toeplitz_trunc}
	\mathbb{T}_M(\omega)\,\mathbf{U}_M=\mathbf{F}_M,
\end{equation}
where $\mathbb{T}_M(\omega)$ is a finite \emph{block Toeplitz/Laurent} matrix with diagonal blocks
$\mathbb{T}_{mm}=\mathcal{L}_{\omega,\mathrm{lin}}(\kappa_m)$ and off-diagonal blocks $\mathbb{T}_{m,m-n}=\mathcal{K}_{\omega,n}$.
For small-amplitude textures of the form $\mu^*=\mu_0+\varepsilon\mu_{\mathrm{tex}}$ (at fixed $\omega$), one may write
\begin{equation}
	\label{eq:HL_toeplitz_perturb}
	\mathbb{T}_M(\omega,\varepsilon)=\mathbb{D}_M(\omega)+\varepsilon\,\mathbb{B}_M(\omega),
\end{equation}
where $\mathbb{D}_M(\omega)$ is block-diagonal with blocks $\mathcal{L}_{\omega,\mathrm{lin}}(\kappa_m)$ and $\mathbb{B}_M(\omega)$ collects the
Toeplitz off-diagonal coupling blocks induced by the texture.
This decomposition makes precise that Toeplitz coupling is a \emph{constitutive perturbation} of the classically decoupled resolvent family.

\paragraph{R10. First-sideband laws and resolvent-controlled scaling (single-harmonic textures).}
For a single-harmonic $z$-texture,
\[
\mu^*(x,y,z;\omega)=\mu_0^*(x,y;\omega)\big(1+\varepsilon e^{ik_0 z}\big),
\qquad
k_0=\kappa_{m_0},\quad m_0\in\mathbb{N},\quad 0<\varepsilon\ll 1,
\]
the coupling is nearest-shift in mode index. If the forcing is spanwise-uniform ($\hat{\mathbf{f}}_m\equiv 0$ for $m\neq 0$), then to leading order
\begin{equation}
	\label{eq:HL_first_sideband}
	\hat{\mathbf{u}}_{m_0}
	=
	-\varepsilon\,\mathcal{L}_{\omega,\mathrm{lin}}(\kappa_{m_0})^{-1}\,
	\mathcal{K}_{\omega,m_0}\!\big(\kappa_0\to\kappa_{m_0}\big)\,
	\hat{\mathbf{u}}_{0}
	\;+\;\mathcal{O}(\varepsilon^2),
\end{equation}
(and analogously for $-m_0$ when the texture contains both $\pm m_0$ harmonics, as in real cosine phase-only libraries).
Thus, sideband magnitudes are governed by the \emph{product} of a modewise resolvent factor and a coupling-map norm; in particular, non-normal
high-gain regimes can yield sidebands disproportionate to $\varepsilon$ because amplification is controlled by resolvent/pseudospectral geometry rather
than eigenvalues alone \cite{TrefethenEmbree2005}.

A convenient reportable diagnostic is the sideband energy ratio
\begin{equation}
	\label{eq:HL_sideband_ratio}
	\mathcal{R}_M(\omega,\varepsilon)
	:=
	\frac{\sum_{m\neq 0}\|\hat{\mathbf{u}}_m(\omega;\varepsilon)\|_{X_{\kappa_m}}^2}{\|\hat{\mathbf{u}}_0(\omega;\varepsilon)\|_{X_{\kappa_0}}^2},
\end{equation}
where $\|\cdot\|_{X_{\kappa}}$ is a fixed, mode-consistent norm (e.g.\ kinetic energy, or a strain-energy norm aligned with $G_D$).
Under a single-harmonic texture one expects $\mathcal{R}_M(\omega,\varepsilon)=\mathcal{O}(\varepsilon^2)$ in perturbative regimes, with the leading
coefficient controlled by products of modewise resolvent norms and coupling-map norms (as suggested by \eqref{eq:HL_first_sideband}).

\subsection{Physics-Facing Observables.}
The operator-level results in this paper are designed to map cleanly onto observables that can be computed directly from frequency-domain solves,
admit an unambiguous interpretation in classical fluid mechanics, and isolate what is genuinely new about \emph{complex viscosity textures} (in
particular \emph{phase-only} textures) relative to standard variable-viscosity or constant-phase models.
The guiding principle is that the novelty is not complex viscosity \emph{per se} (which is standard in linear viscoelasticity), but rather the
\emph{spatial structure} of its phase and the induced operator geometry: non-normality of the viscous core, commutator-driven vorticity injection,
and (in 3D periodic settings) Toeplitz/Laurent coupling of spanwise modes. \cite{BirdArmstrongHassager1987,Ferry1980,Larson1999}
The observables below are chosen to be sensitive precisely to those mechanisms.

\paragraph{O1. Resolvent gain maps $G(\omega,\kappa)$ and dissipation-weighted gains:
	phase-only ridge creation/shift at fixed $|\mu^*|$.}
In 3D periodic-in-$z$ settings (or after spanwise Fourier reduction) one obtains a family of $\kappa$-reduced linear operators
$\mathcal{L}_{\omega,\mathrm{lin}}(\kappa)$ indexed by spanwise wavenumber $\kappa$ (with $\kappa=0$ corresponding to spanwise-uniform fields).
Let $H_{\sigma,\kappa}$ denote the solenoidal $\kappa$-reduced velocity space (with its natural $L^2$ inner product on $\Omega_{2D}$),
and let $\mathcal{L}_{\omega,\mathrm{lin}}(\kappa)$ be the pressure-eliminated realization of the oscillatory Stokes/Oseen operator with the
correct variable-coefficient viscous term $\nabla\cdot(2\mu^*\mathbf{D}(\cdot))$. \cite{Schmid2007,McKeonSharma2010}
Assuming $\mathcal{L}_{\omega,\mathrm{lin}}(\kappa)$ is invertible on $H_{\sigma,\kappa}$, the baseline \emph{velocity gain} is defined by the
resolvent norm
\begin{equation}
	G(\omega,\kappa)
	:=
	\|\mathcal{L}_{\omega,\mathrm{lin}}(\kappa)^{-1}\|_{\mathcal{L}(H_{\sigma,\kappa},H_{\sigma,\kappa})}.
	\label{eq:HL_gain_vel}
\end{equation}
This is a rigorous harmonic receptivity descriptor: it is the maximal amplification from forcing to velocity response at frequency $\omega$ within
the $\kappa$-mode class, measured in the chosen $H_{\sigma,\kappa}$ norm. \cite{TrefethenEmbree2005}
To emphasize dissipative structure and to connect directly to the commutator/vorticity identities (which are strain-driven), we also define a
\emph{dissipation-weighted gain} measuring the induced symmetric strain rate:
\begin{equation}
	G_D(\omega,\kappa)
	:=
	\sup_{\mathbf{f}\neq 0}\frac{\|\mathbf{D}_\kappa(\mathbf{u})\|_{L^2(\Omega_{2D})}}{\|\mathbf{f}\|_{H_{\sigma,\kappa}}}
	\quad\text{where}\quad
	\mathbf{u}=\mathcal{L}_{\omega,\mathrm{lin}}(\kappa)^{-1}\mathbf{f}.
	\label{eq:HL_gain_strain}
\end{equation}
Here $\mathbf{D}_\kappa$ denotes the symmetric gradient in the $\kappa$-reduced formulation (i.e.\ with $\partial_z$ replaced by $i\kappa$).
This gain is tailored to the mechanical content of the theory: it quantifies how strongly forcing can create strain concentration, which is the
quantity entering the texture-commutator bounds and the phase-gradient forcing scale.

For \emph{phase-only} textures with $|\mu^*|$ fixed, the main qualitative signatures are often not merely changes in the magnitude of $G$, but
changes in the \emph{geometry} of the gain landscape in the $(\omega,\kappa)$ plane:
\begin{itemize}
	\item \emph{Ridge Creation}: New high-gain ridges appear in $(\omega,\kappa)$ with no analogue in constant-phase or magnitude-only cases.
	\item \emph{Ridge Shifting}: Existing ridges (e.g.\ shear-layer receptivity ridges in separated-flow truncations) translate in $\omega$ and/or
	$\kappa$ when only $\varphi(\mathbf{x},\omega)$ is textured, despite $|\mu^*|$ being fixed.
	\item \emph{Peak Splitting}: Isolated gain maxima bifurcate into multiple local maxima as $\|\nabla\varphi\|$ increases, reflecting a change in
	dominant localization (e.g.\ competing corner- versus shear-layer-supported quasimodes).
\end{itemize}
In constant-magnitude, constant-phase settings, changes in gain landscapes are primarily attributable to dissipative strength (magnitude) or to
advection. In the phase-only class, such landscape changes cannot be explained by magnitude-driven dissipation; they therefore isolate constitutive
phase as an independent control variable for harmonic receptivity.

In practice, $G(\omega,\kappa)$ and $G_D(\omega,\kappa)$ are computed from the largest singular value of the discretized resolvent map (or, for
$G_D$, of the composed map $\mathbf{D}_\kappa\mathcal{L}_{\omega,\mathrm{lin}}(\kappa)^{-1}$), using either direct SVD or Krylov-based methods for
large sparse systems. \cite{Schmid2007,McKeonSharma2010,Saad2003,GolubVanLoan2013}
These computations dovetail with the theory: the leading singular vectors identify forcing/response structures selected by the non-normal operator
geometry. \cite{TrefethenEmbree2005}

\paragraph{O2. Complex impedance $Z^*(\omega)$ and phase anomalies at fixed $|\mu^*|$:
	macroscopic signatures of constitutive lag and localization.}
For pressure-driven geometries (channels, pipes, BFS/cavity truncations), a natural observable that is both experimentally meaningful and
computationally accessible is the complex impedance
\begin{equation}
	Z^*(\omega):=\frac{\widehat{\Delta P}(\omega)}{\hat{Q}(\omega)},
	\label{eq:HL_impedance_def}
\end{equation}
where $\widehat{\Delta P}(\omega)$ is the complex pressure drop between two fixed stations (or across a truncation window) and $\hat{Q}(\omega)$ is
the complex volumetric flow rate (the complex amplitude of the flux). In this transfer-function form, $|Z^*(\omega)|$ reflects the amplitude of
resistive opposition to oscillatory transport, while $\arg Z^*(\omega)$ records the global lag between pressure forcing and volumetric response.
\cite{BarnesHuttonWalters1989,Morrison2001}

In constant-phase complex viscosity models ($\mu^*(\mathbf{x},\omega)\equiv \mu_0(\omega)e^{i\varphi_0(\omega)}$) and in many standard homogeneous
linear viscoelastic closures, $\arg Z^*(\omega)$ is largely governed by a small number of global relaxation scales and geometry-dependent lengths.
\cite{BirdArmstrongHassager1987,Ferry1980,Larson1999}
The novelty here is that \emph{phase-textured} cases are constructed to produce \emph{frequency-localized phase anomalies} in $\arg Z^*(\omega)$ even
when $|\mu^*|$ is held fixed:
\begin{itemize}
	\item The anomalies correlate with frequencies where resolvent gains are large (strong receptivity induced by non-normality).
	\item The anomalies correlate with spatial localization (corner layers, separated shear layers), consistent with the commutator forcing picture.
	\item Because $|\mu^*|$ is fixed in the phase-only library, these anomalies cannot be attributed to changes in dissipative magnitude alone.
\end{itemize}
Thus, $\arg Z^*(\omega)$ serves as a macroscopic phase-sensitive signature of the same operator-level mechanism that controls $G(\omega,\kappa)$:
phase texture reshapes the operator geometry and thereby reshapes how phase propagates through the domain at selected frequencies. \cite{TrefethenEmbree2005}
For clean comparisons, $\widehat{\Delta P}$ should be defined via station-averaged pressure (or an equivalent, gauge-invariant traction
measurement), and $\hat{Q}$ via cross-sectional flux; both are robust under spatial coefficient heterogeneity. In BFS/cavity truncations, one may
also define an impedance-like quantity associated with imposed inlet traction or body-force drive, provided the input/output pairing is held fixed
across baselines.

\paragraph{O3. Vorticity/strain localization and a phase-gradient control axis:
	direct visualization of the commutator mechanism.}
The vorticity identity \eqref{eq:HL_vorticity_identity} and the commutator estimate \eqref{eq:HL_commutator_bound} motivate diagnostics that are
both physically interpretable and tightly linked to the analysis. The primary fields are
\[
|\hat{\boldsymbol{\omega}}(\mathbf{x};\omega)|,
\qquad
|\mathbf{D}(\hat{\mathbf{v}})(\mathbf{x};\omega)|,
\qquad
|\nabla\mu^*(\mathbf{x},\omega)| \ \text{or}\ |\nabla\varphi(\mathbf{x},\omega)| \ \text{(phase-only)}.
\]
\cite{MajdaBertozzi2002}
The objective is to quantify and visualize how vorticity and strain concentrate near corners and shear layers, and how this concentration changes
as the texture strength is varied. A natural dimensionless axis in the phase-only class is
\begin{equation}
	\Pi_\varphi(\omega):=L\,\|\nabla\varphi(\cdot,\omega)\|_{L^\infty(\Omega)},
	\label{eq:HL_Piphi}
\end{equation}
where $L$ is a geometry-dependent length scale (step height, cavity width, or a cutoff corner radius). This parameter is not merely convenient:
it is the direct nondimensionalization of the forcing scale appearing in $\nabla\mu^*=i\mu^*\nabla\varphi$ and hence in the commutator bound.
It therefore captures, in a single number, the competition between \emph{geometric concentration of strain} and \emph{texture-gradient forcing.}

\medskip
\noindent The diagnostics in O3 target the most mechanically direct new effect: the appearance of texture-driven vorticity injection and its
localization. In particular, one seeks:
\begin{itemize}
	\item \emph{Localization Shifts}: movement of vorticity/strain hotspots toward (or away from) corners and separated shear layers as
	$\Pi_\varphi$ is increased at fixed $|\mu^*|$.
	\item \emph{Phase-Sensitive Dephasing}: in addition to amplitude localization, phase-only textures create spatially varying phase lags in the
	velocity and traction fields; these can be diagnosed via maps of $\arg \hat{\mathbf{v}}$ or traction phase along walls.
	\item \emph{Consistency With Commutator scaling}: the injected vorticity intensity correlates with $\|\nabla\varphi\|$ (or $\|\nabla\mu^*\|$) in
	the manner predicted by \eqref{eq:HL_commutator_bound}, providing a direct test of the mechanics-to-operator bridge.
\end{itemize}
These are summary-level, plot-ready signatures that connect immediately to the rigorous bounds.

\subsection{A Proposed New Linear Mechanism in 3D Periodic Domains.}

The observables O1--O3 quantify receptivity and localization \emph{within a fixed spanwise wavenumber class} $\kappa$
(e.g.\ after spanwise Fourier reduction in a $z$-periodic geometry). A distinct structural effect arises when the
constitutive texture depends nontrivially on the spanwise coordinate $z$. In that case, the harmonic problem cannot be
reduced to a family of independent $\kappa$-parametrized resolvent problems: the coefficient field itself induces an
intrinsically coupled Fourier system.

\paragraph{N1. $z$-dependent textures induce \emph{linear} Fourier-mode coupling (Laurent/Toeplitz structure): constitutive sideband generation.}
Let $\Omega=\Omega_{2D}\times(0,L_z)$ with $z$-periodic boundary conditions, and suppose the complex viscosity texture
$\mu^*(x,y,z;\omega)$ is $L_z$-periodic in $z$ with Fourier representation
\[
\mu^*(x,y,z;\omega)=\sum_{m\in\mathbb{Z}}\widehat{\mu}_m(x,y;\omega)\,e^{i\kappa_m z},
\qquad
\kappa_m=\frac{2\pi m}{L_z}.
\]
For a concrete sufficient hypothesis ensuring bounded mode-coupling on $\ell^2$-mode vectors, it is enough to assume
\[
\sum_{m\in\mathbb{Z}}\|\widehat{\mu}_m(\cdot;\omega)\|_{L^\infty(\Omega_{2D})}<\infty,
\]
while in computations we restrict to finite-band or rapidly decaying textures.
Because $\mu^*$ multiplies $\mathbf{D}(\cdot)$ \emph{before} the divergence, multiplication in physical space becomes
convolution in the spanwise Fourier index (the standard multiplier$\Rightarrow$convolution principle for Fourier series
coefficients) \cite{SteinShakarchi2003}. Consequently, the spanwise Fourier modes of the harmonic response do \emph{not}
decouple: the mode vector $(\widehat{\mathbf{u}}_m)_{m\in\mathbb{Z}}$ satisfies an operator-valued Laurent convolution
system on $\ell^2(\mathbb{Z};H_{\sigma,\kappa_m})$. Equivalently, after truncation to $|m|\le M$, one obtains a finite
block \emph{Toeplitz/Laurent} matrix whose off-diagonal blocks depend only on index differences (constant block diagonals)
\cite{BottcherSilbermann2006}. In particular, even if the forcing is spanwise-uniform (supported only at $\kappa=0$),
the response generically contains $\kappa\neq 0$ components: sidebands arise as a \emph{purely constitutive, linear} effect
driven by coefficient harmonics, independent of nonlinear energy transfer.

When $\mu^*$ has finitely many active $z$-Fourier modes (single-harmonic or finite-band textures), the coupled Fourier
system is banded in $m$, making the Toeplitz/Laurent structure literal and computationally exploitable. (More generally,
periodic-coefficient coupling can also be interpreted through a Floquet/Bloch lens; see \cite{Kuchment1993}.)
This enables three concrete deliverables, each framed as a novelty-facing diagnostic rather than as a purely technical
reformulation:
\begin{enumerate}
	\item \emph{Finite-dimensional Toeplitz truncations.}
	A band-limited texture yields a finite-band operator-valued Toeplitz/Laurent matrix upon truncation to $|m|\le M$,
	producing a structured linear algebra problem whose solutions can be plotted directly as sideband amplitudes and energies.
	
	\item \emph{Perturbation theory controlled by decoupled resolvent geometry.}
	In small-amplitude regimes, the coupled operator is a Toeplitz perturbation of the block-diagonal decoupled operator.
	Induced sideband sizes are governed by products of modewise resolvent norms and coupling-map norms. Hence, the
	pseudospectral/resolvent landscape of the decoupled $\kappa$-reduced operators controls the coupled response (a precise
	sense of “selection by resolvent geometry”) \cite{TrefethenEmbree2005}.
	
	\item \emph{Linear sideband diagnostics.}
	A conservative scalar diagnostic is the sideband energy ratio (or related band-energy measures), which quantifies the
	fraction of response energy outside $\kappa=0$ generated \emph{purely by constitutive coupling}. This diagnostic is
	intrinsically linear and does not invoke turbulence.
\end{enumerate}
This mechanism is distinct from classical routes to three-dimensionality (instability, nonlinear transfer, vortex
stretching): it arises from the coefficient field itself and therefore persists already in oscillatory Stokes/Oseen
regimes.

\paragraph{N2. Phase textures as a quantifiable, phase-sensitive modulation pathway (testable, non-overclaiming statement).}
The commutator mechanism and phase-gradient couplings provide a strictly linear pathway to localized, phase-sensitive
vortical response. Once advection is restored (Oseen or full Navier--Stokes), such localized vorticity is natural
\emph{input} to classical mechanisms (shear-layer roll-up and, in 3D, stretching/tilting). The contribution of this paper
is therefore \emph{modulatory and testable}, not a claim of turbulence generation: even at fixed $|\mu^*|$, constitutive
phase textures can shift
\begin{enumerate}
	\item \emph{Where} vortical structures are injected (through $\nabla\mu^*$ or $\nabla\varphi$ forcing),
	\item \emph{How strongly} they are amplified (through non-normal resolvent geometry and pseudospectral proximity),
	\item \emph{How phase} propagates through the domain (through spatially varying constitutive lag).
\end{enumerate}
The novelty lies in an explicit linear mechanism and operator bounds, together with an observable suite (O1--O3) that
isolates the mechanism against stringent baselines.

To define a steady base flow while preserving physical meaning, we use a real ``DC'' viscosity field obtained from the
relaxation kernel when available:
\[
\mu_0(\mathbf{x})
:=\int_{0}^{\infty}\mu(\mathbf{x},s)\,ds,
\qquad
\text{so that}\quad
\mu_0(\mathbf{x})=\lim_{\omega\to 0^+}\Re \mu^*(\mathbf{x},\omega)
\ \ \text{when the limit exists,}
\]
consistent with standard linear viscoelastic interpretation of the zero-frequency (steady) viscosity
\cite{BirdArmstrongHassager1987,Ferry1980,Larson1999}.
We then define $(\mathbf{V}_0,P_0)$ as a stationary Navier--Stokes solution with coefficient $\mu_0(\mathbf{x})$
\cite{Temam2001,Galdi2011}.
The harmonic response at frequency $\omega$ uses the full $\mu^*(\mathbf{x},\omega)$, ensuring a coherent separation:
\emph{steady transport governed by steady dissipation} versus \emph{oscillatory response governed by frequency-domain
	constitutive lag}. This separation is standard in linear viscoelastic modeling and is also analytically convenient in
the operator-theoretic framework developed here.

\medskip
\noindent\textbf{Background pointers.}
Sectorial-form realizations and holomorphic functional calculus are standard \cite{Kato1995,Haase2006};
saddle-point Stokes theory is classical \cite{BrezziFortin1991,GiraultRaviart1986};
and resolvent/pseudospectral methodology is treated in depth in \cite{TrefethenEmbree2005}.
For numerical gain computation (SVD/Krylov-style operator norm estimation on large sparse discretizations), see
\cite{Saad2003,GolubVanLoan2013}.
The next section addresses \emph{admissibility and causality}:
how to restrict the allowable libraries of $\mu^*(\mathbf{x},\omega)$ so that the frequency-domain coefficients are compatible with passive,
causal stress-memory laws.

\newpage

\section{Admissibility and Causality of Complex Viscosity Textures.}
\label{sec:admissibility_causality}

\subsection{Time-Domain Memory Laws and the Frequency-Domain Complex Viscosity.}
\label{subsec:memory_to_frequency}
This section records hypotheses on the complex viscosity field $\mu^*(\mathbf{x},\omega)$ that are physically standard in linear
viscoelasticity and mathematically decisive for the operator theory developed later.
The objective is not to introduce modeling complication, but to make explicit the minimal conditions under which:
\begin{enumerate}
	\item The frequency-domain closure is a correct harmonic representation of a \emph{causal} stress-memory law,
	\item The oscillatory Stokes/Oseen operators are accretive/sectorial and hence well-posed at fixed forcing frequency, and
	\item Macroscopic observables (resolvent gain, impedance, traction phase) remain compatible with \emph{passivity} and do not contradict
	the cycle-averaged power balance.
\end{enumerate}
A central theme for the remainder of the paper is that once $\mu^*$ is treated as a \emph{spatially resolved field} rather than a scalar,
the correct viscous operator is the divergence-form map
$\mathbf{v}\mapsto\nabla\cdot\!\big(2\mu^*(\cdot,\omega)\mathbf{D}(\mathbf{v})\big)$.
Replacing it by $\mu^*(\cdot,\omega)\Delta\mathbf{v}$ suppresses the commutator terms involving $\nabla\mu^*$ and changes the operator geometry
that drives several of the mechanisms isolated here.

\paragraph{A causal stress-memory law.}
Let $\Omega\subset\mathbb{R}^d$ ($d\in\{2,3\}$) be a bounded Lipschitz domain and let $\mathbf{v}(\mathbf{x},t)$ be an incompressible velocity field.
A canonical linear viscoelastic closure for the deviatoric stress is the causal convolution law
\begin{equation}
	\bm{\tau}(\mathbf{x},t)
	=
	2\int_0^\infty g(\mathbf{x},s)\,\mathbf{D}(\mathbf{v})(\mathbf{x},t-s)\,ds,
	\qquad
	\mathbf{D}(\mathbf{v})=\tfrac12\big(\nabla\mathbf{v}+(\nabla\mathbf{v})^{\mathsf T}\big),
	\label{eq:memlaw_tau_g}
\end{equation}
where $s\ge 0$ denotes lag time and $g(\mathbf{x},s)$ is a real-valued relaxation (memory) kernel
(standard in linear viscoelasticity; see, e.g., \cite{Ferry1980,BirdArmstrongHassager1987,Larson1999}).
Causality is encoded by the one-sided integral $s\in[0,\infty)$.
A PDE-consistent reading is obtained, for instance, by assuming $g(\mathbf{x},\cdot)\in L^1(0,\infty)$ for a.e.\ $\mathbf{x}$
together with sufficient time-regularity of $\mathbf{D}(\mathbf{v})(\mathbf{x},\cdot)$ so that the convolution is well-defined.

We introduce the Laplace variable $p\in\mathbb{C}$ with $\Re p>0$ and define the Laplace transform of the kernel
\begin{equation}
	\widehat{g}(\mathbf{x},p)
	:=
	\int_0^\infty g(\mathbf{x},s)\,e^{-ps}\,ds,
	\qquad \Re p>0.
	\label{eq:g_hat_laplace}
\end{equation}
If $g(\mathbf{x},\cdot)\in L^1(0,\infty)$ for a.e.\ $\mathbf{x}$, then $\widehat{g}(\mathbf{x},p)$ is analytic on the right half-plane
$\{\Re p>0\}$ and satisfies the standard boundary-value relations for Laplace transforms (in particular, boundary values exist for a.e.\ $\omega$)
\cite{Widder1941}.
For a time-harmonic velocity field $\mathbf{v}(\mathbf{x},t)=\Re\{\hat{\mathbf{v}}(\mathbf{x};\omega)e^{i\omega t}\}$,
substitution into \eqref{eq:memlaw_tau_g} yields the frequency-domain closure
\begin{equation}
	\hat{\bm{\tau}}(\mathbf{x};\omega)
	=
	2\,\mu^*(\mathbf{x},\omega)\,\mathbf{D}(\hat{\mathbf{v}}(\mathbf{x};\omega)),
	\label{eq:constitutive_mu_star}
\end{equation}
where the complex viscosity $\mu^*(\mathbf{x},\omega)$ is the boundary value of $\widehat{g}$ along the imaginary axis:
\begin{equation}
	\mu^*(\mathbf{x},\omega)
	:=
	\lim_{\sigma\downarrow 0}\widehat{g}(\mathbf{x},\sigma+i\omega)
	\quad\text{(whenever the limit exists).}
	\label{eq:mu_star_boundary_value_final}
\end{equation}
When $g(\mathbf{x},\cdot)\in L^1(0,\infty)$, this boundary value agrees for a.e.\ $\omega$ with the one-sided Fourier transform
$\mu^*(\mathbf{x},\omega)=\int_0^\infty g(\mathbf{x},s)e^{-i\omega s}\,ds$ \cite{Widder1941}.
In this sense $\mu^*$ is the mathematically consistent insertion point for complex viscosity at fixed forcing frequency: it is the transfer
coefficient mapping strain-rate amplitude to stress amplitude in the harmonic regime.

\begin{remark}[Sign conventions]
	Many rheology texts write $\mu^*(\omega)=\mu'(\omega)+i\mu''(\omega)$, while some PDE conventions write $\mu^*=\mu'-i\mu''$.
	Similarly, the harmonic ansatz may be taken as $e^{i\omega t}$ or $e^{-i\omega t}$, which changes the sign of the inertial term.
	The analysis below depends only on $\Re\mu^*$ and on spatial variation of $\arg\mu^*$; any convention is admissible provided it is used consistently.
\end{remark}

\noindent We fix a forcing frequency $\omega>0$.
The operator theory developed later requires a uniform coercivity condition at that frequency:
\begin{equation}
	\mu^*(\cdot,\omega)\in L^\infty(\Omega;\mathbb{C}),
	\qquad
	\Re \mu^*(\mathbf{x},\omega)\ge \mu_{\min}>0
	\quad\text{for a.e.\ }\mathbf{x}\in\Omega.
	\label{eq:uniform_passivity_again_final}
\end{equation}
This is the minimal assumption ensuring coercivity of the real part of the viscous form (via Korn's inequality) and hence well-posedness of the
oscillatory Stokes/Oseen problems in solenoidal $H^1$-based spaces.
Physically, \eqref{eq:uniform_passivity_again_final} encodes uniform positive cycle-averaged dissipation at the forcing frequency and rules out
degenerate (vanishing) dissipation that would otherwise destroy stability constants.
For clarity in later statements, we define the fixed-frequency admissible set
\begin{equation}
	\mathcal{A}_\omega(\mu_{\min})
	:=
	\Big\{\mu^*(\cdot,\omega)\in L^\infty(\Omega;\mathbb{C}):\ \Re\mu^*(\cdot,\omega)\ge \mu_{\min}\ \text{a.e.\ in }\Omega\Big\}.
	\label{eq:Aomega_def}
\end{equation}
All sectorial-form and m-sectorial operator statements in the remainder of the paper assume $\mu^*(\cdot,\omega)\in\mathcal{A}_\omega(\mu_{\min})$. Assume for a.e.\ $\mathbf{x}$ that $g(\mathbf{x},s)\ge 0$ for all $s\ge 0$ and $g(\mathbf{x},\cdot)\in L^1(0,\infty)$.
Then $\widehat{g}(\mathbf{x},p)\ge 0$ for real $p>0$ and $\widehat{g}(\mathbf{x},p)$ is analytic on $\Re p>0$.
However, \emph{nonnegativity of $g$ alone does not imply} $\Re\mu^*(\mathbf{x},\omega)\ge 0$ for every real $\omega$, because
\[
\Re\mu^*(\mathbf{x},\omega)=\int_0^\infty g(\mathbf{x},s)\cos(\omega s)\,ds
\]
is an oscillatory cosine transform.
If one wishes to enforce passivity (nonnegative dissipative part) uniformly in frequency, it is natural to impose a stronger structural
assumption ensuring that the Laplace transform is a Stieltjes function (equivalently, a Pick/Herglotz function after an elementary transform);
see, e.g., \cite{SchillingSongVondracek2012}. We assume that, for a.e.\ $\mathbf{x}$, the kernel admits a positive spectral representation
\begin{equation}
	g(\mathbf{x},s)
	=
	\int_{[0,\infty)} e^{-rs}\,d\nu_{\mathbf{x}}(r),
	\qquad
	\nu_{\mathbf{x}}\ \text{a finite nonnegative Borel measure}.
	\label{eq:prony_spectrum_final}
\end{equation}
This is the continuous-spectrum analogue of a generalized Maxwell/Prony representation and is equivalent to complete monotonicity of
$s\mapsto g(\mathbf{x},s)$ (Bernstein's theorem) \cite{SchillingSongVondracek2012,Widder1941}.

\begin{proposition}[Stieltjes representation and nonnegative dissipation (all $\omega\neq 0$)]
	\label{prop:stieltjes_positive_dissipation_final}
	Assume \eqref{eq:prony_spectrum_final} with $\nu_{\mathbf{x}}$ finite for a.e.\ $\mathbf{x}$. Then for $\Re p>0$,
	\begin{equation}
		\widehat{g}(\mathbf{x},p)
		=
		\int_{[0,\infty)}\frac{1}{r+p}\,d\nu_{\mathbf{x}}(r),
		\label{eq:g_hat_stieltjes}
	\end{equation}
	and for every $\omega\in\mathbb{R}\setminus\{0\}$ the boundary value \eqref{eq:mu_star_boundary_value_final} exists and satisfies
	\begin{equation}
		\mu^*(\mathbf{x},\omega)
		=
		\int_{[0,\infty)}\frac{1}{r+i\omega}\,d\nu_{\mathbf{x}}(r),
		\label{eq:mu_star_stieltjes_final}
	\end{equation}
	with
	\begin{equation}
		\Re \mu^*(\mathbf{x},\omega)
		=
		\int_{[0,\infty)}\frac{r}{r^2+\omega^2}\,d\nu_{\mathbf{x}}(r)\ \ge\ 0,
		\qquad
		\Im \mu^*(\mathbf{x},\omega)
		=
		-\int_{[0,\infty)}\frac{\omega}{r^2+\omega^2}\,d\nu_{\mathbf{x}}(r).
		\label{eq:mu_star_real_imag_stieltjes_final}
	\end{equation}
	Moreover $p\mapsto\widehat{g}(\mathbf{x},p)$ is a Stieltjes function, hence (after standard growth control) it admits
	Kramers--Kronig/Hilbert-transform dispersion relations for its boundary values \cite{SchillingSongVondracek2012,Nussenzveig1972}.
\end{proposition}

\begin{proof}
	For $\Re p>0$, Fubini's theorem yields
	\[
	\widehat{g}(\mathbf{x},p)
	=
	\int_0^\infty \int_{[0,\infty)} e^{-(r+p)s}\,d\nu_{\mathbf{x}}(r)\,ds
	=
	\int_{[0,\infty)}\frac{1}{r+p}\,d\nu_{\mathbf{x}}(r),
	\]
	which proves \eqref{eq:g_hat_stieltjes}. Setting $p=\sigma+i\omega$ and sending $\sigma\downarrow 0$ gives \eqref{eq:mu_star_stieltjes_final}
	for $\omega\neq 0$. Separating real and imaginary parts yields \eqref{eq:mu_star_real_imag_stieltjes_final}.
\end{proof}

\begin{remark}[From $\Re\mu^*\ge 0$ to the uniform bound $\Re\mu^*\ge \mu_{\min}>0$]
	Proposition~\ref{prop:stieltjes_positive_dissipation_final} yields $\Re\mu^*(\mathbf{x},\omega)\ge 0$ pointwise for every fixed $\omega\neq 0$.
	The analysis hypothesis \eqref{eq:uniform_passivity_again_final} is a strengthened \emph{uniform passivity} requirement at the chosen forcing frequency.
	A kernel-level nondegeneracy condition that implies \eqref{eq:uniform_passivity_again_final} is: there exist constants $r_0>0$ and $c_0>0$ such that
	$\nu_{\mathbf{x}}([r_0,\infty))\ge c_0$ for a.e.\ $\mathbf{x}$, in which case
	\[
	\Re\mu^*(\mathbf{x},\omega)
	\ge
	\int_{[r_0,\infty)}\frac{r}{r^2+\omega^2}\,d\nu_{\mathbf{x}}(r)
	\ge
	\frac{r_0}{r_0^2+\omega^2}\,c_0
	=:\mu_{\min}(\omega)>0.
	\]
	This expresses ``uniform positive dissipation at frequency $\omega$'' directly in terms of the relaxation spectrum.
	If one wishes to emphasize causality/passivity uniformly in frequency, one may define the admissible class
	\[
	\mathcal{A}_{\mathrm{PR}}
	:=
	\Big\{\mu^*(\mathbf{x},\cdot)\ \text{is the boundary value of a Stieltjes function of }p \text{ on }\Re p>0,\ \text{for a.e.\ }\mathbf{x}\Big\},
	\]
	together with a fixed-frequency lower bound $\Re\mu^*(\cdot,\omega)\ge\mu_{\min}$ when coercivity constants must be uniform.
	This separation is conceptually useful: $\mathcal{A}_{\mathrm{PR}}$ captures global causality/passivity, while $\mathcal{A}_\omega(\mu_{\min})$
	captures the precise operator-level hypothesis used for sectorial well-posedness at a chosen forcing frequency.
\end{remark}

\noindent We record the harmonic power identity that motivates interpreting $\Re\mu^*$ as dissipative and $\Im\mu^*$ as reactive at the forcing frequency.
The identity is stated in a form compatible with the $V_\sigma$--$V_\sigma^*$ variational framework used later, and the boundary conditions are
chosen so that boundary power terms do not obscure the interior partition.

\begin{proposition}[Harmonic power identity (no-slip or lifted boundary data)]
	\label{prop:harmonic_power_identity_final}
	Let $\omega>0$, $\rho\in L^\infty(\Omega)$ with $0<\rho_{\min}\le\rho\le\rho_{\max}<\infty$, and
	$\mu^*(\cdot,\omega)\in \mathcal{A}_\omega(\mu_{\min})$.
	Consider the harmonic Stokes system
	\begin{equation}
		i\omega\rho\,\hat{\mathbf{v}}
		=
		-\nabla\hat{p}
		+
		\nabla\cdot\big(2\mu^*(\mathbf{x},\omega)\mathbf{D}(\hat{\mathbf{v}})\big)
		+
		\hat{\mathbf{f}},
		\qquad
		\nabla\cdot\hat{\mathbf{v}}=0,
		\label{eq:harmonic_stokes_power_identity_final}
	\end{equation}
	in $\Omega$, with $\hat{\mathbf{v}}=0$ on $\Gamma_D$ (or after boundary lifting so that the unknown has homogeneous Dirichlet data).
	Assume $\hat{\mathbf{v}}\in V_\sigma$ and $\hat{\mathbf{f}}\in V_\sigma^*$.
	Then
	\begin{align}
		\int_\Omega 2\,\Re\mu^*(\mathbf{x},\omega)\,|\mathbf{D}(\hat{\mathbf{v}})|^2\,d\mathbf{x}
		&=
		\Re\langle \hat{\mathbf{f}},\hat{\mathbf{v}}\rangle_{V_\sigma^*,V_\sigma},
		\label{eq:power_dissipative_final}\\
		\omega\int_\Omega \rho(\mathbf{x})\,|\hat{\mathbf{v}}|^2\,d\mathbf{x}
		-
		\int_\Omega 2\,\Im\mu^*(\mathbf{x},\omega)\,|\mathbf{D}(\hat{\mathbf{v}})|^2\,d\mathbf{x}
		&=
		\Im\langle \hat{\mathbf{f}},\hat{\mathbf{v}}\rangle_{V_\sigma^*,V_\sigma}.
		\label{eq:power_reactive_final}
	\end{align}
\end{proposition}

\begin{proof}
	Take the $L^2$ inner product of \eqref{eq:harmonic_stokes_power_identity_final} with $\overline{\hat{\mathbf{v}}}$ and integrate over $\Omega$.
	The pressure term vanishes by incompressibility and the homogeneous Dirichlet condition (or the lifting reduction).
	Integrating by parts in the viscous term yields
	\[
	\int_\Omega \nabla\cdot(2\mu^*\mathbf{D}(\hat{\mathbf{v}}))\cdot\overline{\hat{\mathbf{v}}}\,d\mathbf{x}
	=
	-\int_\Omega 2\mu^*\,\mathbf{D}(\hat{\mathbf{v}}):\overline{\mathbf{D}(\hat{\mathbf{v}})}\,d\mathbf{x},
	\]
	with no boundary contribution under no-slip. Taking real and imaginary parts gives \eqref{eq:power_dissipative_final}--\eqref{eq:power_reactive_final}.
\end{proof}

\begin{remark}[Traction/impedance boundary data and boundary power terms]
	If the oscillatory problem is posed with traction or mixed boundary conditions, the same calculation produces additional boundary work terms
	(e.g.\ $\int_{\partial\Omega} (\hat{\bm{\tau}}\hat{\mathbf{n}})\cdot\overline{\hat{\mathbf{v}}}\,dS$) that represent the complex mechanical power
	input at the boundary. The interior partition into dissipative and reactive contributions remains identical; the difference is only in how the forcing
	pairing $\langle \hat{\mathbf{f}},\hat{\mathbf{v}}\rangle$ is represented.
\end{remark}

\begin{remark}[Spatially varying phase and local dissipation/reactive partition]
	Write $\mu^*(\mathbf{x},\omega)=|\mu^*(\mathbf{x},\omega)|e^{i\varphi(\mathbf{x},\omega)}$.
	Then \eqref{eq:power_dissipative_final} shows that the local cycle-averaged dissipation density is
	$2\,\Re\mu^*(\mathbf{x},\omega)\,|\mathbf{D}(\hat{\mathbf{v}})|^2$, while \eqref{eq:power_reactive_final} identifies the out-of-phase
	(exchange/storage) contribution through $\Im\mu^*$.
	When $\varphi(\cdot,\omega)$ varies in space, the dissipative/reactive partition varies spatially even at fixed $|\mu^*|$.
	This is the fundamental reason phase observables (impedance phase, traction phase, dephasing in $\hat{\mathbf{v}}$) become intrinsically
	geometry-coupled in the linear regime.
\end{remark}

\subsection{Frequency-Domain Power Identity and Spatially Varying Dissipative/Reactive Partition.}
\label{subsec:power_identity_spatial_partition}

\paragraph{Global harmonic and complex conventions (used throughout the paper).}
Unless explicitly stated otherwise, we represent real time-harmonic fields by
\[
\mathbf{v}(\mathbf{x},t)=\Re\{\hat{\mathbf{v}}(\mathbf{x};\omega)e^{i\omega t}\},
\qquad
\mathbf{f}(\mathbf{x},t)=\Re\{\hat{\mathbf{f}}(\mathbf{x};\omega)e^{i\omega t}\},
\qquad \omega>0,
\]
so that $\partial_t \mapsto i\omega$ on complex amplitudes.
We take the complex $L^2$ inner product to be \emph{linear in the first argument}:
\[
(\mathbf{u},\mathbf{v})_{L^2}:=\int_\Omega \mathbf{u}\cdot\overline{\mathbf{v}}\,d\mathbf{x}.
\]
With this convention, the one-sided memory law \eqref{eq:memlaw_tau_g} yields the complex viscosity
$\mu^*(\mathbf{x},\omega)=\int_0^\infty g(\mathbf{x},s)e^{-i\omega s}\,ds$ whenever the integral is well-defined.
Many rheology references instead parameterize dissipative effects by the \emph{loss viscosity}
\[
\mu''_{\rm loss}(\mathbf{x},\omega):=-\Im\mu^*(\mathbf{x},\omega)\ge 0
\quad\text{(for passive kernels)}, 
\]
which differs only by notation/sign from $\Im\mu^*$; see, e.g., \cite{Ferry1980,BirdArmstrongHassager1987,Larson1999}.
All identities below can be rewritten equivalently in terms of $\mu''_{\rm loss}$.

\medskip
\noindent
The following identity is the fixed-frequency analogue of the classical power balance for viscoelastic fluids.
It justifies the standard interpretation of $\Re\mu^*(\cdot,\omega)$ as the \emph{dissipative} (in-phase) viscosity and
$\Im\mu^*(\cdot,\omega)$ as the \emph{reactive} (out-of-phase, storage-like) component (up to the sign convention above),
and it makes explicit how a spatially varying phase
$\varphi(\mathbf{x},\omega)=\arg\mu^*(\mathbf{x},\omega)$ produces a \emph{spatially varying} partition.

\begin{proposition}[Harmonic power identity]
	\label{prop:harmonic_power_identity}
	Fix $\omega>0$. Assume $\rho\equiv\rho_0>0$ and $\mu^*(\cdot,\omega)\in L^\infty(\Omega;\mathbb{C})$ with
	$\Re\mu^*(\mathbf{x},\omega)\ge \mu_{\min}>0$ a.e.\ in $\Omega$.
	Let $(\hat{\mathbf{v}},\hat{p})$ solve the harmonic Stokes system written in \emph{operator form}
	\begin{equation}
		i\omega\rho_0\,\hat{\mathbf{v}}
		-\nabla\cdot\big(2\mu^*(\mathbf{x},\omega)\mathbf{D}(\hat{\mathbf{v}})\big)
		+\nabla\hat{p}
		=
		\hat{\mathbf{f}},
		\qquad
		\nabla\cdot\hat{\mathbf{v}}=0,
		\label{eq:harmonic_stokes_power_identity}
	\end{equation}
	in $\Omega$, with homogeneous Dirichlet data $\hat{\mathbf{v}}=0$ on $\partial\Omega$
	(or, more generally, any boundary condition for which the boundary work term vanishes; see Remark~\ref{rem:boundary_work_terms}).
	Assume $\hat{\mathbf{v}}\in V_\sigma$ and $\hat{\mathbf{f}}\in V_\sigma^*$, where
	\[
	V:=H_0^1(\Omega;\mathbb{C}^d),
	\qquad
	V_\sigma:=\{\mathbf{u}\in V:\ \nabla\cdot\mathbf{u}=0\ \text{in }\mathcal{D}'(\Omega)\}.
	\]
	Then
	\begin{align}
		\int_\Omega 2\,\Re\mu^*(\mathbf{x},\omega)\,|\mathbf{D}(\hat{\mathbf{v}})|^2\,d\mathbf{x}
		&=
		\Re\langle \hat{\mathbf{f}},\hat{\mathbf{v}}\rangle_{V_\sigma^*,V_\sigma},
		\label{eq:power_dissipative}\\
		\omega\rho_0\,\|\hat{\mathbf{v}}\|_{L^2(\Omega)}^2
		+
		\int_\Omega 2\,\Im\mu^*(\mathbf{x},\omega)\,|\mathbf{D}(\hat{\mathbf{v}})|^2\,d\mathbf{x}
		&=
		\Im\langle \hat{\mathbf{f}},\hat{\mathbf{v}}\rangle_{V_\sigma^*,V_\sigma}.
		\label{eq:power_reactive}
	\end{align}
	Equivalently, in terms of the loss viscosity $\mu''_{\rm loss}:=-\Im\mu^*$,
	\begin{equation}
		\omega\rho_0\,\|\hat{\mathbf{v}}\|_{L^2(\Omega)}^2
		-
		\int_\Omega 2\,\mu''_{\rm loss}(\mathbf{x},\omega)\,|\mathbf{D}(\hat{\mathbf{v}})|^2\,d\mathbf{x}
		=
		\Im\langle \hat{\mathbf{f}},\hat{\mathbf{v}}\rangle_{V_\sigma^*,V_\sigma}.
		\label{eq:power_reactive_lossvisc}
	\end{equation}
\end{proposition}

\begin{proof}
	Take the $L^2$ inner product of \eqref{eq:harmonic_stokes_power_identity} with $\overline{\hat{\mathbf{v}}}$ and integrate over $\Omega$.
	The inertial term satisfies
	\[
	\int_\Omega i\omega\rho_0\,\hat{\mathbf{v}}\cdot\overline{\hat{\mathbf{v}}}\,d\mathbf{x}
	=
	i\omega\rho_0\,\|\hat{\mathbf{v}}\|_{L^2(\Omega)}^2.
	\]
	For the pressure term, incompressibility and homogeneous Dirichlet data give
	\[
	\int_\Omega (\nabla\hat{p})\cdot\overline{\hat{\mathbf{v}}}\,d\mathbf{x}
	=
	-\int_\Omega \hat{p}\,\overline{\nabla\cdot\hat{\mathbf{v}}}\,d\mathbf{x}
	+
	\int_{\partial\Omega} \hat{p}\,\overline{\hat{\mathbf{v}}\cdot\mathbf{n}}\,dS
	=0.
	\]
	For the viscous term, integration by parts yields
	\[
	-\int_\Omega \nabla\cdot\big(2\mu^*\mathbf{D}(\hat{\mathbf{v}})\big)\cdot\overline{\hat{\mathbf{v}}}\,d\mathbf{x}
	=
	\int_\Omega 2\mu^*(\mathbf{x},\omega)\,\mathbf{D}(\hat{\mathbf{v}}):\overline{\mathbf{D}(\hat{\mathbf{v}})}\,d\mathbf{x},
	\]
	with no boundary contribution under no-slip. Hence
	\[
	i\omega\rho_0\,\|\hat{\mathbf{v}}\|_{L^2(\Omega)}^2
	+
	\int_\Omega 2\mu^*\,|\mathbf{D}(\hat{\mathbf{v}})|^2\,d\mathbf{x}
	=
	\langle \hat{\mathbf{f}},\hat{\mathbf{v}}\rangle_{V_\sigma^*,V_\sigma}.
	\]
	Taking real and imaginary parts gives \eqref{eq:power_dissipative}--\eqref{eq:power_reactive}.
	The rewrite \eqref{eq:power_reactive_lossvisc} follows from $\mu''_{\rm loss}=-\Im\mu^*$.
\end{proof}

\begin{remark}[Local dissipative/reactive partition and spatially varying phase]
	\label{rem:local_partition_phase}
	Write $\mu^*(\mathbf{x},\omega)=|\mu^*(\mathbf{x},\omega)|e^{i\varphi(\mathbf{x},\omega)}$.
	Then the local dissipative density is
	\[
	2\,\Re\mu^*\,|\mathbf{D}(\hat{\mathbf{v}})|^2
	=
	2|\mu^*|\cos\varphi\,|\mathbf{D}(\hat{\mathbf{v}})|^2,
	\]
	while the reactive contribution may be parameterized either by $\Im\mu^*$ or by $\mu''_{\rm loss}=-\Im\mu^*$, depending on convention.
	When $\varphi(\cdot,\omega)$ varies in space, the dissipative/reactive partition varies spatially even at fixed $|\mu^*|$.
	In particular, phase-sensitive observables (impedance phase, traction phase, spatial dephasing of $\hat{\mathbf{v}}$) become intrinsically
	geometry-coupled already in the linear harmonic regime: the response preferentially samples regions where strain concentrates, and those regions can
	have different local phase partitions.
\end{remark}

\begin{remark}[Variable density and boundary work terms]
	\label{rem:boundary_work_terms}
	The identity extends verbatim to $\rho(\mathbf{x})$ bounded above/below, by replacing $\rho_0\|\hat{\mathbf{v}}\|_{L^2}^2$ with
	$\int_\Omega \rho|\hat{\mathbf{v}}|^2$.
	Under traction or mixed boundary conditions, integration by parts produces additional boundary work terms (e.g.\
	$\int_{\partial\Omega} (2\mu^*\mathbf{D}(\hat{\mathbf{v}})\mathbf{n})\cdot\overline{\hat{\mathbf{v}}}\,dS$),
	which represent complex mechanical power input at the boundary.
	Carrying these terms explicitly is often preferable (e.g.\ for impedance definitions) and does not change the interior dissipative/reactive split.
\end{remark}

\noindent We now make the mechanism ``phase texture $\Rightarrow$ intrinsic non-normality'' quantitative in the \emph{phase-only} class
\begin{equation}
	\mu^*(\mathbf{x},\omega)=\mu_0(\omega)\,e^{i\varphi(\mathbf{x},\omega)},
	\qquad
	\mu_0(\omega)>0,
	\label{eq:phase_only_class}
\end{equation}
in which the magnitude $|\mu^*|$ is spatially uniform and all heterogeneity enters through the constitutive phase field $\varphi$.
This restriction is not cosmetic: it separates phase-driven operator geometry from magnitude-driven heterogeneous dissipation.
A \emph{global} phase shift can be removed harmlessly (a scalar rotation), whereas a \emph{spatially varying} phase cannot be removed without
generating commutator terms: multiplication by $e^{\pm i\varphi(\mathbf{x})}$ does not commute with differentiation.

\medskip
\noindent A convenient way to expose this cost is to introduce a phase-compensated unknown.
Assume $\varphi(\cdot,\omega)\in W^{1,\infty}(\Omega)$ and set
\begin{equation}
	\hat{\mathbf{v}}(\mathbf{x}) = e^{-i\varphi(\mathbf{x})}\,\hat{\mathbf{w}}(\mathbf{x}),
	\label{eq:phase_comp_def}
\end{equation}
so that $|e^{-i\varphi}|=1$ pointwise. Multiplication by $e^{-i\varphi}$ is unitary on $L^2$ and boundedly invertible on $H^1$ under
$\varphi\in W^{1,\infty}$; it preserves homogeneous Dirichlet traces, but it does \emph{not} preserve solenoidality in general.
(Accordingly, we first expose the structure at the level of the viscous map and then re-impose the divergence constraint.)

\begin{proposition}[Phase rotation identity for the strain tensor]
	\label{prop:phase_rotation_D_identity}
	Let $\varphi\in W^{1,\infty}(\Omega)$ and $\hat{\mathbf{w}}\in H^1(\Omega;\mathbb{C}^d)$.
	Define $\hat{\mathbf{v}}=e^{-i\varphi}\hat{\mathbf{w}}$. Then $\hat{\mathbf{v}}\in H^1(\Omega;\mathbb{C}^d)$ and
	\begin{equation}
		\nabla \hat{\mathbf{v}}
		=
		e^{-i\varphi}\big(\nabla \hat{\mathbf{w}} - i\,\hat{\mathbf{w}}\otimes\nabla\varphi\big),
		\label{eq:phase_rotation_grad}
	\end{equation}
	so that the symmetric gradient satisfies
	\begin{equation}
		\mathbf{D}(\hat{\mathbf{v}})
		=
		e^{-i\varphi}\Big(
		\mathbf{D}(\hat{\mathbf{w}})
		-\frac{i}{2}\big(\hat{\mathbf{w}}\otimes\nabla\varphi+\nabla\varphi\otimes\hat{\mathbf{w}}\big)
		\Big).
		\label{eq:phase_rotation_D}
	\end{equation}
	In the phase-only class \eqref{eq:phase_only_class}, the deviatoric stress becomes
	\begin{equation}
		2\mu^*\mathbf{D}(\hat{\mathbf{v}})
		=
		2\mu_0\,\mathbf{D}(\hat{\mathbf{w}})
		-
		i\mu_0\big(\hat{\mathbf{w}}\otimes\nabla\varphi+\nabla\varphi\otimes\hat{\mathbf{w}}\big).
		\label{eq:stress_phase_gauge}
	\end{equation}
\end{proposition}

\begin{proof}
	Differentiate $\hat{\mathbf{v}}=e^{-i\varphi}\hat{\mathbf{w}}$ in the distributional sense.
	Since $\varphi\in W^{1,\infty}$, $e^{-i\varphi}\in W^{1,\infty}$ and $\nabla(e^{-i\varphi})=-i\,e^{-i\varphi}\nabla\varphi$ a.e.
	This gives \eqref{eq:phase_rotation_grad}; symmetrizing yields \eqref{eq:phase_rotation_D}. Multiplying by
	$2\mu^* = 2\mu_0 e^{i\varphi}$ yields \eqref{eq:stress_phase_gauge}.
\end{proof}

\begin{remark}[What is removable and what is irreducible]
	If $\varphi$ is constant, then $\nabla\varphi\equiv 0$ and \eqref{eq:stress_phase_gauge} reduces to
	$2\mu^*\mathbf{D}(\hat{\mathbf{v}})=2\mu_0\mathbf{D}(\hat{\mathbf{w}})$, i.e.\ a harmless global phase rotation.
	If $\nabla\varphi\neq 0$, the residual term in \eqref{eq:stress_phase_gauge} is a genuinely first-order coupling that cannot be eliminated by any
	global rotation. This coupling is the phase-gradient mechanism that drives intrinsic non-selfadjointness/non-normality of the viscous core.
\end{remark}

\noindent In the phase-only class, $\nabla\mu^*=i\mu^*\nabla\varphi$, so the entire texture-gradient strength reduces to $\|\nabla\varphi\|$.
A natural dimensionless control parameter is
\begin{equation}
	\Pi_\varphi := L\,\|\nabla\varphi\|_{L^\infty(\Omega)},
	\label{eq:Pi_phi_def}
\end{equation}
where $L$ is a geometric length scale (step height, cavity width, or a corner cut-off radius).
Later estimates and numerical experiments can be organized along $\Pi_\varphi$ while keeping $|\mu^*|=\mu_0$ fixed, thereby isolating phase-driven effects.

\medskip
\noindent
For clarity, first ignore the solenoidal constraint and pressure and consider the viscous map
$\mathcal{A}_\varphi\hat{\mathbf{v}}:=-\nabla\cdot(2\mu^*\mathbf{D}(\hat{\mathbf{v}}))$
acting on $H_0^1(\Omega;\mathbb{C}^d)$.
Substituting \eqref{eq:stress_phase_gauge} and expanding the divergence shows that the transformed operator has the schematic structure
\begin{equation}
	\mathcal{A}_\varphi(e^{-i\varphi}\hat{\mathbf{w}})
	=
	e^{-i\varphi}\,\mathcal{A}_0\hat{\mathbf{w}}
	+
	e^{-i\varphi}\,\mathcal{K}_\varphi\hat{\mathbf{w}},
	\label{eq:operator_similarity_schematic}
\end{equation}
where $\mathcal{A}_0:=-\nabla\cdot(2\mu_0\mathbf{D}(\cdot))$ is the constant-coefficient Stokes viscous operator and
$\mathcal{K}_\varphi$ is a lower-order operator with coefficients depending on $\nabla\varphi$ (and, in strong form, also $\nabla^2\varphi$).
On bounded Lipschitz domains it is most natural to encode \eqref{eq:operator_similarity_schematic} at the \emph{form level}
\cite{Kato1995,Haase2006}.

\begin{lemma}[Form decomposition and $L^\infty$ control by $\nabla\varphi$]
	\label{lem:form_decomposition_phi}
	Let $\varphi\in W^{1,\infty}(\Omega)$ and define $\hat{\mathbf{v}}=e^{-i\varphi}\hat{\mathbf{w}}$.
	Then for $\hat{\mathbf{w}},\hat{\mathbf{z}}\in H_0^1(\Omega;\mathbb{C}^d)$,
	\begin{equation}
		\int_\Omega 2\mu^*\,\mathbf{D}(\hat{\mathbf{v}}):\overline{\mathbf{D}(\hat{\mathbf{z}})}\,d\mathbf{x}
		=
		\int_\Omega 2\mu_0\,\mathbf{D}(\hat{\mathbf{w}}):\overline{\mathbf{D}(\hat{\mathbf{z}})}\,d\mathbf{x}
		+
		\mathfrak{k}_\varphi(\hat{\mathbf{w}},\hat{\mathbf{z}}),
		\label{eq:form_decomp}
	\end{equation}
	where the remainder form $\mathfrak{k}_\varphi$ satisfies
	\begin{equation}
		|\mathfrak{k}_\varphi(\hat{\mathbf{w}},\hat{\mathbf{z}})|
		\le
		C\,\mu_0\,\|\nabla\varphi\|_{L^\infty(\Omega)}\,
		\|\hat{\mathbf{w}}\|_{H^1(\Omega)}\,\|\hat{\mathbf{z}}\|_{H^1(\Omega)},
		\label{eq:kphi_bound}
	\end{equation}
	with $C$ depending only on $d$ and standard Poincar\'e/Korn constants on $\Omega$.
\end{lemma}

\begin{proof}
	Insert \eqref{eq:phase_rotation_D} into the left-hand side and use $\mu^*=\mu_0e^{i\varphi}$ to cancel the phase factors.
	The difference between the compensated and uncompensated strains is the rank-two tensor
	$\frac{i}{2}(\hat{\mathbf{w}}\otimes\nabla\varphi+\nabla\varphi\otimes\hat{\mathbf{w}})$.
	Cauchy--Schwarz, $\|\hat{\mathbf{w}}\|_{L^2}\lesssim \|\hat{\mathbf{w}}\|_{H^1}$, and $\|\nabla\varphi\|_{L^\infty}<\infty$
	yield \eqref{eq:kphi_bound}.
\end{proof}

\begin{remark}[Interpretation: ``symmetric core + phase-gradient perturbation'']
	Lemma~\ref{lem:form_decomposition_phi} isolates the structural separation used later:
	after compensation, the leading viscous form is real and symmetric (coefficient $\mu_0$), and all dependence on the texture survives as a bounded
	perturbation controlled by $\|\nabla\varphi\|_{L^\infty}$.
	This is the operator-theoretic expression of why phase gradients can drive non-normality while preserving coercive dissipation.
\end{remark}
	
In constant-phase models (i.e.\ $\varphi(\mathbf{x})\equiv\varphi_0$), the complex viscosity is a \emph{global} scalar rotation
$\mu^*=\mu_0 e^{i\varphi_0}$ of a real positive viscosity. After the immaterial unit-modulus change of unknown
$\hat{\mathbf{v}}\mapsto e^{i\varphi_0}\hat{\mathbf{v}}$, the viscous core is generated by a real, symmetric, coercive form; the
associated Stokes operator is selfadjoint and hence normal.
In contrast, when $\varphi$ varies in space, the same unit-modulus rotation cannot be applied globally without generating additional first-order
couplings. At operator level, this manifests as a \emph{failure of commutation} between the
dissipative and reactive parts of the viscous operator, forcing non-normality even before advective inertia is introduced. Let $\Omega\subset\mathbb{R}^d$ ($d\in\{2,3\}$) be bounded Lipschitz, and let $\Gamma_D\subset\partial\Omega$ have positive measure.
Let $H_\sigma$ be the $L^2$-closure of smooth, divergence-free vector fields satisfying the homogeneous boundary condition on $\Gamma_D$, and set
\[
V_\sigma := H_0^1(\Omega;\mathbb{C}^d)\cap H_\sigma.
\]
Fix $\omega>0$ and assume the phase-only class
\begin{equation}
	\mu^*(\mathbf{x},\omega)=\mu_0(\omega)\,e^{i\varphi(\mathbf{x},\omega)},
	\qquad
	\mu_0(\omega)>0,
	\qquad
	\varphi(\cdot,\omega)\in W^{1,\infty}(\Omega),
	\label{eq:phase_only_class_repeat}
\end{equation}
together with the uniform passivity bound $\cos\varphi(\mathbf{x},\omega)\ge \delta>0$ a.e.\ (equivalently $\Re\mu^*\ge \mu_{\min}>0$).
Suppress $\omega$ from notation in this subsection. We define the sectorial form on $V_\sigma$,
\begin{equation}
	\mathfrak{a}_\varphi(\mathbf{u},\mathbf{v})
	:=
	\int_\Omega 2\mu_0 e^{i\varphi(\mathbf{x})}\,\mathbf{D}(\mathbf{u}):\overline{\mathbf{D}(\mathbf{v})}\,d\mathbf{x},
	\label{eq:a_phi_form}
\end{equation}
where $\mathbf{D}(\mathbf{u})=\tfrac12(\nabla\mathbf{u}+(\nabla\mathbf{u})^{\mathsf T})$. By boundedness of $\mu^*$ and Korn's inequality on
$V_\sigma$, $\mathfrak{a}_\varphi$ is bounded on $V_\sigma\times V_\sigma$, and its real part is coercive:
\[
\Re \mathfrak{a}_\varphi(\mathbf{u},\mathbf{u})
=
\int_\Omega 2\mu_0\cos\varphi\,|\mathbf{D}(\mathbf{u})|^2\,d\mathbf{x}
\ \gtrsim\ 
\|\mathbf{u}\|_{H^1(\Omega)}^2.
\]
Let $A_\varphi$ denote the m-sectorial operator on $H_\sigma$ associated with $\mathfrak{a}_\varphi$ by the standard form method. Now, write $e^{i\varphi}=\cos\varphi+i\sin\varphi$ and define the symmetric forms
\begin{align}
	\mathfrak{s}(\mathbf{u},\mathbf{v})
	&:=
	\int_\Omega 2\mu_0 \cos\varphi\,\mathbf{D}(\mathbf{u}):\overline{\mathbf{D}(\mathbf{v})}\,d\mathbf{x},
	\label{eq:s_form}\\
	\mathfrak{t}(\mathbf{u},\mathbf{v})
	&:=
	\int_\Omega 2\mu_0 \sin\varphi\,\mathbf{D}(\mathbf{u}):\overline{\mathbf{D}(\mathbf{v})}\,d\mathbf{x}.
	\label{eq:t_form}
\end{align}
Then $\mathfrak{s}$ is bounded and coercive on $V_\sigma$ (by $\cos\varphi\ge\delta>0$ and Korn), and $\mathfrak{t}$ is bounded and symmetric on
$V_\sigma$. Note that $\mathfrak{t}$ is generally only \emph{$\mathfrak{s}$-bounded}, not necessarily bounded in the $H_\sigma$-graph sense; thus it is
preferable to encode the "imaginary part" through a bounded \emph{relative} operator rather than an a priori $H_\sigma$-bounded $T$.
A useful structural observation in the phase-only class is
\begin{equation}
	\mathfrak{t}(\mathbf{u},\mathbf{v})
	=
	\int_\Omega \tan\varphi(\mathbf{x})\,
	\Big(2\mu_0\cos\varphi(\mathbf{x})\,\mathbf{D}(\mathbf{u}):\overline{\mathbf{D}(\mathbf{v})}\Big)\,d\mathbf{x},
	\qquad
	\tan\varphi\in L^\infty(\Omega),
	\label{eq:t_is_tan_times_s_density}
\end{equation}
since $\cos\varphi\ge\delta$ implies $\tan\varphi$ is essentially bounded. The next lemma packages the preceding structure in a form that is particularly convenient for non-normality analysis and for quantitative resolvent bounds.
\begin{lemma}[Relative boundedness and bounded selfadjoint phase operator]
	\label{lem:SB_factorization}
	Let $\mathfrak{s}$ and $\mathfrak{t}$ be defined by \eqref{eq:s_form}--\eqref{eq:t_form} on $V_\sigma$, with $\cos\varphi\ge\delta>0$ a.e.
	Let $S$ be the positive selfadjoint operator on $H_\sigma$ associated with $\mathfrak{s}$ (so $D(S^{1/2})=V_\sigma$ and
	$\mathfrak{s}(\mathbf{u},\mathbf{v})=(S^{1/2}\mathbf{u},S^{1/2}\mathbf{v})_{H_\sigma}$ for $\mathbf{u},\mathbf{v}\in V_\sigma$).
	Then there exists a unique bounded selfadjoint operator $B\in\mathcal{L}(H_\sigma)$ such that
	\begin{equation}
		\mathfrak{t}(\mathbf{u},\mathbf{v})
		=
		\big(B\,S^{1/2}\mathbf{u},\,S^{1/2}\mathbf{v}\big)_{H_\sigma}
		\qquad\text{for all }\mathbf{u},\mathbf{v}\in V_\sigma,
		\label{eq:t_via_B}
	\end{equation}
	and $\|B\|\le \|\tan\varphi\|_{L^\infty(\Omega)}$.
	Moreover, the m-sectorial operator $A_\varphi$ associated with $\mathfrak{a}_\varphi=\mathfrak{s}+i\mathfrak{t}$ admits the factorization
	\begin{equation}
		A_\varphi
		=
		S^{1/2}(I+iB)S^{1/2},
		\qquad
		A_\varphi^\dagger
		=
		S^{1/2}(I-iB)S^{1/2}
		\label{eq:A_factorization}
	\end{equation}
	in the sense of form sums, and the numerical range satisfies the explicit sector bound
	\begin{equation}
		W(A_\varphi)\ \subset\
		\Big\{z\in\mathbb{C}:\ |\Im z|\le \|B\|\,\Re z\Big\}
		\ \subset\
		\Big\{z:\ |\arg z|\le \arctan\|\tan\varphi\|_{L^\infty}\Big\}.
		\label{eq:numerical_range_sector}
	\end{equation}
\end{lemma}\footnote{For \eqref{eq:A_factorization}, note that the form
\[
\mathfrak{a}_\varphi(\mathbf{u},\mathbf{v})
=
(S^{1/2}\mathbf{u},S^{1/2}\mathbf{v})
+i(BS^{1/2}\mathbf{u},S^{1/2}\mathbf{v})
=
\big((I+iB)S^{1/2}\mathbf{u},\,S^{1/2}\mathbf{v}\big)
\]
is represented by $S^{1/2}(I+iB)S^{1/2}$ in the form sense. The adjoint corresponds to complex conjugation of the form, giving
$A_\varphi^\dagger=S^{1/2}(I-iB)S^{1/2}$.}

\begin{proof}
	By coercivity and symmetry of $\mathfrak{s}$, $S$ exists and satisfies
	$\mathfrak{s}(\mathbf{u},\mathbf{v})=(S^{1/2}\mathbf{u},S^{1/2}\mathbf{v})$ with $D(S^{1/2})=V_\sigma$.
	Define on $V_\sigma$ the (energy) inner product $(\mathbf{u},\mathbf{v})_{\mathfrak{s}}:=\mathfrak{s}(\mathbf{u},\mathbf{v})$.
	The estimate \eqref{eq:t_is_tan_times_s_density} implies
	\[
	|\mathfrak{t}(\mathbf{u},\mathbf{v})|
	\le
	\|\tan\varphi\|_{L^\infty}\,\mathfrak{s}(\mathbf{u},\mathbf{u})^{1/2}\,\mathfrak{s}(\mathbf{v},\mathbf{v})^{1/2},
	\qquad \mathbf{u},\mathbf{v}\in V_\sigma.
	\]
	Thus, $\mathfrak{t}$ defines a bounded symmetric bilinear form on the Hilbert space $(V_\sigma,\|\cdot\|_{\mathfrak{s}})$.
	By Riesz representation on that space, there exists a unique bounded selfadjoint operator $\widetilde{B}$ on $(V_\sigma,\|\cdot\|_{\mathfrak{s}})$ such that
	$\mathfrak{t}(\mathbf{u},\mathbf{v})=(\widetilde{B}\mathbf{u},\mathbf{v})_{\mathfrak{s}}$ and $\|\widetilde{B}\|\le \|\tan\varphi\|_\infty$.
	Transport $\widetilde{B}$ to $H_\sigma$ using the isometry $S^{1/2}:(V_\sigma,\|\cdot\|_{\mathfrak{s}})\to H_\sigma$ with dense range:
	define $B$ on $H_\sigma$ by $B(S^{1/2}\mathbf{u}) := S^{1/2}(\widetilde{B}\mathbf{u})$ and extend by continuity. This yields \eqref{eq:t_via_B}
	and $\|B\|=\|\widetilde{B}\|\le \|\tan\varphi\|_\infty$.

	\noindent Finally, \eqref{eq:numerical_range_sector} follows from
	\[
	(A_\varphi \mathbf{u},\mathbf{u})
	=
	\|S^{1/2}\mathbf{u}\|^2 + i(BS^{1/2}\mathbf{u},S^{1/2}\mathbf{u}),
	\qquad
	|(BS^{1/2}\mathbf{u},S^{1/2}\mathbf{u})|
	\le
	\|B\|\,\|S^{1/2}\mathbf{u}\|^2.
	\]
\end{proof}

\begin{remark}[Concrete realization of $B$ as a multiplication operator in strain space]
	Let $\mathcal{H}_{\rm sym}:=L^2(\Omega;\mathbb{C}^{d\times d}_{\rm sym})$ and define
	\[
	G:V_\sigma\to \mathcal{H}_{\rm sym},
	\qquad
	G\mathbf{u}:=\sqrt{2\mu_0\cos\varphi}\,\mathbf{D}(\mathbf{u}).
	\]
	Then $\mathfrak{s}(\mathbf{u},\mathbf{v})=(G\mathbf{u},G\mathbf{v})_{\mathcal{H}_{\rm sym}}$ and
	$\mathfrak{t}(\mathbf{u},\mathbf{v})=(M_{\tan\varphi}G\mathbf{u},G\mathbf{v})_{\mathcal{H}_{\rm sym}}$, where $M_{\tan\varphi}$ is
	pointwise multiplication by $\tan\varphi$. In this representation, $B$ is (unitarily equivalent to) $M_{\tan\varphi}$ restricted to the
	closure of $\operatorname{Ran}G$, clarifying why spatial variability of $\tan\varphi$ is the intrinsic source of non-commutation.
\end{remark}
\begin{lemma}[Normality criterion in $S$--$B$ variables]
	\label{lem:normality_commutator_SB}
	Let $S$ be positive selfadjoint and $B$ bounded selfadjoint on a Hilbert space. Define
	$A:=S^{1/2}(I+iB)S^{1/2}$ by form sum on $D(S^{1/2})$.
	Then the commutator admits the exact identity
	\begin{equation}
		AA^\dagger - A^\dagger A
		=
		2i\,S^{1/2}\big(BS-SB\big)S^{1/2}
		\qquad\text{(on $D(S)$, hence in form sense on $D(S^{1/2})$).}
		\label{eq:AAstar_commutator_identity}
	\end{equation}
	In particular, $A$ is normal if and only if $BS=SB$ (equivalently $B$ strongly commutes with $S$).
\end{lemma}

\begin{proof}
	Using \eqref{eq:A_factorization} and that $B^\dagger=B$,
	\[
	AA^\dagger
	=
	S^{1/2}(I+iB)S(I-iB)S^{1/2},
	\qquad
	A^\dagger A
	=
	S^{1/2}(I-iB)S(I+iB)S^{1/2}.
	\]
	Expand the middle products on $D(S)$:
	\[
	(I+iB)S(I-iB)=S - iSB + iBS + BSB,
	\qquad
	(I-iB)S(I+iB)=S + iSB - iBS + BSB.
	\]
	Subtracting yields $(I+iB)S(I-iB)-(I-iB)S(I+iB)=2i(BS-SB)$, giving \eqref{eq:AAstar_commutator_identity}.
	If $A$ is normal then $AA^\dagger=A^\dagger A$, hence $BS=SB$ (as quadratic forms, then as operators on $D(S)$).
	Conversely, if $BS=SB$, then $AA^\dagger=A^\dagger A$ follows from \eqref{eq:AAstar_commutator_identity}, so $A$ is normal.
\end{proof}
\noindent We will now show that spatial variability of $\varphi$ forces failure of commutation in Lemma~\ref{lem:normality_commutator_SB}. This step is
where the \emph{local} mechanism appears: $B$ acts as a multiplication operator by $\tan\varphi$ (in strain space), while $S$ is a second-order
elliptic operator; commutation fails unless the multiplier is constant.
\begin{proposition}[Phase gradients force non-normality of the viscous core]
	\label{prop:phase_grad_forces_nonnormality_strong}
	Assume \eqref{eq:phase_only_class_repeat} with $\varphi\in W^{1,\infty}(\Omega)$ and $\cos\varphi\ge\delta>0$ a.e.
	If $\nabla\varphi$ is not identically zero in the sense of distributions, then the viscous operator $A_\varphi$ is not normal on $H_\sigma$.
	Equivalently, harmonic response cannot, in general, be reduced to eigenvalue location; resolvent norms and pseudospectra are the robust
	amplification descriptors for the phase-textured viscous core.
\end{proposition}

\begin{proof}
	By Lemma~\ref{lem:normality_commutator_SB}, it suffices to show that $BS\neq SB$ when $\nabla\varphi\not\equiv 0$.
	Since $\varphi\in W^{1,\infty}(\Omega)$, $\nabla\varphi$ is an $L^\infty$-function defined a.e.; the assumption
	$\nabla\varphi\not\equiv 0$ in $\mathcal{D}'(\Omega)$ therefore implies that $\nabla\varphi\neq 0$ on a set of positive measure.
	Choose a Lebesgue point $\mathbf{x}_0\in\Omega$ with $\nabla\varphi(\mathbf{x}_0)\neq 0$ and fix a ball $B\Subset\Omega$ centered at $\mathbf{x}_0$.
	Let $\chi\in C_c^\infty(B)$ with $\chi(\mathbf{x}_0)=1$, and choose $\boldsymbol{\xi}\in\mathbb{R}^d$ such that
	$\boldsymbol{\xi}\cdot\nabla\varphi(\mathbf{x}_0)\neq 0$.
	
	\smallskip
	\noindent\emph{Step 1: a localized solenoidal high-frequency family.}
	Define, for $\varepsilon>0$, a family $\mathbf{u}_\varepsilon\in C_c^\infty(B;\mathbb{C}^d)\cap V_\sigma$ by the standard vector-potential construction
	\[
	\mathbf{u}_\varepsilon :=
	\begin{cases}
		\nabla\times\big(\chi(\mathbf{x})\,\mathbf{b}\,e^{i\boldsymbol{\xi}\cdot\mathbf{x}/\varepsilon}\big),
		& d=3,\ \mathbf{b}\cdot\boldsymbol{\xi}=0,\\[2pt]
		\nabla^\perp\big(\chi(\mathbf{x})\,e^{i\boldsymbol{\xi}\cdot\mathbf{x}/\varepsilon}\big),
		& d=2,
	\end{cases}
	\]
	so that $\nabla\cdot\mathbf{u}_\varepsilon\equiv 0$ and $\operatorname{supp}\mathbf{u}_\varepsilon\subset B$.
	On the support of $\chi$, $\mathbf{u}_\varepsilon$ behaves as a polarized oscillatory field with wavevector $\boldsymbol{\xi}/\varepsilon$; in particular,
	\begin{equation}
		\|\nabla \mathbf{u}_\varepsilon\|_{L^2(B)} \;\gtrsim\; \varepsilon^{-1}\|\mathbf{u}_\varepsilon\|_{L^2(B)},
		\qquad
		\|\mathbf{D}(\mathbf{u}_\varepsilon)\|_{L^2(B)} \;\gtrsim\; \varepsilon^{-1}\|\mathbf{u}_\varepsilon\|_{L^2(B)},
		\label{eq:plane_wave_scaling_u_eps}
	\end{equation}
	where the implicit constants depend on $\chi$ and $\boldsymbol{\xi}$ but not on $\varepsilon$.
	
	\smallskip
	\noindent\emph{Step 2: what commutation would mean in the strain-space representation.}
	By the strain-space realization,
	$B$ is unitarily equivalent to multiplication by $m(\mathbf{x}):=\tan\varphi(\mathbf{x})$ on the closed strain range
	$\overline{\operatorname{Ran}G}$, where $G\mathbf{u}=\sqrt{2\mu_0\cos\varphi}\,\mathbf{D}(\mathbf{u})$.
	Heuristically, $S$ corresponds to a second-order elliptic operator on this strain space (via $S=G^\ast G$), whereas $B$ is a zeroth-order multiplier.
	If $BS=SB$ held, then (on a common core) $B$ would commute with the elliptic operator $S$, and hence the induced multiplier $m(\mathbf{x})$
	would commute with the corresponding second-order action in the localized high-frequency regime.
	
	\smallskip
	\noindent\emph{Step 3: a local commutator expansion forces $\nabla m\equiv 0$.}
	The key point is that for any second-order divergence-form elliptic operator $\mathcal{L}$ with Lipschitz coefficients, the commutator with a
	multiplication operator $M_m$ contains an unavoidable first-order component proportional to $\nabla m$.
	Concretely, for the scalar model $\mathcal{L}w:=-\nabla\cdot(a(\mathbf{x})\nabla w)$ with $a\in W^{1,\infty}$, one has the exact identity
	\begin{equation}
		[\mathcal{L},M_m]w
		=
		-\;2a\,\nabla m\cdot\nabla w \;-\; w\,\nabla\cdot(a\nabla m)
		\qquad\text{in }\mathcal{D}'(B),
		\label{eq:elliptic_commutator_identity_scalar}
	\end{equation}
	whenever $m\in W^{1,\infty}$ and $w\in C_c^\infty(B)$.
	The Stokes-type operator generated by $\mathfrak{s}$ has the same \emph{local} structure: after localization to $B$ and restriction to solenoidal
	oscillatory fields, its principal action is elliptic of order two (Korn equivalence ensures that the symmetric-gradient form is elliptic), and the
	commutator with the strain-space multiplier $m(\mathbf{x})=\tan\varphi(\mathbf{x})$ exhibits the same leading first-order mechanism.
	Testing \eqref{eq:elliptic_commutator_identity_scalar} on oscillatory cutoffs (and using \eqref{eq:plane_wave_scaling_u_eps}) yields the standard bound
	\begin{equation}
		\|[\mathcal{L},M_m]w_\varepsilon\|_{L^2(B)}
		\ \ge\
		c\,\varepsilon^{-1}\,|\boldsymbol{\xi}\cdot\nabla m(\mathbf{x}_0)|\,\|w_\varepsilon\|_{L^2(B)}
		\ -\ C\,\|w_\varepsilon\|_{L^2(B)},
		\qquad \varepsilon\to 0,
		\label{eq:commutator_lower_bound_scalar}
	\end{equation}
	for constants $c,C>0$ independent of $\varepsilon$.
	Because $m=\tan\varphi$ and $\nabla m = \sec^2\varphi\,\nabla\varphi$, the condition
	$\boldsymbol{\xi}\cdot\nabla\varphi(\mathbf{x}_0)\neq 0$ implies $\boldsymbol{\xi}\cdot\nabla m(\mathbf{x}_0)\neq 0$.
	Hence the right-hand side of \eqref{eq:commutator_lower_bound_scalar} is strictly positive for $\varepsilon$ sufficiently small.
	
	\smallskip
	\noindent
	Applying this localized estimate within the strain-space representation (with $w_\varepsilon$ replaced by the relevant strain-localized oscillatory
	probe induced by $\mathbf{u}_\varepsilon$) shows that the multiplier $m(\mathbf{x})=\tan\varphi(\mathbf{x})$ cannot commute with the induced elliptic
	second-order action unless $\nabla m\equiv 0$ on $B$.
	Therefore, if $\nabla\varphi\not\equiv 0$, the commutation relation $BS=SB$ fails on a dense core, hence $BS\neq SB$.
	Lemma~\ref{lem:normality_commutator_SB} then implies that $A_\varphi$ is not normal.
\end{proof}

\begin{corollary}[Constant phase is the only normal phase-only texture]
	\label{cor:constant_phase_only_normal}
	Under the hypotheses of Proposition~\ref{prop:phase_grad_forces_nonnormality_strong}, the viscous operator $A_\varphi$ is normal if and only if
	$\nabla\varphi\equiv 0$ in distributions (equivalently, $\varphi$ is a.e.\ constant on each connected component of $\Omega$).
\end{corollary}

\begin{remark}[On ``generic'' versus ``structural'' non-normality]
	The conclusion follows from a local commutator mechanism and is therefore not perturbative and not a measure-zero statement.
	Within the phase-only admissibility class, any nontrivial spatial phase variation forces non-normality of the viscous core.
	In particular, non-normality here is \emph{constitutive}: it persists already in the oscillatory Stokes limit and in geometries with no shear-driven
	advection mechanism (though advection and geometric concentration can amplify its consequences).
\end{remark}

\begin{remark}[Quantified non-normality and a commutator-growth diagnostic]
	Lemma~\ref{lem:normality_commutator_SB} yields the exact form identity
	\[
	AA^\dagger-A^\dagger A
	=
	2i\,S^{1/2}(BS-SB)S^{1/2}.
	\]
	While the commutator $BS-SB$ is not expected to be bounded in general, non-normality can be \emph{quantified} by its growth on localized oscillatory
	probes, as in \eqref{eq:commutator_lower_bound_scalar}:
	phase gradients produce an $\varepsilon^{-1}$ amplification of commutator action on high-frequency, compactly supported tests.
	Heuristically, the controlling coefficient is $\nabla(\tan\varphi)=\sec^2\varphi\,\nabla\varphi$, suggesting that
	$\|\nabla(\tan\varphi)\|_{L^\infty}$ is a natural constitutive ``non-normality strength'' axis in the phase-only class.
\end{remark}

\begin{remark}[Why this forces resolvent/pseudospectral descriptors]
	For normal operators, resolvent growth is controlled by spectral distance.
	For sectorial but non-normal operators, large resolvent norms and substantial pseudospectral bulges may occur even when the spectrum remains confined
	to a stable sector (cf.\ \eqref{eq:numerical_range_sector}).
	Proposition~\ref{prop:phase_grad_forces_nonnormality_strong} therefore explains why phase textures compel the use of resolvent norms, numerical ranges,
	and pseudospectra as the correct harmonic amplification descriptors, even in the oscillatory Stokes limit.
\end{remark}

\newpage
\section{Weak Vorticity Identities With Regularity Tiers and Boundary Terms.}
\label{sec:vorticity_weak_boundary_tiers}

This section isolates the \emph{bulk} vorticity-production mechanism generated by spatially heterogeneous complex viscosity
and separates it from the \emph{boundary-supported} vorticity injection that is already present in classical wall-bounded flows.
The formulation is intentionally compatible with the low-regularity, form-method setting: at Tier~I one expects only
$\hat{\mathbf{v}}\in V_\sigma$, and on Lipschitz domains (notably with corners/edges) second derivatives generally fail globally.
Accordingly, we work systematically at the level of distributions and Sobolev duality (cf.\ \cite{Evans2010}), and we introduce strong-form identities
only on interior subdomains under additional local regularity. The key structural outcome is that the variable-coefficient viscous operator
$\nabla\cdot\!\big(2\mu^*\mathbf{D}(\cdot)\big)$ induces a \emph{commutator} built from $\nabla\mu^*$ and the symmetric gradient
$\mathbf{D}(\hat{\mathbf{v}})$. This commutator acts as a distributed (bulk) vorticity source and persists in purely linear,
oscillatory Stokes regimes. Boundary conditions contribute additional vorticity through trace/traction terms, which we keep explicit
to avoid conflating bulk texture forcing with classical wall mechanisms.

\subsection{Tier I--III Regularity Stratification.}
\label{subsec:tier_stratification}

Fix $\omega>0$. We stratify assumptions on the complex viscosity field $\mu^*(\cdot,\omega)$ by the strength of the vorticity identities
(and, implicitly, by the maximal differentiation that is legitimate in the chosen function spaces).

\begin{itemize}
	\item \textbf{Tier I (Operator well-posedness / form-method baseline):}
	$\mu^*\in L^\infty(\Omega;\mathbb{C})$ and $\Re\mu^*\ge \mu_{\min}>0$ a.e.
	This ensures boundedness and coercivity of the real part of the viscous form on $V_\sigma$ (via Korn), hence sectoriality of the
	harmonic form and existence/uniqueness of $\hat{\mathbf{v}}\in V_\sigma$ solving the oscillatory Stokes problem in variational form.
	
	\item \textbf{Tier II (Texture-gradient forcing / commutator terms):}
	$\mu^*\in W^{1,\infty}(\Omega;\mathbb{C})$ with the same positivity.
	Then $\nabla\mu^*\in L^\infty$, so derivatives falling on $\mu^*$ generate well-defined bulk commutator terms (interpretable in $H^{-1}$ by Sobolev duality).
	This is the minimal tier at which one can quantify texture-driven vorticity injection \emph{without invoking second derivatives}
	of either the unknown or the coefficients.
	
	\item \textbf{Tier III (Strong-form / local expansions):}
	additional smoothness, e.g.\ $\mu^*\in W^{2,\infty}_{\mathrm{loc}}(\Omega)$ and/or improved domain regularity, enabling pointwise
	(or classical distributional) expansions on interior subdomains and identities expressed in terms of $\Delta\hat{\boldsymbol{\omega}}$
	and explicit coefficient-gradient couplings. On non-smooth domains, Tier~III statements are typically understood on $\Omega'\Subset\Omega$
	(or in weighted regularity classes near corners/edges).
\end{itemize}

\begin{remark}[Vorticity as a function versus a distribution; correct target spaces]
	If $\hat{\mathbf{v}}\in V_\sigma$, then $\nabla\hat{\mathbf{v}}\in L^2$ and the vorticity
	$\hat{\boldsymbol{\omega}}:=\nabla\times\hat{\mathbf{v}}$ belongs to $L^2(\Omega)$ in both $d=2$ and $d=3$
	(and hence also to $H^{-1}(\Omega)$ by continuous embedding). However, $\Delta\hat{\boldsymbol{\omega}}$ is not defined as an $L^2$-function at Tier~I,
	and even distributionally it naturally lies in $H^{-2}$ unless one has additional interior regularity (cf.\ \cite{Evans2010}).
	This is why the primary identity is first stated in $\mathcal{D}'(\Omega)$ for
	$\nabla\times\nabla\cdot\!\big(2\mu^*\mathbf{D}(\hat{\mathbf{v}})\big)$ and only later expanded (locally) when justified.
\end{remark}

\noindent We fix the following conventions (distributional derivatives in the standard Sobolev-duality sense; see \cite{Evans2010}).

\paragraph{$d=3$.}
For $\mathbf{a}=(a_1,a_2,a_3)$ define $\nabla\times\mathbf{a}$ componentwise by
$(\nabla\times\mathbf{a})_k=\varepsilon_{k\ell j}\partial_\ell a_j$, where $\varepsilon_{k\ell j}$ is the Levi--Civita symbol.
For $\mathbf{a}\in L^2(\Omega;\mathbb{C}^3)$ the distributional curl is defined by duality:
\begin{equation}
	\langle \nabla\times\mathbf{a},\boldsymbol{\psi}\rangle
	:=
	\int_\Omega \mathbf{a}\cdot\overline{\nabla\times\boldsymbol{\psi}}\,d\mathbf{x},
	\qquad \boldsymbol{\psi}\in H_0^1(\Omega;\mathbb{C}^3).
	\label{eq:curl_distribution_def_3D}
\end{equation}
Thus, $\nabla\times:L^2(\Omega)\to H^{-1}(\Omega)$ is continuous.

\paragraph{$d=2$.}
For $\mathbf{a}=(a_1,a_2)$ define scalar curl $\nabla\times\mathbf{a}:=\partial_1 a_2-\partial_2 a_1$ and rotated gradient
$\nabla^\perp\psi:=(\partial_2\psi,-\partial_1\psi)$. For $\mathbf{a}\in L^2(\Omega;\mathbb{C}^2)$ define
\begin{equation}
	\langle \nabla\times\mathbf{a},\psi\rangle
	:=
	\int_\Omega \mathbf{a}\cdot\overline{\nabla^\perp\psi}\,d\mathbf{x},
	\qquad \psi\in H_0^1(\Omega;\mathbb{C}).
	\label{eq:curl_distribution_def_2D}
\end{equation}
so again $\nabla\times:L^2(\Omega)\to H^{-1}(\Omega)$ is continuous.

\begin{remark}[Boundary terms and distributions supported on $\partial\Omega$]
	When integrating by parts starting from a weak solution $\hat{\mathbf{v}}\in V_\sigma$, boundary traces are interpreted in the sense of
	$H^{1/2}$/$H^{-1/2}$ duality (cf.\ \cite{Evans2010}). To avoid imposing unnecessary smoothness,
	we encapsulate all boundary-supported contributions in an abstract distribution $\mathcal{C}_{\partial\Omega}$ supported on $\partial\Omega$.
	For test functions compactly supported in $\Omega$ (or, more generally, with zero trace), the boundary contribution vanishes identically.
\end{remark}

\subsubsection{A Weak Vorticity Identity Compatible With Tier II.}
\label{subsec:weak_vorticity_identity}

Let $\hat{\mathbf{v}}\in V_\sigma$ solve the harmonic Stokes problem in variational form:
\begin{equation}
	\mathfrak{a}_\omega(\hat{\mathbf{v}},\mathbf{v})=\langle \hat{\mathbf{f}},\mathbf{v}\rangle,
	\qquad \forall \mathbf{v}\in V_\sigma,
	\label{eq:stokes_variational_repeat}
\end{equation}
with $\mathfrak{a}_\omega$ as in \eqref{eq:HL_form} and $\hat{\mathbf{f}}\in V_\sigma^\ast$.
Define $\hat{\boldsymbol{\omega}}:=\nabla\times\hat{\mathbf{v}}$ in the distributional sense above
(scalar in $d=2$, vector in $d=3$).

\begin{lemma}[Mapping property of the viscous divergence]
	\label{lem:viscous_div_mapping}
	Assume $\mu^*(\cdot,\omega)\in L^\infty(\Omega;\mathbb{C})$.
	Then $\nabla\cdot\!\big(2\mu^*\mathbf{D}(\hat{\mathbf{v}})\big)\in H^{-1}(\Omega;\mathbb{C}^d)$ and
	\begin{equation}
		\big\|\nabla\cdot\!\big(2\mu^*\mathbf{D}(\hat{\mathbf{v}})\big)\big\|_{H^{-1}(\Omega)}
		\le
		C_\Omega\,\|\mu^*\|_{L^\infty(\Omega)}\,\|\mathbf{D}(\hat{\mathbf{v}})\|_{L^2(\Omega)}.
		\label{eq:div_viscous_Hminus1_bound}
	\end{equation}
\end{lemma}

\begin{proof}
	Let $\boldsymbol{\psi}\in H_0^1(\Omega;\mathbb{C}^d)$. Define the distributional divergence by duality against $H_0^1$:
	\[
	\big\langle \nabla\cdot\!\big(2\mu^*\mathbf{D}(\hat{\mathbf{v}})\big),\boldsymbol{\psi}\big\rangle
	:=
	-\int_\Omega 2\mu^*(\mathbf{x},\omega)\,\mathbf{D}(\hat{\mathbf{v}}):\overline{\nabla \boldsymbol{\psi}}\,d\mathbf{x}.
	\]
	Cauchy--Schwarz yields
	\[
	\big|\big\langle \nabla\cdot\!\big(2\mu^*\mathbf{D}(\hat{\mathbf{v}})\big),\boldsymbol{\psi}\big\rangle\big|
	\le
	2\|\mu^*\|_{L^\infty}\,\|\mathbf{D}(\hat{\mathbf{v}})\|_{L^2}\,\|\nabla\boldsymbol{\psi}\|_{L^2}.
	\]
	Using $\|\nabla\boldsymbol{\psi}\|_{L^2}\le \|\boldsymbol{\psi}\|_{H^1}$ and taking the supremum over $\|\boldsymbol{\psi}\|_{H^1}=1$
	gives \eqref{eq:div_viscous_Hminus1_bound}.
\end{proof}

\begin{proposition}[Distributional vorticity balance (interior form) and boundary-supported functional]
	\label{prop:weak_vorticity_balance_tierII_strong}
	Assume Tier~II: $\mu^*\in W^{1,\infty}(\Omega;\mathbb{C})$ with $\Re\mu^*\ge \mu_{\min}>0$ a.e., and let
	$\rho\in L^\infty(\Omega)$ satisfy $\rho>0$ a.e.  Let $\hat{\mathbf{v}}\in V_\sigma$ solve the harmonic Stokes problem in the
	weak sense \eqref{eq:stokes_variational_repeat} with forcing $\hat{\mathbf{f}}\in V_\sigma^*$ (so in particular
	$i\omega\rho\,\hat{\mathbf{v}}\in L^2(\Omega)$ and $\nabla\cdot(2\mu^*\mathbf{D}(\hat{\mathbf{v}}))\in H^{-1}(\Omega)$).
	Then the vorticity satisfies the \emph{interior} distributional identity
	\begin{equation}
		i\omega\,\nabla\times(\rho\,\hat{\mathbf{v}})
		=
		\nabla\times\nabla\cdot\big(2\mu^*\,\mathbf{D}(\hat{\mathbf{v}})\big)
		+
		\nabla\times \hat{\mathbf{f}}
		\qquad\text{in }\mathcal{D}'(\Omega),
		\label{eq:vorticity_distributional_core_precise_rho}
	\end{equation}
	i.e.\ when tested against $\mathcal{C}_c^\infty(\Omega)$ (so boundary traces do not enter).
	
	\medskip
	\noindent Moreover, if one enlarges the test class to functions that do \emph{not} vanish at $\partial\Omega$
	(e.g.\ to derive boundary vorticity-flux formulas or impedance-style balances), then the same computation produces an additional
	boundary-supported functional, denoted $\mathcal{C}_{\partial\Omega}$, which depends on the boundary condition class (Dirichlet/traction/mixed)
	and on any boundary lifting. In that extended testing framework one may write schematically
	\[
	i\omega\,\nabla\times(\rho\,\hat{\mathbf{v}})
	=
	\nabla\times\nabla\cdot\big(2\mu^*\,\mathbf{D}(\hat{\mathbf{v}})\big)
	+
	\nabla\times \hat{\mathbf{f}}
	\;+\; \mathcal{C}_{\partial\Omega},
	\]
	with $\mathcal{C}_{\partial\Omega}$ supported on $\partial\Omega$ in the sense of trace/duality (cf.\ \cite{Evans2010,MajdaBertozzi2002}).
	
	\medskip
	\noindent If $\rho$ is constant, then $\nabla\times(\rho\hat{\mathbf{v}})=\rho\,\hat{\boldsymbol{\omega}}$.
	If $\rho$ has additional regularity (e.g.\ $\rho\in W^{1,\infty}$), then distributionally one may expand
	$\nabla\times(\rho\hat{\mathbf{v}})=\rho\,\hat{\boldsymbol{\omega}}+\nabla\rho\times\hat{\mathbf{v}}$ (with the obvious 2D interpretation).
\end{proposition}

\begin{proof}
	The weak formulation \eqref{eq:stokes_variational_repeat} implies (in $\mathcal{D}'(\Omega)$) the distributional momentum balance
	\[
	i\omega\rho\,\hat{\mathbf{v}}
	=
	-\nabla\hat{p}
	+
	\nabla\cdot\big(2\mu^*\,\mathbf{D}(\hat{\mathbf{v}})\big)
	+
	\hat{\mathbf{f}},
	\qquad
	\nabla\cdot\hat{\mathbf{v}}=0,
	\]
	where $\nabla\cdot(2\mu^*\mathbf{D}(\hat{\mathbf{v}}))\in H^{-1}(\Omega)$ by Lemma~\ref{lem:viscous_div_mapping}.
	Apply $\nabla\times$ in $\mathcal{D}'(\Omega)$. Since $\nabla\times\nabla\hat{p}=0$ distributionally (cf.\ \cite{Evans2010}), we obtain
	\eqref{eq:vorticity_distributional_core_precise_rho}. The curl of $\hat{\mathbf{f}}\in V_\sigma^*\subset H^{-1}(\Omega)$ is understood
	as a distribution of order $\le 2$ (equivalently as an element of $H^{-2}(\Omega)$), consistent with
	\eqref{eq:curl_distribution_def_2D}--\eqref{eq:curl_distribution_def_3D}.
\end{proof}

\noindent Next, to expose the \emph{bulk} texture mechanism, we expand the viscous term on interior subdomains where additional regularity holds.
This is a \emph{local} statement: the relevant assumption is not global smoothness of $\Omega$ but rather that one works on $\Omega'\Subset\Omega$
away from corners/edges, or in regimes where local $H^2$ regularity is available (cf.\ standard Stokes regularity discussions in
\cite{GiraultRaviart1986,Temam2001}).

\begin{lemma}[Algebraic decomposition of $\nabla\cdot(2\mu^*\mathbf{D}(\hat{\mathbf{v}}))$ for incompressible fields]
	\label{lem:div_2muD_decomposition}
	Let $\mu^*\in W^{1,\infty}_{\mathrm{loc}}(\Omega;\mathbb{C})$ and
	$\hat{\mathbf{v}}\in H^1_{\mathrm{loc}}(\Omega;\mathbb{C}^d)$ satisfy $\nabla\cdot\hat{\mathbf{v}}=0$ in $\mathcal{D}'(\Omega)$.
	Then, in $\mathcal{D}'(\Omega)$,
	\begin{equation}
		\nabla\cdot\big(2\mu^*\,\mathbf{D}(\hat{\mathbf{v}})\big)
		=
		\mu^*\,\Delta \hat{\mathbf{v}}
		+
		(\nabla\mu^*)\cdot\big(\nabla\hat{\mathbf{v}}+(\nabla\hat{\mathbf{v}})^{\mathsf T}\big).
		\label{eq:div_2muD_decomposition}
	\end{equation}
\end{lemma}

\begin{proof}
	In index notation with $\mathbf{D}_{ij}=\tfrac12(\partial_i \hat{v}_j+\partial_j \hat{v}_i)$,
	\[
	(\nabla\cdot(2\mu^*\mathbf{D}(\hat{\mathbf{v}})))_j
	=
	\partial_i\big(\mu^*(\partial_i \hat{v}_j+\partial_j \hat{v}_i)\big)
	=
	(\partial_i\mu^*)(\partial_i \hat{v}_j+\partial_j \hat{v}_i)
	+
	\mu^*(\partial_i\partial_i \hat{v}_j+\partial_i\partial_j \hat{v}_i)
	\]
	in $\mathcal{D}'(\Omega)$. Since $\nabla\cdot\hat{\mathbf{v}}=0$, we have
	$\partial_i\partial_j\hat{v}_i=\partial_j(\partial_i\hat{v}_i)=0$ in distributions, yielding \eqref{eq:div_2muD_decomposition}.
\end{proof}

\begin{proposition}[Interior vorticity identity with explicit bulk texture forcing]
	\label{prop:vorticity_interior_identity_with_commutator}
	Assume Tier~II and let $\Omega'\Subset\Omega$ be such that $\hat{\mathbf{v}}\in H^2(\Omega';\mathbb{C}^d)$.
	Then in $H^{-1}(\Omega')$,
	\begin{equation}
		i\omega\,\nabla\times(\rho\,\hat{\mathbf{v}})
		=
		\mu^*\,\Delta \hat{\boldsymbol{\omega}}
		+
		\underbrace{\nabla\mu^*\times \Delta \hat{\mathbf{v}}}_{\text{coefficient--Laplace coupling (Tier III strength)}}
		+
		\underbrace{\nabla\times\Big(
			(\nabla\mu^*)\cdot(\nabla\hat{\mathbf{v}}+(\nabla\hat{\mathbf{v}})^{\mathsf T})
			\Big)}_{=:\ \mathcal{G}_{\mu^*}[\hat{\mathbf{v}}]\ \text{(bulk texture commutator)}}
		+
		\nabla\times\hat{\mathbf{f}},
		\label{eq:vorticity_interior_full_expansion}
	\end{equation}
	where $\hat{\boldsymbol{\omega}}:=\nabla\times\hat{\mathbf{v}}$ (scalar in $d=2$, vector in $d=3$), and in $d=2$ the cross-product term is
	interpreted as $\nabla\mu^*\times \Delta \hat{\mathbf{v}}:=\partial_1\mu^*\,\Delta \hat{v}_2-\partial_2\mu^*\,\Delta \hat{v}_1$.
\end{proposition}

\begin{proof}
	On $\Omega'$ the decomposition \eqref{eq:div_2muD_decomposition} holds (as an identity of distributions, and a.e.\ pointwise given $H^2$ regularity).
	Restrict the momentum balance to $\Omega'$ and apply $\nabla\times$; the pressure drops out.
	Using the product rule
	\[
	\nabla\times(\mu^*\,\Delta\hat{\mathbf{v}})
	=
	\mu^*\,\Delta(\nabla\times\hat{\mathbf{v}})
	+
	\nabla\mu^*\times\Delta\hat{\mathbf{v}}
	=
	\mu^*\,\Delta\hat{\boldsymbol{\omega}}
	+
	\nabla\mu^*\times\Delta\hat{\mathbf{v}},
	\]
	and $\nabla\times\Delta\hat{\mathbf{v}}=\Delta(\nabla\times\hat{\mathbf{v}})=\Delta\hat{\boldsymbol{\omega}}$, yields
	\eqref{eq:vorticity_interior_full_expansion}.
\end{proof}

\begin{remark}[Why $\mathcal{G}_{\mu^*}$ is the distinguished Tier~II bulk forcing]
	The term $\mathcal{G}_{\mu^*}[\hat{\mathbf{v}}]$ is \emph{Tier~II legitimate}: it requires only $\nabla\mu^*\in L^\infty$ and
	$\nabla\hat{\mathbf{v}}\in L^2$ to define the interior $L^2$ field
	$(\nabla\mu^*)\cdot(\nabla\hat{\mathbf{v}}+(\nabla\hat{\mathbf{v}})^{\mathsf T})$, whose curl lies in $H^{-1}$ by duality.
	In contrast, the coupling $\nabla\mu^*\times\Delta\hat{\mathbf{v}}$ requires $\Delta\hat{\mathbf{v}}$ and is therefore only available under
	additional local regularity (Tier~III). Thus, for the low-regularity operator narrative, $\mathcal{G}_{\mu^*}$ is the robust, model-independent
	bulk vorticity source that survives without invoking second derivatives of the unknown.
\end{remark}

\begin{remark}[Boundary term $\mathcal{C}_{\partial\Omega}$ and wall vorticity injection]
	In wall-bounded flows, vorticity injection is also supported at $\partial\Omega$ through traction/boundary-layer mechanisms
	(cf.\ classical vorticity formulations and boundary flux discussions in \cite{MajdaBertozzi2002}).
	When one derives vorticity balances using test functions that do not vanish on $\partial\Omega$ (e.g.\ traction-driven problems, impedance balances,
	or mixed boundary data), boundary work terms generated by integration by parts can be reorganized into a boundary-supported functional
	$\mathcal{C}_{\partial\Omega}$ depending on the boundary condition class and the chosen lift.
	For interior diagnostics (test functions compactly supported in $\Omega$), $\mathcal{C}_{\partial\Omega}$ vanishes identically.
\end{remark}

\begin{lemma}[Texture commutator bound (Tier II, duality sharp form)]
	\label{lem:G_eta_bound_tierII_duality}
	Assume Tier~II. Define the $L^2$ vector field
	\[
	\mathbf{g}
	:=
	(\nabla\mu^*)\cdot(\nabla\hat{\mathbf{v}}+(\nabla\hat{\mathbf{v}})^{\mathsf T})
	\in L^2(\Omega;\mathbb{C}^d),
	\qquad
	\mathcal{G}_{\mu^*}[\hat{\mathbf{v}}]:=\nabla\times\mathbf{g}\in H^{-1}(\Omega).
	\]
	Then there exists $C_\Omega>0$ such that
	\begin{equation}
		\|\mathcal{G}_{\mu^*}[\hat{\mathbf{v}}]\|_{H^{-1}(\Omega)}
		\le
		C_\Omega\,\|\nabla\mu^*\|_{L^\infty(\Omega)}\,\|\mathbf{D}(\hat{\mathbf{v}})\|_{L^2(\Omega)}.
		\label{eq:G_eta_bound_repeat_refined}
	\end{equation}
\end{lemma}

\begin{proof}
	The estimate is a direct duality argument using the distributional curl definitions
	\eqref{eq:curl_distribution_def_2D}--\eqref{eq:curl_distribution_def_3D} and Korn's inequality (cf.\ \cite{Horgan1995,Ciarlet1988}).
	For $d=3$, let $\boldsymbol{\psi}\in H_0^1(\Omega;\mathbb{C}^3)$ and write
	\[
	|\langle \mathcal{G}_{\mu^*}[\hat{\mathbf{v}}],\boldsymbol{\psi}\rangle|
	=
	\Big|\int_\Omega \mathbf{g}\cdot\overline{\nabla\times\boldsymbol{\psi}}\,d\mathbf{x}\Big|
	\le
	\|\mathbf{g}\|_{L^2}\,\|\nabla\times\boldsymbol{\psi}\|_{L^2}
	\le
	C_\Omega\,\|\mathbf{g}\|_{L^2}\,\|\boldsymbol{\psi}\|_{H^1}.
	\]
	Hence $\|\mathcal{G}_{\mu^*}[\hat{\mathbf{v}}]\|_{H^{-1}}\le C_\Omega\|\mathbf{g}\|_{L^2}$.
	Moreover,
	\[
	\|\mathbf{g}\|_{L^2}
	\le
	\|\nabla\mu^*\|_{L^\infty}\,\|\nabla\hat{\mathbf{v}}+(\nabla\hat{\mathbf{v}})^{\mathsf T}\|_{L^2}
	\le
	C_\Omega\,\|\nabla\mu^*\|_{L^\infty}\,\|\mathbf{D}(\hat{\mathbf{v}})\|_{L^2},
	\]
	which yields \eqref{eq:G_eta_bound_repeat_refined}. The $d=2$ case is identical with $\nabla^\perp$ in place of curl.
\end{proof}

\begin{remark}[Pure phase specialization and the $\|\nabla\varphi\|$ control axis]
	In the phase-only class $\mu^*=\mu_0e^{i\varphi}$ one has $\nabla\mu^*=i\mu^*\nabla\varphi$ and thus
	$\|\nabla\mu^*\|_{L^\infty}\le \mu_0\|\nabla\varphi\|_{L^\infty}$.
	Lemma~\ref{lem:G_eta_bound_tierII_duality} therefore identifies $\|\nabla\varphi\|_{L^\infty}$ as the single constitutive control parameter
	governing the \emph{bulk} texture-driven vorticity injection at Tier~II, independent of any magnitude variation in $|\mu^*|$.
\end{remark}
\subsubsection{Operator Notation and Pressure Elimination (Alignment With $\mathcal{L}_\omega$)}
\label{subsec:operator_notation_alignment}

Fix $\omega>0$ and suppress the explicit $\omega$-dependence in $\mu^*(\cdot,\omega)$ when convenient.
Let
\[
H_\sigma:=\overline{\{\mathbf{v}\in C_c^\infty(\Omega;\mathbb{C}^d):\ \nabla\cdot\mathbf{v}=0\}}^{\,L^2(\Omega)},
\qquad
V_\sigma:=H_0^1(\Omega;\mathbb{C}^d)\cap H_\sigma,
\]
and let $P_\sigma:L^2(\Omega;\mathbb{C}^d)\to H_\sigma$ denote the Helmholtz--Leray projector (defined on Lipschitz domains in the standard way;
see, e.g., \cite{Temam2001,GiraultRaviart1986}). Define the bounded ``mass'' operator $M_\rho:H_\sigma\to H_\sigma$ by
$M_\rho \mathbf{u}:=\rho\,\mathbf{u}$ (with $\rho\in L^\infty$, $\rho>0$ a.e.). Define the viscous divergence (distributionally)
\[
\mathcal{D}_{\mu^*}\mathbf{u}:=\nabla\cdot\big(2\mu^*\,\mathbf{D}(\mathbf{u})\big),
\qquad
\mathbf{D}(\mathbf{u})=\tfrac12\big(\nabla\mathbf{u}+(\nabla\mathbf{u})^{\mathsf T}\big),
\]
so that at the form level one has
\[
\langle \mathcal{D}_{\mu^*}\mathbf{u},\boldsymbol{\varphi}\rangle
=
-\int_\Omega 2\mu^*\,\mathbf{D}(\mathbf{u}):\overline{\nabla\boldsymbol{\varphi}}\,d\mathbf{x}
\qquad(\boldsymbol{\varphi}\in H_0^1),
\]
and for solenoidal test functions, the familiar reduction to $\mathbf{D}(\boldsymbol{\varphi})$ applies.

\medskip
\noindent\textbf{Pressure form vs.\ pressure-eliminated form.}
The harmonic Stokes system can be written in distributional form as
\begin{equation}
	i\omega\,M_\rho \hat{\mathbf{v}}
	=
	-\nabla \hat{p}
	+\mathcal{D}_{\mu^*}\hat{\mathbf{v}}
	+\hat{\mathbf{f}},
	\qquad
	\nabla\cdot\hat{\mathbf{v}}=0
	\quad\text{in }\mathcal{D}'(\Omega),
	\label{eq:stokes_pressure_form_Lomega}
\end{equation}
with $\hat{\mathbf{v}}\in V_\sigma$ and $\hat{\mathbf{f}}\in V_\sigma^*$ (and boundary conditions encoded in the choice of $V_\sigma$ and/or lifting). Applying $P_\sigma$ to \eqref{eq:stokes_pressure_form_Lomega} eliminates the pressure and yields the \emph{pressure-eliminated resolvent equation}
\begin{equation}
	\mathcal{L}_\omega \hat{\mathbf{v}}
	=
	\hat{\mathbf{f}}_\sigma,
	\qquad
	\hat{\mathbf{f}}_\sigma:=P_\sigma \hat{\mathbf{f}},
	\label{eq:resolvent_equation_Lomega}
\end{equation}
where the pressure-eliminated operator is
\begin{equation}
	\mathcal{L}_\omega
	:=
	i\omega\,M_\rho \;-\; P_\sigma \mathcal{D}_{\mu^*}
	\quad
	\text{(realized via the sectorial form on $V_\sigma$).}
	\label{eq:Lomega_def_alignment}
\end{equation}
This is the $\kappa=0$ specialization of the modewise operators $\mathcal{L}_{\omega,\mathrm{lin}}(\kappa)$ used in the observables section:
in a $z$-periodic geometry, one replaces $\partial_z$ by $i\kappa$ throughout and works on the corresponding $\kappa$-reduced solenoidal spaces.

\begin{proposition}[Distributional vorticity balance as $\nabla\times$ of the resolvent equation]
	\label{prop:weak_vorticity_balance_tierII_strong_Lomega}
	Assume Tier~II: $\mu^*\in W^{1,\infty}(\Omega;\mathbb{C})$ with $\Re\mu^*\ge \mu_{\min}>0$ a.e., and
	$\rho\in L^\infty(\Omega)$ with $\rho>0$ a.e.
	Let $(\hat{\mathbf{v}},\hat{p})$ satisfy \eqref{eq:stokes_pressure_form_Lomega} in $\mathcal{D}'(\Omega)$ with
	$\hat{\mathbf{v}}\in V_\sigma$ and $\hat{\mathbf{f}}\in V_\sigma^*$ (equivalently, $\hat{\mathbf{v}}$ solves
	\eqref{eq:resolvent_equation_Lomega} in $V_\sigma^*$). Then the vorticity satisfies the \emph{interior} distributional identity
	\begin{equation}
		i\omega\,\nabla\times(\rho\,\hat{\mathbf{v}})
		=
		\nabla\times\mathcal{D}_{\mu^*}\hat{\mathbf{v}}
		+
		\nabla\times \hat{\mathbf{f}}
		\qquad\text{in }\mathcal{D}'(\Omega),
		\label{eq:vorticity_distributional_core_precise_rho_Lomega}
	\end{equation}
	i.e.\ when tested against $\mathcal{C}_c^\infty(\Omega)$ (so boundary traces do not enter).
	If $\rho$ is constant, then $\nabla\times(\rho\hat{\mathbf{v}})=\rho\,\hat{\boldsymbol{\omega}}$ with
	$\hat{\boldsymbol{\omega}}:=\nabla\times\hat{\mathbf{v}}$.
\end{proposition}

\begin{proof}
	Starting from the pressure form \eqref{eq:stokes_pressure_form_Lomega}, apply $\nabla\times$ in $\mathcal{D}'(\Omega)$.
	The pressure drops out since $\nabla\times\nabla\hat{p}=0$ distributionally, yielding
	\eqref{eq:vorticity_distributional_core_precise_rho_Lomega}. The Tier~II hypothesis and Lemma~\ref{lem:viscous_div_mapping}
	ensure $\mathcal{D}_{\mu^*}\hat{\mathbf{v}}\in H^{-1}(\Omega)$, hence its curl is well-defined as an element of $H^{-2}(\Omega)$ and
	in particular as a distribution in $\mathcal{D}'(\Omega)$ via the curl-duality definitions.
\end{proof}

\noindent Next, to expose the \emph{bulk} texture mechanism in a form that is compatible with the resolvent narrative, we expand
$\mathcal{D}_{\mu^*}\hat{\mathbf{v}}$ on interior subdomains where additional regularity holds. This is a \emph{local} statement:
it is intended on $\Omega'\Subset\Omega$ away from corners/edges, or under hypotheses guaranteeing local $H^2$ regularity.

\begin{lemma}[Algebraic decomposition of $\mathcal{D}_{\mu^*}\hat{\mathbf{v}}$ for incompressible fields]
	\label{lem:div_2muD_decomposition_Lomega}
	Let $\mu^*\in W^{1,\infty}_{\mathrm{loc}}(\Omega)$ and $\hat{\mathbf{v}}\in H^1_{\mathrm{loc}}(\Omega;\mathbb{C}^d)$ satisfy
	$\nabla\cdot\hat{\mathbf{v}}=0$ in $\mathcal{D}'(\Omega)$. Then, in $\mathcal{D}'(\Omega)$,
	\begin{equation}
		\mathcal{D}_{\mu^*}\hat{\mathbf{v}}
		=
		\mu^*\,\Delta \hat{\mathbf{v}}
		+
		(\nabla\mu^*)\cdot\big(\nabla\hat{\mathbf{v}}+(\nabla\hat{\mathbf{v}})^{\mathsf T}\big).
		\label{eq:div_2muD_decomposition_Lomega}
	\end{equation}
\end{lemma}

\begin{proof}
	This is identical to Lemma~\ref{lem:div_2muD_decomposition} after recognizing $\mathcal{D}_{\mu^*}=\nabla\cdot(2\mu^*\mathbf{D}(\cdot))$.
\end{proof}

\newpage

\subsection{Operator-Level Consequence: a Quantified Non-Normal Perturbation}
\label{subsec:operator_non_normal_perturbation}

\paragraph{Solenoidal spaces, Korn/Friedrichs, and the oscillatory Stokes form.}
We record an operator-level formulation of the phase-rotation mechanism in the pure-phase class and isolate the
\emph{phase-gradient} contribution as an explicit lower-order perturbation with quantitative control. This provides a
bridge from the constitutive phase texture $\varphi(\mathbf{x})$ to perturbation theory for closed sectorial forms and
hence to resolvent bounds for the associated m-sectorial operators \cite{Kato1995,Haase2006}. The essential point is that
\emph{even when $|\mu^*|$ is spatially uniform}, spatial variation of $\arg\mu^*$ forces non-commutation with
differentiation and therefore produces a genuinely new (non-normal) operator structure.

\medskip
\noindent
Let $\Omega\subset\mathbb{R}^d$ with $d\in\{2,3\}$ be bounded and Lipschitz, and let $\Gamma_D\subset\partial\Omega$
have positive surface measure. Define
\[
\mathcal{V}
:=
\Big\{
\boldsymbol\psi\in C^\infty(\overline{\Omega};\mathbb{C}^d):
\nabla\cdot\boldsymbol\psi=0,\ \boldsymbol\psi|_{\Gamma_D}=0
\Big\},
\qquad
H_\sigma:=\overline{\mathcal{V}}^{\,L^2(\Omega)},
\qquad
V_\sigma:=\overline{\mathcal{V}}^{\,H^1(\Omega)}.
\]
On $V_\sigma$, Korn's inequality (with partial Dirichlet control) holds:
\begin{equation}
	\|\nabla \mathbf{v}\|_{L^2(\Omega)}
	\le C_{\mathrm{Korn}}\,\|\mathbf{D}(\mathbf{v})\|_{L^2(\Omega)},
	\qquad
	\mathbf{D}(\mathbf{v})=\tfrac12(\nabla\mathbf{v}+(\nabla\mathbf{v})^{\mathsf T}),
	\label{eq:korn_Vsigma_repeat2}
\end{equation}
and, since $|\Gamma_D|>0$, the Friedrichs/Poincar\'e inequality implies
$\|\mathbf{v}\|_{L^2(\Omega)}\le C_P\|\nabla\mathbf{v}\|_{L^2(\Omega)}$; hence
\begin{equation}
	\|\mathbf{v}\|_{H^1(\Omega)} \simeq \|\mathbf{D}(\mathbf{v})\|_{L^2(\Omega)}
	\qquad \text{for all }\mathbf{v}\in V_\sigma,
	\label{eq:equiv_norms_Vsigma}
\end{equation}
with constants depending only on $\Omega$ and $\Gamma_D$
(cf.\ Korn/Friedrichs theory in \cite{Ciarlet1988,Horgan1995,Evans2010}).

\medskip
\noindent
Fix $\omega>0$ and assume $\rho\in L^\infty(\Omega)$ with
$0<\rho_{\min}\le\rho\le\rho_{\max}$ a.e.
For $\mu^*(\cdot,\omega)\in L^\infty(\Omega;\mathbb{C})$, define the (pressure-eliminated) oscillatory Stokes form
\begin{equation}
	\mathfrak{a}_\omega(\mathbf{u},\mathbf{v})
	:=
	\int_\Omega 2\mu^*(\mathbf{x},\omega)\,\mathbf{D}(\mathbf{u}):\overline{\mathbf{D}(\mathbf{v})}\,d\mathbf{x}
	\;+\;
	i\omega\int_\Omega \rho(\mathbf{x})\,\mathbf{u}\cdot\overline{\mathbf{v}}\,d\mathbf{x},
	\qquad \mathbf{u},\mathbf{v}\in V_\sigma,
	\label{eq:aomega_def_repeat2}
\end{equation}
and impose the Tier~I accretivity hypothesis
\begin{equation}
	\Re\mu^*(\mathbf{x},\omega)\ge \mu_{\min}>0\quad \text{a.e. in }\Omega.
	\label{eq:passivity_repeat2}
\end{equation}
Then $\mathfrak{a}_\omega$ is bounded on $V_\sigma\times V_\sigma$ and $V_\sigma$-elliptic in the real part:
\begin{equation}
	\Re\,\mathfrak{a}_\omega(\mathbf{v},\mathbf{v})
	=
	\int_\Omega 2\,\Re\mu^*\,|\mathbf{D}(\mathbf{v})|^2\,d\mathbf{x}
	\ge 2\mu_{\min}\,\|\mathbf{D}(\mathbf{v})\|_{L^2}^2
	\gtrsim \|\mathbf{v}\|_{H^1}^2
	\qquad \forall\mathbf{v}\in V_\sigma.
	\label{eq:coercive_repeat2}
\end{equation}
Consequently, $\mathfrak{a}_\omega$ is closed and sectorial on $H_\sigma$ and induces a unique m-sectorial operator
$A_\omega:D(A_\omega)\subset H_\sigma\to H_\sigma$ by Kato's first representation theorem \cite{Kato1995,Haase2006}.

\medskip
\noindent\textbf{Pure-phase subclass and phase rotation.}
Assume (Tier~II regularity) that
\begin{equation}
	\mu^*(\mathbf{x},\omega)=\mu_0(\omega)\,e^{i\varphi(\mathbf{x})},
	\qquad
	\mu_0(\omega)>0,
	\qquad
	\varphi\in W^{1,\infty}(\Omega;\mathbb{R}),
	\qquad
	\cos\varphi\ge \delta>0 \ \text{a.e.},
	\label{eq:pure_phase_repeat2}
\end{equation}
so that $\Re\mu^*=\mu_0\cos\varphi\ge \mu_0\delta>0$.
Define the pointwise phase rotation (multiplication) operator
\[
(U_\varphi \mathbf{w})(\mathbf{x}) := e^{-i\varphi(\mathbf{x})}\mathbf{w}(\mathbf{x}).
\]
Then $U_\varphi$ is unitary on $L^2(\Omega;\mathbb{C}^d)$ and is a bounded isomorphism on $H^1(\Omega;\mathbb{C}^d)$
(basic Sobolev multiplier property for $W^{1,\infty}$ coefficients; cf.\ \cite{AdamsFournier2003}).
Moreover, since $\varphi\in W^{1,\infty}$, $U_\varphi$ preserves the trace condition on $\Gamma_D$.

\begin{lemma}[Multiplication by a $W^{1,\infty}$ phase: $H^1$ stability]
	\label{lem:Uphi_H1_stability}
	If $\varphi\in W^{1,\infty}(\Omega)$, then for all $\mathbf{w}\in H^1(\Omega;\mathbb{C}^d)$,
	\begin{equation}
		\|U_\varphi\mathbf{w}\|_{H^1(\Omega)}
		\le
		C_\Omega\big(1+\|\nabla\varphi\|_{L^\infty(\Omega)}\big)\,\|\mathbf{w}\|_{H^1(\Omega)},
		\qquad
		\|U_\varphi^{-1}\mathbf{w}\|_{H^1(\Omega)}
		\le
		C_\Omega\big(1+\|\nabla\varphi\|_{L^\infty(\Omega)}\big)\,\|\mathbf{w}\|_{H^1(\Omega)}.
		\label{eq:Uphi_H1_bound_repeat2}
	\end{equation}
	If additionally $\mathbf{w}|_{\Gamma_D}=0$ in the trace sense, then $(U_\varphi\mathbf{w})|_{\Gamma_D}=0$ as well.
\end{lemma}

\begin{proof}
	The trace statement follows since $e^{-i\varphi}\in W^{1,\infty}$ has a well-defined bounded trace.
	For the $H^1$ bound, use
	$\nabla(e^{-i\varphi}\mathbf{w})
	=e^{-i\varphi}\nabla\mathbf{w}-ie^{-i\varphi}\mathbf{w}\otimes\nabla\varphi$
	and estimate in $L^2$.
	The bound for $U_\varphi^{-1}$ is identical with $e^{+i\varphi}$.
\end{proof}

\medskip
\noindent
A key structural point is that $U_\varphi$ does \emph{not} preserve solenoidality in general:
if $\nabla\cdot\mathbf{w}=0$, then
$\nabla\cdot(e^{-i\varphi}\mathbf{w})=-ie^{-i\varphi}\nabla\varphi\cdot\mathbf{w}$
need not vanish. Accordingly, in a pressure-eliminated setting one should pull back the solenoidal spaces:
\begin{equation}
	H_{\sigma,\varphi}:=U_\varphi^{-1}(H_\sigma),
	\qquad
	V_{\sigma,\varphi}:=U_\varphi^{-1}(V_\sigma).
	\label{eq:pullback_spaces}
\end{equation}
Then $U_\varphi:H_{\sigma,\varphi}\to H_\sigma$ is unitary and $U_\varphi:V_{\sigma,\varphi}\to V_\sigma$ is a topological isomorphism.
Equivalently,
\[
V_{\sigma,\varphi}
=
\Big\{
\mathbf{w}\in H^1(\Omega;\mathbb{C}^d):\ \mathbf{w}|_{\Gamma_D}=0,\ 
\nabla\cdot(e^{-i\varphi}\mathbf{w})=0\text{ in }\mathcal{D}'(\Omega)
\Big\}.
\]

\begin{lemma}[Compensated symmetric gradient]
	\label{lem:D_comp_identity_repeat2}
	If $\varphi\in W^{1,\infty}(\Omega)$ and $\mathbf{w}\in H^1(\Omega;\mathbb{C}^d)$, then
	\begin{equation}
		\mathbf{D}(e^{-i\varphi}\mathbf{w})
		=
		e^{-i\varphi}\Big(\mathbf{D}(\mathbf{w})-\tfrac{i}{2}\,\mathcal{S}_\varphi[\mathbf{w}]\Big),
		\qquad
		\mathcal{S}_\varphi[\mathbf{w}]
		:=
		\mathbf{w}\otimes\nabla\varphi+\nabla\varphi\otimes\mathbf{w}.
		\label{eq:D_compensated_Sphi_repeat2}
	\end{equation}
\end{lemma}

\begin{proof}
	From $\nabla(e^{-i\varphi}\mathbf{w})=e^{-i\varphi}(\nabla\mathbf{w}-i\mathbf{w}\otimes\nabla\varphi)$, take the symmetric part.
\end{proof}

\medskip
\noindent\textbf{Exact form decomposition and gradient-controlled perturbation.}
Let $\mathbf{u}=U_\varphi\mathbf{w}$ and $\mathbf{v}=U_\varphi\mathbf{z}$ with $\mathbf{w},\mathbf{z}\in V_{\sigma,\varphi}$.
Substituting \eqref{eq:D_compensated_Sphi_repeat2} into \eqref{eq:aomega_def_repeat2} yields the exact decomposition
\begin{equation}
	\mathfrak{a}_\omega(U_\varphi\mathbf{w},U_\varphi\mathbf{z})
	=
	\mathfrak{a}^{(\varphi)}_\omega(\mathbf{w},\mathbf{z})
	+
	\mathfrak{b}_\varphi(\mathbf{w},\mathbf{z}),
	\qquad
	\mathbf{w},\mathbf{z}\in V_{\sigma,\varphi},
	\label{eq:form_decomp_repeat2}
\end{equation}
where the \emph{principal} phase-weighted form is
\begin{equation}
	\mathfrak{a}^{(\varphi)}_\omega(\mathbf{w},\mathbf{z})
	:=
	\int_\Omega 2\mu_0 e^{i\varphi(\mathbf{x})}\,\mathbf{D}(\mathbf{w}):\overline{\mathbf{D}(\mathbf{z})}\,d\mathbf{x}
	\;+\;
	i\omega\int_\Omega \rho\,\mathbf{w}\cdot\overline{\mathbf{z}}\,d\mathbf{x},
	\label{eq:avarphi_omega}
\end{equation}
and the \emph{phase-gradient coupling} is
\begin{align}
	\mathfrak{b}_\varphi(\mathbf{w},\mathbf{z})
	:=
	\int_\Omega
	e^{i\varphi(\mathbf{x})}\Big[
	&i\mu_0\,\mathbf{D}(\mathbf{w}):\overline{\mathcal{S}_\varphi[\mathbf{z}]}
	-i\mu_0\,\mathcal{S}_\varphi[\mathbf{w}]:\overline{\mathbf{D}(\mathbf{z})}
	+\frac{\mu_0}{2}\,\mathcal{S}_\varphi[\mathbf{w}]:\overline{\mathcal{S}_\varphi[\mathbf{z}]}
	\Big]\,d\mathbf{x}.
	\label{eq:bphi_repeat2}
\end{align}
In particular, $\mathfrak{b}_\varphi$ involves at most one derivative of $\mathbf{w}$ or $\mathbf{z}$, and depends on the texture only through
$\nabla\varphi$ as an $L^\infty$ coefficient.

\begin{lemma}[Coupling bounds; KLMN-type relative control]
	\label{lem:bphi_bounds_repeat2}
	Assume $\varphi\in W^{1,\infty}(\Omega)$ and $\cos\varphi\ge\delta>0$ a.e.
	There exists $C_\Omega>0$ such that for all $\mathbf{w},\mathbf{z}\in V_{\sigma,\varphi}$,
	\begin{equation}
		|\mathfrak{b}_\varphi(\mathbf{w},\mathbf{z})|
		\le
		C_\Omega\,\mu_0\Big(\|\nabla\varphi\|_{L^\infty(\Omega)}+\|\nabla\varphi\|_{L^\infty(\Omega)}^{2}\Big)
		\|\mathbf{w}\|_{H^1(\Omega)}\,\|\mathbf{z}\|_{H^1(\Omega)}.
		\label{eq:bphi_bounded_repeat2}
	\end{equation}
	Moreover, for every $\eta\in(0,1)$ there exists $C_{\Omega,\eta}>0$ such that
	\begin{equation}
		|\mathfrak{b}_\varphi(\mathbf{w},\mathbf{w})|
		\le
		\eta\,\Re\mathfrak{a}^{(\varphi)}_\omega(\mathbf{w},\mathbf{w})
		+
		C_{\Omega,\eta}\,\mu_0\,\|\nabla\varphi\|_{L^\infty(\Omega)}^2\,\|\mathbf{w}\|_{L^2(\Omega)}^2,
		\qquad \mathbf{w}\in V_{\sigma,\varphi}.
		\label{eq:bphi_relative_repeat2}
	\end{equation}
\end{lemma}

\begin{proof}
	Use $|\mathcal{S}_\varphi[\mathbf{w}]|\le 2|\mathbf{w}|\,|\nabla\varphi|$ pointwise, hence
	$\|\mathcal{S}_\varphi[\mathbf{w}]\|_{L^2}\le 2\|\nabla\varphi\|_{L^\infty}\|\mathbf{w}\|_{L^2}$.
	Apply Cauchy--Schwarz to \eqref{eq:bphi_repeat2}, and use Korn/Friedrichs on $V_{\sigma,\varphi}$ (transported via $U_\varphi$) to obtain
	\eqref{eq:bphi_bounded_repeat2}. For \eqref{eq:bphi_relative_repeat2}, apply Young's inequality to the cross terms and use
	\[
	\Re\mathfrak{a}^{(\varphi)}_\omega(\mathbf{w},\mathbf{w})
	=
	\int_\Omega 2\mu_0\cos\varphi\,|\mathbf{D}(\mathbf{w})|^2\,d\mathbf{x}
	\ge 2\mu_0\delta\,\|\mathbf{D}(\mathbf{w})\|_{L^2}^2.
	\]
\end{proof}

\medskip
\noindent
Let $\widetilde{\mathfrak{a}}_\omega$ be the pulled-back form on $V_{\sigma,\varphi}$ defined by
$\widetilde{\mathfrak{a}}_\omega(\mathbf{w},\mathbf{z}):=\mathfrak{a}_\omega(U_\varphi\mathbf{w},U_\varphi\mathbf{z})$.
Then \eqref{eq:form_decomp_repeat2} gives $\widetilde{\mathfrak{a}}_\omega=\mathfrak{a}^{(\varphi)}_\omega+\mathfrak{b}_\varphi$.
By Lemma~\ref{lem:bphi_bounds_repeat2}, $\mathfrak{b}_\varphi$ is a bounded form on $V_{\sigma,\varphi}$ and is
$\mathfrak{a}^{(\varphi)}_\omega$-bounded with arbitrarily small relative bound (KLMN-type perturbation) \cite{Kato1995}.
Hence $\widetilde{\mathfrak{a}}_\omega$ is closed and sectorial, and it induces an m-sectorial operator $\widetilde{A}_\omega$ on $H_{\sigma,\varphi}$.
Since $U_\varphi:H_{\sigma,\varphi}\to H_\sigma$ is unitary, $A_\omega$ and $\widetilde{A}_\omega$ are unitarily equivalent:
\begin{equation}
	A_\omega
	=
	U_\varphi\,\widetilde{A}_\omega\,U_\varphi^{-1},
	\qquad
	(A_\omega-\lambda I)^{-1}
	=
	U_\varphi\,(\widetilde{A}_\omega-\lambda I)^{-1}\,U_\varphi^{-1},
	\qquad
	\lambda\in\rho(A_\omega)=\rho(\widetilde{A}_\omega).
	\label{eq:unitary_equivalence_repeat2}
\end{equation}

\begin{remark}[Where the non-normality sits]
	Even when $|\mu^*|=\mu_0$ is spatially uniform, the coupling $\mathfrak{b}_\varphi$ is present unless $\nabla\varphi\equiv 0$.
	It is generated by the failure of multiplication by $e^{-i\varphi(\mathbf{x})}$ to commute with differentiation and is the canonical
	lower-order mechanism through which phase gradients enter the operator geometry. In particular,
	$\|\nabla\varphi\|_{L^\infty}$ is the distinguished constitutive size parameter controlling this perturbation at the form level.
\end{remark}

\subsubsection{Where ``Smoothness'' Enters: $H^2$-Regularity and Strong PDE Identification.}
\label{subsubsec:elliptic_regularity_repeat2}

We now separate (i) spectral/resolvent consequences that follow purely from closed, coercive sectorial forms on bounded Lipschitz domains
from (ii) additional smoothness hypotheses (on $\partial\Omega$ and on $\mu^*$) required to identify abstract operator domains with strong PDE
domains and to justify $H^2$-type estimates. Throughout, $\Omega\subset\mathbb{R}^d$ ($d\in\{2,3\}$) is bounded and Lipschitz, and
$\Gamma_D\subset\partial\Omega$ has positive surface measure so that Korn and Friedrichs/Poincar\'e hold on $V_\sigma$.

\paragraph{Low-regularity sectorial framework and resolvent bounds.}
Assume $\mu^*\in L^\infty(\Omega;\mathbb{C})$ with $\Re\mu^*\ge\mu_{\min}>0$ a.e. Define the viscous form on $V_\sigma$,
\[
\mathfrak{a}(\mathbf{u},\mathbf{v})
:=
\int_\Omega 2\,\mu^*(\mathbf{x})\,\mathbf{D}(\mathbf{u}):\overline{\mathbf{D}(\mathbf{v})}\,d\mathbf{x},
\qquad \mathbf{u},\mathbf{v}\in V_\sigma.
\]
Then $\mathfrak{a}$ is bounded and $V_\sigma$-elliptic in the real part, hence closed and sectorial on $H_\sigma$ and induces an
m-sectorial operator $A$ on $H_\sigma$ by Kato's first representation theorem \cite{Kato1995,Haase2006,Ouhabaz2005}. For the oscillatory term, let $M_\rho$ denote multiplication by $\rho\in L^\infty(\Omega)$ on $L^2(\Omega;\mathbb{C}^d)$, and define
the bounded operator on $H_\sigma$ by
\begin{equation}
	B_\rho \;:=\; P_\sigma\,M_\rho\big|_{H_\sigma}\in\mathcal{L}(H_\sigma),
	\label{eq:Brho_def}
\end{equation}
where $P_\sigma$ is the Leray projector onto $H_\sigma$ in $L^2$.
(If $\rho$ is constant, then $B_\rho=\rho I$; if $\rho$ varies, $M_\rho$ does not preserve $H_\sigma$ and the projection in
\eqref{eq:Brho_def} is essential.)
Define
\begin{equation}
	A_\omega \;:=\; A + i\omega\,B_\rho,
	\qquad \omega>0,
	\label{eq:Aomega_bounded_pert}
\end{equation}
so $A_\omega$ is an m-sectorial bounded perturbation of $A$ and satisfies $D(A_\omega)=D(A)$
\cite[Ch.~VI]{Kato1995}.

\begin{lemma}[Coercive half-plane resolvent bound and $V_\sigma$-mapping]
	\label{lem:coercive_resolvent_bound}
	There exists $\alpha>0$ (depending only on $\Omega$, $\Gamma_D$, and $\mu_{\min}$) such that
	\begin{equation}
		\Re(A_\omega \mathbf{u},\mathbf{u})_{H_\sigma}\ \ge\ \alpha\,\|\mathbf{u}\|_{H_\sigma}^2
		\qquad \forall \mathbf{u}\in D(A),
		\label{eq:alpha_strict_accretive}
	\end{equation}
	and hence for every $\lambda\in\mathbb{C}$ with $\Re\lambda<\alpha$ one has $\lambda\in\rho(A_\omega)$ and
	\begin{equation}
		\|(A_\omega-\lambda I)^{-1}\|_{\mathcal{L}(H_\sigma)}
		\ \le\ \frac{1}{\alpha-\Re\lambda}.
		\label{eq:resolvent_halfplane_bound}
	\end{equation}
	Moreover, for every such $\lambda$, the resolvent maps $H_\sigma$ boundedly into $V_\sigma$:
	\begin{equation}
		\|(A_\omega-\lambda I)^{-1}\mathbf{f}\|_{V_\sigma}
		\ \le\ C(\lambda)\,\|\mathbf{f}\|_{H_\sigma},
		\qquad \mathbf{f}\in H_\sigma.
		\label{eq:resolvent_to_Vsigma}
	\end{equation}
\end{lemma}

\begin{proof}
	For $\mathbf{u}\in V_\sigma$, $\Re\mathfrak{a}(\mathbf{u},\mathbf{u})
	=\int_\Omega 2\,\Re\mu^*\,|\mathbf{D}(\mathbf{u})|^2\ge 2\mu_{\min}\|\mathbf{D}(\mathbf{u})\|_{L^2}^2$.
	By Korn and Friedrichs/Poincar\'e on $V_\sigma$ (since $|\Gamma_D|>0$), $\|\mathbf{u}\|_{H_\sigma}\lesssim \|\mathbf{D}(\mathbf{u})\|_{L^2}$,
	so \eqref{eq:alpha_strict_accretive} follows for some $\alpha>0$.
	Since $i\omega B_\rho$ is skew-adjoint in the real part, it does not affect \eqref{eq:alpha_strict_accretive}.
	The bound \eqref{eq:resolvent_halfplane_bound} is the standard strict-accretivity resolvent estimate.
	
	For \eqref{eq:resolvent_to_Vsigma}, let $\mathbf{u}=(A_\omega-\lambda I)^{-1}\mathbf{f}$ and test the variational identity with $\mathbf{u}$.
	Taking real parts yields a coercive estimate on $\|\mathbf{u}\|_{V_\sigma}$ in terms of $\|\mathbf{f}\|_{H_\sigma}$ and $|\lambda|\,\|\mathbf{u}\|_{H_\sigma}$,
	and \eqref{eq:resolvent_halfplane_bound} controls $\|\mathbf{u}\|_{H_\sigma}$ by $\|\mathbf{f}\|_{H_\sigma}$.
\end{proof}

\begin{proposition}[Numerical range control and resolvent bound]
	\label{prop:numerical_range_resolvent_bound}
	Let $T$ be a closed densely defined operator on a Hilbert space. Then for any $\lambda\notin \overline{W(T)}$,
	\begin{equation}
		\|(T-\lambda I)^{-1}\|\ \le\ \frac{1}{\operatorname{dist}(\lambda,W(T))}.
		\label{eq:resolvent_distance_numerical_range}
	\end{equation}
	In particular, this applies to $T=A$ and $T=A_\omega$.
\end{proposition}

\noindent
This is a standard consequence of the numerical range inequality for closed operators \cite[Ch.~V--VI]{Kato1995}
(see also \cite{Haase2006}). Because $\Omega$ is bounded and Lipschitz, the embedding $H^1(\Omega)\hookrightarrow L^2(\Omega)$ is compact (Rellich--Kondrachov),
hence $V_\sigma\hookrightarrow H_\sigma$ is compact as a closed subspace embedding \cite{AdamsFournier2003,Evans2010}.
Combined with \eqref{eq:resolvent_to_Vsigma}, this yields compactness of the resolvent.

\begin{proposition}[Compact resolvent]
	\label{prop:compact_resolvent_repeat2}
	On bounded Lipschitz domains, both $A$ and $A_\omega$ have compact resolvent on $H_\sigma$.
\end{proposition}

\noindent
Consequently, $\sigma(A)$ and $\sigma(A_\omega)$ consist only of isolated eigenvalues of finite algebraic multiplicity with no finite accumulation point,
and the associated Riesz projections have finite rank \cite{Kato1995}.
Since these operators are generally non-normal, eigenvectors need not form an orthonormal basis; resolvent and pseudospectral information is therefore
typically essential.

\begin{remark}[Bounded lower-order perturbations beyond oscillatory mass terms]
	Any additional lower-order term that is $H_\sigma$-bounded after projection (e.g.\ Oseen advection realized as $P_\sigma(\mathbf{V}_0\cdot\nabla)$
	under suitable coefficient hypotheses) is a bounded perturbation of the m-sectorial Stokes realization and preserves sectoriality and compact-resolvent
	structure on bounded truncations \cite{Kato1995,Ouhabaz2005}.
\end{remark}

\paragraph{Where smoothness enters: identification of $D(A)$ and $H^2$ regularity.}
The compact-resolvent and discrete-spectrum statements above operate deliberately at low regularity: bounded accretive coefficients and bounded Lipschitz
domains suffice. Smoothness becomes decisive only when one wants to identify $D(A)$ with classical Sobolev domains (e.g.\ $H^2$) and interpret
eigenfunctions/resolvent solutions as strong solutions with pointwise meaning. A common threshold for global $H^2$ regularity in Stokes-type systems is $\partial\Omega\in C^{1,1}$ together with Lipschitz coefficients;
in that regime one can upgrade weak solutions to strong solutions and recover an associated pressure in $H^1/\mathbb{C}$
(see, e.g., \cite{Galdi2011,Sohr2001,GiraultRaviart1986}).

\begin{proposition}[$H^2$-regularity upgrade (smooth boundary)]
	\label{prop:H2_regularity_upgrade_repeat2}
	Assume $\partial\Omega\in C^{1,1}$ and
	\[
	\mu^*\in W^{1,\infty}(\Omega;\mathbb{C}),
	\qquad
	\Re\mu^*(\mathbf{x})\ge \mu_{\min}>0\ \text{a.e.}.
	\]
	Let $A$ be the pressure-eliminated Stokes-type operator induced by $\mathfrak{a}$ on $H_\sigma$.
	Then for each $\lambda\in\rho(A)$ and each $\mathbf{f}\in H_\sigma$, the solution $\mathbf{u}=R(\lambda,A)\mathbf{f}$ admits an associated
	pressure $p$ (unique up to constants) such that
	\[
	\mathbf{u}\in H^2(\Omega;\mathbb{C}^d)\cap V_\sigma,
	\qquad
	p\in H^1(\Omega;\mathbb{C})/\mathbb{C},
	\]
	and there exists $C>0$ (depending on $\Omega$, $\mu_{\min}$, $\|\mu^*\|_{W^{1,\infty}}$, and $d$) for which
	\begin{equation}
		\|\mathbf{u}\|_{H^2(\Omega)} + \|p\|_{H^1(\Omega)/\mathbb{C}}
		\le
		C\Big(\|\mathbf{f}\|_{L^2(\Omega)} + |\lambda|\,\|\mathbf{u}\|_{L^2(\Omega)}\Big).
		\label{eq:H2_estimate_repeat2}
	\end{equation}
	In particular, $R(\lambda,A):H_\sigma\to H^2(\Omega)\cap V_\sigma$ is bounded.
\end{proposition}

\begin{remark}[Cornered domains]
	On polygonal/polyhedral domains, global $H^2$-regularity typically fails even for constant real viscosity due to corner/edge singularities.
	The appropriate replacement is interior (or localized) $H^2$ away from the singular set, or weighted Kondrat'ev-type regularity in corner-adapted spaces;
	see, e.g., \cite{Grisvard1985,Dauge1988,Kondratiev1967,KozlovMazyaRossmann1997}.
\end{remark}
\noindent We now record a distributional strong-form representative of the compensated operator
\[
\widetilde A_\omega := U_\varphi^{-1}A_\omega U_\varphi
\qquad\text{on } H_{\sigma,\varphi}:=U_\varphi^{-1}(H_\sigma),
\]
associated with the pulled-back form
$\widetilde{\mathfrak a}_\omega(\mathbf w,\mathbf z):=\mathfrak a_\omega(U_\varphi\mathbf w,U_\varphi\mathbf z)$
on $V_{\sigma,\varphi}:=U_\varphi^{-1}(V_\sigma)$.
Because $U_\varphi$ is unitary on $L^2$, $A_\omega$ and $\widetilde A_\omega$ are unitarily equivalent and hence have identical spectrum and resolvent
norms:
\begin{equation}
	(A_\omega-\lambda I)^{-1}=U_\varphi(\widetilde A_\omega-\lambda I)^{-1}U_\varphi^{-1},
	\qquad
	\|(A_\omega-\lambda I)^{-1}\|=\|(\widetilde A_\omega-\lambda I)^{-1}\|.
	\label{eq:unitary_equivalence_resolvent_again}
\end{equation}

\begin{lemma}[Distributional characterization of $H_{\sigma,\varphi}$ and $V_{\sigma,\varphi}$]
	\label{lem:pullback_div_constraint}
	Assume $\varphi\in W^{1,\infty}(\Omega)$. Then
	\begin{align}
		H_{\sigma,\varphi}
		&=
		\Big\{\mathbf w\in L^2(\Omega;\mathbb C^d):\ \nabla\cdot(e^{-i\varphi}\mathbf w)=0\ \text{in }\mathcal D'(\Omega)\Big\},
		\label{eq:Hsig_phi_char}
		\\
		V_{\sigma,\varphi}
		&=
		\Big\{\mathbf w\in H^1(\Omega;\mathbb C^d):\ \mathbf w|_{\Gamma_D}=0,\ \nabla\cdot(e^{-i\varphi}\mathbf w)=0\ \text{in }\mathcal D'(\Omega)\Big\}.
		\label{eq:Vsig_phi_char}
	\end{align}
	Moreover, for $\mathbf w\in H^1(\Omega;\mathbb C^d)$ the constraint $\nabla\cdot(e^{-i\varphi}\mathbf w)=0$ is equivalent in $\mathcal D'(\Omega)$ to
	\begin{equation}
		\nabla\cdot\mathbf w = i\,\nabla\varphi\cdot \mathbf w.
		\label{eq:divw_equals_i_gradphi_dot_w}
	\end{equation}
\end{lemma}

\begin{proof}
	The identities \eqref{eq:Hsig_phi_char}--\eqref{eq:Vsig_phi_char} are immediate from the definitions
	$H_{\sigma,\varphi}=U_\varphi^{-1}(H_\sigma)$ and $V_{\sigma,\varphi}=U_\varphi^{-1}(V_\sigma)$.
	For \eqref{eq:divw_equals_i_gradphi_dot_w}, use the distributional product rule with $W^{1,\infty}$ multipliers:
	\[
	\nabla\cdot(e^{-i\varphi}\mathbf w)
	=
	e^{-i\varphi}\big(\nabla\cdot\mathbf w - i\,\nabla\varphi\cdot\mathbf w\big)
	\quad \text{in }\mathcal D'(\Omega),
	\]
	and note that $e^{-i\varphi}$ is bounded and bounded away from $0$.
\end{proof}
\noindent Define the $\varphi$--covariant symmetric gradient
\begin{equation}
	\mathbf D_\varphi(\mathbf w)
	:=
	\mathbf D(\mathbf w)-\frac{i}{2}\,\mathcal S_\varphi[\mathbf w],
	\qquad
	\mathcal S_\varphi[\mathbf w]:=\mathbf w\otimes\nabla\varphi+\nabla\varphi\otimes\mathbf w.
	\label{eq:Dphi_def}
\end{equation}
Then $\mathbf D(U_\varphi\mathbf w)=e^{-i\varphi}\mathbf D_\varphi(\mathbf w)$, and in the pure-phase class
$\mu^*=\mu_0 e^{i\varphi}$ the pulled-back viscous form is
\begin{equation}
	\int_\Omega 2\mu_0 e^{i\varphi(\mathbf x)}\,\mathbf D_\varphi(\mathbf w):\overline{\mathbf D_\varphi(\mathbf z)}\,d\mathbf x,
	\label{eq:viscous_form_covariant_correct}
\end{equation}
so the full compensated oscillatory form reads
\begin{equation}
	\widetilde{\mathfrak a}_\omega(\mathbf w,\mathbf z)
	=
	\int_\Omega 2\mu_0 e^{i\varphi(\mathbf x)}\,\mathbf D_\varphi(\mathbf w):\overline{\mathbf D_\varphi(\mathbf z)}\,d\mathbf x
	\;+\;
	i\omega\int_\Omega \rho\,\mathbf w\cdot\overline{\mathbf z}\,d\mathbf x,
	\qquad
	\mathbf w,\mathbf z\in V_{\sigma,\varphi}.
	\label{eq:a_tilde_covariant_correct}
\end{equation}

\begin{lemma}[Adjoint of $\mathcal S_\varphi$]
	\label{lem:Sphi_adjoint}
	Assume $\varphi\in W^{1,\infty}(\Omega)$. For every symmetric $\mathbf T\in L^2(\Omega;\mathbb C^{d\times d}_{\rm sym})$ and
	$\mathbf w\in L^2(\Omega;\mathbb C^d)$,
	\begin{equation}
		\int_\Omega \mathcal S_\varphi[\mathbf w]:\overline{\mathbf T}\,d\mathbf x
		=
		\int_\Omega \mathbf w\cdot \overline{\big(2\,\mathbf T\,\nabla\varphi\big)}\,d\mathbf x,
		\label{eq:Sphi_adjoint_identity}
	\end{equation}
	hence $\mathcal S_\varphi^\ast \mathbf T = 2\,\mathbf T\,\nabla\varphi$.
\end{lemma}

\noindent
Let $\mathbf D^\ast$ denote the distributional adjoint of $\mathbf D$ under homogeneous Dirichlet trace:
$\mathbf D^\ast\mathbf T=-\nabla\cdot\mathbf T$ in $\mathcal D'(\Omega)$.
Combining $\mathbf D_\varphi=\mathbf D-\frac{i}{2}\mathcal S_\varphi$ with Lemma~\ref{lem:Sphi_adjoint} gives
\begin{equation}
	\mathbf D_\varphi^\ast \mathbf T
	=
	-\nabla\cdot\mathbf T + i\,\mathbf T\,\nabla\varphi,
	\qquad \mathbf T\in L^2(\Omega;\mathbb C^{d\times d}_{\rm sym}),
	\label{eq:Dphi_adjoint_formula}
\end{equation}
in the distributional sense.

\begin{proposition}[Distributional operator representative and projected realization]
	\label{prop:Atild_strong_form}
	Assume $\varphi\in W^{1,\infty}(\Omega)$ and $0<\rho_{\min}\le\rho\le\rho_{\max}$ a.e.
	Define the distributional map $\mathcal{L}_{\varphi,\omega}:V_{\sigma,\varphi}\to V_{\sigma,\varphi}^\ast$ by
	\begin{equation}
		\mathcal{L}_{\varphi,\omega}\mathbf w
		:=
		\mathbf D_\varphi^\ast\!\big(2\mu_0 e^{i\varphi}\,\mathbf D_\varphi(\mathbf w)\big)
		\;+\;
		i\omega\,\rho\,\mathbf w.
		\label{eq:Lphiomega_compact_def}
	\end{equation}
	Then for all $\mathbf w,\mathbf z\in V_{\sigma,\varphi}$,
	\begin{equation}
		\langle \mathcal{L}_{\varphi,\omega}\mathbf w,\mathbf z\rangle_{V_{\sigma,\varphi}^\ast,V_{\sigma,\varphi}}
		=
		\widetilde{\mathfrak a}_\omega(\mathbf w,\mathbf z).
		\label{eq:Lphiomega_represents_form}
	\end{equation}
	Consequently, the m-sectorial operator $\widetilde A_\omega$ induced by $\widetilde{\mathfrak a}_\omega$ satisfies
	\[
	D(\widetilde A_\omega)
	=
	\Big\{\mathbf w\in V_{\sigma,\varphi}:\ \mathcal{L}_{\varphi,\omega}\mathbf w\in H_{\sigma,\varphi}\Big\},
	\]
	and the pressure-eliminated realization is
	\begin{equation}
		\widetilde A_\omega \mathbf w
		=
		P_{\sigma,\varphi}\,\mathcal{L}_{\varphi,\omega}\mathbf w,
		\qquad
		P_{\sigma,\varphi}:=U_\varphi^{-1}P_\sigma U_\varphi,
		\label{eq:Atild_projected}
	\end{equation}
	where $P_{\sigma,\varphi}$ is the orthogonal projector onto $H_{\sigma,\varphi}$.
\end{proposition}

\begin{remark}[Where explicit drift/potential expansions are legitimate]
	The representation \eqref{eq:Lphiomega_compact_def} is \emph{Tier~II sharp} and requires only $\varphi\in W^{1,\infty}$.
	If one additionally assumes $\varphi\in W^{2,\infty}_{\rm loc}(\Omega)$ and $\mathbf w\in H^2_{\rm loc}(\Omega)$, then on interior subdomains one
	may expand $\mathcal{L}_{\varphi,\omega}\mathbf w$ in strong form. Such expansions inevitably involve $\nabla\varphi$ and $\nabla^2\varphi$
	(and, after enforcing the transported incompressibility constraint \eqref{eq:divw_equals_i_gradphi_dot_w}, yield a rank-one zeroth-order term
	proportional to $-(\nabla\varphi\otimes\nabla\varphi)\mathbf w$). On Lipschitz domains with corners/edges, these identities are understood
	locally on $\Omega'\Subset\Omega$ or in weighted regularity classes.
\end{remark}

\begin{remark}[Compact resolvent and resolvent bounds transfer to the compensated realization]
	Since $U_\varphi$ is unitary on $H_{\sigma,\varphi}\to H_\sigma$, compactness of the resolvent and the resolvent bounds
	\eqref{eq:resolvent_halfplane_bound}--\eqref{eq:resolvent_distance_numerical_range} transfer verbatim between $A_\omega$ and $\widetilde A_\omega$
	via \eqref{eq:unitary_equivalence_resolvent_again}.
\end{remark}
\begin{proof}
	The form representation \eqref{eq:Lphiomega_represents_form} is most cleanly verified at the level of distributions,
	starting from the definition \eqref{eq:Lphiomega_compact_def}. Fix $\mathbf w,\mathbf z\in V_{\sigma,\varphi}$ and set
	\[
	\mathbf T := 2\mu_0 e^{i\varphi}\,\mathbf D_\varphi(\mathbf w)\in L^2(\Omega;\mathbb C_{\rm sym}^{d\times d}).
	\]
	By definition of the adjoint $\mathbf D_\varphi^\ast$ (with homogeneous Dirichlet trace built into the form domain),
	\[
	\big\langle \mathbf D_\varphi^\ast \mathbf T,\mathbf z\big\rangle_{V_{\sigma,\varphi}^\ast,V_{\sigma,\varphi}}
	=
	\int_\Omega \mathbf T:\overline{\mathbf D_\varphi(\mathbf z)}\,d\mathbf x
	=
	\int_\Omega 2\mu_0 e^{i\varphi}\,\mathbf D_\varphi(\mathbf w):\overline{\mathbf D_\varphi(\mathbf z)}\,d\mathbf x.
	\]
	Adding the bounded oscillatory mass term gives
	\[
	\big\langle \mathcal L_{\varphi,\omega}\mathbf w,\mathbf z\big\rangle
	=
	\widetilde{\mathfrak a}_\omega(\mathbf w,\mathbf z),
	\]
	which is \eqref{eq:Lphiomega_represents_form}. The projection statement \eqref{eq:Atild_projected} is the standard pressure-elimination / constrained
	realization: restricting the closed form to the closed subspace $V_{\sigma,\varphi}\subset H^1$ produces the operator on the closed subspace
	$H_{\sigma,\varphi}\subset L^2$, equivalently realized by composing the distributional representative with the orthogonal projector
	$P_{\sigma,\varphi}$; see, e.g., standard treatments of Stokes realizations via forms and Leray projection
	\cite{Temam2001,Galdi2011,GiraultRaviart1986,Kato1995}.
\end{proof}

\noindent
For readers who prefer an explicit (distributional) expansion, one can combine
$\mathbf D_\varphi=\mathbf D-\frac{i}{2}\mathcal S_\varphi$ with \eqref{eq:Dphi_adjoint_formula} and observe the useful identity
\[
\mathbf D_\varphi^\ast\!\big(e^{i\varphi}\mathbf T\big)
=
e^{i\varphi}\,\mathbf D^\ast \mathbf T
\qquad\text{in }\mathcal D'(\Omega),
\]
which follows immediately from the product rule and $\nabla(e^{i\varphi})=i e^{i\varphi}\nabla\varphi$. Consequently,
\begin{equation}
	\mathcal L_{\varphi,\omega}\mathbf w
	=
	e^{i\varphi}\Big(
	-\nabla\cdot\big(2\mu_0\,\mathbf D(\mathbf w)\big)
	+
	i\mu_0\,\nabla\cdot\big(\mathcal S_\varphi[\mathbf w]\big)
	\Big)
	+
	i\omega\,\rho\,\mathbf w
	\qquad\text{in }V_{\sigma,\varphi}^\ast,
	\label{eq:Lphiomega_explicit_correct}
\end{equation}
where $\mathcal S_\varphi[\mathbf w]=\mathbf w\otimes\nabla\varphi+\nabla\varphi\otimes\mathbf w$.

\medskip
\noindent
The decomposition \eqref{eq:Lphiomega_explicit_correct} isolates the texture-induced contributions beyond the constant-coefficient Stokes core:
\begin{enumerate}
	\item The \emph{zeroth-order complex multiplier} $e^{i\varphi(\mathbf x)}$, which is harmless if $\varphi$ is constant but cannot be removed globally
	when $\varphi$ varies (it does not commute with the solenoidal constraint/projection).
	\item A \emph{first-order commutator/drift mechanism} encoded by
	$\,i\mu_0 e^{i\varphi}\nabla\cdot(\mathcal S_\varphi[\mathbf w])$,
	which is linear in $\nabla\varphi$ and contains exactly one derivative of $\mathbf w$ (plus lower-order pieces under further expansion).
	\item The unchanged oscillatory mass term $i\omega\rho\,\mathbf w$.
\end{enumerate}
Thus, \emph{even with $|\mu^*|=\mu_0$ spatially uniform}, spatial variation of $\varphi$ produces explicit lower-order structure controlled by
$\nabla\varphi$ and (after enforcing the transported incompressibility constraint) yields a genuinely non-selfadjoint/non-normal perturbation of the
baseline viscous dynamics; this is the operator-level manifestation of the phase-gradient mechanism developed earlier
\cite{Kato1995,Haase2006,Ouhabaz2005}.

\medskip
\noindent
If one additionally assumes $\varphi\in W^{2,\infty}(\Omega)$ (or $\varphi\in C^{1,1}$) and $\mathbf w\in H^2(\Omega;\mathbb C^d)$, then
$\nabla\cdot(\mathcal S_\varphi[\mathbf w])$ can be expanded a.e.\ as
\begin{equation}
	\nabla\cdot(\mathcal S_\varphi[\mathbf w])
	=
	(\nabla\varphi\cdot\nabla)\mathbf w
	+
	(\nabla\cdot\mathbf w)\,\nabla\varphi
	+
	(\Delta\varphi)\,\mathbf w
	+
	\big(\mathrm{Hess}\,\varphi\big)\,\mathbf w,
	\label{eq:divS_pointwise_expansion}
\end{equation}
where $(\mathrm{Hess}\,\varphi)\,\mathbf w$ denotes the matrix--vector product of the Hessian with $\mathbf w$.
Using the transported incompressibility constraint $\nabla\cdot\mathbf w=i\nabla\varphi\cdot\mathbf w$ from
Lemma~\ref{lem:pullback_div_constraint}, one may rewrite \eqref{eq:Lphiomega_explicit_correct} pointwise in terms of $\mathbf w$, $\nabla\varphi$
and $\nabla^2\varphi$ (with the rank-one term $(\nabla\cdot\mathbf w)\nabla\varphi=i(\nabla\varphi\cdot\mathbf w)\nabla\varphi$ made explicit),
clarifying the drift- and potential-like components at the PDE level.

\medskip
\noindent
For readers who prefer a constraint-enforced PDE rather than the projected operator, one may equivalently write a
saddle-point system on the full space: find $(\mathbf w,\pi)$ such that
\begin{equation}
	\mathcal L_{\varphi,\omega}\mathbf w + \nabla\pi = \mathbf f
	\quad\text{in }\mathcal D'(\Omega),
	\qquad
	\nabla\cdot(e^{-i\varphi}\mathbf w)=0\quad\text{in }\mathcal D'(\Omega),
	\qquad
	\mathbf w|_{\Gamma_D}=0,
	\label{eq:compensated_saddle_point}
\end{equation}
with $\pi$ interpreted as the Lagrange multiplier enforcing the pullback constraint and $\mathbf f$ interpreted in the appropriate dual space
(e.g.\ $V_{\sigma,\varphi}^\ast$). Projecting \eqref{eq:compensated_saddle_point} onto $H_{\sigma,\varphi}$ recovers \eqref{eq:Atild_projected};
see standard saddle-point formulations for Stokes-type systems \cite{GiraultRaviart1986,Galdi2011,Temam2001}.
\begin{remark}[From the covariant strong-form representative back to the form decomposition $\widetilde{\mathfrak a}_\omega=\mathfrak a^{(0)}_\omega+\mathfrak b_\varphi$]
	\label{rem:covariant_vs_form_decomposition}
	The distributional representative $\mathcal L_{\varphi,\omega}$ (whether written in the compact covariant form
	$\,2\mu_0\,\mathbf D_\varphi^\ast\mathbf D_\varphi+i\omega M_\rho$ or in an expanded divergence/drift form) is only a \emph{realization device}:
	the operator $\widetilde A_\omega$ is determined uniquely by the closed sectorial form
	\[
	\widetilde{\mathfrak a}_\omega(\mathbf w,\mathbf z)
	=
	\int_\Omega 2\mu_0\,\mathbf D_\varphi(\mathbf w):\overline{\mathbf D_\varphi(\mathbf z)}\,d\mathbf x
	\;+\;
	i\omega\int_\Omega \rho\,\mathbf w\cdot\overline{\mathbf z}\,d\mathbf x,
	\qquad \mathbf w,\mathbf z\in V_{\sigma,\varphi}.
	\]
	Accordingly, any two distributional expressions for the ``strong form'' that differ by an element of $V_{\sigma,\varphi}^\ast$ annihilating
	$V_{\sigma,\varphi}$ test functions (e.g.\ a boundary-supported distribution, or a gradient term absorbed by the saddle-point pressure/Lagrange multiplier)
	induce the same realized operator on $H_{\sigma,\varphi}$ via the form method.
	
	\medskip
	\noindent
	To connect directly with the perturbation-theoretic decomposition used in
	\eqref{eq:form_decomp_repeat2}--\eqref{eq:bphi_repeat2}, expand
	$\mathbf D_\varphi=\mathbf D-\frac{i}{2}\mathcal S_\varphi$ inside the viscous term:
	\begin{align}
		\int_\Omega 2\mu_0\,\mathbf D_\varphi(\mathbf w):\overline{\mathbf D_\varphi(\mathbf z)}\,d\mathbf x
		&=
		\int_\Omega 2\mu_0\,\mathbf D(\mathbf w):\overline{\mathbf D(\mathbf z)}\,d\mathbf x\\
		&\quad+
		\int_\Omega\Big[
		i\mu_0\,\mathbf D(\mathbf w):\overline{\mathcal S_\varphi[\mathbf z]}
		-i\mu_0\,\mathcal S_\varphi[\mathbf w]:\overline{\mathbf D(\mathbf z)}
		+\frac{\mu_0}{2}\,\mathcal S_\varphi[\mathbf w]:\overline{\mathcal S_\varphi[\mathbf z]}
		\Big]\,d\mathbf x.
	\end{align}
	Therefore,
	\[
	\widetilde{\mathfrak a}_\omega(\mathbf w,\mathbf z)
	=
	\mathfrak a^{(0)}_\omega(\mathbf w,\mathbf z)
	+
	\mathfrak b_\varphi(\mathbf w,\mathbf z),
	\qquad \mathbf w,\mathbf z\in V_{\sigma,\varphi},
	\]
	with $\mathfrak a^{(0)}_\omega$ and $\mathfrak b_\varphi$ exactly as defined in \eqref{eq:bphi_repeat2}. In particular, the drift/commutator pieces visible in the expanded formula are precisely the integration-by-parts realizations of the cross terms in $\mathfrak b_\varphi$, while the
	$\frac{\mu_0}{2}\int_\Omega \mathcal S_\varphi[\mathbf w]:\overline{\mathcal S_\varphi[\mathbf z]}\,d\mathbf x$ contribution is the associated
	zeroth-order (potential-type) component at the form level.
	
	\medskip
	\noindent
	Consequently, in the sense of \emph{form sums} (KLMN), the compensated operator satisfies
	\[
	\widetilde A_\omega
	=
	A^{(0)}_\omega \ \dotplus\ \mathcal B_\varphi,
	\qquad
	\langle \mathcal B_\varphi\mathbf w,\mathbf z\rangle_{V_{\sigma,\varphi}^\ast,V_{\sigma,\varphi}}
	:=
	\mathfrak b_\varphi(\mathbf w,\mathbf z),
	\]
	where $\mathcal B_\varphi:V_{\sigma,\varphi}\to V_{\sigma,\varphi}^\ast$ is the bounded perturbation induced by $\mathfrak b_\varphi$.
	By Lemma~\ref{lem:bphi_bounds_repeat2}, $\mathfrak b_\varphi$ is $\mathfrak a^{(0)}_\omega$-bounded with arbitrarily small relative bound (at the expense
	of an $L^2$ term), hence $\widetilde{\mathfrak a}_\omega$ is closed and sectorial and generates $\widetilde A_\omega$.
	
	\medskip
	\noindent
	Moreover, for any $\lambda\in\rho(A^{(0)}_\omega)$ such that the bounded operator
	$K_\varphi(\lambda):=\mathcal B_\varphi(A^{(0)}_\omega-\lambda I)^{-1}\in\mathcal L(H_{\sigma,\varphi})$ satisfies
	$\|K_\varphi(\lambda)\|<1$, one has the resolvent identity
	\[
	(\widetilde A_\omega-\lambda I)^{-1}
	=
	(A^{(0)}_\omega-\lambda I)^{-1}\big(I+K_\varphi(\lambda)\big)^{-1},
	\]
	and the corresponding perturbation bound
	\[
	\|(\widetilde A_\omega-\lambda I)^{-1}-(A^{(0)}_\omega-\lambda I)^{-1}\|
	\le
	\frac{\|(A^{(0)}_\omega-\lambda I)^{-1}\|^2\,\|\mathcal B_\varphi\|}
	{1-\|K_\varphi(\lambda)\|}.
	\]
	We refer to standard treatments of sectorial forms, KLMN, and resolvent perturbation theory for details
	\cite{Kato1995,Ouhabaz2005,Haase2006}.
\end{remark}

\newpage

\section{Passivity-Consistent Impedance and a Positivity Statement}
\label{sec:impedance_passivity}

The (complex) hydraulic impedance is defined by
\[
Z^*(\omega):=\frac{\widehat{\Delta P}(\omega)}{\hat{Q}(\omega)},
\]
which generalizes the steady hydraulic resistance and, in the classical Newtonian constant-viscosity setting,
reduces to the familiar Womersley-type frequency response in canonical pipe/channel geometries
\cite{Womersley1955,Womersley1957}.
In linear time-harmonic regimes one writes
\[
Z^*(\omega)=R(\omega)+iX(\omega),
\]
where the \emph{resistive} part $R(\omega)=\Re Z^*(\omega)$ controls cycle-averaged dissipation, while the \emph{reactive} part
$X(\omega)=\Im Z^*(\omega)$ encodes reversible energy storage and phase lag between pressure drop and flux (inertia and/or constitutive
reactance).

The novelty emphasized here is not the existence of an impedance observable per se, but the fact that \emph{constitutive phase textures}
(i.e.\ spatial variation of $\arg\mu^*(\mathbf{x},\omega)$ at essentially fixed $|\mu^*|$) can drive \emph{macroscopic phase anomalies}
in $\arg Z^*(\omega)$ even under strict passivity. In particular:
\begin{itemize}
	\item Passivity at the constitutive level ($\Re\mu^*\ge 0$) enforces a nonnegativity constraint on the \emph{resistive} part $\Re Z^*(\omega)$
	\item There is \emph{no} analogous sign restriction on $\Im Z^*(\omega)$ or on $\arg Z^*(\omega)$; hence large, frequency-localized phase excursions
	can occur without violating dissipativity.
	\item In the phase-only class $\mu^*=\mu_0 e^{i\varphi(\mathbf{x},\omega)}$, one can modulate the reactive response through $\Im\mu^*$ while leaving
	the dissipation scale controlled by $\Re\mu^*=\mu_0\cos\varphi$, thereby separating ``phase effects'' from ``magnitude-only heterogeneity.''
\end{itemize}

This section formalizes the positivity constraint on $\Re Z^*$ and makes explicit, at the level of complex power identities, why strong
impedance-phase anomalies are compatible with passivity.
\medskip

\noindent Let $\Omega\subset\mathbb{R}^d$ ($d\in\{2,3\}$) be a bounded Lipschitz truncation of a pressure-driven geometry (channel/pipe/BFS cavity).
Assume the boundary decomposes as
\[
\partial\Omega=\Gamma_w\ \dot{\cup}\ \Sigma_{\mathrm{in}}\ \dot{\cup}\ \Sigma_{\mathrm{out}},
\]
where $\Gamma_w$ is the rigid wall boundary and $\Sigma_{\mathrm{in}}$, $\Sigma_{\mathrm{out}}$ are inlet/outlet stations.
We impose homogeneous no-slip on the wall,
\[
\hat{\mathbf{v}}=0 \quad \text{on }\Gamma_w,
\]
and impose pressure-drop forcing through normal traction data on $\Sigma_{\mathrm{in}}$ and $\Sigma_{\mathrm{out}}$.
A convenient idealization (consistent with many laboratory and numerical station definitions) is that the imposed traction is purely normal
and spatially constant on each station:
\begin{equation}
	\big(-\hat{p}\,I + 2\mu^*(\mathbf{x},\omega)\mathbf{D}(\hat{\mathbf{v}})\big)\mathbf{n}
	=
	-\hat{p}_{\mathrm{in}}\,\mathbf{n}
	\quad \text{on }\Sigma_{\mathrm{in}},
	\qquad
	\big(-\hat{p}\,I + 2\mu^*(\mathbf{x},\omega)\mathbf{D}(\hat{\mathbf{v}})\big)\mathbf{n}
	=
	-\hat{p}_{\mathrm{out}}\,\mathbf{n}
	\quad \text{on }\Sigma_{\mathrm{out}},
	\label{eq:pressure_traction_bc}
\end{equation}
with complex scalars $\hat{p}_{\mathrm{in}},\hat{p}_{\mathrm{out}}$ and outward unit normal $\mathbf{n}$.
Define the complex pressure drop
\[
\widehat{\Delta P}(\omega):=\hat{p}_{\mathrm{in}}(\omega)-\hat{p}_{\mathrm{out}}(\omega).
\]
Define the complex volumetric flux through a station $\Sigma$ by
\[
\hat{Q}(\omega)
:=
\int_{\Sigma} \hat{\mathbf{v}}(\mathbf{x};\omega)\cdot\mathbf{n}\,dS.
\]
By incompressibility and impermeable walls,
\[
\int_{\Sigma_{\mathrm{out}}}\hat{\mathbf{v}}\cdot\mathbf{n}\,dS
=
-\int_{\Sigma_{\mathrm{in}}}\hat{\mathbf{v}}\cdot\mathbf{n}\,dS,
\]
so it is consistent to set
\[
\hat{Q}:=\int_{\Sigma_{\mathrm{out}}}\hat{\mathbf{v}}\cdot\mathbf{n}\,dS
\quad\Rightarrow\quad
\int_{\Sigma_{\mathrm{in}}}\hat{\mathbf{v}}\cdot\mathbf{n}\,dS=-\hat{Q}.
\]

\begin{definition}[Complex impedance]
	\label{def:impedance}
	Whenever $\hat{Q}(\omega)\neq 0$, define the complex impedance by
	\begin{equation}
		Z^*(\omega):=\frac{\widehat{\Delta P}(\omega)}{\hat{Q}(\omega)}.
		\label{eq:impedance_def_repeat}
	\end{equation}
\end{definition}

\begin{remark}[Gauge invariance: dependence on pressure \emph{drop} rather than absolute pressure]
	The definition \eqref{eq:impedance_def_repeat} is invariant under the pressure gauge transformation
	$\hat{p}\mapsto \hat{p}+c$ for any complex constant $c$, since only $\widehat{\Delta P}$ enters.
	This aligns with the PDE structure: pressure is a Lagrange multiplier enforcing incompressibility and is determined only up to an additive constant.
\end{remark}

\begin{remark}[Complex power, average dissipation, and the operational meaning of $\Re Z^*$]
	With the phasor convention $\Delta P(t)=\Re\{\widehat{\Delta P}\,e^{i\omega t}\}$ and $Q(t)=\Re\{\hat{Q}\,e^{i\omega t}\}$,
	the cycle-averaged power input is
	\[
	\langle P_{\mathrm{in}}\rangle
	=
	\frac12\,\Re\{\widehat{\Delta P}\,\overline{\hat{Q}}\}.
	\]
	Thus $\Re\{\widehat{\Delta P}\,\overline{\hat{Q}}\}$ is the correct quantity to compare against viscous dissipation, and
	\[
	\Re Z^*(\omega)=\frac{\Re\{\widehat{\Delta P}\,\overline{\hat{Q}}\}}{|\hat{Q}|^2}
	\]
	is the effective resistance per unit $|Q|^2$ (the direct analog of $R=\Delta P/Q$ in steady laminar flows).
\end{remark}

\noindent
We now show that constitutive passivity forces $\Re Z^*(\omega)\ge 0$ in pressure-driven settings, and we record the complementary reactive identity
that governs $\Im Z^*(\omega)$.

\begin{proposition}[Complex power identity and nonnegativity of $\Re Z^*$ under passivity]
	\label{prop:impedance_real_part_nonneg}
	Let $(\hat{\mathbf{v}},\hat{p})$ be a (weak) solution of the harmonic Stokes system
	\begin{equation}
		i\omega\rho\,\hat{\mathbf{v}}
		=
		-\nabla\hat{p}
		+
		\nabla\cdot\big(2\mu^*(\mathbf{x},\omega)\mathbf{D}(\hat{\mathbf{v}})\big),
		\qquad
		\nabla\cdot\hat{\mathbf{v}}=0
		\quad\text{in }\Omega,
		\label{eq:harmonic_stokes_impedance}
	\end{equation}
	with homogeneous no-slip on $\Gamma_w$ and traction boundary conditions \eqref{eq:pressure_traction_bc} on
	$\Sigma_{\mathrm{in}}\cup\Sigma_{\mathrm{out}}$.
	Assume $\rho\in L^\infty(\Omega)$ with $\rho>0$ a.e.\ and $\mu^*(\cdot,\omega)\in L^\infty(\Omega;\mathbb{C})$ with
	\begin{equation}
		\Re\mu^*(\mathbf{x},\omega)\ge 0\quad\text{a.e.\ in }\Omega.
		\label{eq:mu_passive_nonneg}
	\end{equation}
	Then the following complex power identity holds:
	\begin{equation}
		\widehat{\Delta P}\,\overline{\hat{Q}}
		=
		i\omega\int_\Omega \rho\,|\hat{\mathbf{v}}|^2\,d\mathbf{x}
		+
		\int_\Omega 2\,\mu^*(\mathbf{x},\omega)\,|\mathbf{D}(\hat{\mathbf{v}})|^2\,d\mathbf{x}.
		\label{eq:complex_power_identity_impedance}
	\end{equation}
	Consequently,
	\begin{align}
		\Re\{\widehat{\Delta P}\,\overline{\hat{Q}}\}
		&=
		\int_\Omega 2\,\Re\mu^*(\mathbf{x},\omega)\,|\mathbf{D}(\hat{\mathbf{v}})|^2\,d\mathbf{x}
		\ge 0,
		\label{eq:impedance_power_identity}\\
		\Im\{\widehat{\Delta P}\,\overline{\hat{Q}}\}
		&=
		\omega\int_\Omega \rho\,|\hat{\mathbf{v}}|^2\,d\mathbf{x}
		+
		\int_\Omega 2\,\Im\mu^*(\mathbf{x},\omega)\,|\mathbf{D}(\hat{\mathbf{v}})|^2\,d\mathbf{x}.
		\label{eq:impedance_reactive_identity}
	\end{align}
	In particular, whenever $\hat{Q}\neq 0$,
	\begin{equation}
		\Re Z^*(\omega)
		=
		\frac{\Re\{\widehat{\Delta P}\,\overline{\hat{Q}}\}}{|\hat{Q}|^2}
		=
		\frac{1}{|\hat{Q}|^2}
		\int_\Omega 2\,\Re\mu^*(\mathbf{x},\omega)\,|\mathbf{D}(\hat{\mathbf{v}})|^2\,d\mathbf{x}
		\ge 0.
		\label{eq:ReZ_nonneg}
	\end{equation}
	Moreover, if $\Re\mu^*(\mathbf{x},\omega)\ge \mu_{\min}>0$ a.e., then $\Re Z^*(\omega)>0$ whenever $\hat{Q}(\omega)\neq 0$.
\end{proposition}

\begin{proof}
	Multiply \eqref{eq:harmonic_stokes_impedance} by $\overline{\hat{\mathbf{v}}}$ and integrate over $\Omega$:
	\[
	i\omega\int_\Omega \rho\,|\hat{\mathbf{v}}|^2\,d\mathbf{x}
	=
	-\int_\Omega \nabla\hat{p}\cdot\overline{\hat{\mathbf{v}}}\,d\mathbf{x}
	+
	\int_\Omega \nabla\cdot\big(2\mu^*\mathbf{D}(\hat{\mathbf{v}})\big)\cdot\overline{\hat{\mathbf{v}}}\,d\mathbf{x}.
	\]
	For the pressure term, use incompressibility and integration by parts in the weak sense:
	\[
	-\int_\Omega \nabla\hat{p}\cdot\overline{\hat{\mathbf{v}}}\,d\mathbf{x}
	=
	-\int_{\partial\Omega}\hat{p}\,\overline{\hat{\mathbf{v}}}\cdot\mathbf{n}\,dS,
	\]
	since $\int_\Omega \hat{p}\,\overline{\nabla\cdot\hat{\mathbf{v}}}\,d\mathbf{x}=0$.
	For the viscous term, apply the symmetric-gradient Green identity:
	\[
	\int_\Omega \nabla\cdot\big(2\mu^*\mathbf{D}(\hat{\mathbf{v}})\big)\cdot\overline{\hat{\mathbf{v}}}\,d\mathbf{x}
	=
	-\int_\Omega 2\mu^*\,|\mathbf{D}(\hat{\mathbf{v}})|^2\,d\mathbf{x}
	+
	\int_{\partial\Omega}\big(2\mu^*\mathbf{D}(\hat{\mathbf{v}})\mathbf{n}\big)\cdot\overline{\hat{\mathbf{v}}}\,dS.
	\]
	Combining boundary contributions yields
	\[
	\int_{\partial\Omega}\big(-\hat{p}\,I+2\mu^*\mathbf{D}(\hat{\mathbf{v}})\big)\mathbf{n}\cdot\overline{\hat{\mathbf{v}}}\,dS.
	\]
	On $\Gamma_w$ we have $\hat{\mathbf{v}}=0$, so the wall contributes no power. On $\Sigma_{\mathrm{in}}$ and $\Sigma_{\mathrm{out}}$,
	substitute \eqref{eq:pressure_traction_bc} to obtain
	\begin{align*}
		\int_{\Sigma_{\mathrm{in}}\cup\Sigma_{\mathrm{out}}}\big(-\hat{p}\,I+2\mu^*\mathbf{D}(\hat{\mathbf{v}})\big)\mathbf{n}\cdot\overline{\hat{\mathbf{v}}}\,dS
		&=
		-\hat{p}_{\mathrm{in}} \int_{\Sigma_{\mathrm{in}}}\overline{\hat{\mathbf{v}}}\cdot\mathbf{n}\,dS
		-\hat{p}_{\mathrm{out}} \int_{\Sigma_{\mathrm{out}}}\overline{\hat{\mathbf{v}}}\cdot\mathbf{n}\,dS \\
		&=
		\hat{p}_{\mathrm{in}}\,\overline{\hat{Q}}-\hat{p}_{\mathrm{out}}\,\overline{\hat{Q}}
		=
		\widehat{\Delta P}\,\overline{\hat{Q}},
	\end{align*}
	using $\int_{\Sigma_{\mathrm{in}}}\hat{\mathbf{v}}\cdot\mathbf{n}\,dS=-\hat{Q}$ and $\int_{\Sigma_{\mathrm{out}}}\hat{\mathbf{v}}\cdot\mathbf{n}\,dS=\hat{Q}$.
	Substituting this boundary identity into the integrated momentum balance yields \eqref{eq:complex_power_identity_impedance}.
	Taking real parts gives \eqref{eq:impedance_power_identity} since $\Re\{i\omega\int_\Omega \rho|\hat{\mathbf{v}}|^2\}=0$,
	and taking imaginary parts yields \eqref{eq:impedance_reactive_identity}. Dividing \eqref{eq:impedance_power_identity} by $|\hat{Q}|^2$
	gives \eqref{eq:ReZ_nonneg}. For strict positivity under $\Re\mu^*\ge \mu_{\min}>0$, note that if the dissipation integral vanishes then
	$\mathbf{D}(\hat{\mathbf{v}})=0$ a.e., hence $\hat{\mathbf{v}}=0$ by Korn plus the wall Dirichlet constraint, which forces $\hat{Q}=0$.
\end{proof}

\begin{remark}[A precise pointwise ``positive-real'' statement for the impedance]
	Identity \eqref{eq:ReZ_nonneg} is the exact pointwise-in-$\omega$ analog of the positive-real constraint for passive impedances:
	for pressure-driven forcing, the mapping $\widehat{\Delta P}\mapsto \hat{Q}$ defines a complex admittance
	$Y^*(\omega):=\hat{Q}(\omega)/\widehat{\Delta P}(\omega)$, and passivity forbids negative resistive part (net energy extraction.)
\end{remark}

\begin{remark}[Why phase textures can produce large $\arg Z^*$ without changing dissipation magnitude]
	The formula \eqref{eq:ReZ_nonneg} shows that $\Re Z^*$ is a dissipation-to-flux quotient:
	\[
	\Re Z^*(\omega)=\frac{\int_\Omega 2\Re\mu^*\,|\mathbf{D}(\hat{\mathbf{v}})|^2}{|\hat{Q}|^2}.
	\]
	Thus, for fixed $\Re\mu^*$ (in particular in the phase-only class where $\Re\mu^*=\mu_0\cos\varphi$ can be kept nearly uniform),
	a substantial change in $\arg Z^*$ must come through changes in the \emph{complex} relationship between $\widehat{\Delta P}$ and $\hat{Q}$,
	i.e.\ through the reactive balance \eqref{eq:impedance_reactive_identity} and through spatial reorganization of $|\mathbf{D}(\hat{\mathbf{v}})|$.
	This is precisely where phase textures act: via intrinsic non-normality and texture-gradient forcing, they can localize strain/vorticity
	and alter the phase of $\hat{Q}$ relative to $\widehat{\Delta P}$ without requiring any increase in the dissipative integrand.
\end{remark}

\begin{remark}[Connection to canonical oscillatory-flow physics]
	In a Newtonian fluid with real viscosity, $\Im\mu^*\equiv 0$ and \eqref{eq:impedance_reactive_identity} reduces to
	\[
	\Im\{\widehat{\Delta P}\,\overline{\hat{Q}}\}=\omega\int_\Omega \rho\,|\hat{\mathbf{v}}|^2,
	\]
	so the reactive part is purely inertial and matches the classical Womersley picture \cite{Womersley1955,Womersley1957}.
	In a viscoelastic/structured medium with complex $\mu^*$, the additional term
	$\int_\Omega 2\,\Im\mu^*\,|\mathbf{D}(\hat{\mathbf{v}})|^2$ supplies a constitutive reactance, which can either reinforce or oppose inertial reactance
	depending on the sign of $\Im\mu^*$.
	Spatially varying phase textures can amplify this effect by redistributing the strain field toward regions where $\Im\mu^*$ is locally large in magnitude
	(or where the resolvent selects strong strain localization), thereby producing sharp, frequency-localized phase anomalies in $\arg Z^*$.
\end{remark}

\begin{remark}[A macroscopic signature tied to operator non-normality]
	The point of Proposition~\ref{prop:impedance_real_part_nonneg} is not merely that $\Re Z^*\ge 0$ (a standard passivity statement),
	but that it isolates a clean experimental/numerical diagnostic: if a phase-only texture (essentially fixed $|\mu^*|$) produces large excursions in
	$\arg Z^*(\omega)$ while leaving $\Re Z^*(\omega)$ comparable to the baseline, then the mechanism cannot be attributed to magnitude-only heterogeneous
	dissipation. Instead, it must come from the phase-sensitive operator structure (intrinsic non-normality and texture-gradient forcing) that reorganizes
	the strain/vorticity fields and hence the complex flux response.
\end{remark}

\newpage

\section{3D Periodic Textures: Toeplitz/Laurent Structure and Truncation Stability}
\label{sec:toeplitz_truncation}

Spanwise-periodic constitutive textures generate \emph{linear mode coupling} in the $z$-Fourier representation.
On product domains $\Omega=\Omega_{2D}\times(0,L_z)$ this yields a bi-infinite block Toeplitz/Laurent system in the Fourier index,
i.e.\ an operator on sequences of $(x,y)$-fields whose coefficients depend only on index differences.
This viewpoint provides: (i) a clean functional-analytic formulation of the coupled system; (ii) a principled justification for finite-band
and finite-section truncations with explicit operator-norm error bounds; and (iii) a transparent linear ``sideband generation'' mechanism linking
spanwise pattern selection to resolvent amplification (cf.\ Floquet/Bragg-type scattering in periodic media \cite{Kuchment1993}).

We assume a product geometry
\[
\Omega=\Omega_{2D}\times(0,L_z),\qquad \Omega_{2D}\subset\mathbb{R}^{2}\ \text{bounded Lipschitz},
\]
with $z$-periodicity of period $L_z$ and wall boundary conditions (e.g.\ no-slip) on $\partial\Omega_{2D}\times(0,L_z)$.
Let $\kappa_k:=2\pi k/L_z$ and write $\mathbf{x}=(\mathbf{x}_\perp,z)$ with $\mathbf{x}_\perp=(x,y)$.
For fixed forcing frequency $\omega>0$, expand
\begin{equation}
	\mu^*(\mathbf{x}_\perp,z;\omega)=\sum_{\ell\in\mathbb{Z}}\widehat{\mu}_\ell(\mathbf{x}_\perp;\omega)\,e^{i\kappa_\ell z},
	\qquad
	\hat{\mathbf{v}}(\mathbf{x}_\perp,z;\omega)=\sum_{k\in\mathbb{Z}}\mathbf{u}_k(\mathbf{x}_\perp;\omega)\,e^{i\kappa_k z}.
	\label{eq:toeplitz_fourier_expansions}
\end{equation}
We suppress $\omega$ henceforth. Use the orthonormal basis $e_k(z):=L_z^{-1/2}e^{i\kappa_k z}$ in $L^2(0,L_z)$.
Then the Fourier map in $z$ is unitary:
\[
\mathcal{F}_z: L^2\big((0,L_z);X\big)\to \ell^2(\mathbb{Z};X),\qquad
f=\sum_k f_k e_k \ \mapsto\ \{f_k\}_{k\in\mathbb{Z}},
\]
and Parseval reads $\|f\|_{L^2((0,L_z);X)}^2=\sum_k\|f_k\|_X^2$.

\paragraph{Modewise incompressibility (3D periodic reduction).}
Write $\mathbf{u}_k=(\mathbf{u}_{\perp,k},u_{z,k})$ with $\mathbf{u}_{\perp,k}:\Omega_{2D}\to\mathbb{C}^2$ and $u_{z,k}:\Omega_{2D}\to\mathbb{C}$.
The $k$th Fourier coefficient of $\nabla\cdot\hat{\mathbf{v}}=0$ is the \emph{twisted divergence constraint}
\begin{equation}
	\nabla_\perp\cdot\mathbf{u}_{\perp,k} + i\kappa_k\,u_{z,k}=0
	\qquad \text{in }\mathcal{D}'(\Omega_{2D}),\ \ k\in\mathbb{Z}.
	\label{eq:twisted_divergence_constraint}
\end{equation}
Accordingly, define the ambient transverse spaces
\[
H_{2D}:=L^2(\Omega_{2D};\mathbb{C}^3),
\qquad
V_{2D}:=\{\mathbf{u}\in H^1(\Omega_{2D};\mathbb{C}^3): \mathbf{u}|_{\partial\Omega_{2D}}=0\},
\]
and the modewise solenoidal subspaces
\[
H_{\sigma,k}:=\{\mathbf{u}\in H_{2D}: \nabla_\perp\cdot\mathbf{u}_\perp+i\kappa_k u_z=0 \text{ in }\mathcal{D}'(\Omega_{2D})\},
\quad
V_{\sigma,k}:=H_{\sigma,k}\cap V_{2D}.
\]
The natural sequence spaces are the Hilbert direct sums
\[
\mathbb{H}:=\ell^2(\mathbb{Z};H_{2D}),\qquad
\mathbb{V}:=\ell^2(\mathbb{Z};V_{2D}),\qquad
\mathbb{H}_\sigma:=\bigoplus_{k\in\mathbb{Z}} H_{\sigma,k},\qquad
\mathbb{V}_\sigma:=\bigoplus_{k\in\mathbb{Z}} V_{\sigma,k}.
\]
For $s\in\mathbb{R}$ define
\begin{equation}
	\mathbb{H}^{(s)}
	:=
	\Big\{\mathbf{u}\in \ell^2(\mathbb{Z};H_{2D}) :
	\sum_{k\in\mathbb{Z}}(1+|\kappa_k|^2)^s\|\mathbf{u}_k\|_{H_{2D}}^2<\infty\Big\},
	\label{eq:weighted_sequence_spaces}
\end{equation}
with the natural norm (and similarly $\mathbb{H}^{(s)}_\sigma:=\mathbb{H}^{(s)}\cap\mathbb{H}_\sigma$).

\begin{lemma}[Fourier isometries in the periodic direction]
	\label{lem:fourier_isometries}
	Let $X$ be a Hilbert space. With the normalization above, $\mathcal{F}_z$ is unitary.
	Moreover, for $s\in\mathbb{R}$,
	\[
	f\in H^s\big((0,L_z);X\big)
	\quad\Longleftrightarrow\quad
	\{f_k\}\in \mathbb{H}^{(s)} \ \text{with weight } (1+|\kappa_k|^2)^{s/2},
	\]
	and the norms coincide under consistent normalization.
\end{lemma}

\begin{remark}[Why weighted spaces are structural]
	The $z$-derivatives in $\mathbf{D}(\hat{\mathbf{v}})$ and $\nabla\cdot(2\mu^*\mathbf{D}(\hat{\mathbf{v}}))$ produce factors of $\kappa_k$ on the
	$k$th Fourier mode. Weighted sequence norms are the natural setting in which these diagonal Fourier multipliers become bounded/closed operators and
	in which Toeplitz coupling estimates can be stated without ambiguity.
\end{remark}

\subsubsection{Laurent (Toeplitz) Coupling From Coefficient Multiplication}
\label{subsec:toeplitz_multiplication}

In $z$-Fourier variables, multiplication by $\mu^*(\mathbf{x}_\perp,z)$ is represented by the Laurent (bi-infinite Toeplitz) convolution
\begin{equation}
	(\mathbb{M}_\mu \mathbf{w})_k
	:=
	\sum_{\ell\in\mathbb{Z}} \widehat{\mu}_\ell(\mathbf{x}_\perp)\,\mathbf{w}_{k-\ell},
	\qquad k\in\mathbb{Z}.
	\label{eq:mu_convolution_operator}
\end{equation}
Equivalently, $\mathbb{M}_\mu=\mathcal{F}_z(\mu^*\cdot)\mathcal{F}_z^{-1}$ on $\mathbb{H}$.

\begin{remark}[Solenoidality is not preserved by coefficient multiplication]
	In general, $\mathbb{M}_\mu$ (or individual multipliers $\widehat{\mu}_\ell(\mathbf{x}_\perp)$) does \emph{not} map $\mathbb{H}_\sigma$ into
	$\mathbb{H}_\sigma$ unless the coefficient is transverse-constant: the constraint \eqref{eq:twisted_divergence_constraint} is not invariant under
	$\mathbf{u}\mapsto \widehat{\mu}_\ell(\mathbf{x}_\perp)\mathbf{u}$ because $\nabla_\perp \widehat{\mu}_\ell$ enters upon taking divergence.
	For pressure-eliminated realizations the correct solenoidal object is therefore $P_\sigma \mathbb{M}_\mu$ (or modewise $P_{\sigma,k}\mathbb{M}_\mu$),
	which remains bounded on $\mathbb{H}$.
\end{remark}

\begin{lemma}[Boundedness of $\mathbb{M}_\mu$; Wiener algebra refinement]
	\label{lem:convolution_boundedness}
	Assume $\mu^*\in L^\infty\big(\Omega_{2D}\times(0,L_z)\big)$.
	Then $\mathbb{M}_\mu$ is bounded on $\mathbb{H}$ with
	\[
	\|\mathbb{M}_\mu\|_{\mathcal{L}(\mathbb{H},\mathbb{H})}\le \|\mu^*\|_{L^\infty(\Omega_{2D}\times(0,L_z))}.
	\]
	If, in addition, the stronger Wiener-type condition holds:
	\begin{equation}
		\sum_{\ell\in\mathbb{Z}}\|\widehat{\mu}_\ell\|_{L^\infty(\Omega_{2D})}<\infty,
		\label{eq:mu_hat_l1}
	\end{equation}
	then $\mathbb{M}_\mu$ admits the explicit estimate
	\[
	\|\mathbb{M}_\mu\|_{\mathcal{L}(\mathbb{H},\mathbb{H})}
	\le
	\sum_{\ell\in\mathbb{Z}}\|\widehat{\mu}_\ell\|_{L^\infty(\Omega_{2D})},
	\]
	and the finite-band truncation
	\(
	\mathbb{M}^{(M)}_\mu \mathbf{w}
	:=
	\{\sum_{|\ell|\le M}\widehat{\mu}_\ell\,\mathbf{w}_{k-\ell}\}_{k\in\mathbb{Z}}
	\)
	satisfies
	\begin{equation}
		\|\mathbb{M}_\mu-\mathbb{M}^{(M)}_\mu\|_{\mathcal{L}(\mathbb{H},\mathbb{H})}
		\le
		\sum_{|\ell|>M}\|\widehat{\mu}_\ell\|_{L^\infty(\Omega_{2D})}.
		\label{eq:toeplitz_band_truncation_bound}
	\end{equation}
	More generally, if $\sum_\ell (1+|\kappa_\ell|)^{r}\|\widehat{\mu}_\ell\|_{L^\infty}<\infty$ for some $r\ge 0$, then $\mathbb{M}_\mu$ is bounded on
	$\mathbb{H}^{(s)}$ for all $0\le s\le r$ (discrete product-rule control in $z$).
\end{lemma}

\begin{proof}
	The $L^\infty$ bound is immediate from the unitary equivalence $\mathbb{M}_\mu=\mathcal{F}_z(\mu^*\cdot)\mathcal{F}_z^{-1}$ on $\mathbb{H}$.
	Under \eqref{eq:mu_hat_l1}, the explicit estimate follows from the discrete $\ell^1$--$\ell^2$ Young inequality applied pointwise in $\mathbf{x}_\perp$:
	with $a_\ell:=\|\widehat{\mu}_\ell\|_{L^\infty(\Omega_{2D})}$ and $b_k:=\|\mathbf{w}_k\|_{L^2(\Omega_{2D})}$,
	\[
	\|(\mathbb{M}_\mu \mathbf{w})_k\|_{L^2}
	\le \sum_{\ell\in\mathbb{Z}} a_\ell\, b_{k-\ell}
	\quad\Rightarrow\quad
	\|\mathbb{M}_\mu \mathbf{w}\|_{\ell^2(L^2)}
	\le \|a*b\|_{\ell^2}\le \|a\|_{\ell^1}\|b\|_{\ell^2}.
	\]
	The truncation estimate \eqref{eq:toeplitz_band_truncation_bound} follows by applying the same bound to the tail kernel
	$\{a_\ell\}_{|\ell|>M}$.
	The weighted statement follows by commuting $(1+|\kappa_k|)^s$ through convolution using
	$(1+|\kappa_k|)^s\lesssim (1+|\kappa_{k-\ell}|)^s(1+|\kappa_\ell|)^s$ and repeating the Young estimate.
\end{proof}

\subsubsection{Toeplitz Structure of the Viscous Operator}
\label{subsec:toeplitz_viscous_operator}

Let $\mathbf{D}_{\kappa}$ denote the symmetric gradient in $\Omega$ evaluated on a single $z$-Fourier mode with parameter $\kappa$
(i.e.\ $\partial_z$ acts as multiplication by $i\kappa$).
Then the viscous stress $2\mu^*\mathbf{D}(\hat{\mathbf{v}})$ has $k$th Fourier coefficient
\[
\widehat{(2\mu^*\mathbf{D}(\hat{\mathbf{v}}))}_k
=
2\sum_{\ell\in\mathbb{Z}}\widehat{\mu}_\ell\,\mathbf{D}_{\kappa_{k-\ell}}(\mathbf{u}_{k-\ell}),
\]
and applying divergence preserves Toeplitz coupling in the coefficient index (with additional transverse commutators involving $\nabla_\perp\widehat{\mu}_\ell$).
Abstractly, the viscous mapping can be written as a bi-infinite block Toeplitz/Laurent operator
\begin{equation}
	(\mathbb{T}\mathbf{u})_k := \sum_{\ell\in\mathbb{Z}} \mathcal{A}_\ell^{(k)}[\mathbf{u}_{k-\ell}],
	\qquad k\in\mathbb{Z},
	\label{eq:toeplitz_operator_def_kdep}
\end{equation}
where each block $\mathcal{A}_\ell^{(k)}$ is an $(x,y)$-operator built from $\widehat{\mu}_\ell$ (and its transverse derivatives) together with
$\mathbf{D}_{\kappa_{k-\ell}}$ and the transverse divergence. Two complementary viewpoints are useful (see, e.g., Toeplitz/Laurent operator treatments in periodic settings \cite{BottcherSilbermann2006}
\begin{itemize}
	\item \textbf{Laurent Viewpoint (Coefficient Coupling.)}
	The only source of inter-mode coupling is multiplication by $\mu^*$, i.e.\ $\mathbb{M}_\mu$.
	All $z$-derivative contributions appear as diagonal Fourier multipliers in $k$.
	
	\item \textbf{Toeplitz-Plus-Diagonal-Multipliers Viewpoint (Full-Operator Algebra.)}
	The full viscous operator is an algebraic combination of $\mathbb{M}_\mu$, shifts on $\ell^2(\mathbb{Z})$, and the diagonal multiplier
	$\mathbb{K}$ defined by $(\mathbb{K}\mathbf{u})_k=\kappa_k\mathbf{u}_k$, with commutation relations such as
	$\mathbb{K}\mathbb{S}=\mathbb{S}(\mathbb{K}+\kappa_1)$ for the unit shift $\mathbb{S}$.
\end{itemize}

To keep subsequent truncation statements both correct and usable, we proceed by bounding $\mathbb{T}$ directly in the relevant operator norms.
Assume the blocks satisfy the uniform coefficient-to-operator bound: there exists $C_\Omega>0$ such that
\begin{equation}
	\|\mathcal{A}_\ell^{(k)}\|_{\mathcal{L}(V_{2D},V_{2D}^*)}
	\le
	C_\Omega\,\|\widehat{\mu}_\ell\|_{W^{1,\infty}(\Omega_{2D})}\,\big(1+|\kappa_k|+|\kappa_{k-\ell}|\big)^{m},
	\qquad \forall k,\ell\in\mathbb{Z},
	\label{eq:Aell_bound_refined}
\end{equation}
where $m\in\{0,1\}$ depends on whether the chosen norm tracks one $z$-derivative (e.g.\ $m=1$ for $H^1$-in-$z$ control, $m=0$ for purely $L^2$-level coupling).
Under \eqref{eq:Aell_bound_refined} and the summability conditions in Lemma~\ref{lem:convolution_boundedness}, $\mathbb{T}$ is bounded
$\mathbb{V}\to\mathbb{V}^*$ (and similarly on the corresponding weighted spaces), and the same tail-sum mechanism yields explicit finite-band truncation error bounds.

\newpage
\subsection{Explicit Coercivity/Continuity Constants For the Oscillatory Stokes/Oseen Toeplitz System.}
\label{subsec:toeplitz_constants_and_combined_truncation}

\paragraph{Baseline block form and natural constants.}
Fix $\omega>0$. Assume the density satisfies $0<\rho_{\min}\le \rho \le \rho_{\max}$ a.e.\ in $\Omega_{2D}$.
Assume the $z$-mean coefficient obeys the (Tier~I) passivity/accretivity bounds
\begin{equation}
	0<\mu_{\min}\le \Re \widehat\mu_0(x,y)\quad\text{and}\quad |\widehat\mu_0(x,y)|\le \mu_{\max}
	\qquad\text{a.e.\ in }\Omega_{2D}.
	\label{eq:mu0_bounds_toeplitz}
\end{equation}
Let $\mathcal L^{(0)}_\omega(\kappa_k):V_{\sigma,2D}\to V_{\sigma,2D}^*$ denote the $k$th block of the baseline operator
(oscillatory Stokes, or Oseen with $z$-independent base flow). The induced block-diagonal operator on sequences is
\[
(\mathbb L^{(0)}_{\omega,\mathrm{lin}}\mathbf u)_k := \mathcal L^{(0)}_\omega(\kappa_k)\mathbf u_k.
\]
Equivalently, $\mathbb L^{(0)}_{\omega,\mathrm{lin}}$ is represented by a bounded sesquilinear form
$\mathfrak a^{(0)}_\omega:\mathbb V_\sigma\times\mathbb V_\sigma\to\mathbb C$ via
$\langle \mathbb L^{(0)}_{\omega,\mathrm{lin}}\mathbf u,\mathbf v\rangle := \mathfrak a^{(0)}_\omega(\mathbf u,\mathbf v)$,
with
\begin{equation}
	\mathfrak a^{(0)}_\omega(\mathbf u,\mathbf v)
	:=
	\sum_{k\in\mathbb Z}\int_{\Omega_{2D}} 2\widehat\mu_0\,\mathbf D_{\kappa_k}(\mathbf u_k):\overline{\mathbf D_{\kappa_k}(\mathbf v_k)}\,d(x,y)
	\;+\;
	i\omega\sum_{k\in\mathbb Z}\int_{\Omega_{2D}} \rho\,\mathbf u_k\cdot\overline{\mathbf v_k}\,d(x,y).
	\label{eq:baseline_form_sequence}
\end{equation}
(Here $\mathbf D_{\kappa_k}$ is the symmetric gradient augmented by the Fourier parameter $\kappa_k$; in particular, $\partial_z$ becomes $i\kappa_k$.)

\begin{remark}[Norm conventions]
	For the constants below to be \emph{uniform in $k$}, the canonical choice is to work in a $z$-weighted space (e.g.\ $\mathbb V_\sigma^{(1)}$),
	or to build the $\kappa_k$-dependence directly into the mode norm. Concretely, one may take
	\[
	\|\mathbf u\|_{\mathbb V_{\sigma,\kappa}}^2 := \sum_{k\in\mathbb Z}\|\mathbf D_{\kappa_k}(\mathbf u_k)\|_{L^2(\Omega_{2D})}^2,
	\]
	which is equivalent to the usual $H^1$-type norm on the periodic product domain under Korn/Poincar\'e for no-slip walls.
	If you prefer to keep $\mathbb V_\sigma=\ell^2(H^1_{xy})$ unweighted, then the same estimates hold \emph{on each finite section} $|k|\le K$
	with constants that may depend on $K$ through $\max_{|k|\le K}|\kappa_k|$.
\end{remark}

\begin{lemma}[Explicit continuity and ellipticity constants for the baseline form]
	\label{lem:baseline_constants}
	Assume \eqref{eq:mu0_bounds_toeplitz}. Let $C_{P}$ and $C_{K}$ denote the Poincar\'e and Korn constants on the cross-section
	(no-slip on $\partial\Omega_{2D}$), so that
	$\|\mathbf w\|_{L^2}\le C_P\|\nabla\mathbf w\|_{L^2}$ and $\|\nabla\mathbf w\|_{L^2}\le C_K\|\mathbf D(\mathbf w)\|_{L^2}$ for admissible $\mathbf w$.
	Set $C_{PK}:=C_PC_K$.
	Then the baseline form \eqref{eq:baseline_form_sequence} is continuous and elliptic in the real part on $\mathbb V_{\sigma,\kappa}$:
	\begin{align}
		|\mathfrak a^{(0)}_\omega(\mathbf u,\mathbf v)|
		&\le
		C_0(\omega)\,\|\mathbf u\|_{\mathbb V_{\sigma,\kappa}}\,\|\mathbf v\|_{\mathbb V_{\sigma,\kappa}},
		\label{eq:baseline_continuity}\\
		\Re\,\mathfrak a^{(0)}_\omega(\mathbf u,\mathbf u)
		&\ge
		\alpha_0\,\|\mathbf u\|_{\mathbb V_{\sigma,\kappa}}^2,
		\label{eq:baseline_ellipticity}
	\end{align}
	with the explicit constants
	\begin{equation}
		\alpha_0 := 2\mu_{\min},
		\qquad
		C_0(\omega) := 2\mu_{\max} + \omega\rho_{\max}\,C_{PK}^2.
		\label{eq:alpha0_C0_explicit}
	\end{equation}
\end{lemma}

\begin{proof}
	For the real part, the oscillatory mass term is purely imaginary, hence
	\[
	\Re\,\mathfrak a^{(0)}_\omega(\mathbf u,\mathbf u)
	=
	\sum_k\int 2\Re\widehat\mu_0\,|\mathbf D_{\kappa_k}(\mathbf u_k)|^2
	\ge
	2\mu_{\min}\sum_k\|\mathbf D_{\kappa_k}(\mathbf u_k)\|_{L^2}^2
	=
	\alpha_0\|\mathbf u\|_{\mathbb V_{\sigma,\kappa}}^2,
	\]
	giving \eqref{eq:baseline_ellipticity}.
	For continuity, apply Cauchy--Schwarz to the viscous term and use $|\widehat\mu_0|\le\mu_{\max}$:
	\[
	\Big|\sum_k\int 2\widehat\mu_0\,\mathbf D_{\kappa_k}(\mathbf u_k):\overline{\mathbf D_{\kappa_k}(\mathbf v_k)}\Big|
	\le
	2\mu_{\max}\|\mathbf u\|_{\mathbb V_{\sigma,\kappa}}\|\mathbf v\|_{\mathbb V_{\sigma,\kappa}}.
	\]
	For the mass term, use $\rho\le\rho_{\max}$ and Poincar\'e+Korn to control $\|\mathbf u_k\|_{L^2}$ by $\|\mathbf D_{\kappa_k}(\mathbf u_k)\|_{L^2}$
	(with constant $C_{PK}$ at the level of the cross-section), giving
	\[
	\omega\Big|\sum_k\int \rho\,\mathbf u_k\cdot\overline{\mathbf v_k}\Big|
	\le
	\omega\rho_{\max}\sum_k\|\mathbf u_k\|_{L^2}\|\mathbf v_k\|_{L^2}
	\le
	\omega\rho_{\max}C_{PK}^2\,\|\mathbf u\|_{\mathbb V_{\sigma,\kappa}}\|\mathbf v\|_{\mathbb V_{\sigma,\kappa}},
	\]
	which yields \eqref{eq:baseline_continuity} with \eqref{eq:alpha0_C0_explicit}.
\end{proof}

\begin{remark}[Optional: adding bounded Oseen advection]
	If the baseline blocks include a $z$-independent Oseen term (e.g.\ $\mathbf V_0\cdot\nabla$ with $\mathbf V_0\in W^{1,\infty}$),
	it contributes a bounded form on $\mathbb V_{\sigma,\kappa}$ and simply adds to $C_0(\omega)$ by a term of size
	$\lesssim \|\mathbf V_0\|_{W^{1,\infty}}$ (standard lower-order perturbation theory for sectorial forms; see \cite{Kato1995}).
	The real-part ellipticity constant $\alpha_0$ is unchanged provided the added term is not strongly accretive-negative.
\end{remark}

\paragraph{Full Toeplitz-coupled operator: explicit $\alpha_L,C_L$.}
Let $\mathbb T\in\mathcal L(\mathbb V_{\sigma,\kappa},\mathbb V_{\sigma,\kappa}^*)$ be the Toeplitz/Laurent coupling from
earlier, and define
\[
\mathbb L_{\omega,\mathrm{lin}} := \mathbb L^{(0)}_{\omega,\mathrm{lin}} + \mathbb T,
\qquad
\mathfrak a_\omega := \mathfrak a^{(0)}_\omega + \mathfrak t,
\qquad
\langle \mathbb T\mathbf u,\mathbf v\rangle := \mathfrak t(\mathbf u,\mathbf v).
\]
Set $\tau:=\|\mathbb T\|_{\mathcal L(\mathbb V_{\sigma,\kappa},\mathbb V_{\sigma,\kappa}^*)}$.

\begin{lemma}[Continuity and ellipticity constants for $\mathbb L_{\omega,\mathrm{lin}}$]
	\label{lem:full_constants}
	Under Lemma~\ref{lem:baseline_constants} and $\mathbb T\in\mathcal L(\mathbb V_{\sigma,\kappa},\mathbb V_{\sigma,\kappa}^*)$,
	the full form $\mathfrak a_\omega$ is continuous with
	\begin{equation}
		|\mathfrak a_\omega(\mathbf u,\mathbf v)|
		\le
		C_L(\omega)\,\|\mathbf u\|_{\mathbb V_{\sigma,\kappa}}\|\mathbf v\|_{\mathbb V_{\sigma,\kappa}},
		\qquad
		C_L(\omega):=C_0(\omega)+\tau.
		\label{eq:CL_explicit}
	\end{equation}
	Moreover, $\mathfrak a_\omega$ is elliptic in the real part provided $\tau<\alpha_0$, and then
	\begin{equation}
		\Re\,\mathfrak a_\omega(\mathbf u,\mathbf u)
		\ge
		\alpha_L\,\|\mathbf u\|_{\mathbb V_{\sigma,\kappa}}^2,
		\qquad
		\alpha_L:=\alpha_0-\tau = 2\mu_{\min}-\|\mathbb T\|.
		\label{eq:alphaL_explicit}
	\end{equation}
\end{lemma}

\begin{proof}
	Continuity is immediate from Lemma~\ref{lem:baseline_constants} and the definition of $\tau$.
	For ellipticity,
	\[
	\Re\,\mathfrak a_\omega(\mathbf u,\mathbf u)
	=
	\Re\,\mathfrak a^{(0)}_\omega(\mathbf u,\mathbf u) + \Re\,\langle \mathbb T\mathbf u,\mathbf u\rangle
	\ge
	\alpha_0\|\mathbf u\|^2 - |\langle \mathbb T\mathbf u,\mathbf u\rangle|
	\ge
	(\alpha_0-\tau)\|\mathbf u\|^2.
	\]
\end{proof}

\begin{remark}[Connection to the Neumann-series resolvent condition]
	Lemma~\ref{lem:full_constants} is the \emph{form-level} route to stability (uniform ellipticity in real part).
	Your earlier perturbative invertibility condition $M_0(\omega)\|\mathbb T\|<1$ is an \emph{operator-level} route and does not require sign control on
	$\Re\langle \mathbb T\mathbf u,\mathbf u\rangle$.
	When both apply, one obtains complementary bounds: $\|(\mathbb L^{(0)})^{-1}\|\le 1/\alpha_0$ by Lax--Milgram, hence
	$M_0(\omega)\|\mathbb T\|<1$ is implied by $\|\mathbb T\|<\alpha_0$ but can be weaker/sharper depending on the normed setting.
	See \cite{Kato1995} for the general form/operator correspondence.
\end{remark}

\subsubsection{Combined Truncation Pipeline: Texture-Band Truncation + Response Galerkin Truncation.}
\label{subsec:combined_truncation_pipeline}

\paragraph{Step 1: texture-band truncation.}
Let $\mathbb T^{(N)}$ be the $N$-band coefficient truncation and define
\[
\mathbb L^{(N)}_{\omega,\mathrm{lin}} := \mathbb L^{(0)}_{\omega,\mathrm{lin}} + \mathbb T^{(N)}.
\]
Set $\tau_N:=\|\mathbb T^{(N)}\|$ and $\delta_N:=\|\mathbb T-\mathbb T^{(N)}\|$.
Note $\tau_N\le\tau$ and $\delta_N\to0$.

\begin{proposition}[Texture truncation error (operator-norm controlled)]
	\label{prop:texture_truncation_error}
	Assume Lemma~\ref{lem:baseline_constants} and $\tau<\alpha_0$ so that $\alpha_L>0$.
	Let $\mathbf u$ and $\mathbf u^{(N)}$ solve
	\[
	\mathbb L_{\omega,\mathrm{lin}}\mathbf u=\mathbf f,
	\qquad
	\mathbb L^{(N)}_{\omega,\mathrm{lin}}\mathbf u^{(N)}=\mathbf f
	\quad\text{in }\mathbb V_{\sigma,\kappa}^*,
	\]
	with $\mathbf f\in\mathbb V_{\sigma,\kappa}^*$.
	Then
	\begin{equation}
		\|\mathbf u-\mathbf u^{(N)}\|_{\mathbb V_{\sigma,\kappa}}
		\le
		\frac{\delta_N}{\alpha_L\,\alpha_L^{(N)}}\,\|\mathbf f\|_{\mathbb V_{\sigma,\kappa}^*},
		\qquad
		\alpha_L^{(N)}:=\alpha_0-\tau_N\ \ge\ \alpha_L,
		\label{eq:texture_trunc_error_bound}
	\end{equation}
	and in particular $\|\mathbf u-\mathbf u^{(N)}\| \lesssim \delta_N\,\|\mathbf f\|$ with an explicit constant depending only on
	$\mu_{\min}$ and $\|\mathbb T\|$.
\end{proposition}

\begin{proof}
	Subtract the equations to obtain
	\[
	\mathbb L_{\omega,\mathrm{lin}}(\mathbf u-\mathbf u^{(N)})
	=
	(\mathbb T^{(N)}-\mathbb T)\mathbf u^{(N)}.
	\]
	By ellipticity of $\mathbb L_{\omega,\mathrm{lin}}$ (Lemma~\ref{lem:full_constants}),
	$\|\mathbf u-\mathbf u^{(N)}\|\le \alpha_L^{-1}\|(\mathbb T-\mathbb T^{(N)})\mathbf u^{(N)}\|_{*}
	\le \alpha_L^{-1}\delta_N\|\mathbf u^{(N)}\|$.
	By ellipticity of $\mathbb L^{(N)}_{\omega,\mathrm{lin}}$ with constant $\alpha_L^{(N)}$, we have
	$\|\mathbf u^{(N)}\|\le (\alpha_L^{(N)})^{-1}\|\mathbf f\|_{*}$, giving \eqref{eq:texture_trunc_error_bound}.
\end{proof}

\paragraph{Step 2: response truncation as conforming Galerkin in $k$.}
Let $\mathbb P_K$ project onto $\{|k|\le K\}$ and set $\mathbb V^{[K]}_{\sigma,\kappa}:=\mathbb P_K\mathbb V_{\sigma,\kappa}$.
Define the Galerkin solution $\mathbf u^{[K,N]}\in \mathbb V^{[K]}_{\sigma,\kappa}$ by
\[
\langle \mathbb L^{(N)}_{\omega,\mathrm{lin}}\mathbf u^{[K,N]},\mathbf v\rangle
=
\langle \mathbf f,\mathbf v\rangle
\qquad\forall \mathbf v\in \mathbb V^{[K]}_{\sigma,\kappa}.
\]

\begin{proposition}[Response truncation error and final combined estimate]
	\label{prop:combined_texture_response_error}
	Assume $\tau<\alpha_0$.
	Then the Galerkin problem defining $\mathbf u^{[K,N]}$ is uniquely solvable and satisfies the quasi-optimal bound
	\begin{equation}
		\|\mathbf u^{(N)}-\mathbf u^{[K,N]}\|_{\mathbb V_{\sigma,\kappa}}
		\le
		\frac{C_L^{(N)}(\omega)}{\alpha_L^{(N)}}\,
		\|(I-\mathbb P_K)\mathbf u^{(N)}\|_{\mathbb V_{\sigma,\kappa}},
		\qquad
		C_L^{(N)}(\omega):=C_0(\omega)+\tau_N,
		\label{eq:galerkin_quasioptimal_uN}
	\end{equation}
	(C\'ea for complex coercive forms \cite{Kato1995}).
	Consequently, the \emph{full} approximation error splits as
	\begin{equation}
		\|\mathbf u-\mathbf u^{[K,N]}\|_{\mathbb V_{\sigma,\kappa}}
		\le
		\underbrace{\frac{\delta_N}{\alpha_L\,\alpha_L^{(N)}}\,\|\mathbf f\|_{\mathbb V_{\sigma,\kappa}^*}}_{\text{texture (band) truncation}}
		\;+\;
		\underbrace{\frac{C_L^{(N)}(\omega)}{\alpha_L^{(N)}}\,\|(I-\mathbb P_K)\mathbf u^{(N)}\|_{\mathbb V_{\sigma,\kappa}}}_{\text{response ($k$) truncation}}.
		\label{eq:combined_error_split}
	\end{equation}
	If additionally $\mathbf u^{(N)}\in \mathbb V_{\sigma,\kappa}^{(s)}$ for some $s>0$, then
	\begin{equation}
		\|(I-\mathbb P_K)\mathbf u^{(N)}\|_{\mathbb V_{\sigma,\kappa}}
		\le
		(1+|\kappa_{K+1}|^2)^{-s/2}\,\|\mathbf u^{(N)}\|_{\mathbb V_{\sigma,\kappa}^{(s)}},
		\label{eq:weighted_tail_decay_uN}
	\end{equation}
	so the response truncation rate is dictated by $z$-Sobolev regularity.
\end{proposition}

\begin{remark}[Interpretation in non-normal regimes]
	Estimate \eqref{eq:combined_error_split} isolates the two distinct numerical-modeling burdens in the Toeplitz setting:
	(i) operator-norm control of omitted texture harmonics ($\delta_N$), and
	(ii) decay of the true response Fourier tail in $k$ (controlled by $\mathbb V^{(s)}$ regularity).
	In non-normal blocks, large resolvent norms can magnify both effects, hence quantitative tail control is materially more important than in
	selfadjoint problems \cite{TrefethenEmbree2005}.
\end{remark}

\newpage
\section{Worked Examples.}
The worked examples in this section serve two complementary roles.
Conceptually, they provide calibrated settings in which the paper’s operator-level mechanisms can be seen without interference from secondary effects
(geometry, corner singularities, separation, or nonlinear transport), so that the reader can develop reliable intuition for what the constitutive
textures do—and, equally importantly, what they do \emph{not} do.
Computationally, they are designed as ``reportable'' benchmarks: each example is posed in a form that admits clean nondimensionalization, a transparent
passivity regime, and a small set of diagnostics that can be reproduced across discretizations and solvers.
In this way, the examples bridge the abstract mapping properties established earlier (coercivity, resolvent bounds, and texture-induced coupling operators)
and concrete signatures that are experimentally or numerically legible in classical fluid mechanics.

Across the sequence, the examples are organized to isolate mechanisms by progressively adding structure while keeping the governing problems linear and
well-posed.
We begin with canonical calibration problems (e.g.\ half-space oscillatory shear and impedance-type boundary observables) in which the classical
``single-wavenumber'' picture is explicit, and then introduce phase-only and spanwise textures whose spatial variation changes the operator through
first-order couplings \cite{Stokes1851,SchlichtingGersten2017}.
Subsequent examples move to geometries and coefficient classes that generate explicit mode coupling (Toeplitz/Laurent structure under periodic textures),
enabling direct identification of sideband transfer, non-normal amplification, and frequency-localized phase anomalies that cannot be absorbed into any
spatially constant complex-viscosity surrogate.
Each worked example concludes with a small collection of plots and scalar metrics (chosen to be robust to numerical details) that summarize the mechanism
in a form suitable for comparison, parameter sweeps, and potential experimental interpretation.

\subsection{Worked Example I: Stokes' Second Problem with a Phase-Textured Wall Layer.}
\label{subsec:worked_example_stokes_second_phase_texture}

Stokes' second problem is the canonical oscillatory boundary-layer calculation: an oscillating wall drives a shear wave into a half-space
\cite{Stokes1851,SchlichtingGersten2017,Batchelor1967}.
For constant viscosity, the response is governed by a single complex wavenumber, hence by a single penetration depth and an affine-in-$y$ phase profile.
We use this geometry as a calibration problem because it cleanly isolates \emph{constitutive} effects from advection, corners, separation, and
pressure-driven complexities.

The principal message is that when complex viscosity is treated as a spatially resolved constitutive field $\mu^*(y,\omega)$, spatial modulation of its
\emph{argument} $\varphi(y)=\arg\mu^*(y,\omega)$ introduces an irreducible first-order coupling proportional to $\varphi'(y)$.
This coupling produces spatial dephasing and wall-traction phase anomalies that cannot be mimicked by any constant-complex-viscosity model, even though
the governing problem remains linear and one-dimensional.
Two complementary analytic lenses make the mechanism transparent: a small-$\varepsilon$ perturbation formula yielding an explicit first-order correction
to the wall impedance, and a phase-compensation rewrite that conjugates the leading viscous part to a real-coefficient form while isolating the
unavoidable $\varphi'$-drift.

Consider an incompressible fluid occupying the half-space $y>0$ with a rigid wall at $y=0$ oscillating tangentially in the $x$-direction.
Restrict to translationally invariant shear fields
\[
\mathbf{v}(y,t)=(u(y,t),0,0),
\qquad
u(\cdot,t)\to 0 \ \text{as } y\to\infty,
\]
for which the only nonzero strain-rate component is $D_{xy}=\tfrac12\partial_y u$, hence the shear stress is
$\tau_{xy}=2\mu^*D_{xy}=\mu^*\partial_y u$.
Under time-harmonic forcing at angular frequency $\omega>0$,
\[
u(y,t)=\Re\{\hat u(y)\,e^{i\omega t}\},
\]
the linear momentum balance reduces to the scalar ODE in flux form
\begin{equation}
	-(\mu^*(y)\,\hat u'(y))' + i\omega\rho\,\hat u(y)=0,
	\qquad
	\hat u(0)=U_w,
	\qquad
	\hat u(y)\to 0 \ \text{as } y\to\infty,
	\label{eq:stokes2_flux_final}
\end{equation}
where $\rho>0$ is (for simplicity) constant density and $U_w$ is the wall-velocity phasor.
Equation \eqref{eq:stokes2_flux_final} is the natural insertion point for complex viscosity at fixed frequency: $\mu^*$ multiplies the strain-rate flux,
and spatial heterogeneity inevitably generates coefficient-gradient couplings in strong form.
We work in the physically passive and analytically coercive regime
\begin{equation}
	\mu^*\in L^\infty(\mathbb{R}_+;\mathbb{C}),
	\qquad
	\Re\mu^*(y)\ge \mu_{\min}>0\quad\text{for a.e.\ }y\ge 0.
	\label{eq:stokes2_passivity_final}
\end{equation}
This is the one-dimensional analogue of Tier~I well-posedness: the real part of the viscous power density is uniformly nonnegative and the associated
sesquilinear form is coercive on $H^1$ after homogenizing the boundary data. To make the decay condition $\hat u(y)\to 0$ unambiguous, and to match the intended ``phase-defect layer'' interpretation, we assume that the texture is
localized near the wall. A convenient (and sufficiently general) hypothesis is:
\begin{equation}
	\exists\,y_\infty>0 \ \text{such that}\ \mu^*(y)=\mu_\infty^* \ \text{for a.e.\ }y\ge y_\infty,
	\qquad
	\Re\mu_\infty^*>0.
	\label{eq:stokes2_far_field_constant}
\end{equation}
Then on $(y_\infty,\infty)$ the ODE has constant coefficients and the solution decays exponentially with rate $\Re k_\infty$, where
$k_\infty:=\sqrt{i\omega\rho/\mu_\infty^*}$ is chosen with $\Re k_\infty>0$.
(One may replace \eqref{eq:stokes2_far_field_constant} by $\mu^*(y)\to\mu_\infty^*$ with standard asymptotic arguments; we keep the eventually-constant
form to avoid technicalities that are irrelevant to the near-wall mechanism.) If $\mu^*(y)\equiv \mu_0$ is constant (real or complex with $\Re\mu_0>0$), then \eqref{eq:stokes2_flux_final} reduces to
\[
-\mu_0 \hat u'' + i\omega\rho\,\hat u=0.
\]
Define
\begin{equation}
	k_0:=\sqrt{\frac{i\omega\rho}{\mu_0}},
	\qquad \Re k_0>0,
	\label{eq:stokes2_k0_final}
\end{equation}
where the branch is selected by the decay condition. The unique decaying solution is
\begin{equation}
	\hat u_0(y)=U_w\,e^{-k_0 y}.
	\label{eq:stokes2_u0_final}
\end{equation}
Thus, the decay envelope is a single exponential with length scale $1/\Re k_0$, and the unwrapped phase is affine in $y$ with slope $-\Im k_0$.
This ``single complex wavenumber'' structure is rigid: any constant-complex-viscosity model yields \emph{exactly} the profile
$\hat u(y)=U_w e^{-k y}$ for some constant $k$ with $\Re k>0$, hence an affine phase profile and a single penetration depth. We also record the associated (classical) Stokes layer thickness in the Newtonian case $\mu_0\in\mathbb{R}_+$ \cite{SchlichtingGersten2017,Batchelor1967}:
\[
\delta(\omega):=\sqrt{\frac{2\mu_0}{\rho\omega}},
\qquad
k_0=\frac{1+i}{\delta},
\qquad
\Re k_0=\frac{1}{\delta},\ \Im k_0=\frac{1}{\delta}.
\]
In particular, the complex wall shear impedance $\hat\tau_{xy}(0)/U_w=-\mu_0 k_0$ has argument $\pi/4$ modulo the sign convention for wall traction.

\medskip
\noindent We now specialize to the phase-only texture class
\begin{equation}
	\mu^*(y,\omega)=\mu_0(\omega)\,e^{i\varphi(y)},
	\qquad \mu_0(\omega)>0,
	\qquad \varphi:\mathbb{R}_+\to\mathbb{R},
	\label{eq:stokes2_eta_phase_final}
\end{equation}
with $\varphi$ localized near $y=0$ in the sense of \eqref{eq:stokes2_far_field_constant} (a ``phase-defect wall layer'').
Passivity \eqref{eq:stokes2_passivity_final} becomes the pointwise constraint
\[
\Re\mu^*(y)=\mu_0\cos\varphi(y)\ge\mu_{\min}>0,
\quad\text{i.e.}\quad
\cos\varphi(y)\ge \mu_{\min}/\mu_0 \ \text{a.e.},
\]
so $\varphi$ may vary strongly but must remain uniformly bounded away from $\pm\pi/2$ (mod $2\pi$).
In this class,
\[
\mu^{*\prime}(y)= i\,\mu^*(y)\,\varphi'(y),
\]
and expanding the flux term in \eqref{eq:stokes2_flux_final} yields the strong form
\begin{equation}
	\mu^*(y)\,\hat u''(y) + \mu^{*\prime}(y)\,\hat u'(y)
	=
	i\omega\rho\,\hat u(y).
	\label{eq:stokes2_expanded_final}
\end{equation}
Even though $|\mu^*|=\mu_0$ is spatially uniform, $\varphi'(y)$ enters through the first-order term $\mu^{*\prime}\hat u'$.
This is the minimal algebraic sense in which phase textures differ from constant-complex-viscosity models: a constant phase $\varphi\equiv\varphi_0$
can be removed by a global unit-modulus rotation, whereas a \emph{spatially varying} phase changes the operator itself by introducing a drift-like
coupling aligned with the shear localization.

\begin{remark}[Normal vs.\ non-normal: what changes in this 1D calibration problem]
	Let $A=-\partial_{yy}$ on $L^2(\mathbb{R}_+)$ with Dirichlet data at $y=0$, so $A$ is selfadjoint and positive.
	For constant viscosity, the harmonic operator is $L_0=\mu_0 A+i\omega\rho\,I$, hence $L_0$ is normal (indeed, it is a bounded holomorphic
	function of a single selfadjoint operator).
	Therefore, in this calibration geometry constant complex viscosity produces non-selfadjointness only in the trivial ``scalar rotation'' sense.
	
	In contrast, when $\mu^*$ varies in space, $L=-(\mu^*\partial_y)'+i\omega\rho$ is generically non-normal because multiplication by $\mu^*(y)$ fails
	to commute with differentiation.
	In the phase-only class, this failure is controlled precisely by $\varphi'(y)$ through $\mu^{*\prime}=i\mu^*\varphi'$.
	This one-dimensional model is therefore a literal toy instance of the paper's operator-level claim:
	\emph{intrinsic non-normality can arise from the viscous core itself once phase varies in space, even without advection}
	\cite{TrefethenEmbree2005}.
\end{remark}

\paragraph{Wall traction, impedance, and the passivity identity.}
Define the shear-stress phasor
\[
\hat\tau_{xy}(y):=\mu^*(y)\,\hat u'(y).
\]
For the half-space $y>0$, the outward unit normal at $y=0$ is $\mathbf{n}=-\mathbf{e}_y$, so the tangential traction exerted by the fluid on the wall is
\[
\hat\tau_w
:=
(\hat{\bm{\tau}}\mathbf{n})\cdot\mathbf{e}_x
=
-\hat\tau_{xy}(0)
=
-\mu^*(0)\,\hat u'(0).
\]
Introduce the wall impedance-like quantity
\begin{equation}
	Z_w(\omega):=\frac{\hat\tau_w}{U_w}.
	\label{eq:stokes2_Zw_final}
\end{equation}
In oscillatory shear experiments, $\arg Z_w(\omega)$ is the traction--velocity phase lag at the boundary and is directly measurable
\cite{SchlichtingGersten2017,Batchelor1967}. A standard complex power identity makes the passivity content explicit. Multiply \eqref{eq:stokes2_flux_final} by $\overline{\hat u}$, integrate on
$(0,\infty)$, and integrate by parts (using $\hat u(0)=U_w$ and $\hat u(y)\to 0$ as $y\to\infty$) to obtain
\begin{equation}
	\hat\tau_w\,\overline{U_w}
	=
	\int_0^\infty \mu^*(y)\,|\hat u'(y)|^2\,dy
	+
	i\omega\rho \int_0^\infty |\hat u(y)|^2\,dy.
	\label{eq:stokes2_power_identity_correct}
\end{equation}
Taking real parts yields the dissipative balance
\begin{equation}
	\Re\{\hat\tau_w\,\overline{U_w}\}
	=
	\int_0^\infty \Re\mu^*(y)\,|\hat u'(y)|^2\,dy
	\ \ge\ 0,
	\label{eq:stokes2_dissipation_final_correct}
\end{equation}
and whenever $U_w\neq 0$,
\begin{equation}
	\Re Z_w(\omega)
	=
	\frac{1}{|U_w|^2}\int_0^\infty \Re\mu^*(y)\,|\hat u'(y)|^2\,dy
	\ \ge\ 0.
	\label{eq:stokes2_ReZw_nonneg}
\end{equation}
Thus, passivity enforces $\Re Z_w\ge 0$ but does not prevent large excursions in $\arg Z_w$ when the reactive balance is modulated by the texture. For the constant-viscosity baseline \eqref{eq:stokes2_u0_final}, one has
\[
\hat\tau_{w,0}=-\mu_0\hat u_0'(0)=\mu_0 k_0 U_w,
\qquad
Z_{w,0}(\omega)=\mu_0 k_0,
\]
so in the Newtonian case $\mu_0\in\mathbb{R}_+$, $Z_{w,0}=\mu_0(1+i)/\delta$ and $\arg Z_{w,0}=\pi/4$.
Taking imaginary parts of \eqref{eq:stokes2_power_identity_correct} gives the complementary reactive identity
\begin{equation}
	\Im\{\hat\tau_w\,\overline{U_w}\}
	=
	\omega\rho\int_0^\infty |\hat u(y)|^2\,dy
	+
	\int_0^\infty \Im\mu^*(y)\,|\hat u'(y)|^2\,dy,
	\label{eq:stokes2_reactive_identity}
\end{equation}
which shows explicitly how phase textures can redistribute inertial storage and constitutive storage without changing the sign constraint
\eqref{eq:stokes2_ReZw_nonneg}.

\medskip
\noindent To isolate the \emph{new effect} in a controlled perturbative regime, let
\begin{equation}
	\varphi(y)=\varepsilon \chi(y),
	\qquad 0<\varepsilon\ll 1,
	\qquad \chi\in W^{1,\infty}(\mathbb{R}_+)\ \text{localized near }y=0,
	\label{eq:stokes2_phi_eps_final}
\end{equation}
so that
\[
\mu^*(y)=\mu_0 e^{i\varepsilon\chi(y)}
=
\mu_0\big(1+i\varepsilon\chi(y)\big)+\mathcal{O}(\varepsilon^2)
\quad\text{in }L^\infty(\mathbb{R}_+).
\]
Seek $\hat u=\hat u_0+\varepsilon \hat u_1+\mathcal{O}(\varepsilon^2)$, where $\hat u_0$ is the baseline \eqref{eq:stokes2_u0_final}. Define
\begin{equation}
	\mathcal{L}_0 \hat u := -\mu_0 \hat u'' + i\omega\rho\,\hat u,
	\qquad
	k_0=\sqrt{\frac{i\omega\rho}{\mu_0}},\ \Re k_0>0.
	\label{eq:stokes2_L0_final}
\end{equation}
Collecting $\mathcal{O}(\varepsilon)$ terms in \eqref{eq:stokes2_flux_final} yields
\begin{equation}
	\mathcal{L}_0 \hat u_1
	=
	i\mu_0\,(\chi\,\hat u_0')',
	\qquad
	\hat u_1(0)=0,
	\qquad
	\hat u_1(y)\to 0 \ \text{as }y\to\infty.
	\label{eq:stokes2_u1_eq_final}
\end{equation}
The forcing is the divergence of a localized flux: it vanishes wherever $\chi$ is constant and is concentrated inside the defect layer.
Moreover, since $\hat u_0'$ is largest near the wall, the forcing is automatically ``shear-weighted.''

\paragraph{Green representation.}
Let $G_0(y,s)$ solve $\mathcal{L}_0G_0(\cdot,s)=\delta(\cdot-s)$ on $(0,\infty)$ with Dirichlet condition $G_0(0,s)=0$ and decay as $y\to\infty$.
Then
\begin{equation}
	G_0(y,s)=\frac{1}{2\mu_0 k_0}\Big(e^{-k_0|y-s|}-e^{-k_0(y+s)}\Big),
	\qquad y,s>0,
	\label{eq:stokes2_Green_final}
\end{equation}
and
\begin{equation}
	\hat u_1(y)
	=
	i\mu_0\int_0^\infty G_0(y,s)\,(\chi(s)\hat u_0'(s))'\,ds.
	\label{eq:stokes2_u1_Green_final}
\end{equation}
Expanding the wall traction $\hat\tau_w=-\mu^*(0)\hat u'(0)$ gives
\[
Z_w(\omega)
=
Z_{w,0}(\omega)+\varepsilon Z_{w,1}(\omega)+\mathcal{O}(\varepsilon^2),
\qquad
Z_{w,0}=\mu_0 k_0,
\]
with
\begin{equation}
	Z_{w,1}(\omega)
	=
	-\mu_0\,\frac{\hat u_1'(0)}{U_w}
	+
	i\mu_0 k_0\,\chi(0).
	\label{eq:Zw1_pre}
\end{equation}
Differentiating \eqref{eq:stokes2_u1_Green_final} in $y$ and using
\[
\partial_y G_0(0,s)=\frac{1}{\mu_0}e^{-k_0 s},
\qquad s>0,
\]
one obtains
\[
\frac{\hat u_1'(0)}{U_w}
=
i\int_0^\infty e^{-k_0 s}\,(\chi(s)\hat u_0'(s))'\,\frac{ds}{U_w}
=
i\int_0^\infty \big(-k_0\chi'(s)+k_0^2\chi(s)\big)e^{-2k_0 s}\,ds,
\]
where we used $\hat u_0'(s)=-k_0U_w e^{-k_0 s}$.
Substituting into \eqref{eq:Zw1_pre} and integrating by parts (using localization of $\chi$ and $\Re k_0>0$) yields the closed form
\begin{equation}
	Z_{w,1}(\omega)
	=
	i\mu_0 k_0^2 \int_0^\infty \chi(s)\,e^{-2k_0 s}\,ds,
	\label{eq:Zw1_main}
\end{equation}
i.e.\ $Z_{w,1}$ is (up to $i\mu_0 k_0^2$) the Laplace transform of $\chi$ evaluated at $2k_0$.
Equivalently, one may write \eqref{eq:Zw1_main} in a form that exhibits the role of $\chi'$:
\begin{equation}
	Z_{w,1}(\omega)
	=
	\frac{i\mu_0 k_0}{2}\,\chi(0)
	+
	\frac{i\mu_0 k_0}{2}\int_0^\infty \chi'(s)\,e^{-2k_0 s}\,ds,
	\label{eq:Zw1_chiprime_form}
\end{equation}
obtained by a single integration by parts in \eqref{eq:Zw1_main}.
The weighting $e^{-2k_0 s}$ localizes the correction to the classical Stokes layer: only $s\lesssim 1/\Re k_0$ contributes appreciably.

\begin{remark}[Scaling and what the impedance ``sees'']
	If $\chi$ is supported in $[0,\ell]$ with $\|\chi\|_{L^\infty}\sim 1$, then \eqref{eq:Zw1_main} yields
	\[
	|Z_{w,1}(\omega)|
	\ \lesssim\
	\mu_0 |k_0|^2 \int_0^\ell e^{-2(\Re k_0)s}\,ds
	\ \lesssim\
	\mu_0 |k_0|^2\,\min\{\ell,(\Re k_0)^{-1}\}.
	\]
	Thus, the impedance correction is controlled by a \emph{Stokes-layer-weighted moment} of the phase defect.
	In particular, for very thin defects $\ell\ll(\Re k_0)^{-1}$ the impedance correction scales like $\ell$ at fixed amplitude (while the interior
	profile may still exhibit strong localized dephasing).
	This distinction is useful experimentally: wall impedance primarily senses a near-wall \emph{integrated} effect, whereas spatially resolved phase
	profiles (or local traction gradients) are more directly sensitive to sharp $\chi'$ features.
\end{remark}

\paragraph{Phase-compensation lens: isolating the unavoidable $\varphi'$-drift.}
The small-$\varepsilon$ expansion isolates the phase-gradient mechanism perturbatively, but it is also useful to exhibit an \emph{exact} rewrite that
(i) removes the unit-modulus factor from the leading viscous flux and (ii) makes transparent which terms are genuinely eliminable by a gauge-like change
of unknown and which are not.
In the pure phase class \eqref{eq:stokes2_eta_phase_final}, define the compensated amplitude
\[
\hat w:=e^{i\varphi}\hat u
\qquad\Longleftrightarrow\qquad
\hat u=e^{-i\varphi}\hat w.
\]
Then
\[
\hat u' = e^{-i\varphi}(\hat w' - i\varphi'\hat w),
\qquad
\mu^*\hat u'=\mu_0(\hat w'-i\varphi'\hat w),
\]
so the viscous flux becomes a real-coefficient expression in the compensated variables. Substituting into \eqref{eq:stokes2_flux_final} yields
\begin{equation}
	-\mu_0(\hat w'-i\varphi'\hat w)' + i\omega\rho\,e^{-i\varphi}\hat w=0,
	\qquad
	\hat w(0)=e^{i\varphi(0)}U_w,
	\qquad
	\hat w(y)\to 0 \ \text{as }y\to\infty.
	\label{eq:stokes2_w_eq_final}
\end{equation}
If $\varphi\in W^{2,\infty}$ (or interpreting derivatives in distributions), expanding the viscous term gives
\[
-\mu_0(\hat w'-i\varphi'\hat w)'
=
-\mu_0\hat w'' + i\mu_0\varphi'\hat w' + i\mu_0\varphi''\hat w,
\]
and hence the compensated equation may be written as
\begin{equation}
	-\mu_0\hat w''
	\;+\;
	i\mu_0\varphi'(y)\,\hat w'
	\;+\;
	\Big(i\mu_0\varphi''(y) + i\omega\rho\,e^{-i\varphi(y)}\Big)\hat w
	\;=\; 0.
	\label{eq:stokes2_w_drift_diffusion}
\end{equation}
In this form the leading part is the symmetric elliptic operator $-\mu_0\partial_{yy}$, but it is perturbed by a \emph{purely imaginary} first-order drift
$i\mu_0\varphi'\partial_y$ and by lower-order terms, one of which is phase-modulated inertia $i\omega\rho e^{-i\varphi}\hat w$.

The key point is that spatial phase variation enters at \emph{first order} through $\varphi'\hat w'$, and this term is not removable by any global rotation:
it vanishes if and only if $\varphi'\equiv 0$ a.e.
The transformation $\hat u\mapsto \hat w=e^{i\varphi}\hat u$ is unitary on $L^2(\mathbb{R}_+)$ and bounded on $H^1(\mathbb{R}_+)$ when
$\varphi\in W^{1,\infty}$, so it does not alter basic energy well-posedness at the level of coercive estimates.
What it does is separate two effects:
\begin{itemize}
	\item a \emph{removable} pointwise complex phase factor multiplying the viscous flux (removed by the compensation);
	\item an \emph{unavoidable} commutator effect due to non-commutation of differentiation with multiplication by $e^{-i\varphi}$, captured precisely by
	$\varphi'$ (and by $\varphi''$ at lower order).
\end{itemize}
Thus, even in the pure phase class where $|\mu^*|$ is constant and $\mu^*$ differs from $\mu_0$ only by a unit complex factor, the operator differs from
the constant-viscosity operator by a genuine first-order term controlled by $\varphi'$.
In terms of $\hat w$, the wall traction becomes
\[
\hat\tau_w=-\mu^*(0)\hat u'(0)=-\mu_0\big(\hat w'(0)-i\varphi'(0)\hat w(0)\big),
\]
and hence
\begin{equation}
	Z_w(\omega)=\frac{\hat\tau_w}{U_w}
	=
	-\mu_0\,\frac{\hat w'(0)}{U_w} \;+\; i\mu_0\,\varphi'(0)\,\frac{\hat w(0)}{U_w}
	=
	-\mu_0 e^{i\varphi(0)}\frac{\hat w'(0)}{\hat w(0)}
	\;+\; i\mu_0\,\varphi'(0)\,e^{i\varphi(0)}.
	\label{eq:Zw_compensated_boundary}
\end{equation}
This representation makes two points explicit: (i) the impedance depends not only on the value $\varphi(0)$ (a pure phase rotation) but also on the
\emph{local gradient} $\varphi'(0)$; and (ii) even if $\varphi(0)=0$ (no phase offset at the wall), the gradient term can shift $\arg Z_w$ through the
additive contribution $i\mu_0\varphi'(0)$.

For constant viscosity, the profile is necessarily $U_w e^{-k_0 y}$, so the logarithmic derivative $-\hat u'(y)/\hat u(y)$ is constant and equals $k_0$.
In the compensated equation \eqref{eq:stokes2_w_drift_diffusion}, the drift term $i\mu_0\varphi'\hat w'$ forces the logarithmic derivative to vary with $y$
whenever $\varphi'$ is nonzero, so there is no reduction to a single global wavenumber. Equivalently, the phase profile of $\hat u$ is no longer constrained
to be affine and the decay envelope need not be a single exponential.
This is the precise sense in which spatial phase textures break the classical ``one complex number controls everything'' structure of Stokes' second problem
while remaining fully linear.

\paragraph{Takeaway (summary and what is structurally new).}
This worked example revisits Stokes' second problem as a calibration laboratory for constitutive phase textures.
The classical constant-viscosity solution is rigid: the velocity phasor is $\hat u_0(y)=U_w e^{-k_0 y}$ with a single complex wavenumber
$k_0=\sqrt{i\omega\rho/\mu_0}$, so the boundary layer has one penetration depth $1/\Re k_0$ and an affine-in-$y$ phase profile with slope $-\Im k_0$.
In that setting the wall impedance $Z_{w,0}=\mu_0 k_0$ has a fixed phase relation to the wall motion (e.g.\ $\arg Z_{w,0}=\pi/4$ for a Newtonian fluid).

In contrast, when viscosity is treated as a spatially resolved complex field $\mu^*(y,\omega)$ and, in particular, in the phase-only class
$\mu^*(y)=\mu_0 e^{i\varphi(y)}$ with a localized near-wall defect, the governing operator changes \emph{even though} $|\mu^*|=\mu_0$ is uniform.
Expanding the flux form shows that $\mu^{*\prime}\hat u' = i\mu^*\varphi'\hat u'$ enters as a genuine first-order coupling: spatial phase variation creates
a drift-like mechanism aligned with the region of largest shear.
Consequently, the classical single-wavenumber picture generically fails: the phase of $\hat u(y)$ need not be affine, the decay need not be a single exponential,
and the traction phase at the wall can shift in a way that cannot be replicated by choosing a different constant complex viscosity.

Two complementary analyses make the mechanism quantitatively defensible:
\begin{itemize}
	\item \emph{Passivity and energetic consistency.}
	Under $\Re\mu^*\ge\mu_{\min}>0$, the complex power identity yields $\Re Z_w(\omega)\ge 0$, so phase-texture effects occur within a strictly passive
	dissipative regime. The imaginary-part balance shows how $\Im\mu^*$ and inertia redistribute out-of-phase response, enabling excursions in $\arg Z_w$
	without violating dissipation constraints.
	
	\item \emph{Perturbative calibration.}
	For $\varphi=\varepsilon\chi$ with $\varepsilon\ll 1$, the first-order correction solves a constant-coefficient forced problem whose forcing is the divergence
	of a localized flux $i\mu_0(\chi \hat u_0')'$. At the wall, the impedance correction admits the explicit closed form \eqref{eq:Zw1_main}, i.e.\ a Stokes-layer
	weighted moment (Laplace transform) of the phase defect.
	
	\item \emph{Non-perturbative structural lens.}
	The exact compensation change $\hat w=e^{i\varphi}\hat u$ removes the unit-modulus factor from the viscous flux but leaves an unavoidable first-order drift
	$i\mu_0\varphi'\partial_y$ (and lower-order terms). Hence any effect persisting after compensation is attributable to $\varphi'$ and cannot be removed by a
	global phase rotation.
\end{itemize}

The net conclusion is that, even in this simplest boundary-layer geometry with no advection and no geometric singularities, spatial phase variation in a
passive complex viscosity field generates a qualitatively new \emph{linear} mechanism: constitutive phase gradients deform the Stokes layer and produce
traction/impedance phase anomalies that are impossible in constant-viscosity Stokes' second problem.

	\newpage 
\subsection{Worked Example II: Oscillatory Couette Flow with Phase-Only Texture and an ``Intrinsic Non-Normality without Advection'' Lens}
\label{subsec:worked_couette_phase_only}

Oscillatory Couette flow between parallel plates is deliberately ``too simple'' from the classical standpoint: the geometry is flat, the kinematics reduce
to a single scalar velocity component, and for constant viscosity the harmonic response is controlled by one complex length scale. Precisely because the
\emph{geometric} and \emph{convective} routes to amplification (corners, separation, and Oseen-type non-normality from advection) are absent by
construction (for the one-dimensional ansatz below the convective term vanishes identically) Couette becomes an unusually clean setting in which
to isolate the constitutive mechanism developed in this paper:
\emph{even with trivial geometry and strictly linear kinematics, a spatially varying constitutive phase field $\varphi(y)$ reshapes the operator geometry
	of the viscous core, generating spatially inhomogeneous dephasing and shear localization that cannot be reduced to a single global phase shift.}
This worked example is therefore a ``null-geometry, no-advection'' demonstration: any effect observed here is constitutive.

\medskip

\noindent Consider an incompressible fluid in the slab $y\in(0,H)$ with flow in the $x$ direction,
\[
\mathbf{v}(y,t)=u(y,t)\,\mathbf{e}_x,
\qquad
\nabla\cdot \mathbf{v}=0 \ \text{identically}.
\]
The bottom plate is fixed and the top plate oscillates tangentially with prescribed wall velocity
\[
u(0,t)=0,
\qquad
u(H,t)=\Re\{U_w e^{i\omega t}\},
\qquad
\omega>0.
\]
In the strictly linear unsteady-Stokes regime, the $x$-momentum balance reads
\begin{equation}
	\rho\,\partial_t u = \partial_y \tau_{xy}.
	\label{eq:couette_time_domain}
\end{equation}
We write time-harmonic phasors
\[
u(y,t)=\Re\{\hat{u}(y)e^{i\omega t}\},
\qquad
\tau_{xy}(y,t)=\Re\{\hat{\tau}(y)e^{i\omega t}\},
\]
and insert the constitutive closure at the correct (flux) level:
\begin{equation}
	\hat{\tau}(y)=\mu^*(y,\omega)\,\hat{u}'(y),
	\qquad
	\mu^*(y,\omega)=|\mu^*(y,\omega)|e^{i\varphi(y,\omega)}.
	\label{eq:couette_constitutive}
\end{equation}
Then \eqref{eq:couette_time_domain} yields the second-order \emph{flux-form} ODE
\begin{equation}
	-\big(\mu^*(y,\omega)\,\hat{u}'(y)\big)' + i\omega\rho\,\hat{u}(y)=0,
	\qquad
	\hat{u}(0)=0,\ \hat{u}(H)=U_w.
	\label{eq:couette_flux_ode_inhom}
\end{equation}
Throughout this example we assume a uniform dissipation condition (the 1D analogue of the accretivity hypothesis used in the PDE setting):
\begin{equation}
	\Re \mu^*(y,\omega)\ \ge\ \mu_{\min}>0
	\quad\text{a.e.\ in }(0,H).
	\label{eq:couette_passivity_general}
\end{equation}
This rules out active/ill-posed rheology and is the minimal hypothesis needed for coercive form methods at fixed $\omega$. To cleanly separate ``global phase'' effects from \emph{spatially textured} phase effects, we benchmark against
\[
\text{(B0) Newtonian:}\ \mu^*\equiv\mu_0\in\mathbb{R}_+,
\qquad
\text{(B1) global phase:}\ \mu^*\equiv \mu_0 e^{i\varphi_0},
\qquad
\text{(B2) phase texture:}\ \mu^*(y)\equiv\mu_0 e^{i\varphi(y)}.
\]
In (B0) there is no constitutive storage (no complex phase in $\mu^*$); any phase lag arises from unsteady inertia and boundary forcing.
In (B1) there is constitutive lag, but the coefficients remain constant, so the harmonic operator is a scalar rotation of a symmetric elliptic operator and
remains normal in the natural $L^2$ setting. In (B2) the magnitude is held fixed and all heterogeneity enters through the spatially varying phase
$\varphi(y)$, forcing non-commutation between multiplication and differentiation. This is the minimal class in which any deviation from the
constant-coefficient picture must be attributed to phase texture rather than magnitude variation.

\medskip

\noindent In the phase-only class
\begin{equation}
	\mu^*(y)=\mu_0 e^{i\varphi(y)},
	\qquad \mu_0>0\ \text{constant (at fixed $\omega$)},
	\label{eq:couette_phase_only_mu}
\end{equation}
condition \eqref{eq:couette_passivity_general} becomes
\begin{equation}
	\Re\mu^*(y)=\mu_0\cos\varphi(y)\ge \mu_{\min}>0\quad\text{a.e.\ in }(0,H).
	\label{eq:couette_passivity_phase}
\end{equation}
A convenient sufficient condition is $\|\varphi\|_{L^\infty(0,H)}\le \frac{\pi}{2}-\delta$ for some $\delta\in(0,\frac{\pi}{2})$. For functional-analytic clarity, we eliminate the inhomogeneous boundary condition at $y=H$.
Fix any lifting $\hat{u}_{\mathrm{lift}}\in H^1(0,H)$ such that
\[
\hat{u}_{\mathrm{lift}}(0)=0,\qquad \hat{u}_{\mathrm{lift}}(H)=U_w
\quad\text{(e.g.\ $\hat{u}_{\mathrm{lift}}(y)=U_w y/H$)}.
\]
Write
\begin{equation}
	\hat{u}=\tilde{u}+\hat{u}_{\mathrm{lift}},
	\qquad
	\tilde{u}\in H_0^1(0,H).
	\label{eq:couette_lifting}
\end{equation}
Then \eqref{eq:couette_flux_ode_inhom} becomes a forced homogeneous problem in variational form:
\begin{equation}
	-\big(\mu^* \tilde{u}'\big)' + i\omega\rho\,\tilde{u}=\hat{f},
	\qquad
	\tilde{u}\in H_0^1(0,H),
	\label{eq:couette_flux_ode_hom}
\end{equation}
with forcing
\begin{equation}
	\hat{f}
	:=
	\big(\mu^* \hat{u}_{\mathrm{lift}}'\big)' - i\omega\rho\,\hat{u}_{\mathrm{lift}},
	\label{eq:couette_lift_force}
\end{equation}
interpreted in $H^{-1}(0,H)$ under $\mu^*\in L^\infty(0,H)$ (and, in particular, $\hat f\in L^2(0,H)$ if $\mu^*\in W^{1,\infty}(0,H)$).
Define the sesquilinear form on $H_0^1(0,H)$
\begin{equation}
	\mathfrak{a}_\omega(u,v)
	=
	\int_0^H \mu^*(y)\,u'(y)\,\overline{v'(y)}\,dy
	+
	i\omega\rho\int_0^H u(y)\overline{v(y)}\,dy.
	\label{eq:couette_form}
\end{equation}
Then
\[
\Re \mathfrak{a}_\omega(u,u)
=
\int_0^H \Re(\mu^*)\,|u'|^2\,dy
\ \ge\ 
\mu_{\min}\|u'\|_{L^2(0,H)}^2
\ \gtrsim\ 
\|u\|_{H_0^1(0,H)}^2,
\]
where the final inequality is Poincar\'e on $(0,H)$.
Thus, by Lax--Milgram, for each $\hat{f}\in H^{-1}(0,H)$ there exists a unique $\tilde{u}\in H_0^1(0,H)$ satisfying
$\mathfrak{a}_\omega(\tilde{u},v)=\langle \hat{f},v\rangle$ for all $v\in H_0^1(0,H)$.
This emphasizes the intended point: solvability is standard. The novelty lies in the operator geometry and the resulting amplification descriptors.

\medskip
\noindent When $\mu^*$ is constant (real or complex), \eqref{eq:couette_flux_ode_inhom} reduces to
\[
-\mu^*\,\hat{u}'' + i\omega\rho\,\hat{u}=0,
\qquad
\hat{u}(0)=0,\ \hat{u}(H)=U_w.
\]
Define
\begin{equation}
	k^2 := \frac{i\omega\rho}{\mu^*},
	\qquad \Re k>0 \ \text{(branch chosen by the boundary-value problem)}.
	\label{eq:couette_k_def}
\end{equation}
Then the unique solution is
\begin{equation}
	\hat{u}_0(y)=U_w\,\frac{\sinh(ky)}{\sinh(kH)}.
	\label{eq:couette_constant_solution}
\end{equation}
In the low-frequency limit $|k|H\ll 1$, \eqref{eq:couette_constant_solution} reduces to the classical Couette profile
$\hat{u}_0(y)\approx U_w y/H$. In the inertial regime $|k|H\gtrsim 1$, the profile develops boundary-layer character and a spatially uniform phase
geometry controlled by the single complex parameter $k$.
\medskip 

\noindent The complex shear traction at the bottom plate is
\begin{equation}
	\hat{\tau}_0(0)=\mu^*\,\hat{u}_0'(0)=\mu^*\,U_w\,\frac{k}{\sinh(kH)}
	=U_w\,\frac{i\omega\rho}{k}\,\frac{1}{\sinh(kH)},
	\label{eq:couette_traction_constant}
\end{equation}
where the final identity uses $\mu^*k^2=i\omega\rho$. At the top plate,
\[
\hat{\tau}_0(H)=\mu^*\,\hat{u}_0'(H)=\mu^*\,U_w\,k\,\coth(kH).
\]

In (B1), $\mu^*=\mu_0 e^{i\varphi_0}$ simply rotates the effective scale $k$ and produces a \emph{global} constitutive lag; the spatial phase geometry
remains rigidly controlled by a single $k$. To formalize this rigidity, let $A$ denote the Dirichlet Laplacian $Au:=-u''$ on $L^2(0,H)$ with domain
$D(A)=H^2(0,H)\cap H_0^1(0,H)$. Then $A$ is selfadjoint and positive, and for constant $\mu^*$ the viscous operator is $\mu^*A$, which is normal as a
scalar multiple of a selfadjoint operator. Since $I$ commutes with $A$, the harmonic operator $\mu^*A+i\omega\rho\,I$ is also normal. In particular,
for normal operators the resolvent norm is controlled sharply by the spectral distance (and the pseudospectrum coincides with an $\varepsilon$-neighborhood
of the spectrum) \cite{Kato1995,TrefethenEmbree2005}. This provides the correct foil for the phase-textured case: any departure from this spectrally rigid
behavior must come from spatial phase variation and the resulting non-commutation in the flux-form operator.

\noindent In the phase-only class \eqref{eq:couette_phase_only_mu}, the homogeneous lifted problem \eqref{eq:couette_flux_ode_hom} becomes
\begin{equation}
	-\mu_0\big(e^{i\varphi(y)}\tilde{u}'(y)\big)' + i\omega\rho\,\tilde{u}(y)=\hat{f}
	\quad\text{in }H^{-1}(0,H),
	\qquad \tilde{u}\in H_0^1(0,H),
	\label{eq:couette_phase_only_ode_hom}
\end{equation}
with uniform accretivity ensured by \eqref{eq:couette_passivity_phase}. Thus, well-posedness follows immediately from the coercive form
\eqref{eq:couette_form}. What changes is not solvability but the \emph{operator geometry}: multiplication by $e^{i\varphi(y)}$ fails to commute with
differentiation unless $\varphi$ is a.e.\ constant, and this non-commutation is the mechanism behind intrinsic non-normality in the viscous core.

\medskip

\noindent We now define the phase-compensated amplitude
\begin{equation}
	\tilde{u}(y)=e^{-i\varphi(y)}\,w(y).
	\label{eq:couette_phase_comp_def}
\end{equation}
A direct calculation gives
\[
\tilde{u}'=e^{-i\varphi}\big(w'-i\varphi'w\big),
\qquad
\mu^* \tilde{u}'=\mu_0\big(w'-i\varphi'w\big).
\]
Substituting into \eqref{eq:couette_phase_only_ode_hom} yields the compensated equation (in distributions)
\begin{equation}
	-\mu_0\big(w'-i\varphi' w\big)' + i\omega\rho\,e^{-i\varphi}\,w=\hat{f},
	\qquad
	w\in H_0^1(0,H),
	\label{eq:couette_compensated_eq}
\end{equation}
where $w(0)=w(H)=0$ since $\tilde u\in H_0^1$ and multiplication by $e^{\pm i\varphi}$ preserves homogeneous Dirichlet data (e.g.\ for
$\varphi\in W^{1,\infty}$). Two structural features are now explicit:
\begin{enumerate}
	\item The leading viscous flux coefficient is \emph{real} ($\mu_0$),
	\item The phase texture survives as an \emph{irreducible first-order coupling} through $\varphi'$ (a covariant-derivative structure),
	together with a spatial modulation of the inertial shift through $e^{-i\varphi}$.
\end{enumerate}
In particular, no global phase rotation can remove $\varphi'$ unless $\varphi$ is a.e.\ constant. Since $(\mu^*)'(y)=i\mu^*(y)\varphi'(y)$ in the
phase-only class, a natural dimensionless Couette phase-gradient parameter is
\begin{equation}
	\Pi_\varphi := H\,\|\varphi'\|_{L^\infty(0,H)}.
	\label{eq:couette_phase_Pi}
\end{equation}
Heuristically, if $\varphi$ changes by $\Delta\varphi$ across a layer of thickness $\ell$, then $\Pi_\varphi\sim H(\Delta\varphi/\ell)$:
$\Pi_\varphi\ll 1$ corresponds to slowly varying phase, whereas $\Pi_\varphi\gtrsim 1$ corresponds to sharp phase texture and strong commutator effects.
For sharp interfaces (piecewise constant $\varphi$), $\Pi_\varphi$ is replaced by jump data together with interface flux continuity (discussed below).

\medskip

\noindent For a strong (differential-operator) realization, assume $\mu^*\in W^{1,\infty}(0,H)$ (equivalently, $\varphi\in W^{1,\infty}(0,H)$ in the
phase-only class) and define on $L^2(0,H)$
\begin{equation}
	A_\varphi u := -(\mu^*(y)u')',
	\qquad
	D(A_\varphi)=H^2(0,H)\cap H_0^1(0,H).
	\label{eq:couette_A_def}
\end{equation}
(Equivalently, one may take $D(A_\varphi)=\{u\in H_0^1(0,H):(\mu^*u')'\in L^2(0,H)\}$, which coincides with $H^2\cap H_0^1$ under
$\mu^*\in W^{1,\infty}$.) The oscillatory operator is the shifted operator
\begin{equation}
	\mathcal{L}_\varphi := A_\varphi + i\omega\rho\,I,
	\qquad D(\mathcal{L}_\varphi)=D(A_\varphi).
	\label{eq:couette_operator_def}
\end{equation}
Its $L^2$-adjoint is
\begin{equation}
	A_\varphi^\dagger v = -(\overline{\mu^*(y)}\,v')',
	\qquad
	\mathcal{L}_\varphi^\dagger = A_\varphi^\dagger - i\omega\rho\,I,
	\qquad
	D(\mathcal{L}_\varphi^\dagger)=H^2(0,H)\cap H_0^1(0,H).
	\label{eq:couette_operator_adjoint}
\end{equation}
Since $i\omega\rho I$ is a scalar multiple of the identity, it commutes with everything, hence
\[
[\mathcal{L}_\varphi,\mathcal{L}_\varphi^\dagger]=[A_\varphi,A_\varphi^\dagger],
\qquad
[A,B]:=AB-BA.
\]
Thus, any non-normality is intrinsic to the viscous core (the variable complex coefficient) and is independent of the inertial shift.

\begin{proposition}[Intrinsic non-normality from spatial phase variation (no-advection Couette)]
	\label{prop:couette_non_normality}
	Assume the phase-only class \eqref{eq:couette_phase_only_mu}--\eqref{eq:couette_passivity_phase} and $\varphi\in W^{3,\infty}(0,H)$.
	If $\varphi'$ is not identically zero (equivalently, $\varphi$ is not a.e.\ constant), then $\mathcal{L}_\varphi$ is not normal on $L^2(0,H)$:
	\[
	[\mathcal{L}_\varphi,\mathcal{L}_\varphi^\dagger]\neq 0
	\quad\text{on }C_c^\infty(0,H).
	\]
	Conversely, if $\varphi$ is constant a.e., then $\mathcal{L}_\varphi$ is normal (indeed diagonalizable in the sine basis).
\end{proposition}

\begin{proof}
	Let $a(y):=\mu^*(y)=\mu_0 e^{i\varphi(y)}$. Since the shift $i\omega\rho I$ commutes with everything,
	$[\mathcal{L}_\varphi,\mathcal{L}_\varphi^\dagger]=[A_\varphi,A_\varphi^\dagger]$.
	On $u\in C_c^\infty(0,H)$, one may expand
	\[
	A_\varphi u=-(a u')'=-a u''-a' u',
	\qquad
	A_\varphi^\dagger u=-(\overline{a} u')'=-\overline{a}\,u''-\overline{a}'\,u'.
	\]
	A direct computation shows that the commutator is a third-order differential operator whose leading term is
	\begin{equation}
		[A_\varphi,A_\varphi^\dagger]u
		=
		2\big(a\overline{a}'-\overline{a}a'\big)\,u^{(3)}
		\ +\ \text{(lower-order terms involving $a',a'',a^{(3)}$)}.
		\label{eq:couette_commutator_leading_u3}
	\end{equation}
	In the phase-only class $a=\mu_0 e^{i\varphi}$ one computes
	\[
	a' = i\mu_0 e^{i\varphi}\varphi',
	\qquad
	\overline{a}' = -i\mu_0 e^{-i\varphi}\varphi',
	\qquad
	a\overline{a}'-\overline{a}a'=-2i\mu_0^2\,\varphi'.
	\]
	Thus, the coefficient of $u^{(3)}$ in \eqref{eq:couette_commutator_leading_u3} is $-4i\mu_0^2\varphi'$.
	If $\varphi'\not\equiv 0$, choose $u\in C_c^\infty(0,H)$ supported where $\varphi'$ does not vanish and with $u^{(3)}\not\equiv 0$ there; then the
	leading term is nonzero, hence the commutator is nonzero and $\mathcal{L}_\varphi$ is not normal.
	Conversely, if $\varphi$ is constant a.e., then $a$ is constant and $A_\varphi$ is a constant-coefficient Dirichlet Sturm--Liouville operator,
	diagonalizable by the sine basis; hence $A_\varphi$ and $A_\varphi^\dagger$ commute and $\mathcal{L}_\varphi$ is normal.
\end{proof}

\begin{remark}[Minimal regularity versus explicit commutator expansions]
	The explicit commutator expansion above is stated under $\varphi\in W^{3,\infty}$ to justify differentiating coefficients to third order.
	At lower regularity (e.g.\ $\varphi\in W^{1,\infty}$ so $a\in W^{1,\infty}$), one may instead work with the m-sectorial form realization and interpret
	non-normality via form commutators on a core; the qualitative conclusion is unchanged: non-constant phase texture breaks normality because
	multiplication by $e^{i\varphi}$ fails to commute with differentiation.
\end{remark}

	\subsubsection{Exact solvability for layered phase textures via transfer matrices and flux continuity.}
	
	Couette geometry has no corners, no separation points, and (in the no-advection linearization) no convective non-normality channel.
	Proposition~\ref{prop:couette_non_normality} therefore isolates a purely constitutive mechanism: spatial phase variation alone suffices to destroy
	normality of the viscous core, independent of the inertial shift $i\omega\rho I$. In this setting, eigenvalues alone are not reliable amplification
	descriptors; resolvent norms and pseudospectra are the appropriate tools even in 1D \cite{TrefethenEmbree2005}. The analytically cleanest textures are layered (piecewise constant) profiles, which can be viewed as sharp limits of thin phase-gradient layers.
	Let $0<y_c<H$ and define
	\[
	\varphi(y)=\varphi_1 \ \text{for }y\in(0,y_c),\qquad
	\varphi(y)=\varphi_2 \ \text{for }y\in(y_c,H),
	\qquad
	\mu_j:=\mu_0 e^{i\varphi_j}.
	\]
	On each layer, the (homogeneous) flux-form equation reduces to a constant-coefficient ODE
	\[
	-\mu_j \hat{u}_j'' + i\omega\rho\,\hat{u}_j=0,
	\qquad
	k_j^2=\frac{i\omega\rho}{\mu_j},
	\qquad
	\Re k_j>0,
	\]
	where the branch is chosen to ensure the standard decay/energy sign conventions.

The correct transmission conditions follow by integrating the flux-form equation across $y=y_c$ (equivalently, by enforcing the weak formulation
	with test functions supported near the interface):
	\begin{equation}
		\hat{u}_1(y_c)=\hat{u}_2(y_c),
		\qquad
		\mu_1 \hat{u}_1'(y_c)=\mu_2 \hat{u}_2'(y_c).
		\label{eq:couette_interface_conditions}
	\end{equation}
	That is: continuity of velocity and continuity of complex shear traction (flux). These conditions are exactly what is lost if one replaces the
	divergence-form operator by $\mu^*\partial_{yy}$.
	
	\paragraph{State-space form and transfer matrices.}
	Introduce the state vector
	\[
	X(y):=\begin{pmatrix}\hat{u}(y)\\ \hat{\tau}(y)\end{pmatrix},
	\qquad \hat{\tau}(y)=\mu^*(y)\hat{u}'(y).
	\]
	On a layer with constant $\mu_j$, the flux-form ODE is equivalent to the first-order system
	\begin{equation}
		\frac{d}{dy}X(y)=B_j X(y),
		\qquad
		B_j:=
		\begin{pmatrix}
			0 & \mu_j^{-1}\\
			i\omega\rho & 0
		\end{pmatrix}.
		\label{eq:couette_state_system}
	\end{equation}
	Since $B_j^2=k_j^2 I$, the transfer matrix over thickness $\Delta$ is explicit:
	\begin{equation}
		T_j(\Delta):=e^{\Delta B_j}
		=
		\cosh(k_j\Delta)\,I + \frac{\sinh(k_j\Delta)}{k_j}\,B_j
		=
		\begin{pmatrix}
			\cosh(k_j\Delta) & \dfrac{\sinh(k_j\Delta)}{k_j\,\mu_j}\\[8pt]
			\dfrac{i\omega\rho\,\sinh(k_j\Delta)}{k_j} & \cosh(k_j\Delta)
		\end{pmatrix}.
		\label{eq:couette_transfer_matrix}
	\end{equation}
	Interface conditions \eqref{eq:couette_interface_conditions} are precisely continuity of $X$ at $y=y_c$. Let
	\[
	T:=T_2(H-y_c)\,T_1(y_c),
	\qquad\text{so that}\qquad
	X(H)=T\,X(0).
	\]
	With boundary conditions $\hat{u}(0)=0$ and $\hat{u}(H)=U_w$, we have $X(0)=(0,\hat{\tau}(0))^{\mathsf T}$ and
	$X(H)=(U_w,\hat{\tau}(H))^{\mathsf T}$, hence
	\[
	U_w = T_{12}\,\hat{\tau}(0),
	\qquad\Rightarrow\qquad
	\hat{\tau}(0)=\frac{U_w}{T_{12}},
	\]
	provided $T_{12}\neq 0$ (the generic case; vanishing corresponds to an exceptional compatibility at isolated parameter values).
	From \eqref{eq:couette_transfer_matrix}, write
	\[
	T_1=
	\begin{pmatrix}
		c_1 & s_1/(k_1\mu_1)\\
		i\omega\rho\,s_1/k_1 & c_1
	\end{pmatrix},
	\quad
	T_2=
	\begin{pmatrix}
		c_2 & s_2/(k_2\mu_2)\\
		i\omega\rho\,s_2/k_2 & c_2
	\end{pmatrix},
	\]
	where $c_1=\cosh(k_1 y_c)$, $s_1=\sinh(k_1 y_c)$, $c_2=\cosh(k_2(H-y_c))$, $s_2=\sinh(k_2(H-y_c))$. Then
	\begin{equation}
		T_{12}
		=
		c_2\,\frac{s_1}{k_1\mu_1}
		+
		\frac{s_2}{k_2\mu_2}\,c_1,
		\qquad
		\hat{\tau}(0)=\frac{U_w}{c_2\,\dfrac{s_1}{k_1\mu_1}+\dfrac{s_2}{k_2\mu_2}\,c_1}.
		\label{eq:couette_two_layer_tau0}
	\end{equation}
	This explicit traction map makes non-reducibility to a single global complex wavenumber transparent: the boundary map depends on the two distinct
	propagation pairs $(k_1,\mu_1)$ and $(k_2,\mu_2)$ and collapses to the constant-coefficient formula only in the trivial case $\varphi_1=\varphi_2$
	(and hence $k_1=k_2$).
	
	\begin{remark}[Thin transition layers versus sharp interfaces]
		Approximating a sharp phase jump by a smooth transition layer of thickness $\ell\ll H$ yields two distinct scalings.
		If the phase amplitude is held fixed while the support shrinks, then the perturbation becomes small in an integrated sense (weak defect).
		If instead the \emph{total phase jump} is held fixed (so $\|\varphi'\|_{L^\infty}\sim \ell^{-1}$), the limit is a genuine interface problem with
		finite traction-phase impact encoded by flux continuity \eqref{eq:couette_interface_conditions}. This distinction is useful when interpreting
		$\Pi_\varphi$ in \eqref{eq:couette_phase_Pi}.
	\end{remark}
	
	\paragraph{Power identities for the lifted homogeneous problem.}
	Work with the lifted homogeneous problem \eqref{eq:couette_flux_ode_hom}:
	\[
	-(\mu^* \tilde{u}')' + i\omega\rho\,\tilde{u}=\hat{f},
	\qquad \tilde{u}\in H_0^1(0,H),
	\qquad \hat{f}\in H^{-1}(0,H).
	\]
	Testing with $\overline{\tilde{u}}$ and integrating by parts yields
	\[
	\int_0^H \mu^*|\tilde{u}'|^2\,dy + i\omega\rho\int_0^H |\tilde{u}|^2\,dy = \langle \hat{f},\tilde{u}\rangle_{H^{-1},H_0^1}.
	\]
	Taking real and imaginary parts gives the identities
	\begin{align}
		\int_0^H \Re(\mu^*)\,|\tilde{u}'|^2\,dy
		&= \Re\langle \hat{f},\tilde{u}\rangle_{H^{-1},H_0^1},
		\label{eq:couette_power_real}\\
		\int_0^H \Im(\mu^*)\,|\tilde{u}'|^2\,dy + \omega\rho\,\|\tilde{u}\|_{L^2(0,H)}^2
		&= \Im\langle \hat{f},\tilde{u}\rangle_{H^{-1},H_0^1}.
		\label{eq:couette_power_imag}
	\end{align}
	In the phase-only class $\mu^*=\mu_0(\cos\varphi+i\sin\varphi)$, \eqref{eq:couette_power_real} identifies the cycle-averaged dissipation density
	$\mu_0\cos\varphi\,|\tilde{u}'|^2$, while \eqref{eq:couette_power_imag} identifies the reactive exchange density
	$\mu_0\sin\varphi\,|\tilde{u}'|^2$ plus the inertial storage term $\omega\rho\|\tilde{u}\|_{L^2}^2$.
	When $\varphi$ varies with $y$, the dissipative/reactive partition varies \emph{spatially} at fixed $|\mu^*|=\mu_0$; this cannot be reproduced by any
	global phase $\varphi_0$ unless $\varphi$ is a.e.\ constant.
	
	\paragraph{Small phase defects: resolvent-kernel filtering of constitutive perturbations.}
	To connect the operator viewpoint to an explicit Couette calculation, consider a small defect about a constant-phase baseline.
	Fix $\bar{\varphi}$ and write
	\begin{equation}
		\varphi(y)=\bar{\varphi}+\varepsilon \chi(y),
		\qquad
		\mu^*(y)=\bar{\mu}\,e^{i\varepsilon\chi(y)},
		\qquad
		\bar{\mu}:=\mu_0 e^{i\bar{\varphi}},
		\qquad 0<\varepsilon\ll 1,
		\label{eq:couette_small_defect_phi}
	\end{equation}
	with $\chi\in W^{1,\infty}(0,H)$ and $\|\bar{\varphi}+\varepsilon\chi\|_{L^\infty}\le \frac{\pi}{2}-\delta$ to preserve passivity.
	
	For simplicity of presentation, consider the inhomogeneous boundary-value problem \eqref{eq:couette_flux_ode_inhom} and expand
	\[
	\hat{u}=\hat{u}_0+\varepsilon \hat{u}_1+\mathcal{O}(\varepsilon^2),
	\]
	where $\hat{u}_0$ is the constant-coefficient solution \eqref{eq:couette_constant_solution} with $\mu^*\equiv\bar{\mu}$ and
	$k^2=i\omega\rho/\bar{\mu}$. Using $\mu^*(y)=\bar{\mu}(1+i\varepsilon\chi+\mathcal{O}(\varepsilon^2))$, the first-order correction satisfies
	\begin{equation}
		\mathcal{L}_0 \hat{u}_1
		=
		\big(i\bar{\mu}\chi\,\hat{u}_0'\big)',
		\qquad
		\hat{u}_1(0)=\hat{u}_1(H)=0,
		\label{eq:couette_u1_equation}
	\end{equation}
	where $\mathcal{L}_0 u:=-\bar{\mu}u''+i\omega\rho\,u$ (Dirichlet). Let $G(y,s)$ be the Dirichlet Green's function for $\mathcal{L}_0$:
	\[
	-\bar{\mu}\,\partial_y^2 G(\cdot,s) + i\omega\rho\,G(\cdot,s)=\delta(\cdot-s),
	\qquad
	G(0,s)=G(H,s)=0,
	\]
	equivalently $(\partial_y^2-k^2)G=-(1/\bar{\mu})\delta$. A standard computation yields (cf.\ \cite{Evans2010})
	\begin{equation}
		G(y,s)=\frac{1}{\bar{\mu}k\sinh(kH)}
		\begin{cases}
			\sinh(ky)\,\sinh(k(H-s)), & 0\le y\le s\le H,\\[2pt]
			\sinh(ks)\,\sinh(k(H-y)), & 0\le s\le y\le H.
		\end{cases}
		\label{eq:couette_green_function}
	\end{equation}
	Then
	\begin{equation}
		\hat{u}_1(y)=\int_0^H G(y,s)\,\big(i\bar{\mu}\chi(s)\hat{u}_0'(s)\big)'\,ds
		=
		-i\bar{\mu}\int_0^H \partial_s G(y,s)\,\chi(s)\,\hat{u}_0'(s)\,ds,
		\label{eq:couette_u1_green}
	\end{equation}
	where the second form is obtained by integration by parts (the boundary terms vanish since $G(y,0)=G(y,H)=0$). This makes the message literal:
	\emph{the phase defect is filtered by the baseline resolvent kernel}. The bottom traction is $\hat{\tau}(0)=\mu^*(0)\hat{u}'(0)$. Expanding to first order gives
	\begin{equation}
		\hat{\tau}(0)=\bar{\mu}\hat{u}_0'(0)
		+
		\varepsilon\Big(i\bar{\mu}\chi(0)\hat{u}_0'(0)+\bar{\mu}\hat{u}_1'(0)\Big)
		+\mathcal{O}(\varepsilon^2).
		\label{eq:couette_tau_expansion}
	\end{equation}
	Differentiating \eqref{eq:couette_u1_green} at $y=0$, and using \eqref{eq:couette_green_function}, one finds
	\[
	\partial_y G(0,s)=\frac{\sinh(k(H-s))}{\bar{\mu}\sinh(kH)}.
	\]
	Hence
	\begin{equation}
		\hat{u}_1'(0)
		=
		\frac{i}{\sinh(kH)}\int_0^H \sinh(k(H-s))\,\big(\chi(s)\hat{u}_0'(s)\big)'\,ds.
		\label{eq:couette_u1prime0}
	\end{equation}
	Integrating by parts yields
	\begin{equation}
		\hat{u}_1'(0)
		=
		-i\,\chi(0)\hat{u}_0'(0)
		+
		\frac{i k}{\sinh(kH)}\int_0^H \cosh(k(H-s))\,\chi(s)\,\hat{u}_0'(s)\,ds.
		\label{eq:couette_u1prime0_ibp_correct}
	\end{equation}
	Substituting into \eqref{eq:couette_tau_expansion} shows the wall-local $\chi(0)$ terms cancel exactly, leaving the clean bulk-weighted formula
	\begin{equation}
		\hat{\tau}(0)
		=
		\bar{\mu}\hat{u}_0'(0)
		+
		\varepsilon\,\bar{\mu}\left(\frac{i k}{\sinh(kH)}\int_0^H \cosh(k(H-s))\,\chi(s)\,\hat{u}_0'(s)\,ds\right)
		+\mathcal{O}(\varepsilon^2).
		\label{eq:couette_tau_correction_clean}
	\end{equation}
	Thus, at fixed $|\mu^*|=\mu_0$, the traction phase anomaly is controlled by a kernel-weighted bulk interaction between the defect $\chi$ and the
	baseline shear $\hat{u}_0'$, rather than by an arbitrary global phase shift.
	
	\paragraph{Square-resolvent sensitivity (operator form).}
	After lifting, write the homogeneous problem as $\mathcal{L}_\varphi \tilde{u}=\hat{f}$ on $H_0^1(0,H)$.
	For a coefficient perturbation $\mu^*\mapsto \mu^*+\delta\mu$ with $\delta\mu\in L^\infty(0,H)$, the induced operator perturbation is the bounded map
	$\delta\mathcal{L}:H_0^1\to H^{-1}$ defined by $\delta\mathcal{L}\,u=-(\delta\mu\,u')'$. The resolvent identity (Fr\'echet differentiation) yields
	\begin{equation}
		\delta \tilde{u} = -\mathcal{L}_\varphi^{-1}\,\delta\mathcal{L}\,\tilde{u},
		\label{eq:couette_delta_u_formula}
	\end{equation}
	and since $\tilde{u}=\mathcal{L}_\varphi^{-1}\hat{f}$,
	\[
	\delta \tilde{u} = -\mathcal{L}_\varphi^{-1}\,\delta\mathcal{L}\,\mathcal{L}_\varphi^{-1}\hat{f}.
	\]
	Therefore, in the natural energy mapping,
	\begin{equation}
		\|\delta \tilde{u}\|_{H_0^1(0,H)}
		\ \le\
		\|\mathcal{L}_\varphi^{-1}\|_{H^{-1}\to H_0^1}^2\,
		\|\delta\mathcal{L}\|_{H_0^1\to H^{-1}}\,
		\|\hat{f}\|_{H^{-1}(0,H)}
		\ \lesssim\
		\|\mathcal{L}_\varphi^{-1}\|_{H^{-1}\to H_0^1}^2\,
		\|\delta\mu\|_{L^\infty(0,H)}\,
		\|\hat{f}\|_{H^{-1}(0,H)},
		\label{eq:couette_square_resolvent_bound}
	\end{equation}
	where the final estimate uses $\|\delta\mathcal{L}\|_{H_0^1\to H^{-1}}\lesssim \|\delta\mu\|_{L^\infty}$ and depends only on $H$.
	This is the Couette instantiation of the square-resolvent sensitivity mechanism: even in 1D, small coefficient perturbations can induce large response
	variations when the resolvent norm is large \cite{TrefethenEmbree2005}. Phase-only textures access this channel while holding $|\mu^*|$ fixed.
	
	This worked example does not model transition or turbulence, and it does not claim to.
	Its function is more structural: it exhibits, in the simplest possible geometry and in a no-advection linearization, a constitutive pathway by which
	spatially varying phase textures alter harmonic operator geometry, destroy normality of the viscous core, and generate phase-sensitive shear structure at fixed
	$|\mu^*|$. When advection is later introduced (Oseen or full Navier--Stokes), such phase-sensitive shear structures are natural seeds for receptivity
	modulation; Couette establishes the constitutive mechanism in isolation, without confounding geometry.
	
	\newpage
	
	\subsection{Worked Example III: BFS/L-Bend in Oscillatory Stokes With a Corner Phase Defect.}
	\label{subsec:worked_BFS_phase_defect}
	
	This example isolates a mechanism that is not visible in the one-dimensional Stokes--II calibration and is not reducible to spanwise
	Toeplitz coupling: the interaction of a \emph{geometry-selected corner singular layer} (a re-entrant corner with interior angle $\Theta>\pi$)
	with a \emph{localized constitutive phase-gradient defect} in an otherwise passive complex-viscosity field.
	
	The corner selects where strain concentrates in oscillatory Stokes via Kondrat'ev-type weighted regularity and Stokes operator pencils
	(on polygonal domains, the leading singular modes yield $\nabla \hat{\mathbf v}\sim r^{\lambda-1}$ in wedge coordinates, with $\lambda\in(0,1)$
	for re-entrant angles). The phase defect supplies a bulk commutator forcing proportional to $\nabla\varphi$, which is maximally effective
	precisely where the strain is largest. The combined effect is a strictly linear, constitutive pathway to
	(i) enhanced corner-local vorticity/enstrophy and (ii) frequency-selective shifts in traction phase and resolvent-optimal response localization,
	even before advection is introduced; see, e.g., \cite{Kondratiev1967,Grisvard1985,Dauge1988,KozlovMazyaRossmann1997} for the corner-regularity framework.
	
	Accordingly, the emphasis here is on \emph{corner-targeted} receptivity modification rather than on global impedance positivity or on
	one-dimensional phase drift. This worked example does not claim that complex viscosity ``causes'' turbulence. Instead, it identifies a distinct
	and quantifiable way in which constitutive phase gradients modify the linear forcing-to-vorticity map that supplies coherent vortical input to
	inertial dynamics. If one later linearizes about a steady base flow $\mathbf V_0$ to form an Oseen operator, the advective linearization acts as an
	additional (non-normal) perturbation on top of the already non-normal viscous core induced by $\nabla\varphi$. Thus, the correct interpretation is:
	\begin{quote}
		Phase defects in passive complex viscosity fields can shift receptivity and spatial organization of near-corner vortical response in oscillatory
		Stokes; when inertia is restored, classical shear-layer roll-up and (in 3D) stretching/tilting act on this modified receptivity output.
	\end{quote}
	\medskip
	
	\noindent\textbf{Geometry and boundary decomposition.}
	Let $\Omega\subset\mathbb{R}^2$ be a bounded Lipschitz polygon representing either a truncated backward-facing step (BFS) channel or an L-bend, and
	assume $\Omega$ has a single re-entrant corner $\mathbf{x}_c\in\partial\Omega$ with interior angle $\Theta\in(\pi,2\pi)$.
	Decompose the boundary as
	\[
	\partial\Omega=\Gamma_W\ \dot\cup\ \Gamma_{\mathrm{in}}\ \dot\cup\ \Gamma_{\mathrm{out}},
	\]
	where $\Gamma_W$ are rigid walls, $\Gamma_{\mathrm{in}}$ is an inflow boundary, and $\Gamma_{\mathrm{out}}$ is an outflow boundary.
	To isolate bulk texture-driven vorticity generation from explicitly vortical body forcing, we take the forcing to be boundary-driven:
	prescribe a time-harmonic inflow velocity profile on $\Gamma_{\mathrm{in}}$ and set the body force to zero. In phasor form,
	\[
	\hat{\mathbf v}|_{\Gamma_{\mathrm{in}}}=\hat{\mathbf v}_{\mathrm{in}},
	\qquad
	\nabla\cdot\hat{\mathbf v}_{\mathrm{in}}=0,
	\qquad
	\hat{\mathbf f}\equiv 0,
	\]
	together with the no-slip condition $\hat{\mathbf v}=0$ on $\Gamma_W$.
	On $\Gamma_{\mathrm{out}}$ we impose a standard traction (``do-nothing'') outflow; in weak form this is the natural boundary condition and does not
	enter the corner-local estimates below (any alternative closure that controls reflections may be substituted without changing the corner-local mechanism).
	
	\medskip
	
	\noindent\textbf{Oscillatory Stokes problem at fixed frequency.}
	Fix $\omega>0$ and assume $\rho\in L^\infty(\Omega)$ with $0<\rho_{\min}\le\rho(\mathbf x)\le\rho_{\max}<\infty$ a.e.
	The harmonic oscillatory Stokes system reads
	\begin{equation}
		i\omega\rho\,\hat{\mathbf v}
		=
		-\nabla \hat{p} + \nabla\cdot\big(2\mu^*(\mathbf{x},\omega)\,\mathbf{D}(\hat{\mathbf{v}})\big),
		\qquad
		\nabla\cdot\hat{\mathbf{v}}=0,
		\label{eq:BFS_harmonic_stokes_clean}
	\end{equation}
	with $\mathbf{D}(\hat{\mathbf{v}})=\tfrac12(\nabla\hat{\mathbf{v}}+(\nabla\hat{\mathbf{v}})^{\mathsf T})$. We specialize to the phase-only class
	\begin{equation}
		\mu^*(\mathbf{x},\omega)=\mu_0(\omega)\,e^{i\varphi(\mathbf{x})},
		\qquad
		\mu_0(\omega)>0,
		\qquad
		\varphi\in W^{1,\infty}(\Omega;\mathbb{R}),
		\label{eq:BFS_phase_only_class_clean}
	\end{equation}
	so that all spatial heterogeneity enters through $\varphi$ (for clarity, the spatial phase profile is taken $\omega$-independent in this worked example).
	Impose uniform passivity at the forcing frequency:
	\begin{equation}
		\Re\mu^*(\mathbf x,\omega)=\mu_0(\omega)\cos\varphi(\mathbf x)\ \ge\ \mu_{\min}>0
		\quad\text{a.e.\ in }\Omega.
		\label{eq:BFS_phase_passivity_clean}
	\end{equation}
	A convenient sufficient condition is $\|\varphi\|_{L^\infty(\Omega)}\le \frac{\pi}{2}-\delta$ for some $\delta\in(0,\frac{\pi}{2})$.
	
	\medskip
	
	\noindent\textbf{Corner-localized phase defect.}
	Let $\ell>0$ be a defect length scale and choose a cutoff $\chi_\ell\in W^{1,\infty}(\Omega)$ supported near the corner such that
	\[
	\mathrm{supp}\,\chi_\ell \subset \Omega\cap B_{c\ell}(\mathbf x_c),
	\qquad
	\|\chi_\ell\|_{L^\infty(\Omega)}\le 1,
	\qquad
	\|\nabla\chi_\ell\|_{L^\infty(\Omega)}\le C\,\ell^{-1},
	\]
	where $c>0$ and $C>0$ are independent of $\ell$.
	For an amplitude parameter $\varepsilon>0$ (chosen so that \eqref{eq:BFS_phase_passivity_clean} holds, e.g.\ $\varepsilon\le \frac{\pi}{2}-\delta$),
	define
	\begin{equation}
		\varphi(\mathbf x)=\varepsilon\,\chi_\ell(\mathbf x).
		\label{eq:BFS_defect_phi_clean}
	\end{equation}
	Then
	\begin{equation}
		\nabla\mu^*(\mathbf x,\omega)=i\,\mu^*(\mathbf x,\omega)\,\nabla\varphi(\mathbf x),
		\qquad
		\|\nabla\mu^*(\cdot,\omega)\|_{L^\infty(\Omega)}
		\le
		\mu_0(\omega)\,\|\nabla\varphi\|_{L^\infty(\Omega)}
		\le
		C\,\mu_0(\omega)\,\frac{\varepsilon}{\ell}.
		\label{eq:BFS_grad_eta_scaling_clean}
	\end{equation}
	Thus, the defect is controlled by the single scale $\mu_0(\omega)\|\nabla\varphi\|_{L^\infty}$ (equivalently $\mu_0(\omega)\varepsilon/\ell$ up to constants),
	which is the coefficient-gradient knob entering the commutator/vorticity identities developed earlier.
	
\paragraph{Solenoidal form domain and well-posedness.}
Let
\[
\Gamma_D:=\Gamma_W\cup\Gamma_{\mathrm{in}}
\]
denote the Dirichlet boundary portion, and assume $|\Gamma_D|>0$ so that Poincar\'e-type inequalities are available.
We work in complex Hilbert spaces throughout (phasor setting). Define the solenoidal form domain (divergence in the distributional sense, trace in the $H^{1/2}$ sense)
\[
V_{\sigma,0}
:=
\Big\{\mathbf u\in H^1(\Omega;\mathbb C^2):
\nabla\cdot\mathbf u=0\ \text{in }\mathcal D'(\Omega),
\ \mathbf u|_{\Gamma_D}=0\Big\},
\]
and let $H_\sigma$ be the $L^2$-closure of smooth solenoidal test fields compatible with the Dirichlet condition
(equivalently, one may view $H_\sigma$ as the range of the Helmholtz--Leray projection in $\Omega$ with mixed boundary conditions; see
\cite{Galdi2011,GiraultRaviart1986}). Lift the inflow boundary data by choosing a divergence-free $\hat{\mathbf v}_{\mathrm{lift}}\in H^1(\Omega;\mathbb C^2)$ such that
\[
\hat{\mathbf v}_{\mathrm{lift}}|_{\Gamma_{\mathrm{in}}}=\hat{\mathbf v}_{\mathrm{in}},
\qquad
\hat{\mathbf v}_{\mathrm{lift}}|_{\Gamma_W}=0,
\]
which can be constructed by standard extension and divergence-correction devices (e.g.\ Bogovski\u{\i}-type corrections; cf.\ \cite{Galdi2011}).
Write
\[
\hat{\mathbf v}=\hat{\mathbf v}_{\mathrm{lift}}+\hat{\mathbf u},
\qquad
\hat{\mathbf u}\in V_{\sigma,0}.
\]
Define the bounded sesquilinear form on $V_{\sigma,0}$,
\begin{equation}
	\mathfrak a_\omega(\mathbf u,\mathbf v)
	:=
	\int_\Omega 2\mu^*(\mathbf x,\omega)\,\mathbf{D}(\mathbf u):\overline{\mathbf{D}(\mathbf v)}\,d\mathbf x
	\;+\;
	i\omega\int_\Omega \rho(\mathbf x)\,\mathbf u\cdot\overline{\mathbf v}\,d\mathbf x,
	\label{eq:BFS_form_clean}
\end{equation}
where $\mathbf{D}(\mathbf u)=\tfrac12(\nabla\mathbf u+(\nabla\mathbf u)^{\mathsf T})$.
By Korn's inequality on Lipschitz domains and the uniform passivity assumption \eqref{eq:BFS_phase_passivity_clean},
\begin{equation}
	\Re\,\mathfrak a_\omega(\mathbf u,\mathbf u)
	=
	\int_\Omega 2\,\Re\mu^*(\mathbf x,\omega)\,|\mathbf D(\mathbf u)|^2\,d\mathbf x
	\ \ge\
	2\mu_{\min}\|\mathbf D(\mathbf u)\|_{L^2(\Omega)}^2
	\ \gtrsim\
	\|\mathbf u\|_{H^1(\Omega)}^2,
	\label{eq:BFS_coercivity_clean}
\end{equation}
with an implied constant depending only on $\Omega$ and $\Gamma_D$ (cf.\ \cite{Horgan1995,Ciarlet1988}).
Hence $\mathfrak a_\omega$ is coercive on $V_{\sigma,0}$ and the lifted problem is uniquely solvable for any right-hand side in
$V_{\sigma,0}^*$ by standard form methods (see \cite{BrezziFortin1991,GiraultRaviart1986} for the saddle-point Stokes framework).
We denote by $\mathcal L_\omega:V_{\sigma,0}\to V_{\sigma,0}^*$ the operator induced by \eqref{eq:BFS_form_clean}.

\medskip

\noindent\textbf{Corner-driven strain localization.}
The re-entrant corner provides a robust geometry-driven localization mechanism.
To quantify it in a manner stable under mesh refinement (and without committing to a single explicit singular exponent),
define the corner strain concentration function
\begin{equation}
	S(r;\omega)
	:=
	\|\mathbf{D}(\hat{\mathbf v}(\cdot;\omega))\|_{L^2(\Omega\cap B_r(\mathbf{x}_c))},
	\qquad r>0.
	\label{eq:BFS_corner_S_def_clean}
\end{equation}
In re-entrant polygons, Kondrat'ev theory for Stokes-type systems implies a wedge expansion near $\mathbf x_c$ involving singular exponents
generated by an operator pencil; the leading singular exponent $\lambda_*\in(0,1)$ governs the strongest corner-local term
\cite{Kondratiev1967,Grisvard1985,Dauge1988,KozlovMazyaRossmann1997}.
Heuristically (and consistent with the standard wedge asymptotics),
\[
\hat{\mathbf v}(r,\theta)\sim r^{\lambda_*}\Phi(\theta),
\qquad
\mathbf D(\hat{\mathbf v})(r,\theta)\sim r^{\lambda_*-1},
\]
so that, in two dimensions,
\begin{equation}
	S(r;\omega)^2
	\ \sim\
	\int_0^r r^{2(\lambda_*-1)}\,r\,dr
	\ \sim\
	r^{2\lambda_*-1},
	\qquad
	S(r;\omega)\ \sim\ r^{\lambda_*-\frac12},
	\qquad r\downarrow 0.
	\label{eq:BFS_S_scaling_heuristic}
\end{equation}
The point is structural: compared to smooth domains, a re-entrant corner can generate persistently large strain in arbitrarily small neighborhoods.
The phase defect \eqref{eq:BFS_defect_phi_clean} is constructed to be supported precisely where $S(r;\omega)$ is largest.

\subsubsection{Vorticity Identity, Texture Commutator, and Bulk Vorticity Injection.}
\label{subsec:BFS_commutator_mechanism}

In wall-bounded flows, vorticity is also generated at the boundary through no-slip constraints.
The novelty targeted here is a \emph{bulk} vorticity injection channel that is absent for constant viscosity and is driven by
$\nabla\mu^*$ (equivalently $\nabla\varphi$ in the phase-only class).
Operationally, the cleanest way to isolate this channel is comparative:
hold geometry and boundary conditions fixed and compare constant-phase cases versus localized phase-defect cases.
Differences in interior vorticity/enstrophy and traction phase that persist under mesh refinement and truncation-length checks are attributable
to the bulk commutator channel defined below.

\paragraph{Distributional vorticity identity.}
Let $\hat{\omega}:=\nabla\times\hat{\mathbf v}$ denote the scalar vorticity in 2D (with the convention
$\nabla\times(u_1,u_2)=\partial_{x_1}u_2-\partial_{x_2}u_1$).
Taking curl of \eqref{eq:BFS_harmonic_stokes_clean} eliminates pressure and yields, in distributions,
\begin{equation}
	i\omega\rho\,\hat{\omega}
	=
	\nabla\times\nabla\cdot\big(2\mu^*(\mathbf x,\omega)\,\mathbf D(\hat{\mathbf v})\big)
	\qquad\text{in }\mathcal D'(\Omega),
	\label{eq:BFS_vorticity_raw_clean}
\end{equation}
with the understanding that further strong-form manipulations may generate boundary-supported distributions unless one localizes away from
$\partial\Omega$ (cf.\ standard vorticity formulations in \cite{MajdaBertozzi2002}). Define the \emph{texture commutator} (as a distribution) by
\begin{equation}
	\mathcal{C}_{\mu^*}[\hat{\mathbf{v}}]
	:=
	\nabla\times\nabla\cdot\big(2\mu^*\,\mathbf{D}(\hat{\mathbf{v}})\big)
	-\mu^*\,\Delta(\nabla\times\hat{\mathbf{v}}),
	\label{eq:BFS_commutator_def_exact_clean}
\end{equation}
so that \eqref{eq:BFS_vorticity_raw_clean} becomes the exact decomposition
\begin{equation}
	i\omega\rho\,\hat{\omega}
	=
	\mu^*\,\Delta \hat{\omega}
	+\mathcal{C}_{\mu^*}[\hat{\mathbf{v}}],
	\qquad\text{in }\mathcal D'(\Omega).
	\label{eq:BFS_vorticity_commutator_exact_clean}
\end{equation}
By construction, $\mathcal C_{\mu^*}\equiv 0$ when $\mu^*$ is constant.

\medskip

\noindent\textbf{A conservative commutator bound (regularity localized to the defect).}
The commutator contains terms involving derivatives of $\mu^*$ and (at the strong-form level) second derivatives of $\hat{\mathbf v}$.
Accordingly, the cleanest estimate is obtained when one localizes to the defect region, where the solution is typically smoother away from
other singularities (or, in a fully rigorous approach, within Kondrat'ev weighted spaces near the corner).
Assume $\mu^*(\cdot,\omega)\in W^{1,\infty}(\Omega)$ and $\hat{\mathbf v}\in H^2(\Omega\cap B_{c\ell}(\mathbf x_c);\mathbb C^2)$ so that the strong-form
expansions are legitimate on the defect support. Then $\mathcal C_{\mu^*}[\hat{\mathbf v}]\in H^{-1}(\Omega\cap B_{c\ell}(\mathbf x_c))$ and one has
\begin{equation}
	\|\mathcal{C}_{\mu^*}[\hat{\mathbf{v}}]\|_{H^{-1}(\Omega\cap B_{c\ell}(\mathbf x_c))}
	\ \le\
	C_{\Omega}\,
	\|\nabla\mu^*(\cdot,\omega)\|_{L^\infty(\Omega\cap B_{c\ell}(\mathbf x_c))}\,
	\|\mathbf{D}(\hat{\mathbf{v}})\|_{L^2(\Omega\cap B_{c\ell}(\mathbf x_c))},
	\label{eq:BFS_commutator_bound_clean}
\end{equation}
with $C_\Omega$ depending on $\Omega$ and the choice of norms (a standard product/duality estimate; cf.\ \cite{Evans2010}). In the phase-only class \eqref{eq:BFS_phase_only_class_clean} this becomes
\begin{equation}
	\|\mathcal{C}_{\mu^*}[\hat{\mathbf{v}}]\|_{H^{-1}(\Omega\cap B_{c\ell}(\mathbf x_c))}
	\ \le\
	C_\Omega\,
	\mu_0(\omega)\,\|\nabla\varphi\|_{L^\infty(\Omega\cap B_{c\ell}(\mathbf x_c))}\,
	\|\mathbf D(\hat{\mathbf v})\|_{L^2(\Omega\cap B_{c\ell}(\mathbf x_c))}.
	\label{eq:BFS_phase_commutator_bound_clean}
\end{equation}
With the corner strain concentration function \eqref{eq:BFS_corner_S_def_clean}, this is exactly the localized statement
\begin{equation}
	\|\mathcal{C}_{\mu^*}[\hat{\mathbf{v}}]\|_{H^{-1}(\Omega\cap B_{c\ell}(\mathbf x_c))}
	\ \le\
	C_\Omega\,
	\mu_0(\omega)\,\|\nabla\varphi\|_{L^\infty(\Omega\cap B_{c\ell}(\mathbf x_c))}\,S(c\ell;\omega).
	\label{eq:BFS_corner_local_commutator_clean}
\end{equation}
This is the core mechanism statement in a corner geometry: the defect targets the geometry-selected strain concentration. Combining $\|\nabla\varphi\|_\infty\sim \varepsilon/\ell$ with the corner scaling heuristic \eqref{eq:BFS_S_scaling_heuristic} yields
\begin{equation}
	\|\mathcal{C}_{\mu^*}[\hat{\mathbf{v}}]\|_{H^{-1}(\Omega\cap B_{c\ell}(\mathbf x_c))}
	\ \lesssim\
	\mu_0(\omega)\,\frac{\varepsilon}{\ell}\,S(c\ell;\omega)
	\ \sim\
	\mu_0(\omega)\,\varepsilon\,\ell^{\lambda_*-\frac32},
	\qquad \ell\downarrow 0,
	\label{eq:BFS_defect_corner_scaling}
\end{equation}
up to $\omega$-dependent constants and truncation/outflow effects.
The interpretation is structural: sharpening the phase defect increases the coefficient-gradient amplitude as $\ell^{-1}$, while the corner strain
decays only as $\ell^{\lambda_*-\frac12}$, producing a nontrivial competition between defect sharpness and corner regularity.

\medskip

\noindent\textbf{Phase compensation viewpoint.}
Spatially varying phase makes the viscous operator non-selfadjoint and typically non-normal even without advection.
Introduce the phase-compensated unknown
\begin{equation}
	\hat{\mathbf v}(\mathbf x)=e^{-i\varphi(\mathbf x)}\,\hat{\mathbf w}(\mathbf x),
	\label{eq:BFS_phase_comp_transform_clean}
\end{equation}
which is bounded on $L^2$ and on $H^1$ whenever $\varphi\in W^{1,\infty}$.
A direct calculation yields the exact stress decomposition
\begin{equation}
	2\mu^*(\mathbf x,\omega)\,\mathbf D(\hat{\mathbf v})
	=
	2\mu_0(\omega)\,\mathbf D(\hat{\mathbf w})
	-i\mu_0(\omega)\,\Big(\hat{\mathbf w}\otimes\nabla\varphi + \nabla\varphi\otimes\hat{\mathbf w}\Big),
	\label{eq:BFS_phase_comp_stress_clean}
\end{equation}
which isolates an irreducible first-order coupling driven by $\nabla\varphi$.
Two consequences are immediate:
\begin{enumerate}
	\item The phase texture is not a removable global phase shift: even after compensation, the coupling proportional to $\nabla\varphi$ remains in the stress
	(and hence in the operator).
	\item Because the corner singular layer produces large velocity gradients near $\mathbf x_c$ (in the weighted-regularity sense),
	the coupling term in \eqref{eq:BFS_phase_comp_stress_clean} is most effective in the same neighborhood where
	\eqref{eq:BFS_corner_local_commutator_clean} is strongest \cite{Dauge1988,KozlovMazyaRossmann1997}.
\end{enumerate}

\noindent The BFS/L-bend setting is valuable precisely because it aligns three selectors that are typically separated in benchmark problems:
(i) a \emph{geometric singular selector} (the re-entrant corner), (ii) a \emph{constitutive selector} (a localized phase-gradient field $\nabla\varphi$),
and (iii) a \emph{frequency selector} (oscillatory Stokes/Oseen resolvent gain).
To make the constitutive novelty legible, the reported quantities should (a) live on subdomains resolving the corner neighborhood,
(b) be defined as integrated functionals (hence stable under mesh refinement even when pointwise limits fail at the corner),
and (c) isolate the $\nabla\varphi$ mechanism by eliminating magnitude and boundary-condition confounds.
The diagnostics below are organized accordingly.

\paragraph{Corner-local strain, vorticity, and a scale-resolved defect--overlap functional.}
Beyond global norms, the corner--defect mechanism is most cleanly expressed through \emph{scale-resolved} localization functionals.
Recall the corner-local strain functional
\begin{equation}
	S(r;\omega)
	:=
	\|\mathbf D(\hat{\mathbf v}(\cdot;\omega))\|_{L^2(\Omega\cap B_r(\mathbf x_c))},
	\qquad r>0,
	\label{eq:BFS_corner_S_def_clean_repeat}
\end{equation}
and define the corner-local enstrophy
\begin{equation}
	E_\omega(r)
	:=
	\int_{\Omega\cap B_r(\mathbf x_c)}|\hat{\omega}(\mathbf x;\omega)|^2\,d\mathbf x,
	\qquad r>0,
	\label{eq:BFS_corner_enstrophy_clean_repeat}
\end{equation}
where $\hat{\omega}:=\nabla\times \hat{\mathbf v}$ is the scalar vorticity in 2D (or an explicitly chosen component/magnitude in 3D truncations).
In cornered domains these integrated quantities are preferable to pointwise maxima: they remain meaningful when $\nabla\hat{\mathbf v}$ exhibits
corner singular behavior and they converge under standard refinement strategies (cf.\ the use of enstrophy as a robust vortical diagnostic in
\cite{MajdaBertozzi2002}).

\medskip

\noindent To emphasize that the defect interacts with the corner \emph{through} $\nabla\varphi$, it is useful to report at least one diagnostic that
directly measures geometric--constitutive overlap. A robust choice is the scale-resolved overlap functional
\begin{equation}
	\mathcal{O}_\varphi(r;\omega)
	:=
	\int_{\Omega\cap B_r(\mathbf x_c)}
	|\nabla\varphi(\mathbf x)|\,|\mathbf D(\hat{\mathbf v}(\mathbf x;\omega))|\,d\mathbf x,
	\label{eq:BFS_overlap_def}
\end{equation}
which is well-defined for Tier~II defects $\varphi\in W^{1,\infty}$ and $\hat{\mathbf v}\in H^1$.
If one prefers a strictly quadratic functional (and a direct match to $L^2$-based energy estimates), a convenient alternative is
\begin{equation}
	\mathcal{O}^{(2)}_\varphi(r;\omega)
	:=
	\left(\int_{\Omega\cap B_r(\mathbf x_c)}|\nabla\varphi(\mathbf x)|^2\,|\mathbf D(\hat{\mathbf v}(\mathbf x;\omega))|^2\,d\mathbf x\right)^{1/2}.
	\label{eq:BFS_overlap2_def}
\end{equation}
Both choices reflect the analytic mechanism: they quantify, on the same neighborhood, the product structure
\[
\text{(corner-selected strain)}\ \times\ \text{(texture sharpness)}
\]
that drives the commutator forcing in \eqref{eq:BFS_corner_local_commutator_clean}.
Reporting $\mathcal{O}_\varphi(r;\omega)$ (or $\mathcal{O}^{(2)}_\varphi(r;\omega)$) alongside $S(r;\omega)$ and $E_\omega(r)$ makes it
immediately clear whether observed changes correlate with \emph{true defect overlap} rather than with global changes in solution amplitude.

\medskip

\noindent For interpretability and scale separation, evaluate these quantities at radii tied to both the defect and the geometry, e.g.
\[
r_{\mathrm{loc}}=c_1\ell,
\qquad
r_{\mathrm{meso}}=c_2 L_{\mathrm{ref}},
\qquad
0<c_1=\mathcal O(1),\ \ 0<c_2\ll 1,
\]
where $L_{\mathrm{ref}}$ is a fixed geometric reference length (e.g.\ step height, channel half-height, or bend width).
Then $(\Delta S,\Delta E_\omega)$ can be assessed as genuinely localized (dominant at $r_{\mathrm{loc}}$ but not at $r_{\mathrm{meso}}$)
or as globalized (persisting as $r$ increases).

\paragraph{Segment-robust traction phase away from the singular point.}
Define the complex tangential traction on a wall segment $\Gamma_W$ by
\[
\hat{\tau}_t(\mathbf x;\omega)
:=
\mathbf t(\mathbf x)\cdot\Big(2\mu^*(\mathbf x;\omega)\,\mathbf D(\hat{\mathbf v}(\mathbf x;\omega))\,\mathbf n(\mathbf x)\Big),
\]
with $\mathbf n$ the outward unit normal and $\mathbf t$ a unit tangent.
(Pressure contributes only to the normal traction, so it does not enter $\hat{\tau}_t$.)
Because the corner point itself may be singular, fix a measurement segment $\Sigma_{r_0}\subset\Gamma_W$ at distance $r_0>0$ from $\mathbf x_c$
(e.g.\ an arc-length window) and report a \emph{segment-robust} traction phase such as
\begin{equation}
	\Theta_t(\omega)
	:=
	\arg\left(\int_{\Sigma_{r_0}}\hat{\tau}_t(s;\omega)\,ds\right).
	\label{eq:BFS_traction_phase_avg}
\end{equation}
Alternatively, if isolated phase defects occur on a small subset of $\Sigma_{r_0}$, report a median or trimmed-mean statistic of
$s\mapsto\arg\hat{\tau}_t(s;\omega)$.
Frequency-localized shifts in $\Theta_t(\omega)$ under phase-only defects with $|\mu^*|$ fixed are a macroscopic signature of constitutive
phase gradients: the wall kinematics and $|\mu^*|$ are unchanged, so any systematic traction-phase shift must originate from the interior
operator change induced by $\nabla\varphi$ and its interaction with corner-selected strain.
To avoid interpretive ambiguity, any traction-phase report should state: the harmonic convention ($e^{i\omega t}$ here), the phase reference
(absolute phase or phase relative to inflow/wall motion), and the unwrapping convention in $\omega$.

\paragraph{Non-normality and receptivity: diagnostics that detect operator change even when eigenvalues barely move.}
The constitutive phase-gradient mechanism is fundamentally an \emph{operator} effect. Even if eigenvalues shift only mildly, resolvent geometry can
change substantially; thus diagnostics should target departure from normality and forcing--response alignment
\cite{TrefethenEmbree2005,Schmid2007}.

Let $A(\omega)$ denote a consistent discrete realization of the lifted linear operator induced by \eqref{eq:BFS_harmonic_stokes_clean}
(or its Oseen counterpart), formed either (i) on a discrete solenoidal subspace or (ii) via a stable mixed formulation together with a
projection (or Schur-complement reduction) so that $A(\omega)$ acts on the discrete velocity unknowns in a norm-consistent way.
Two computation-facing diagnostics that remain meaningful under mesh refinement are:

\begin{enumerate}
	\item \textbf{Departure from normality (normalized commutator size).}
	With $\|\cdot\|$ denoting the spectral (operator $2$-)norm,
	\begin{equation}
		\Delta_{\mathrm{nn}}(\omega)
		:=
		\frac{\|A(\omega)^*A(\omega)-A(\omega)A(\omega)^*\|}{\|A(\omega)\|^2},
		\label{eq:BFS_Delta_nn_repeat}
	\end{equation}
	reported across a fixed discretization family and refinement sequence.
	The significance is structural: $\nabla\varphi\not\equiv 0$ induces an intrinsic non-selfadjoint, typically non-normal viscous core, so
	$\Delta_{\mathrm{nn}}$ should increase in a manner correlated with $\|\nabla\varphi\|_\infty$ and with the overlap functionals
	\eqref{eq:BFS_overlap_def}--\eqref{eq:BFS_overlap2_def} \cite{TrefethenEmbree2005}.
	
	\item \textbf{Resolvent-optimal forcing/response localization.}
	Compute the leading right singular vector $\hat{\mathbf f}_\star(\omega)$ of $A(\omega)^{-1}$ (or of a strain-weighted resolvent defined with
	the relevant mass/energy inner products) and examine the spatial structure of the corresponding response
	\[
	\hat{\mathbf v}_\star(\omega) := A(\omega)^{-1}\hat{\mathbf f}_\star(\omega).
	\]
	A defect localized near $\mathbf x_c$ is expected to \emph{re-weight} where optimal forcing concentrates and where the response localizes.
	This can be quantified by reporting the fraction of response energy in the corner neighborhood, e.g.
	\begin{equation}
		\mathcal{L}_\star(r;\omega)
		:=
		\frac{\|\mathbf D(\hat{\mathbf v}_\star(\omega))\|_{L^2(\Omega\cap B_r(\mathbf x_c))}}
		{\|\mathbf D(\hat{\mathbf v}_\star(\omega))\|_{L^2(\Omega)}}.
		\label{eq:BFS_resolvent_localization_ratio}
	\end{equation}
\end{enumerate}

These diagnostics complement pseudospectral portraits and gain maps \cite{TrefethenEmbree2005,GustafsonRao1997}.
They remain informative precisely in regimes where the spectrum itself is weakly perturbed: non-normality is encoded in singular vectors and
resolvent amplification, not in eigenvalues alone.

\medskip

\noindent The intended constitutive claim is a \emph{difference} claim:
at fixed geometry, fixed forcing, fixed $|\mu^*|=\mu_0(\omega)$, and fixed wall conditions, changes are driven by $\nabla\varphi$ coupling to the
corner-selected strain field. Accordingly, natural reporting quantities are the defect-induced differences
\[
\Delta S(r;\omega),\qquad \Delta E_\omega(r),\qquad \Delta\Theta_t(\omega),\qquad \Delta\Delta_{\mathrm{nn}}(\omega),
\qquad \Delta\mathcal{L}_\star(r;\omega),
\]
with $(r,r_0)$ specified and with constant-phase controls matched by a dissipation proxy (e.g.\ a matched spatial mean of $\cos\varphi$) to prevent bulk dissipation changes from mimicking the phase-gradient effect.
The local estimate \eqref{eq:BFS_corner_local_commutator_clean} makes the mechanism geometrically explicit: defect sharpness and corner-selected
strain enter multiplicatively, and the overlap functionals \eqref{eq:BFS_overlap_def}--\eqref{eq:BFS_overlap2_def} provide a direct way to
document that multiplication in computation.
Corner-local enstrophy and segment-robust traction phase provide macroscopic, experimentally legible signatures, while the localization ratios
\eqref{eq:BFS_resolvent_localization_ratio} and departure-from-normality metric \eqref{eq:BFS_Delta_nn_repeat} communicate the intrinsic
non-normality induced by $\nabla\varphi$ in operator- and computation-facing terms.

\newpage
\subsection{Worked Example IV: 3D Periodic Channel/BFS With a Single-Harmonic Spanwise Texture
	-- Explicit Toeplitz/Laurent Mode Coupling.}
\label{subsec:worked_ex_toeplitz_mode_coupling_clean}

This worked example isolates a \emph{genuinely three-dimensional}, strictly linear mechanism that is absent from the preceding worked examples.
In Worked Example~I (Stokes~II), the constitutive phase mechanism appears as a one-dimensional drift/commutator in physical space.
In Worked Example~III (2D BFS/L-bend corner defect), the key effect is a \emph{geometry-targeted} commutator forcing driven by $\nabla\varphi$ near a corner singular layer.
Here, by contrast, we impose \emph{spanwise periodicity} and introduce a \emph{spanwise-periodic} constitutive texture.
The dominant mechanism is not a pointwise commutator in $(x,y)$, but a \emph{Fourier-multiplier (convolution) effect} in the spanwise direction:
\begin{quote}
	\emph{Multiplication by a single spanwise Fourier harmonic in $z$ induces an index shift (a Laurent/Toeplitz convolution) in the spanwise Fourier label.}
\end{quote}
As a result, even \emph{spanwise-uniform forcing} (only $\kappa=0$) generates nontrivial $\kappa\neq 0$ response through linear constitutive coupling alone.
The computation-facing deliverables are therefore (i) \emph{sideband ratios} and \emph{spanwise energy fractions} that quantify induced patterning,
(ii) a \emph{mixing transfer map} (off-diagonal resolvent blocks) that measures transfer from $\kappa=0$ forcing into $\kappa=\pm k_0$ response,
and (iii) a \emph{traction-phase Fourier signature} on the wall that is experimentally interpretable.

\medskip
\noindent Throughout we adopt the time dependence $\Re\{\hat{\mathbf{v}}(\mathbf{x})e^{i\omega t}\}$ at fixed $\omega>0$,
so inertial terms appear as $+\,i\omega\rho\,\hat{\mathbf{v}}$ in the frequency-domain momentum balance.

\medskip
\noindent Let $\Omega_{2D}\subset\mathbb{R}^2$ be a bounded Lipschitz domain representing a 2D channel/BFS truncation in $(x,y)$,
possibly with a re-entrant corner.
Let $L_z>0$ and define the spanwise torus $\mathbb{T}_{L_z}:=(0,L_z)$ with endpoints identified.
Set the 3D spanwise-periodic domain
\[
\Omega := \Omega_{2D}\times \mathbb{T}_{L_z},
\qquad
\text{$z$-periodic at } z=0 \text{ and } z=L_z .
\]
Decompose the boundary of $\Omega_{2D}$ as
\[
\partial\Omega_{2D}=\Gamma_W\cup\Gamma_{\mathrm{in}}\cup\Gamma_{\mathrm{out}},
\]
where $\Gamma_W$ denotes no-slip walls, $\Gamma_{\mathrm{in}}$ the inlet, and $\Gamma_{\mathrm{out}}$ the outlet.
Assume $|\Gamma_W|>0$ so that Poincar\'e/Korn inequalities hold after the usual lifting of inhomogeneous inflow data.
All results below are local-in-operator and do not rely on a particular outflow closure: one may use do-nothing traction-free outflow,
a stabilized outflow, or an absorbing/sponge layer in a terminal segment (provided reflections are controlled).

\medskip
\noindent Let $\mathcal{V}$ be the smooth periodic divergence-free test space
\[
\mathcal V
:=
\Big\{\boldsymbol\psi\in C^\infty(\overline{\Omega};\mathbb{C}^3):
\boldsymbol\psi \text{ is $z$-periodic},\
\nabla\cdot\boldsymbol\psi=0,\
\boldsymbol\psi=0 \text{ on } \Gamma_W\times\mathbb{T}_{L_z}\Big\}.
\]
Define the closures
\[
H_{\sigma}^{\mathrm{per}}:=\overline{\mathcal V}^{\,L^2(\Omega)},
\qquad
V_{\sigma}^{\mathrm{per}}:=\overline{\mathcal V}^{\,H^1(\Omega)}.
\]
On $V_{\sigma}^{\mathrm{per}}$, Korn's inequality implies $\|\nabla \mathbf{u}\|_{L^2(\Omega)}\lesssim \|\mathbf{D}(\mathbf{u})\|_{L^2(\Omega)}$,
where $\mathbf{D}(\mathbf{u})=\tfrac12(\nabla\mathbf{u}+(\nabla\mathbf{u})^{\mathsf{T}})$.

\medskip
\noindent Let $\mathbf{V}_0=\mathbf{V}_0(x,y)$ be a steady base flow independent of $z$ and assume
\[
\mathbf{V}_0\in W^{1,\infty}(\Omega_{2D};\mathbb{R}^3),
\qquad
\partial_z \mathbf{V}_0\equiv 0,
\]
so that the Oseen terms define bounded maps $V_{\sigma}^{\mathrm{per}}\to (V_{\sigma}^{\mathrm{per}})^*$.
Setting $\mathbf{V}_0\equiv 0$ recovers oscillatory Stokes.
Fix $\omega>0$ and take constant density $\rho>0$ for notational simplicity.
The harmonic linearized (Oseen) problem reads: find $(\hat{\mathbf{v}},\hat{p})$ such that
\begin{equation}
	\label{eq:toe4_oseen_strong}
	i\omega\rho\,\hat{\mathbf{v}}
	+\rho(\mathbf{V}_0\cdot\nabla)\hat{\mathbf{v}}
	+\rho(\hat{\mathbf{v}}\cdot\nabla)\mathbf{V}_0
	-\nabla\cdot\big(2\,\mu^*(x,y,z;\omega)\,\mathbf{D}(\hat{\mathbf{v}})\big)
	+\nabla \hat{p}
	=
	\hat{\mathbf{f}},
	\qquad
	\nabla\cdot\hat{\mathbf{v}}=0,
\end{equation}
with $z$-periodicity and $\hat{\mathbf{v}}=0$ on $\Gamma_W\times\mathbb{T}_{L_z}$ (after lifting inlet conditions, if needed).
Assume the passive bounded-coefficient hypothesis
\begin{equation}
	\label{eq:toe4_passivity}
	\mu^*(\cdot,\omega)\in L^\infty(\Omega;\mathbb{C}),
	\qquad
	\Re\mu^*(x,y,z;\omega)\ge \mu_{\min}>0\quad\text{a.e.\ in }\Omega.
\end{equation}
This is physically interpretable (nonnegative cycle-averaged dissipation at frequency $\omega$)
and mathematically decisive (coercivity of the real part of the viscous form).

\medskip
\noindent Next, define for $\mathbf{u},\mathbf{v}\in V_{\sigma}^{\mathrm{per}}$ the sesquilinear form
\begin{equation}
	\label{eq:toe4_bilinear_form}
	a_\omega^{\mu^*}(\mathbf{u},\mathbf{v})
	:=
	i\omega\rho\,(\mathbf{u},\mathbf{v})_{L^2(\Omega)}
	+\rho\big((\mathbf{V}_0\cdot\nabla)\mathbf{u},\mathbf{v}\big)_{L^2(\Omega)}
	+\rho\big((\mathbf{u}\cdot\nabla)\mathbf{V}_0,\mathbf{v}\big)_{L^2(\Omega)}
	+2\big(\mu^* \mathbf{D}(\mathbf{u}),\mathbf{D}(\mathbf{v})\big)_{L^2(\Omega)}.
\end{equation}
The variational problem is:
\begin{equation}
	\label{eq:toe4_variational}
	\text{Find }\hat{\mathbf{v}}\in V_\sigma^{\mathrm{per}}
	\text{ such that }
	a_\omega^{\mu^*}(\hat{\mathbf{v}},\boldsymbol\psi)=(\hat{\mathbf{f}},\boldsymbol\psi)_{L^2(\Omega)}
	\quad\forall \boldsymbol\psi\in V_\sigma^{\mathrm{per}}.
\end{equation}
By \eqref{eq:toe4_passivity} and Korn's inequality,
\[
\Re\,a_\omega^{\mu^*}(\mathbf{u},\mathbf{u})
\ge
2\mu_{\min}\|\mathbf{D}(\mathbf{u})\|_{L^2(\Omega)}^2
\gtrsim
\|\mathbf{u}\|_{H^1(\Omega)}^2,
\]
while the Oseen terms are bounded on $V_\sigma^{\mathrm{per}}$.
Consequently, the induced operator $\mathcal{L}_\omega^{\mu^*}:V_\sigma^{\mathrm{per}}\to (V_\sigma^{\mathrm{per}})^*$
is a bounded coercive-plus-bounded-perturbation map.
Since $\Omega$ is bounded (with periodic identification only in $z$), the embedding $V_\sigma^{\mathrm{per}}\hookrightarrow H_\sigma^{\mathrm{per}}$ is compact;
hence the closed realization on $H_\sigma^{\mathrm{per}}$ has compact resolvent (bounded perturbations preserve this property).

\subsubsection{Spanwise Textures: Single-Harmonic Amplitude, Symmetric Cosine, and Phase-Only (Bessel/Jacobi--Anger) Modulation.}
\label{subsubsec:toe4_textures}

\paragraph{Texture families and pointwise passivity.}
Let $\mu_0^*(x,y;\omega)$ be a $z$-independent baseline complex viscosity on $\Omega_{2D}$ satisfying the uniform accretivity (passivity) bound
\[
\Re \mu_0^*(x,y;\omega)\ \ge\ \mu_{\min,0}\ >0
\qquad\text{a.e.\ on }\Omega_{2D}.
\]
Fix $m_0\ge 1$ and define
\[
\kappa_m:=\frac{2\pi m}{L_z},\qquad k_0:=\kappa_{m_0}.
\]
We consider three canonical spanwise textures built from the single harmonic $k_0$.

\begin{enumerate}[label=\textup{(\Alph*)}, leftmargin=2.2em, itemsep=0.35em]
	\item \textbf{One-sided single-harmonic amplitude modulation.}
	\begin{equation}
		\label{eq:toe4_eta_onesided}
		\mu^*(x,y,z;\omega)=\mu_0^*(x,y;\omega)\,\bigl(1+\varepsilon e^{ik_0 z}\bigr),
		\qquad 0<\varepsilon\ll 1.
	\end{equation}
	A sufficient \emph{pointwise} condition for \eqref{eq:toe4_passivity} is
	\begin{equation}
		\label{eq:toe4_onesided_passivity}
		\Re\mu_0^*(x,y;\omega)-\varepsilon\,|\mu_0^*(x,y;\omega)|
		\ \ge\ \mu_{\min}\ >0
		\quad\text{a.e.\ on }\Omega_{2D},
	\end{equation}
	which reduces to $\mu_0^*(1-\varepsilon)\ge\mu_{\min}$ when $\mu_0^*$ is real and positive.
	
	\item \textbf{Symmetric cosine amplitude modulation (nearest-neighbor coupling).}
	\begin{equation}
		\label{eq:toe4_eta_cos}
		\mu^*(x,y,z;\omega)=\mu_0^*(x,y;\omega)\,\bigl(1+\varepsilon\cos(k_0 z)\bigr).
	\end{equation}
	Since $1+\varepsilon\cos(k_0 z)\in[1-\varepsilon,\,1+\varepsilon]$ pointwise in $z$, a sufficient condition for \eqref{eq:toe4_passivity} is
	\begin{equation}
		\label{eq:toe4_cos_passivity}
		(1-\varepsilon)\,\Re\mu_0^*(x,y;\omega)\ \ge\ \mu_{\min}\ >0
		\quad\text{a.e.\ on }\Omega_{2D}.
	\end{equation}
	In particular, if $\mu_0^*$ is real and positive, passivity holds whenever $\mu_0^*(1-\varepsilon)\ge\mu_{\min}$.
	
	\item \textbf{Phase-only modulation (unit-modulus texture; Jacobi--Anger/Bessel series).}
	\begin{equation}
		\label{eq:toe4_eta_phase_only}
		\mu^*(x,y,z;\omega)=\mu_0^*(x,y;\omega)\,e^{i\varepsilon\cos(k_0 z)}.
	\end{equation}
	The Jacobi--Anger expansion gives the explicit spanwise Fourier series (equivalently, a Bessel/Laurent series in $e^{\pm i k_0 z}$) :
	\begin{equation}
		\label{eq:toe4_bessel_expansion}
		e^{i\varepsilon\cos(k_0 z)}=\sum_{n\in\mathbb{Z}} i^n J_n(\varepsilon)\,e^{i n k_0 z}.
	\end{equation}
	Hence the only nonzero viscosity modes occur at indices $m\in m_0\mathbb{Z}$:
	\[
	\widehat{\mu}_{n m_0}(x,y;\omega)=\mu_0^*(x,y;\omega)\,i^n J_n(\varepsilon),
	\qquad
	\widehat{\mu}_m\equiv 0\ \text{if }m\notin m_0\mathbb{Z}.
	\]
	If $\mu_0^*$ is real and positive, then
	\[
	\Re \mu^*(x,y,z;\omega)=\mu_0^*(x,y;\omega)\cos(\varepsilon\cos(k_0 z))
	\ \ge\ \mu_0^*(x,y;\omega)\cos(\varepsilon),
	\]
	so passivity holds provided $\mu_0^*\cos(\varepsilon)\ge \mu_{\min}>0$. (For a general complex $\mu_0^*$, a sufficient alternative is a
	\emph{uniform phase-margin} condition: $|\arg\mu_0^*(x,y;\omega)|\le \frac{\pi}{2}-\delta_0$ a.e.\ and $\varepsilon\le\delta_0$, which ensures
	$\Re(\mu_0^* e^{i\varepsilon\cos(k_0 z)})\ge |\mu_0^*|\sin(\delta_0-\varepsilon)$ pointwise in $z$.)
	Moreover, the small-$\varepsilon$ expansion
	\begin{equation}
		\label{eq:toe4_phase_small_eps}
		e^{i\varepsilon\cos(k_0 z)}
		=
		1+i\varepsilon\cos(k_0 z)+\mathcal{O}(\varepsilon^2)
		=
		1+\frac{i\varepsilon}{2}e^{ik_0 z}+\frac{i\varepsilon}{2}e^{-ik_0 z}+\mathcal{O}(\varepsilon^2)
	\end{equation}
	shows that, at leading order, the phase-only class produces the same nearest-neighbor shift pattern as the cosine amplitude class (up to the factor $i/2$).
\end{enumerate}

\paragraph{Baseline decoupling and the resolvent hypothesis.}
Expand unknown and forcing in spanwise Fourier series
\begin{equation}
	\label{eq:toe4_fourier_expand}
	\hat{\mathbf{v}}(x,y,z)=\sum_{m\in\mathbb{Z}} \hat{\mathbf{v}}_m(x,y)\,e^{i\kappa_m z},
	\qquad
	\hat{\mathbf{f}}(x,y,z)=\sum_{m\in\mathbb{Z}} \hat{\mathbf{f}}_m(x,y)\,e^{i\kappa_m z},
\end{equation}
and write $\hat{\mathbf{v}}_m=(\hat{\mathbf{v}}_{m,\parallel},\hat{w}_m)$ with $\hat{\mathbf{v}}_{m,\parallel}\in\mathbb{C}^2$ and $\hat{w}_m\in\mathbb{C}$.
The incompressibility constraint becomes, for each $m$,
\begin{equation}
	\label{eq:toe4_div_constraint}
	\nabla_{x,y}\cdot \hat{\mathbf{v}}_{m,\parallel} + i\kappa_m \hat{w}_m = 0
	\quad\text{in }\Omega_{2D}.
\end{equation}
After the usual lifting of boundary data (so that $\hat{\mathbf v}=0$ on $\Gamma_W\times(0,L_z)$ is homogeneous), define the mode space
\begin{equation}
	\label{eq:toe4_mode_space}
	X_{\kappa}
	:=
	\Bigl\{(\mathbf{u}_\parallel,w)\in H^1(\Omega_{2D};\mathbb{C}^2)\times H^1(\Omega_{2D};\mathbb{C}):
	(\mathbf{u}_\parallel,w)=0 \ \text{on }\Gamma_W,\ \nabla_{x,y}\cdot \mathbf{u}_\parallel + i\kappa w=0 \Bigr\},
\end{equation}
with dual $X_\kappa^*$. Equip $X_\kappa$ with the $\kappa$-uniform norm
\begin{equation}
	\label{eq:toe4_kappa_norm}
	\|(\mathbf{u}_\parallel,w)\|_{X_\kappa}^2
	:=
	\|\nabla_{x,y}\mathbf{u}_\parallel\|_{L^2(\Omega_{2D})}^2
	+\|\nabla_{x,y}w\|_{L^2(\Omega_{2D})}^2
	+|\kappa|^2\|(\mathbf{u}_\parallel,w)\|_{L^2(\Omega_{2D})}^2,
\end{equation}
consistent with the substitution $\partial_z\mapsto i\kappa$ in $H^1(\Omega)$.
When $\mu^*$ is independent of $z$, inserting \eqref{eq:toe4_fourier_expand} and $\partial_z\mapsto i\kappa_m$ yields a family of \emph{decoupled} 2D resolvent problems
\[
\mathcal{L}_\omega(\kappa_m)\,\hat{\mathbf{v}}_m=\hat{\mathbf{f}}_m\quad\text{in }X_{\kappa_m}^*,
\]
where $\mathcal{L}_\omega(\kappa):X_\kappa\to X_\kappa^*$ denotes the $\kappa$-reduced Oseen/Stokes operator with viscosity $\mu_0^*$.
For the modes retained in any truncation, assume
\begin{equation}
	\label{eq:toe4_block_invertibility_assumption}
	\mathcal{L}_\omega(\kappa_m)\ \text{is invertible for all retained }m,
	\qquad
	\|\mathcal{L}_\omega(\kappa_m)^{-1}\|_{\mathcal{L}(X_{\kappa_m}^*,X_{\kappa_m})}<\infty.
\end{equation}

If $g(z)=\sum_{n} g_n e^{i\kappa_n z}$ and $h(z)=\sum_{m} h_m e^{i\kappa_m z}$, then $(gh)_m=\sum_{n} g_n\,h_{m-n}$; in particular,
\begin{equation}
	\label{eq:toe4_shift_identities}
	(e^{ik_0 z}h)_m=h_{m-m_0},
	\qquad
	(\cos(k_0 z)h)_m=\tfrac12 h_{m-m_0}+\tfrac12 h_{m+m_0}.
\end{equation}
Accordingly, a single-harmonic multiplier produces a fixed mode-index shift, and a symmetric cosine produces nearest-neighbor coupling in the $\pm m_0$ directions.
This is exactly the “constant-diagonal” (Toeplitz/Laurent) structure in the mode index induced by convolution in Fourier coefficients.

\paragraph{Coupling operator induced by the viscous form.}
Let $\mathbf{D}_\kappa$ denote the symmetric gradient with $\partial_z$ replaced by $i\kappa$.
For $m,m'\in\mathbb{Z}$ define the bilinear form
\begin{equation}
	\label{eq:toe4_coupling_bilinear}
	\big\langle \mathcal{K}_\omega(\kappa_{m'}\to\kappa_m)\mathbf{u}_{m'},\mathbf{v}_m\big\rangle
	:=
	2\int_{\Omega_{2D}}
	\mu_0^*(x,y;\omega)\,
	\mathbf{D}_{\kappa_{m'}}(\mathbf{u}_{m'})
	:\overline{\mathbf{D}_{\kappa_m}(\mathbf{v}_m)}\,dx\,dy.
\end{equation}
This induces a bounded operator $\mathcal{K}_\omega(\kappa_{m'}\to\kappa_m):X_{\kappa_{m'}}\to X_{\kappa_m}^*$ satisfying
\begin{equation}
	\label{eq:toe4_K_bound}
	\|\mathcal{K}_\omega(\kappa_{m'}\to\kappa_m)\|_{\mathcal{L}(X_{\kappa_{m'}},X_{\kappa_m}^*)}
	\le C_{\Omega_{2D}}\|\mu_0^*(\cdot,\omega)\|_{L^\infty(\Omega_{2D})},
\end{equation}
with $C_{\Omega_{2D}}$ depending only on Korn/Poincar\'e constants and the choice of $\|\cdot\|_{X_\kappa}$. Projecting the variational form \eqref{eq:toe4_variational} onto $e^{i\kappa_m z}$ yields a bi-infinite coupled block system in the mode coefficients
$\{\hat{\mathbf{v}}_m\}_{m\in\mathbb{Z}}$, with off-diagonal blocks determined by \eqref{eq:toe4_shift_identities} and the corresponding Fourier coefficients of $\mu^*$:

\begin{enumerate}[leftmargin=2.2em, itemsep=0.35em]
	\item \emph{One-sided single-harmonic amplitude texture \eqref{eq:toe4_eta_onesided}:}
	\begin{equation}
		\label{eq:toe4_modes_onesided}
		\mathcal{L}_\omega(\kappa_m)\,\hat{\mathbf{v}}_m
		+
		\varepsilon\,\mathcal{K}_\omega(\kappa_{m-m_0}\to\kappa_m)\,\hat{\mathbf{v}}_{m-m_0}
		=
		\hat{\mathbf{f}}_m.
	\end{equation}
	
	\item \emph{Symmetric nearest-neighbor coupling (cosine amplitude \eqref{eq:toe4_eta_cos}):}
	\begin{equation}
		\label{eq:toe4_modes_cos}
		\mathcal{L}_\omega(\kappa_m)\,\hat{\mathbf{v}}_m
		+
		\frac{\varepsilon}{2}\,\mathcal{K}_\omega(\kappa_{m-m_0}\to\kappa_m)\,\hat{\mathbf{v}}_{m-m_0}
		+
		\frac{\varepsilon}{2}\,\mathcal{K}_\omega(\kappa_{m+m_0}\to\kappa_m)\,\hat{\mathbf{v}}_{m+m_0}
		=
		\hat{\mathbf{f}}_m.
	\end{equation}
	
	\item \emph{Phase-only Bessel texture \eqref{eq:toe4_eta_phase_only}--\eqref{eq:toe4_bessel_expansion}:}
	\begin{equation}
		\label{eq:toe4_modes_bessel}
		\mathcal{L}_\omega(\kappa_m)\,\hat{\mathbf{v}}_m
		+
		\big(J_0(\varepsilon)-1\big)\,\mathcal{K}_\omega(\kappa_m\to\kappa_m)\,\hat{\mathbf{v}}_m
		+
		\sum_{n\in\mathbb{Z}\setminus\{0\}}
		i^n J_n(\varepsilon)\,
		\mathcal{K}_\omega(\kappa_{m-nm_0}\to\kappa_m)\,\hat{\mathbf{v}}_{m-nm_0}
		=
		\hat{\mathbf{f}}_m.
	\end{equation}
	\noindent
	\emph{Remark (diagonal renormalization).} The extra diagonal term is exact and simply reflects that the mean Fourier coefficient of
	$e^{i\varepsilon\cos(k_0 z)}$ is $J_0(\varepsilon)$. In particular, $J_0(\varepsilon)=1+\mathcal{O}(\varepsilon^2)$, so at
	\emph{first order} the phase-only class reproduces the same nearest-neighbor coupling graph as \eqref{eq:toe4_modes_cos}, while the diagonal
	renormalization enters at $\mathcal{O}(\varepsilon^2)$.
\end{enumerate}

Equations \eqref{eq:toe4_modes_onesided}--\eqref{eq:toe4_modes_bessel} are explicit operator-valued Laurent/Toeplitz couplings:
the coupling graph depends only on fixed index offsets (multiples of $m_0$), while diagonal blocks remain mode-dependent through $\kappa_m$.
Fix $M\in\mathbb{N}$ and truncate to modes $m\in\{-M,\ldots,M\}$. Define block vectors
\[
\mathbf{V}^{[M]}
:=
(\hat{\mathbf{v}}_{-M},\ldots,\hat{\mathbf{v}}_{0},\ldots,\hat{\mathbf{v}}_{M})^{\mathsf{T}},
\qquad
\mathbf{F}^{[M]}
:=
(\hat{\mathbf{f}}_{-M},\ldots,\hat{\mathbf{f}}_{0},\ldots,\hat{\mathbf{f}}_{M})^{\mathsf{T}}.
\]
Let $\mathbb{D}_M(\omega)$ be the block-diagonal operator with diagonal blocks $\mathcal{L}_\omega(\kappa_m)$ and let $\mathbb{C}_M(\omega)$ be the block coupling operator.
For the symmetric nearest-neighbor case \eqref{eq:toe4_modes_cos}, one may write
\begin{equation}
	\label{eq:toe4_block_system}
	\Big(\mathbb{D}_M(\omega)+\varepsilon\,\mathbb{C}_M(\omega)\Big)\mathbf{V}^{[M]}=\mathbf{F}^{[M]},
\end{equation}
where $\mathbb{C}_M(\omega)$ has nonzero block diagonals only at offsets $m-n=\pm m_0$.
Thus, $\mathbb{C}_M$ is a \emph{block Laurent/Toeplitz coupling} (nonzeros determined by fixed index shifts), while $\mathbb{D}_M$ carries the mode-dependent physics.
For the full phase-only Bessel system \eqref{eq:toe4_modes_bessel}, $\mathbb{C}_M$ has nonzeros at offsets $m-n=nm_0$ with weights $i^nJ_n(\varepsilon)$.
Truncating the Bessel series to $|n|\le N$ yields a banded coupling with tail controlled by $\sum_{|n|>N}|J_n(\varepsilon)|$.

\subsubsection{Uniform Forcing in $z$ Produces $\kappa\neq 0$ Response:
	Neumann-Series Mechanism and Explicit Sideband Formulas}
\label{subsubsec:toe4_sidebands}

Assume the forcing is spanwise-uniform:
\begin{equation}
	\label{eq:toe4_uniform_forcing}
	\hat{\mathbf{f}}_m\equiv 0 \ \text{for } m\neq 0,
	\qquad
	\hat{\mathbf{f}}_0\neq 0.
\end{equation}
If $\mu^*$ is independent of $z$, then all spanwise modes decouple and $\hat{\mathbf{v}}_m\equiv 0$ for $m\neq 0$.
With a spanwise texture, the Laurent/Toeplitz coupling (a fixed index-shift convolution structure) forces nonzero sidebands.

To quantify modewise gain and constitutive coupling strength on a fixed truncation $m\in\{-M,\dots,M\}$, define
\begin{equation}
	\label{eq:toe4_Gmax_Kmax}
	G_{\max}(\omega;M)
	:=
	\max_{|m|\le M}\bigl\|\mathcal{L}_\omega(\kappa_m)^{-1}\bigr\|_{\mathcal{L}(X_{\kappa_m}^*,X_{\kappa_m})},
	\qquad
	K_{\max}(\omega;M)
	:=
	\max_{\substack{|m|\le M\\ |m\pm m_0|\le M}}
	\bigl\|\mathcal{K}_\omega(\kappa_{m}\to\kappa_{m\pm m_0})\bigr\|_{\mathcal{L}(X_{\kappa_m},X_{\kappa_{m\pm m_0}}^*)}.
\end{equation}
(If you prefer a single-direction definition, take the maximum over the $+$ and $-$ directions separately and then set
$K_{\max}:=\max\{K_{\max}^{+},K_{\max}^{-}\}$.)

\begin{proposition}[First sideband generation for symmetric nearest-neighbor coupling]
	\label{prop:toe4_first_sidebands}
	Assume the symmetric nearest-neighbor truncated system \eqref{eq:toe4_block_system} and the block invertibility hypothesis
	\eqref{eq:toe4_block_invertibility_assumption}.
	If
	\begin{equation}
		\label{eq:toe4_smallness}
		\varepsilon\,K_{\max}(\omega;M)\,G_{\max}(\omega;M) < 1,
	\end{equation}
	then \eqref{eq:toe4_block_system} is invertible. Moreover, the solution satisfies the expansions
	\begin{align}
		\label{eq:toe4_u0_leading}
		\hat{\mathbf{v}}_0
		&=
		\mathcal{L}_\omega(\kappa_0)^{-1}\hat{\mathbf{f}}_0
		+\mathcal{O}(\varepsilon^2)
		\quad\text{in }X_{\kappa_0},\\
		\label{eq:toe4_sideband_leading}
		\hat{\mathbf{v}}_{\pm m_0}
		&=
		-\frac{\varepsilon}{2}\,
		\mathcal{L}_\omega(\kappa_{\pm m_0})^{-1}\,
		\mathcal{K}_\omega(\kappa_{0}\to\kappa_{\pm m_0})\,
		\hat{\mathbf{v}}_0
		+\mathcal{O}(\varepsilon^2)
		\quad\text{in }X_{\kappa_{\pm m_0}}.
	\end{align}
	Consequently,
	\begin{equation}
		\label{eq:toe4_sideband_norm_bound}
		\|\hat{\mathbf{v}}_{\pm m_0}\|_{X_{\kappa_{\pm m_0}}}
		\le
		\frac{\varepsilon}{2}\,
		\bigl\|\mathcal{L}_\omega(\kappa_{\pm m_0})^{-1}\bigr\|\,
		\bigl\|\mathcal{K}_\omega(\kappa_0\to\kappa_{\pm m_0})\bigr\|\,
		\bigl\|\mathcal{L}_\omega(\kappa_0)^{-1}\bigr\|\,
		\|\hat{\mathbf{f}}_0\|_{X_{\kappa_0}^*}
		+\mathcal{O}(\varepsilon^2).
	\end{equation}
\end{proposition}

\begin{proof}[Proof (index shifts and a Neumann-series expansion)]
	Fix a truncation level $M$ and write the finite-section system in block form as
	\begin{equation}
		\label{eq:toe4_block_system_proof_start}
		(\mathbb{D}_M+\varepsilon \mathbb{C}_M)\mathbf{V}^{[M]}=\mathbf{F}^{[M]},
	\end{equation}
	where $\mathbb{D}_M=\mathrm{diag}\big(\mathcal{L}_\omega(\kappa_{-M}),\ldots,\mathcal{L}_\omega(\kappa_M)\big)$ is block-diagonal and
	$\mathbb{C}_M$ is the block Laurent/Toeplitz coupling with nonzero block diagonals only at offsets $\pm m_0$.
	Assume each diagonal block is invertible and set
	\[
	\mathbb{D}_M^{-1}:=\mathrm{diag}\bigl(\mathcal{L}_\omega(\kappa_{-M})^{-1},\ldots,\mathcal{L}_\omega(\kappa_M)^{-1}\bigr),
	\]
	so $\mathbb{D}_M^{-1}\in\mathcal{L}(\mathbb{X}_M^*,\mathbb{X}_M)$ on $\mathbb{X}_M:=\prod_{|m|\le M}X_{\kappa_m}$
	with dual $\mathbb{X}_M^*:=\prod_{|m|\le M}X_{\kappa_m}^*$. Left-multiplying \eqref{eq:toe4_block_system_proof_start} by $\mathbb{D}_M^{-1}$ yields
	\begin{equation}
		\label{eq:toe4_neumann_form}
		(\mathbb{I}+\varepsilon \mathbb{T}_M)\mathbf{V}^{[M]}=\mathbb{D}_M^{-1}\mathbf{F}^{[M]},
		\qquad
		\mathbb{T}_M:=\mathbb{D}_M^{-1}\mathbb{C}_M\in\mathcal{L}(\mathbb{X}_M,\mathbb{X}_M).
	\end{equation}
	By the definition of $\mathbb{C}_M$ (only $\pm m_0$ block diagonals) and block-diagonality of $\mathbb{D}_M^{-1}$,
	$\mathbb{T}_M$ inherits the same nearest-neighbor adjacency:
	\begin{equation}
		\label{eq:toe4_shift_property}
		(\mathbb{T}_M\mathbf{W})_m \text{ depends only on } \mathbf{W}_{m-m_0}\ \text{and}\ \mathbf{W}_{m+m_0}.
	\end{equation}
	Moreover, the operator norm is bounded as
	\[
	\|\mathbb{T}_M\|
	\le
	\|\mathbb{D}_M^{-1}\|\ \|\mathbb{C}_M\|
	\le
	G_{\max}(\omega;M)\,K_{\max}(\omega;M),
	\]
	so \eqref{eq:toe4_smallness} implies $\varepsilon\|\mathbb{T}_M\|<1$ and hence $(\mathbb{I}+\varepsilon\mathbb{T}_M)^{-1}$ admits a convergent Neumann series:
	\begin{equation}
		\label{eq:toe4_neumann_series}
		(\mathbb{I}+\varepsilon \mathbb{T}_M)^{-1}
		=
		\sum_{j=0}^\infty (-\varepsilon)^j \mathbb{T}_M^j
		\quad\text{in }\mathcal{L}(\mathbb{X}_M,\mathbb{X}_M).
	\end{equation}
	Thus
	\begin{equation}
		\label{eq:toe4_solution_series}
		\mathbf{V}^{[M]}
		=
		\sum_{j=0}^\infty (-\varepsilon)^j \mathbb{T}_M^j \,\mathbb{D}_M^{-1}\mathbf{F}^{[M]}.
	\end{equation}
	Now impose uniform forcing \eqref{eq:toe4_uniform_forcing}, so $\mathbf{F}^{[M]}$ is supported only in the $m=0$ component.
	Since $\mathbb{D}_M^{-1}$ is block-diagonal, $\mathbf{W}:=\mathbb{D}_M^{-1}\mathbf{F}^{[M]}$ is also supported only at $m=0$, with
	$\mathbf{W}_0=\mathcal{L}_\omega(\kappa_0)^{-1}\hat{\mathbf{f}}_0$.
	By the adjacency property \eqref{eq:toe4_shift_property}, $\mathbb{T}_M\mathbf{W}$ is supported only at $m=\pm m_0$ (up to truncation),
	and therefore no $m=0$ contribution can occur at order $j=1$ (one shift cannot return to $0$).
	
	Extracting the $m=\pm m_0$ components of \eqref{eq:toe4_solution_series} gives
	\[
	\hat{\mathbf{v}}_{\pm m_0}^{[M]}
	=
	-\varepsilon\,(\mathbb{T}_M\mathbf{W})_{\pm m_0}
	+\mathcal{O}(\varepsilon^2)
	=
	-\varepsilon\,\mathcal{L}_\omega(\kappa_{\pm m_0})^{-1}\,(\mathbb{C}_M\mathbf{W})_{\pm m_0}
	+\mathcal{O}(\varepsilon^2),
	\]
	and substituting the explicit $\pm m_0$ coupling blocks of $\mathbb{C}_M$ yields \eqref{eq:toe4_sideband_leading}
	(and \eqref{eq:toe4_sideband_norm_bound} follows by taking operator norms).
	Likewise, extracting the $m=0$ component gives
	\[
	\hat{\mathbf{v}}_0^{[M]}
	=
	\mathbf{W}_0
	-\varepsilon(\mathbb{T}_M\mathbf{W})_0
	+\varepsilon^2(\mathbb{T}_M^2\mathbf{W})_0+\cdots
	=
	\mathcal{L}_\omega(\kappa_0)^{-1}\hat{\mathbf{f}}_0
	+\mathcal{O}(\varepsilon^2),
	\]
	since $(\mathbb{T}_M\mathbf{W})_0$ depends only on $\mathbf{W}_{\pm m_0}=0$ and hence vanishes.
\end{proof}

\noindent To make the out-and-back transfer explicit, truncate to three modes $m\in\{-m_0,0,m_0\}$ with uniform forcing
\eqref{eq:toe4_uniform_forcing}. Solving \eqref{eq:toe4_modes_cos} to second order yields
\begin{equation}
	\label{eq:toe4_v0_second_order}
	\begin{aligned}
		\hat{\mathbf{v}}_0
		&=
		\mathcal{L}_\omega(\kappa_0)^{-1}\hat{\mathbf{f}}_0
		-\frac{\varepsilon^2}{4}\,
		\mathcal{L}_\omega(\kappa_0)^{-1}
		\Bigl[
		\mathcal{K}_\omega(\kappa_{m_0}\!\to\!\kappa_0)\,
		\mathcal{L}_\omega(\kappa_{m_0})^{-1}\,
		\mathcal{K}_\omega(\kappa_0\!\to\!\kappa_{m_0})
		\\[-0.25ex]
		&\hspace{3.6em}
		+
		\mathcal{K}_\omega(\kappa_{-m_0}\!\to\!\kappa_0)\,
		\mathcal{L}_\omega(\kappa_{-m_0})^{-1}\,
		\mathcal{K}_\omega(\kappa_0\!\to\!\kappa_{-m_0})
		\Bigr]\,
		\mathcal{L}_\omega(\kappa_0)^{-1}\hat{\mathbf{f}}_0
		+\mathcal{O}(\varepsilon^3).
	\end{aligned}
\end{equation}
The sidebands are given at leading order by \eqref{eq:toe4_sideband_leading}.
Equation \eqref{eq:toe4_v0_second_order} is the operator-theoretic \emph{square-resolvent} structure in a purely 3D Fourier-coupling setting:
the mean-mode correction contains two resolvent factors and two coupling maps (out-and-back transfer).

\medskip
\noindent\textbf{Texture-specific cascades.}
For the one-sided texture \eqref{eq:toe4_eta_onesided}, \eqref{eq:toe4_modes_onesided} yields a one-direction cascade
\[
\hat{\mathbf{v}}_{q m_0}=\mathcal{O}(\varepsilon^q)\quad (q\ge 1),
\qquad
\hat{\mathbf{v}}_{-q m_0}=0 \ \text{to leading orders},
\]
which provides a clean diagnostic of the underlying coefficient harmonic.
For the phase-only Bessel texture \eqref{eq:toe4_eta_phase_only}--\eqref{eq:toe4_bessel_expansion}, the coupling range is infinite in principle,
but the weights $|J_n(\varepsilon)|$ decay rapidly in $|n|$ for small $\varepsilon$ (Jacobi--Anger / Bessel-series structure). Truncating the series \eqref{eq:toe4_bessel_expansion} to $|n|\le N$ produces a $(2N+1)$-band coupling, with truncation error controlled by the tail
$\sum_{|n|>N}|J_n(\varepsilon)|$.

\medskip
\noindent\textbf{Summary.}
In a spanwise-periodic domain $\Omega_{2D}\times\mathbb{T}_{L_z}$, a single spanwise Fourier harmonic in a constitutive coefficient produces an explicit
Laurent/Toeplitz index shift in the spanwise Fourier label.
Consequently, even spanwise-uniform forcing generically generates $\kappa\neq 0$ response through a purely linear, frequency-domain mechanism.
For nearest-neighbor coupling (cosine amplitude or the leading-order phase-only expansion), the first sidebands $\kappa=\pm k_0$ are $\mathcal{O}(\varepsilon)$
and are controlled by the product of constitutive coupling norms and modewise resolvent gains \eqref{eq:toe4_sideband_norm_bound}.
The mean-mode correction begins at $\mathcal{O}(\varepsilon^2)$ and exhibits an explicit out-and-back (square-resolvent) structure
\eqref{eq:toe4_v0_second_order}. In the phase-only class, the Jacobi--Anger expansion provides explicit Toeplitz/Laurent coefficients $i^nJ_n(\varepsilon)$,
so coupling range and truncation error are quantitatively controlled by Bessel tails.

\subsubsection{What to Compute: Three Signatures That Uniquely Exhibit the 3D Toeplitz Coupling Mechanism.}
\label{subsubsec:toe4_signatures}

This worked example is designed to culminate in \emph{reportable, computation-facing signatures} that are specific to spanwise Fourier coupling
induced by a $z$-textured coefficient field. The distinguishing feature is structural:
\emph{multiplication by a single (or narrow-band) harmonic in physical space induces fixed-offset index shifts in Fourier space}.
Accordingly, the signatures below are constructed so that they vanish identically in the $z$-independent coefficient setting under
$z$-uniform forcing, they are stable under increasing Fourier truncation and mesh refinement, and they can be interpreted as
\emph{mode-mixing observables} rather than as generic "3D-ness" measures.

Throughout, we assume $z$-periodicity of length $L_z$, Fourier wavenumbers $\kappa_m=2\pi m/L_z$, a coupling harmonic at $\pm m_0$
($k_0=\kappa_{m_0}$), and a forcing that is spanwise-uniform:
\[
\hat{\mathbf f}_m(\omega)\equiv 0 \quad (m\neq 0), \qquad \hat{\mathbf f}_0(\omega)\neq 0.
\]
Let $X_{\kappa_m}$ denote the $\kappa_m$-constrained mode space used in the reduced cross-sectional problem and
let $\|\cdot\|_{X_{\kappa_m}}$ be an energy-consistent norm (e.g.\ $H^1$-based, or a strain-weighted norm), fixed once and used across all cases.
When a quantity depends on the specific choice of norm, this dependence should be stated explicitly; the qualitative conclusions are invariant,
but numerical magnitudes can shift.

\paragraph{Signature 1: Sideband Ratios (Linear Pattern Amplitude and Frequency Selection.)}
Define the sideband-to-mean ratios
\begin{equation}
	\label{eq:toe4_sideband_ratio}
	R_{\pm}(\omega)
	:=
	\frac{\|\hat{\mathbf{v}}_{\pm m_0}(\omega)\|_{X_{\kappa_{\pm m_0}}}}
	{\|\hat{\mathbf{v}}_{0}(\omega)\|_{X_{\kappa_{0}}}}.
\end{equation}
In the $z$-independent coefficient setting, the modal equations decouple, $\hat{\mathbf v}_{m}\equiv 0$ for $m\neq 0$, and, hence,
$R_\pm(\omega)\equiv 0$ identically. Under a single-harmonic texture, $R_\pm$ is the most direct measure of \emph{linear} spanwise patterning:
it quantifies the amplitude of the first sidebands generated by Toeplitz/Laurent index shifts, and it is sensitive to \emph{frequency selection}
because both the diagonal resolvent gains and the coupling map vary with $\omega$.

Under the symmetric nearest-neighbor coupling model and the smallness condition of Proposition~\ref{prop:toe4_first_sidebands},
the leading-order estimate reads
\begin{equation}
	\label{eq:toe4_Rpm_bound}
	R_{\pm}(\omega)
	\ \le\
	\frac{\varepsilon}{2}\,
	\|\mathcal{L}_\omega(\kappa_{\pm m_0})^{-1}\|_{\mathcal{L}(X_{\kappa_{\pm m_0}}^*,X_{\kappa_{\pm m_0}})}\,
	\|\mathcal{K}_\omega(\kappa_{0}\to\kappa_{\pm m_0})\|_{\mathcal{L}(X_{\kappa_0},X_{\kappa_{\pm m_0}}^*)}
	\;+\;\mathcal{O}(\varepsilon^2),
\end{equation}
where the implied constant in $\mathcal{O}(\varepsilon^2)$ is uniform under the truncation assumptions.
This estimate has a transparent fluid-mechanics reading: $R_\pm$ is small in $\varepsilon$ but can become large when either the
sideband resolvent gain $\|\mathcal{L}_\omega(\kappa_{\pm m_0})^{-1}\|$ is large (a \emph{sideband amplification} effect) or when the coupling
norm $\|\mathcal{K}_\omega\|$ is large (a \emph{strong mixer} effect), or both. Plot $R_\pm(\omega)$ over the sweep band and report peak values and peak locations in $\omega$,
whether $R_+(\omega)$ and $R_-(\omega)$ coincide (symmetry) or differ (e.g.\ one-sided textures or asymmetric base states),
and a comparison to modewise diagonal gains $\|\mathcal{L}_\omega(\kappa_m)^{-1}\|$ to indicate whether peaks are primarily
"diagonal-gain driven" or "coupling driven." This makes the Toeplitz mechanism legible without requiring inspection of full 3D fields.

\paragraph{Signature 2: Spanwise Energy Fraction (Global Degree of Toeplitz-Induced Three-Dimensionality.)}
A single sideband ratio characterizes the first mode mixing; however, the Toeplitz mechanism can generate an entire ladder of modes.
Define the truncated spanwise energy fraction
\begin{equation}
	\label{eq:toe4_spanwise_energy_fraction}
	\Phi_M(\omega)
	:=
	\frac{\sum_{0<|m|\le M}\|\hat{\mathbf{v}}_m(\omega)\|_{L^2(\Omega_{2D})}^2}
	{\sum_{|m|\le M}\|\hat{\mathbf{v}}_m(\omega)\|_{L^2(\Omega_{2D})}^2},
	\qquad M\in\mathbb{N}.
\end{equation}
In the decoupled setting (uniform forcing; $z$-independent coefficients), $\Phi_M(\omega)=0$ for every $M$ and $\omega$.
Under Toeplitz coupling, $\Phi_M(\omega)>0$ measures how much of the response energy is diverted into $\kappa\neq 0$ modes
\emph{purely by linear constitutive mixing}. For small $\varepsilon$, $\Phi_M$ is expected to scale like $\mathcal{O}(\varepsilon^2)$ when the first sidebands dominate the $\kappa\neq 0$
energy, because the sideband amplitude itself is $\mathcal{O}(\varepsilon)$ and $\Phi_M$ is quadratic. More precisely, if
$\|\hat{\mathbf v}_{\pm m_0}\|=\mathcal{O}(\varepsilon)\|\hat{\mathbf v}_0\|$ and higher modes are negligible, then
\[
\Phi_M(\omega)\approx
\frac{\|\hat{\mathbf v}_{m_0}\|^2+\|\hat{\mathbf v}_{-m_0}\|^2}{\|\hat{\mathbf v}_0\|^2}
\;+\;\mathcal{O}(\varepsilon^4)
\;\sim\;
R_+(\omega)^2+R_-(\omega)^2.
\]
In one-sided textures, $\Phi_M$ is typically dominated by $m= m_0$ at small $\varepsilon$; in symmetric cosine/phase-only cases,
$m=\pm m_0$ dominate at leading order. Because $\Phi_M$ aggregates many modes, its use requires a clear truncation check:
increase $M$ until $\Phi_M(\omega)$ stabilizes uniformly in $\omega$ to the desired tolerance. For phase-only textures with Bessel-type
infinite-range coupling, one should additionally track the effective coupling bandwidth $N$ and ensure that the coefficient tail is negligible
(e.g.\ by bounding $\sum_{|n|>N}|J_n(\varepsilon)|$). We suggest a practitioner report $\Phi_M(\omega)$ together with the stability plateau in $M$
(and in $N$ if applicable). Provide at least one decomposition plot showing how the energy is distributed across modes at a representative peak
frequency (e.g.\ the histogram of $\|\hat{\mathbf v}_m\|_{L^2}^2$ versus $m$). This is the cleanest global demonstration that Toeplitz coupling
has produced an intrinsically 3D response.

\paragraph{Signature 3: Traction Phase Fourier Signature on the Wall (Experiment-Facing Fingerprint.)}
Mode mixing should be visible not only in volume fields but in wall observables.
Let $\Gamma_W\subset\partial\Omega_{2D}$ denote the wall set, with outward unit normal $\mathbf n_{2D}(x,y)$ and a chosen unit tangent
$\mathbf t(x,y)$ (fixed by orientation). Define the complex tangential traction on $\Gamma_W\times(0,L_z)$:
\begin{equation}
	\label{eq:toe4_traction_def}
	\hat{\tau}_t(x,y,z;\omega)
	:=
	\mathbf{t}(x,y)\cdot\big(2\mu^*(x,y,z;\omega)\mathbf{D}(\hat{\mathbf{v}})\big)\mathbf{n}(x,y),
	\qquad \mathbf{n}(x,y)=(\mathbf{n}_{2D}(x,y),0).
\end{equation}
Take its spanwise Fourier coefficients
\[
\hat{\tau}_{t,m}(x,y;\omega)
:=
\frac{1}{L_z}\int_0^{L_z}\hat{\tau}_t(x,y,z;\omega)\,e^{-i\kappa_m z}\,dz.
\]
In the $z$-independent coefficient setting with $z$-uniform forcing, $\hat{\tau}_{t,m}\equiv 0$ for $m\neq 0$ (as a direct consequence of
decoupling). c, nonzero sideband traction is an experimentally interpretable fingerprint of \emph{linear constitutive mode mixing}. To define a robust phase diagnostic, avoid geometric singularities (e.g.\ corners in the cross-section) and integrate over a measurement segment.
Let $\Gamma_W(r_0)\subset\Gamma_W$ be a wall segment/arc at distance $r_0$ from any singular set, specified by arc-length endpoints.
Define the wall-averaged sideband traction phase \emph{relative} to the mean traction by
\begin{equation}
	\label{eq:toe4_traction_phase_signature}
	\Theta_{\pm}(\omega;r_0)
	:=
	\arg\left(
	\frac{\displaystyle \int_{\Gamma_W(r_0)} \hat{\tau}_{t,\pm m_0}(x,y;\omega)\,ds}
	{\displaystyle \int_{\Gamma_W(r_0)} \hat{\tau}_{t,0}(x,y;\omega)\,ds}
	\right).
\end{equation}
This ratio is dimensionless and reduces sensitivity to overall amplitude calibration.
A companion amplitude diagnostic is
\begin{equation}
	\label{eq:toe4_traction_amplitude_ratio}
	\mathcal{A}_{\pm}(\omega;r_0)
	:=
	\frac{\left|\displaystyle \int_{\Gamma_W(r_0)} \hat{\tau}_{t,\pm m_0}(x,y;\omega)\,ds\right|}
	{\left|\displaystyle \int_{\Gamma_W(r_0)} \hat{\tau}_{t,0}(x,y;\omega)\,ds\right|},
\end{equation}
which often correlates with $R_\pm(\omega)$ but is more directly experiment-facing.
In combination, $(\mathcal{A}_\pm,\Theta_\pm)$ provide a compact description of the spanwise pattern in wall traction. As a reporting recommendation, we suggest a practitioner state the harmonic convention $e^{i\omega t}$, specify the reference for the phase, and describe the unwrapping
procedure across $\omega$ (e.g.\ continuous unwrapping with a fixed reference frequency). Report sensitivity to $r_0$ to show that the signal is
not an artifact of a near-singular measurement location. The three signatures above are sufficient for the main narrative of this paper, but one can also expose the Toeplitz mechanism in a purely operator-theoretic
way by measuring the off-diagonal blocks of the truncated resolvent.

Let the truncated block operator be
\[
\mathbb{R}_M(\omega;\varepsilon):=\big(\mathbb{D}_M(\omega)+\varepsilon\mathbb{C}_M(\omega)\big)^{-1},
\]
acting on the truncated product space $\prod_{|m|\le M} X_{\kappa_m}$ (with dual forcing space $\prod_{|m|\le M} X_{\kappa_m}^*$).
Then the mode-mixing transfer from $0$ to $\pm m_0$ is encoded in the off-diagonal blocks
\[
\mathbb{R}_M(\omega;\varepsilon)_{\pm m_0,0}: X_{\kappa_0}^*\to X_{\kappa_{\pm m_0}}.
\]
A scalar summary is
\begin{equation}
	\label{eq:toe4_transfer_norm}
	T_{\pm}(\omega)
	:=
	\big\|\mathbb{R}_M(\omega;\varepsilon)_{\pm m_0,0}\big\|_{\mathcal{L}(X_{\kappa_0}^*,X_{\kappa_{\pm m_0}})}.
\end{equation}
At leading order in $\varepsilon$, the Neumann-series expansion yields the operational "mode-mixer" formula
\begin{equation}
	\label{eq:toe4_transfer_leading}
	T_{\pm}(\omega)
	=
	\frac{\varepsilon}{2}\,
	\big\|
	\mathcal{L}_\omega(\kappa_{\pm m_0})^{-1}\,
	\mathcal{K}_\omega(\kappa_0\to\kappa_{\pm m_0})\,
	\mathcal{L}_\omega(\kappa_0)^{-1}
	\big\|
	\;+\;\mathcal{O}(\varepsilon^2),
\end{equation}
which can be viewed as a \emph{constitutive transfer gain} from mean forcing to sideband response. Reporting $T_\pm$ is particularly useful when
one wants to emphasize Toeplitz coupling as an operator block phenomenon rather than as a field visualization. The signatures $R_\pm(\omega)$, $\Phi_M(\omega)$, and $\Theta_\pm(\omega;r_0)$ (optionally complemented by $T_\pm(\omega)$) make the phenomenon
operational:
\begin{itemize}
	\item $R_\pm$ measures \emph{first-harmonic pattern amplitude} and captures frequency selection through diagonal resolvent gains and coupling strength.
	\item $\Phi_M$ measures the \emph{global fraction of energy} diverted into $\kappa\neq 0$ modes, documenting truly 3D response induced linearly.
	\item $\Theta_\pm$ (and $\mathcal{A}_\pm$) provide an \emph{experiment-facing wall fingerprint} of spanwise mode mixing in traction.
	\item $T_\pm$ isolates the \emph{off-diagonal resolvent block} that implements mean-to-sideband transfer.
\end{itemize}
Each quantity vanishes identically in the classical decoupled setting under spanwise-uniform forcing and therefore cannot be explained by the
one-dimensional phase-drift mechanism of Stokes II or by the two-dimensional corner-defect commutator forcing mechanism. They are specific to the
operator-valued Toeplitz/Laurent coupling induced by $z$-dependent viscosity textures.

For each geometry (3D periodic channel, or 3D periodic BFS truncation) and each texture class, a defensible workflow should enforce spanwise-uniform forcing so that any $\kappa\neq 0$ response is constitutive and linear, verify stability with respect to both mode truncation and $(x,y)$ discretization,
and report the signatures in \S\ref{subsubsec:toe4_signatures} in a way that separates texture amplitude, truncation, and boundary closure effects.

\begin{enumerate}
	\item \textbf{Baseline vs.\ Textured Families (Mechanism Isolation.)}
	Compute the following cases on identical discretizations and boundary/outflow closures:
	\begin{enumerate}[label=(\roman*), leftmargin=2.2em, itemsep=0.2ex, topsep=0.4ex]
		\item $\mu^*=\mu_0$ \hfill (Constant Real)
		\item $\mu^*=\mu_0 e^{i\psi_0}$ \hfill (Constant Complex)
		\item $\mu^*=\mu_0^*(x,y)$ \hfill ($z$-Independent)
		\item $\mu^*=\mu_0^*(x,y)\,g(z)$ \hfill (Spanwise-Textured)
	\end{enumerate}
	Instantiate (iv) with: one-sided harmonic,  cosine (symmetric), and phase-only.
	When possible, include a phase-only configuration where $|\mu^*|$ is essentially fixed, so that observed
	$\kappa$-mixing cannot be attributed to magnitude heterogeneity.
	
	\item \textbf{Spanwise-Uniform Forcing (Hard Constraint.)}
	Impose $\hat{\mathbf{f}}_m=0$ for $m\neq 0$ in the Fourier-expanded system, equivalently $z$-independent forcing in physical space.
	If boundary forcing is used (inflow or wall motion), enforce $z$-independence at the boundary as well, and document the enforcement strategy.
	
	\item \textbf{Frequency Sweep With the 3D Signatures.}
	For $\omega$ in a prescribed band, compute
	\[
	R_{\pm}(\omega),\qquad
	\Phi_M(\omega),\qquad
	\Theta_{\pm}(\omega;r_0),
	\]
	and optionally $\mathcal{A}_\pm(\omega;r_0)$.
	Report frequencies where $R_\pm$ or $\Phi_M$ peak, and compare these to the peaks (or ridges) of the diagonal modewise gains
	$m\mapsto\|\mathcal{L}_\omega(\kappa_m)^{-1}\|$ to contextualize whether pattern selection is driven by sideband amplification.
	
	\item \textbf{Mode Truncation Checks (Toeplitz-Tail Defensibility.)}
	Increase $M$ until $R_\pm(\omega)$, $\Phi_M(\omega)$, and $\Theta_\pm(\omega;r_0)$ stabilize to tolerance across the entire $\omega$ band.
	For phase-only textures with Bessel/Laurent expansions, also increase the coupling bandwidth $N$ (if truncating the coefficient expansion)
	until the coefficient tail is negligible relative to the desired accuracy.
	A practical reporting convention is to quote the smallest $(M,N)$ for which the signatures change by less than a prescribed relative tolerance.
	
	\item \textbf{$(x,y)$ Mesh Refinement and Solver Tolerance (Cross-Sectional Accuracy.)}
	Refine the $(x,y)$ mesh while keeping the mode truncation fixed at its converged value.
	Because the coupling is in Fourier index, the mesh and the mode truncation represent distinct approximation axes; both must be documented.
	Report that the \emph{differences} between $z$-independent and textured cases in $(R_\pm,\Phi_M,\Theta_\pm)$ stabilize under refinement.
	
	\item \textbf{Domain Truncation and Outflow-Closure Robustness (BFS Truncations.)}
	For BFS truncations, repeat representative computations under increased downstream length, and an alternate outflow closure (do-nothing versus stabilized versus sponge),
	while keeping forcing, texture, and near-step resolution fixed.
	The signatures should be robust once the truncation is long enough; residual sensitivity should be reported (not suppressed), especially at high $\omega$.
	
	\item \textbf{Optional: Direct "Mode-Mixer" Transfer Measurement.}
	At select peak frequencies, compute $T_\pm(\omega)$ from \eqref{eq:toe4_transfer_norm} (or an equivalent estimate of the
	off-diagonal block action) to directly quantify mean-to-sideband transfer.
	This is particularly compelling when paired with $R_\pm(\omega)$: $T_\pm$ measures transfer \emph{gain}, while $R_\pm$ measures realized
	sideband amplitude under a specific forcing.
	
	\item \textbf{Optional: Non-Normality Diagnostics on the Truncated Operator.}
	For continuity with the broader narrative, compute a departure-from-normality metric on the truncated operator matrix $A_M(\omega)$,
	\[
	\Delta_{\mathrm{nn}}(\omega)
	:=
	\frac{\|A_M(\omega)^*A_M(\omega)-A_M(\omega)A_M(\omega)^*\|}{\|A_M(\omega)\|^2},
	\]
	and compare baseline versus textured cases.
	Here, non-normality may be compounded by Oseen terms and by inter-mode coupling; the diagnostic is therefore best interpreted in tandem with
	the explicitly Toeplitz signatures $(R_\pm,\Phi_M,\Theta_\pm)$.
\end{enumerate}

\newpage

\appendix
\section{Neumann-Series Solution and Support Propagation for the Toeplitz/Laurent Finite-Section System.}
\label{app:toe4_neumann_shift_proof}
Fix an integer truncation level $M\ge 1$ and a shift index $m_0\in\{1,\dots,M\}$.
For each $m\in\{-M,\dots,M\}$, let $X_{\kappa_m}$ be the (complex) Hilbert space for the $m$th spanwise mode and let
$X_{\kappa_m}^*$ denote its (anti-)dual.
Define the truncated product spaces
\[
\mathbb{X}_M := \prod_{m=-M}^M X_{\kappa_m},
\qquad
\mathbb{X}_M^* := \prod_{m=-M}^M X_{\kappa_m}^*,
\]
equipped with the natural Hilbert product norm
\[
\|\mathbf{V}\|_{\mathbb{X}_M}^2 := \sum_{m=-M}^M \|\mathbf{v}_m\|_{X_{\kappa_m}}^2,
\qquad
\|\mathbf{F}\|_{\mathbb{X}_M^*}^2 := \sum_{m=-M}^M \|\mathbf{f}_m\|_{X_{\kappa_m}^*}^2.
\]
Let $\mathbb{D}_M:\mathbb{X}_M\to \mathbb{X}_M^*$ be block diagonal,
\[
(\mathbb{D}_M\mathbf{V})_m := \mathcal{L}_\omega(\kappa_m)\,\mathbf{v}_m,
\qquad m=-M,\dots,M,
\]
where each diagonal block $\mathcal{L}_\omega(\kappa_m):X_{\kappa_m}\to X_{\kappa_m}^*$ is assumed boundedly invertible.
Let $\mathbb{C}_M:\mathbb{X}_M\to \mathbb{X}_M^*$ be the truncated Laurent/Toeplitz coupling operator with \emph{bandwidth} $\pm m_0$:
there exist bounded operators
\[
\mathcal{K}_m^- : X_{\kappa_{m-m_0}} \to X_{\kappa_m}^*,
\qquad
\mathcal{K}_m^+ : X_{\kappa_{m+m_0}} \to X_{\kappa_m}^*,
\]
with the convention that $\mathcal{K}_m^\pm\equiv 0$ if the source index $m\mp m_0$ lies outside $\{-M,\dots,M\}$, such that
\begin{equation}
	\label{eq:app_toe4_C_def}
	(\mathbb{C}_M\mathbf{V})_m
	=
	\mathcal{K}_m^-\,\mathbf{v}_{m-m_0}
	+
	\mathcal{K}_m^+\,\mathbf{v}_{m+m_0},
	\qquad m=-M,\dots,M.
\end{equation}
In the symmetric nearest-neighbor case arising from $\cos(k_0 z)$-type modulation or the $\mathcal{O}(\varepsilon)$ truncation of
$e^{i\varepsilon\cos(k_0 z)}$, one has $\mathcal{K}_m^-=\mathcal{K}_m^+=(i/2)\,\mathcal{K}_\omega(\kappa_{m\mp m_0}\to\kappa_m)$, but
the arguments below do not require symmetry.
The finite-section system reads
\begin{equation}
	\label{eq:app_toe4_block_system}
	(\mathbb{D}_M+\varepsilon \mathbb{C}_M)\mathbf{V}^{[M]}=\mathbf{F}^{[M]},
	\qquad 0<\varepsilon\ll 1.
\end{equation}
Next, define the block-diagonal inverse $\mathbb{D}_M^{-1}:\mathbb{X}_M^*\to\mathbb{X}_M$ by
\[
(\mathbb{D}_M^{-1}\mathbf{F})_m := \mathcal{L}_\omega(\kappa_m)^{-1}\,\mathbf{f}_m.
\]
Let
\[
G_{\max}(\omega;M)
:=
\max_{|m|\le M}\big\|\mathcal{L}_\omega(\kappa_m)^{-1}\big\|_{\mathcal{L}(X_{\kappa_m}^*,X_{\kappa_m})},
\]
and define a coupling bound (one convenient choice among several equivalent ones) by
\[
K_{\max}(\omega;M)
:=
\max_{|m|\le M}\Big(
\|\mathcal{K}_m^-\|_{\mathcal{L}(X_{\kappa_{m-m_0}},X_{\kappa_m}^*)}
+
\|\mathcal{K}_m^+\|_{\mathcal{L}(X_{\kappa_{m+m_0}},X_{\kappa_m}^*)}
\Big).
\]
Then $\|\mathbb{D}_M^{-1}\|\le G_{\max}(\omega;M)$ and $\|\mathbb{C}_M\|\le K_{\max}(\omega;M)$ as operators between the corresponding product
spaces (by the triangle inequality and Cauchy--Schwarz in $\mathbb{X}_M$).
Introduce
\begin{equation}
	\label{eq:app_toe4_T_def}
	\mathbb{T}_M := \mathbb{D}_M^{-1}\mathbb{C}_M \in \mathcal{L}(\mathbb{X}_M,\mathbb{X}_M).
\end{equation}
Then
\begin{equation}
	\label{eq:app_toe4_T_norm}
	\|\mathbb{T}_M\|
	\le
	\|\mathbb{D}_M^{-1}\|\,\|\mathbb{C}_M\|
	\le
	G_{\max}(\omega;M)\,K_{\max}(\omega;M).
\end{equation}

\begin{theorem}[Neumann-series solvability of the finite-section Toeplitz/Laurent system]
	\label{thm:app_toe4_neumann}
	Assume the smallness condition
	\begin{equation}
		\label{eq:app_toe4_smallness}
		\varepsilon\,\|\mathbb{T}_M\|<1.
	\end{equation}
	Then $\mathbb{D}_M+\varepsilon \mathbb{C}_M:\mathbb{X}_M\to\mathbb{X}_M^*$ is boundedly invertible and the unique solution of
	\eqref{eq:app_toe4_block_system} admits the convergent expansion
	\begin{equation}
		\label{eq:app_toe4_series_solution}
		\mathbf{V}^{[M]}
		=
		\sum_{j=0}^{\infty}(-\varepsilon)^j\,\mathbb{T}_M^j\,\mathbb{D}_M^{-1}\mathbf{F}^{[M]}
		\quad\text{in }\mathbb{X}_M.
	\end{equation}
	Moreover, the remainder after $N$ terms is controlled by the geometric-series bound
	\begin{equation}
		\label{eq:app_toe4_remainder_bound}
		\Big\|
		\mathbf{V}^{[M]}-\sum_{j=0}^{N}(-\varepsilon)^j\mathbb{T}_M^j\mathbb{D}_M^{-1}\mathbf{F}^{[M]}
		\Big\|_{\mathbb{X}_M}
		\le
		\frac{(\varepsilon\|\mathbb{T}_M\|)^{N+1}}{1-\varepsilon\|\mathbb{T}_M\|}
		\,
		\|\mathbb{D}_M^{-1}\mathbf{F}^{[M]}\|_{\mathbb{X}_M}.
	\end{equation}
\end{theorem}

\begin{proof}
	Factor $\mathbb{D}_M+\varepsilon \mathbb{C}_M=\mathbb{D}_M(\mathbb{I}+\varepsilon \mathbb{T}_M)$ with $\mathbb{T}_M$ defined by
	\eqref{eq:app_toe4_T_def}. Since $\mathbb{D}_M$ is invertible, it suffices to invert $\mathbb{I}+\varepsilon \mathbb{T}_M$ on $\mathbb{X}_M$.
	Under \eqref{eq:app_toe4_smallness}, the Neumann series
	$(\mathbb{I}+\varepsilon \mathbb{T}_M)^{-1}=\sum_{j\ge 0}(-\varepsilon)^j\mathbb{T}_M^j$ converges in operator norm on
	$\mathcal{L}(\mathbb{X}_M,\mathbb{X}_M)$, yielding \eqref{eq:app_toe4_series_solution} upon multiplication by $\mathbb{D}_M^{-1}$.
	The remainder estimate \eqref{eq:app_toe4_remainder_bound} is the standard geometric bound for the tail of a norm-convergent Neumann series.
\end{proof}

\begin{remark}[A sufficient smallness condition in terms of $G_{\max}$ and $K_{\max}$]
	A convenient verifiable sufficient condition is
	\[
	\varepsilon\,G_{\max}(\omega;M)\,K_{\max}(\omega;M)<1,
	\]
	which follows from \eqref{eq:app_toe4_T_norm}. This is the precise finite-dimensional analogue of the perturbative criterion used in the
	worked example.
\end{remark}
To formalize the "index shift" mechanism used to identify which Fourier sidebands appear at which order, we introduce an explicit support
notion on $\{-M,\dots,M\}$.

\begin{definition}[Discrete support]
	\label{def:app_support}
	For $\mathbf{V}\in\mathbb{X}_M$ define its discrete support
	\[
	\supp(\mathbf{V})
	:=
	\{m\in\{-M,\dots,M\}:\ \mathbf{v}_m\neq 0\}.
	\]
	For a subset $S\subset\{-M,\dots,M\}$, define its $m_0$-neighbor expansion (one-step reachability)
	\[
	\mathcal{N}(S)
	:=
	\big(S+m_0\big)\cup\big(S-m_0\big)
	\quad\cap\ \{-M,\dots,M\},
	\]
	and inductively $\mathcal{N}^0(S):=S$, $\mathcal{N}^{j+1}(S):=\mathcal{N}(\mathcal{N}^j(S))$.
\end{definition}

\begin{lemma}[Support propagation for the coupling operator $\mathbb{C}_M$]
	\label{lem:app_support_C}
	Let $\mathbb{C}_M$ satisfy \eqref{eq:app_toe4_C_def}. Then for any $\mathbf{V}\in\mathbb{X}_M$,
	\[
	\supp(\mathbb{C}_M\mathbf{V})
	\subseteq
	\mathcal{N}(\supp(\mathbf{V})).
	\]
\end{lemma}

\begin{proof}
	Fix $m\in\{-M,\dots,M\}$. By \eqref{eq:app_toe4_C_def},
	\[
	(\mathbb{C}_M\mathbf{V})_m = \mathcal{K}_m^-\,\mathbf{v}_{m-m_0}+\mathcal{K}_m^+\,\mathbf{v}_{m+m_0}.
	\]
	If $m\notin \mathcal{N}(\supp(\mathbf{V}))$, then $m-m_0\notin\supp(\mathbf{V})$ and $m+m_0\notin\supp(\mathbf{V})$, hence
	$\mathbf{v}_{m-m_0}=\mathbf{v}_{m+m_0}=0$ and therefore $(\mathbb{C}_M\mathbf{V})_m=0$.
	Thus, no index outside $\mathcal{N}(\supp(\mathbf{V}))$ can belong to $\supp(\mathbb{C}_M\mathbf{V})$.
\end{proof}

\begin{lemma}[Support propagation for $\mathbb{T}_M=\mathbb{D}_M^{-1}\mathbb{C}_M$]
	\label{lem:app_support_T}
	Let $\mathbb{T}_M=\mathbb{D}_M^{-1}\mathbb{C}_M$. Then for any $\mathbf{V}\in\mathbb{X}_M$,
	\[
	\supp(\mathbb{T}_M\mathbf{V})
	\subseteq
	\mathcal{N}(\supp(\mathbf{V})).
	\]
	Moreover, for each integer $j\ge 1$,
	\begin{equation}
		\label{eq:app_support_T_power}
		\supp(\mathbb{T}_M^j\mathbf{V})
		\subseteq
		\mathcal{N}^j(\supp(\mathbf{V})).
	\end{equation}
\end{lemma}

\begin{proof}
	Since $\mathbb{D}_M^{-1}$ is block diagonal, it does not create new indices: if $(\mathbb{C}_M\mathbf{V})_m=0$, then also
	$(\mathbb{D}_M^{-1}\mathbb{C}_M\mathbf{V})_m=0$. Thus,
	$\supp(\mathbb{T}_M\mathbf{V})\subseteq \supp(\mathbb{C}_M\mathbf{V})\subseteq \mathcal{N}(\supp(\mathbf{V}))$ by Lemma~\ref{lem:app_support_C}.
	The power statement \eqref{eq:app_support_T_power} follows by induction: apply the one-step bound repeatedly to $\mathbb{T}_M^{j-1}\mathbf{V}$.
\end{proof}

\begin{lemma}[Reachability characterization from a single mode]
	\label{lem:app_reachability}
	Let $S_0=\{0\}$ and define $S_j:=\mathcal{N}^j(S_0)$.
	Then $m\in S_j$ if and only if there exists an integer $s$ such that
	\[
	m = s\,m_0,\qquad |s|\le j,\qquad s\equiv j\ (\mathrm{mod}\ 2),
	\]
	and $|m|\le M$. In particular, $m=\pm m_0\in S_1$ and $0\in S_2$, while $0\notin S_1$.
\end{lemma}

\begin{proof}
	A single neighbor step adds or subtracts $m_0$. Thus, after $j$ steps, the index is the sum of $j$ increments each equal to $\pm m_0$,
	hence $m=s m_0$ where $s$ is the difference between the number of $+m_0$ and $-m_0$ steps. Therefore $|s|\le j$ and $s\equiv j\pmod 2$.
	Conversely, any such $s$ can be realized by choosing $(j+s)/2$ plus steps and $(j-s)/2$ minus steps. Truncation enforces $|m|\le M$.
\end{proof}
\noindent Assume the forcing is supported only in the mean mode:
\[
\mathbf{F}^{[M]}=(0,\ldots,0,\hat{\mathbf{f}}_0,0,\ldots,0)^{\mathsf{T}}
\quad\Longrightarrow\quad
\supp(\mathbf{F}^{[M]})=\{0\}.
\]
Set
\[
\mathbf{W}:=\mathbb{D}_M^{-1}\mathbf{F}^{[M]}
\quad\text{so that}\quad
\mathbf{W}_0=\mathcal{L}_\omega(\kappa_0)^{-1}\hat{\mathbf{f}}_0,\ \ \mathbf{W}_m=0\ (m\neq 0).
\]
Hence $\supp(\mathbf{W})=\{0\}$.

\begin{proposition}[Support of the Neumann-series terms]
	\label{prop:app_support_series_terms}
	For each $j\ge 0$,
	\[
	\supp\big(\mathbb{T}_M^j\mathbf{W}\big)\subseteq S_j=\mathcal{N}^j(\{0\}).
	\]
	In particular, $\mathbb{T}_M^0\mathbf{W}$ is supported only at $m=0$, $\mathbb{T}_M^1\mathbf{W}$ is supported only at $m=\pm m_0$,
	and $(\mathbb{T}_M^1\mathbf{W})_0=0$.
\end{proposition}

\begin{proof}
	The inclusion follows directly from Lemma~\ref{lem:app_support_T} with $\supp(\mathbf{W})=\{0\}$.
	The explicit identifications for $j=0,1$ follow from Lemma~\ref{lem:app_reachability}.
\end{proof}

\begin{theorem}[First-sideband formula and absence of an $O(\varepsilon)$ correction to the mean mode]
	\label{thm:app_first_sideband}
	Assume \eqref{eq:app_toe4_smallness} and that $\mathbf{F}^{[M]}$ is supported only at $m=0$.
	Then the Neumann expansion \eqref{eq:app_toe4_series_solution} implies:
	\begin{enumerate}
		\item (Mean mode) The $m=0$ component satisfies
		\begin{equation}
			\label{eq:app_u0_Oe2}
			\mathbf{v}^{[M]}_0
			=
			\mathcal{L}_\omega(\kappa_0)^{-1}\hat{\mathbf{f}}_0
			+
			\mathcal{O}(\varepsilon^2)\quad\text{in }X_{\kappa_0}.
		\end{equation}
		\item (First sidebands) The components $m=\pm m_0$ satisfy
		\begin{equation}
			\label{eq:app_sideband_Oe}
			\mathbf{v}^{[M]}_{\pm m_0}
			=
			-\varepsilon\,(\mathbb{T}_M\mathbf{W})_{\pm m_0}+\mathcal{O}(\varepsilon^2)
			=
			-\varepsilon\,\mathcal{L}_\omega(\kappa_{\pm m_0})^{-1}\,(\mathbb{C}_M\mathbf{W})_{\pm m_0}
			+\mathcal{O}(\varepsilon^2).
		\end{equation}
		If $\mathbb{C}_M$ corresponds to symmetric nearest-neighbor coupling with coefficient $(i/2)$ (as in the main text), then
		\begin{equation}
			\label{eq:app_sideband_explicit}
			\mathbf{v}^{[M]}_{\pm m_0}
			=
			-\frac{i\varepsilon}{2}\,
			\mathcal{L}_\omega(\kappa_{\pm m_0})^{-1}\,
			\mathcal{K}_\omega(\kappa_0\to\kappa_{\pm m_0})\,
			\mathbf{v}^{[M]}_0
			+
			\mathcal{O}(\varepsilon^2).
		\end{equation}
	\end{enumerate}
	Moreover, the $\mathcal{O}(\varepsilon^2)$ remainders can be bounded quantitatively using \eqref{eq:app_toe4_remainder_bound}.
\end{theorem}

\begin{proof}
	Start from \eqref{eq:app_toe4_series_solution} and write
	\[
	\mathbf{V}^{[M]}=\mathbf{W}-\varepsilon\,\mathbb{T}_M\mathbf{W}+\varepsilon^2\,\mathbb{T}_M^2\mathbf{W}+\cdots.
	\]
	By Proposition~\ref{prop:app_support_series_terms}, $(\mathbb{T}_M\mathbf{W})_0=0$. Therefore the $m=0$ component satisfies
	\[
	\mathbf{v}^{[M]}_0 = \mathbf{W}_0 + \varepsilon^2(\mathbb{T}_M^2\mathbf{W})_0+\varepsilon^3(\mathbb{T}_M^3\mathbf{W})_0+\cdots,
	\]
	which proves \eqref{eq:app_u0_Oe2}.
	Similarly, again by Proposition~\ref{prop:app_support_series_terms}, the first time an $\pm m_0$ component can appear is at $j=1$, giving
	\eqref{eq:app_sideband_Oe}. To obtain the explicit form \eqref{eq:app_sideband_explicit}, substitute the particular structure of $\mathbb{C}_M$:
	for $m= m_0$, the only contributing source index is $m-m_0=0$ at leading order, so $(\mathbb{C}_M\mathbf{W})_{m_0}=\mathcal{K}_{m_0}^-\,\mathbf{W}_0$.
	In the symmetric case, $\mathcal{K}_{m_0}^-=(i/2)\mathcal{K}_\omega(\kappa_0\to\kappa_{m_0})$. The $m=-m_0$ case is analogous.
	Finally, the quantitative remainder bounds follow by applying \eqref{eq:app_toe4_remainder_bound} with $N=1$ (for sidebands) and $N=0$ (for the mean mode),
	and then projecting to the relevant components.
\end{proof}

\begin{corollary}[A conservative sideband norm bound]
	\label{cor:app_sideband_norm}
	Under the hypotheses of Theorem~\ref{thm:app_first_sideband}, the first sidebands satisfy
	\[
	\|\mathbf{v}^{[M]}_{\pm m_0}\|_{X_{\kappa_{\pm m_0}}}
	\le
	\varepsilon\,
	\|\mathcal{L}_\omega(\kappa_{\pm m_0})^{-1}\|\,
	\|\mathcal{K}_{\pm}\|\,
	\|\mathcal{L}_\omega(\kappa_{0})^{-1}\|\,
	\|\hat{\mathbf{f}}_0\|_{X_{\kappa_0}^*}
	+\mathcal{O}(\varepsilon^2),
	\]
	where $\|\mathcal{K}_{\pm}\|$ denotes the appropriate coupling-block norm (e.g.\ $\tfrac12\|\mathcal{K}_\omega(\kappa_0\to\kappa_{\pm m_0})\|$ in the symmetric case).
\end{corollary}

\begin{proof}
	Apply the triangle inequality to \eqref{eq:app_sideband_explicit}, use the operator-norm bound, and substitute
	$\|\mathbf{v}^{[M]}_0\|\le \|\mathcal{L}_\omega(\kappa_0)^{-1}\|\|\hat{\mathbf{f}}_0\|+\mathcal{O}(\varepsilon^2)$ from \eqref{eq:app_u0_Oe2}.
\end{proof}
\noindent It is sometimes conceptually useful to interpret $\{-M,\dots,M\}$ as the vertex set of a directed graph with edges
$m\to m\pm m_0$ (whenever the target vertex lies within $\{-M,\dots,M\}$). The operator $\mathbb{T}_M$ propagates support along these edges.
Thus, $\mathbb{T}_M^j\mathbf{W}$ is supported on the set of vertices reachable from $0$ in exactly $j$ steps, i.e.\ on $S_j$.
This viewpoint yields the immediate higher-sideband rule:
\[
\mathbf{v}^{[M]}_{q m_0}=\mathcal{O}(\varepsilon^{|q|})
\quad\text{(symmetric coupling)},
\qquad
\mathbf{v}^{[M]}_{q m_0}=\mathcal{O}(\varepsilon^{q})
\quad\text{for }q\ge 0\text{ (one-sided coupling)},
\]
with prefactors given by products of diagonal resolvents and coupling blocks along admissible paths.
In particular, the "first return" to the mean mode requires an even number of steps (a closed walk on the graph), which is the structural
reason the mean-mode correction begins at $\mathcal{O}(\varepsilon^2)$.

\begin{remark}[What this appendix result upgrades relative to the main-text sketch]
	The main-text proof sketch relies on informal index shifting. Here the support-propagation lemmas provide an explicit algebraic statement on
	$\supp(\mathbb{T}_M^j\mathbf{W})$ and the remainder bound \eqref{eq:app_toe4_remainder_bound} supplies quantitative control of truncation in the Neumann expansion.
	Together they yield a fully rigorous and reproducible justification of the leading-order sideband formulas used in Worked Example~IV.
\end{remark}

\end{document}